\documentclass{article}
\usepackage{geometry}
\usepackage{indentfirst}  

\usepackage{amssymb,amsfonts,amscd}
\usepackage[all,arc]{xy}
\usepackage{enumerate}
\usepackage{graphicx}
\usepackage{mathrsfs}
\usepackage{amsmath}
\usepackage{amsfonts}
\usepackage[utf8]{inputenc}
\usepackage{tikz-cd}

\usepackage{yhmath}

\usepackage{amsmath}
\usepackage{amsthm}
\usepackage{amsfonts}
\usepackage{mathrsfs}
\usepackage{amssymb}
\usepackage{bbm}

\usepackage{tikz}
\usetikzlibrary{cd,matrix,arrows,decorations.pathmorphing}

\usepackage{enumerate}

\usepackage{hyperref}
\hypersetup{colorlinks=true, linkcolor=black, filecolor=black, urlcolor=black, citecolor=black}

\def\calA{\mathcal{A}}

\def\calC{\mathcal{C}}
\def\calD{\mathcal{D}}
\def\calE{\mathcal{E}}
\def\calF{\mathcal{F}}
\def\calG{\mathcal{G}}
\def\calH{\mathcal{H}}

\def\calL{\mathcal{L}}
\def\calM{\mathcal{M}}

\def\calO{\mathcal{O}}
\def\calP{\mathcal{P}}

\def\calX{\mathcal{X}}

\def\calZ{\mathcal{Z}}

\def\frb{\mathfrak{b}}

\def\frg{\mathfrak{g}}
\def\frh{\mathfrak{h}}

\def\frm{\mathfrak{m}}
\def\frn{\mathfrak{n}}

\def\ffrb{\mathfrak{B}}

\def\ffrm{\mathfrak{M}}

\def\ffru{\mathfrak{U}}

\def\ffrx{\mathfrak{X}}

\def\bbA{\mathbb{A}}
\def\bbB{\mathbb{B}}
\def\bbC{\mathbb{C}}

\def\bbF{\mathbb{F}}
\def\bbG{\mathbb{G}}
\def\bbH{\mathbb{H}}

\def\bbN{\mathbb{N}}

\def\bbP{\mathbb{P}}
\def\bbQ{\mathbb{Q}}
\def\bbR{\mathbb{R}}

\def\bbT{\mathbb{T}}

\def\bbZ{\mathbb{Z}}

\DeclareFontFamily{U}{wncy}{}
\DeclareFontShape{U}{wncy}{m}{n}{<->wncyr10}{}
\DeclareSymbolFont{mcy}{U}{wncy}{m}{n}
\DeclareMathSymbol{\Sha}{\mathord}{mcy}{"58} 

\def\vphi{\varphi}

\def\bs{\backslash}

\DeclareMathOperator{\ox}{\otimes}

\def\comp{\circ}
\def\inj{\hookrightarrow}
\def\surj{\twoheadrightarrow}
\def\isom{\cong}
\def\aisom{\xrightarrow{\sim}}
\def\ot{\leftarrow}

\DeclareMathOperator{\im}{im}

\DeclareMathOperator{\coker}{coker}

\DeclareMathOperator{\id}{id}
\DeclareMathOperator{\op}{op}
\DeclareMathOperator{\pr}{pr}

\newcommand{\dlim}{\varinjlim}
\newcommand{\ilim}{\varprojlim}

\DeclareMathOperator{\Fil}{Fil}

\DeclareMathOperator{\Hom}{Hom}

\DeclareMathOperator{\Ext}{Ext}

\DeclareMathOperator{\Ind}{Ind}
\DeclareMathOperator{\Res}{Res}

\DeclareMathOperator{\Gal}{Gal}

\DeclareMathOperator{\Tr}{Tr}

\renewcommand{\mod}{{\text{ }\mathrm{mod}\text{ }}}
\DeclareMathOperator{\GL}{GL}

\DeclareMathOperator{\PGL}{PGL}

\DeclareMathOperator{\End}{End}

\DeclareMathOperator{\Mat}{Mat}

\DeclareMathOperator{\Lie}{Lie}

\DeclareMathOperator{\ord}{ord}

\DeclareMathOperator{\gl}{\mathfrak{gl}}

\DeclareMathOperator{\Rep}{Rep}
\DeclareMathOperator{\ev}{ev}

\def\bb1{\mathbbm{1}}

\DeclareMathOperator{\DR}{DR}

\DeclareMathOperator{\gr}{gr}
\DeclareMathOperator{\ur}{ur}

\DeclareMathOperator{\cont}{cont}

\DeclareMathOperator{\HT}{HT}
\DeclareMathOperator{\dR}{dR}
\DeclareMathOperator{\cris}{cris}
\DeclareMathOperator{\st}{st}

\DeclareMathOperator{\Spf}{Spf}

\DeclareMathOperator{\Spa}{Spa}

\def\LT{\mathrm{LT}}

\def\et{{\text{\'et}} }

\def\an{\mathrm{an}}
\def\lan{\mathrm{la}}
\def\sm{\mathrm{sm}}

\def\alg{\mathrm{alg}}
\def\lalg{\mathrm{lalg}}

\def\hat{\widehat}
\def\bar{\overline}
\def\tilde{\widetilde}

\def\Coh{\mathrm{Coh}}

\def\St{\mathrm{St}}

\def\rig{\mathrm{rig}}

\def\-{\text{-}}

\def\pst{\mathrm{pst}}
\def\WD{\mathrm{WD}}
\def\ss{\mathrm{ss}}

\def\pro{\mathrm{pro}}

\def\-{\text{-}}

\def\fl{\mathscr{F}\ell}

\def\Sen{\mathrm{Sen}}

\def\Sym{\mathrm{Sym}}

\def\GM{\mathrm{GM}}

\def\Ig{\mathrm{Ig}}
\def\Nilp{\mathrm{Nilp}}
\def\KS{\mathrm{KS}}
\def\Dr{\mathrm{Dr}}

\def\ORD{\mathrm{ORD}}
\def\SS{\mathrm{SS}}

\def\Nrd{\mathrm{Nrd}}

\def\tr{\mathrm{tr}}

\def\VB{\mathrm{VB}}

\def\ul{\underline}

\def\ov{\stackrel}

\newcommand{\Q}{\mathbb{Q}}
\newcommand{\Z}{\mathbb{Z}}
\newcommand{\N}{\mathbb{N}}

\newcommand{\X}{\mathcal{X}}
\newcommand{\B}{\mathbb{B}}
\newcommand{\ra}{\rightarrow}

\newcommand{\T}{\mathbb{T}}

\theoremstyle{definition}
\newtheorem{thm}{Theorem}[subsection]

\newtheorem{prop}[thm]{Proposition}

\newtheorem{lem}[thm]{Lemma}
\newtheorem{conj}[thm]{Conjecture}

\newtheorem{lemma}[thm]{Lemma}
\newtheorem{proposition}[thm]{Proposition}
\newtheorem{theorem}[thm]{Theorem}
\newtheorem{corollary}[thm]{Corollary}
\newtheorem{remark}[thm]{Remark}

\theoremstyle{definition}
\newtheorem{defn}[thm]{Definition}

\newtheorem{definition}[thm]{Definition}

\makeatletter
\let\c@equation\c@thm
\makeatother
\numberwithin{equation}{section}

\title{Locally analytic vectors in the completed cohomology of unitary Shimura curves}

\author{Tian Qiu, Benchao Su}
\date{}

\begin{document}	

\maketitle
\begin{abstract}
We use the methods introduced by Lue Pan to study the locally analytic vectors in the completed cohomology of unitary Shimura curves. As an  application, we prove a classicality result on two-dimensional regular $\sigma$-de Rham representations of $\text{Gal}(\bar L/L)$ appearing in the locally $\sigma$-analytic vectors of the completed cohomology, where $L$ is a finite extension of $\Q_p$ and $\sigma:L\hookrightarrow E$ is an embedding of $L$ into a sufficiently large finite extension $E$ of $\Q_p$. We also prove that if a two-dimensional representation of $\text{Gal}(\bar L/L)$ appears in the locally $\sigma$-algebraic vectors of the completed cohomology then it is $\sigma$-de Rham. Finally, we give a geometric realization of some locally $\sigma$-analytic representations of $\GL_2(L)$. This realization has some applications to the $p$-adic local Langlands program for $\GL_2(L)$, including a locality theorem for Galois representations arising from classical automorphic forms, an admissibility result for coherent cohomology of Drinfeld towers of dimension $1$ over $L$ of arbitrary level, and some special cases of the Breuil's locally analytic Ext$^1$-conjecture for $\GL_2(L)$.
\end{abstract}

\tableofcontents

\section{Introduction}
\subsection{Statement of main results}
Let $L$ be a finite extension of $\bbQ_p$ and let $E$ be a sufficiently large finite extension of $\bbQ_p$. The $p$-adic local Langlands program for the group $\GL_2(L)$ aims to construct a correspondence between certain $2$-dimensional $p$-adic representations of $\Gal(\bar L/L)$ over $E$ and certain $p$-adic representations of $\GL_2(L)$ over $E$, suitably matching properties on both sides. Although such a correspondence is well-established when $L=\bbQ_p$ (cf.\cite{breuil2010emerging, colmez2010representations, emerton2011local}), the situation for $L\neq\bbQ_p$ remains largely mysterious. 

However, when the $p$-adic representation $\rho_p$ of $\Gal(\bar L/L)$ comes from a global Galois representation $\rho$ that appears in the completed cohomology \cite{Eme06} of certain unitary Shimura curves (which we will define later), one can associate to $\rho$ an admissible Banach representation $\Pi(\rho)$ of $\GL_2(L)$. When $L=\bbQ_p$ and the curve is the modular curve, the remarkable papers of Lue Pan \cite{Pan22, PanII} establish a result regarding the de Rhamness and the classicality of $\rho$, and give a detailed description of the locally analytic vectors $\Pi(\rho)^{\lan}$ inside $\Pi(\rho)$ when $\rho$ arising from classical modular forms, whose proofs are largely independent of the known $p$-adic local Langlands for $\GL_2(\bbQ_p)$. In this paper, we aim to generalize Pan's results to the case of unitary Shimura curves, and give some applications to the still conjectural $p$-adic local Langlands for the group $\GL_2(L)$.

First, we introduce some notation. Let $F$ be a CM field, with $F^+$ its maximal totally real subfield. Suppose $v$ is a prime of $F^+$ over $p$ that splits in $F$ into two primes, $w$ and $w^c$. Assume the local completions $F_w \cong F^+_v$ are isomorphic to $L$. Let $(G,X)$ be the Shimura datum, with reflex field $F$, such that for a sufficiently small open compact subgroup $K\subset G(\bbA^\infty)$, we get the unitary Shimura curve $X_{K}$ considered in, for example, \cite{Car83,HT01,Din17}. In particular, $G$ is a unitary similitude group over $\bbQ$ such that $G(\bbQ_p)$ has a factor isomorphic to $\GL_2(L)$. We write $G(\bbA^\infty)\isom G_vG^v$ with $G_v:=\GL_2(L)$ and $G^v$ is the product of $G(\bbA^{\infty,p})$ together with factors at $p$ but away from $v$. Fix a sufficiently small open compact subgroup $K^v\subset G^v$. Define 
\begin{align*}
    \tilde{H}^1(K^v,E):=(\ilim_n\dlim_{K_v\subset \GL_2(L)}H^1(X_{K^vK_v}(\bbC),\calO_E/p^n\calO_E))\ox_{\calO_E}E
\end{align*}
This is a $p$-adic Banach space over $E$, equipped with natural continuous actions of $\GL_2(L)$ and the Galois group $\Gal_F$. It is also equipped with an action of the abstract spherical Hecke algebra $\mathbb{T}^S:=\bbZ[K^S\bs G(\bbA^{\infty,S})/K^S]$, where $S$ is a finite set of primes containing $p$, such that $K^v=K^{S}K^v_S$ with $K^S\subset G(\bbA^{\infty,S})$ a product of hyperspecial subgroups. 

We also need to introduce certain conditions on $p$-adic representations of $\Gal_L$ and $p$-adic representations of $\GL_2(L)$ following \cite{Din14}. Let $V$ be a $2$-dimensional continuous $E$-linear representation of $\Gal_L$. Let $B_{\dR}$ denote the de Rham period ring constructed from a fixed algebraic closure $\bar L$ of $L$. Then the usual filtered $L$-vector space $D_{\dR}(V):=(B_{\dR}\ox_{\bbQ_p}V)^{\Gal_L}$ carries an action of $L\ox_{\bbQ_p}E$, so that we can decompose $D_{\dR}(V)\isom \bigoplus_{\sigma\in \Sigma}D_{\dR,\sigma}(V)$ with respect to the decomposition $L\ox_{\bbQ_p}E\isom \prod_{\sigma\in \Sigma}E$, where $\Sigma$ is the set of $\bbQ_p$-embeddings of $L$ into $E$. We assume $E$ contains all Galois conjugates of $L$, so that $|\Sigma|=[L:\bbQ_p]$. For an embedding $\sigma\in \Sigma$, we say $V$ is \emph{$\sigma$-de Rham} if $\dim_ED_{\dR,\sigma}(V)=2$. In general, if $J\subset\Sigma$ is a subset, we say $V$ is \emph{$J$-de Rham} if $V$ is $\sigma$-de Rham for each $\sigma\in J$. The Hodge filtration on $D_{\dR}(V)$ also decomposes and induces natural filtration on $D_{\dR,\sigma}(V)$ for $\sigma\in \Sigma$. We call the filtration on $D_{\dR,\sigma}(V)$ the \emph{$\sigma$-Hodge filtration}. Similarly we can also define the \emph{$J$-Hodge--Tate representation} by demanding $\dim_E D_{\HT,\sigma}(V)=2$ for each $\sigma\in J$. 

For an $E$-Banach space $\Pi$ with a continuous action of $\mathrm{GL}_2(L)$, let $\Pi^{\mathrm{la}}$ be the subspace of locally $\bbQ_p$-analytic vectors in $\Pi$. It has an action of $\mathfrak{gl}_2(L)\otimes_{\bbQ_p}E\cong \prod_{\sigma\in \Sigma}\mathfrak{gl}_2(E)$, where $\mathfrak{gl}_2(L)$ is the Lie algebra of $\mathrm{GL}_2(L)$. For $J$ a subset of embeddings of $L$ into $E$, let $\Pi^{J\-\mathrm{la}}$ be the subspace of $\Pi^\mathrm{la}$ on which the action of $\prod_{\sigma\in \Sigma}\mathfrak{gl}_2(E)$ factors through $\prod_{\sigma\in J}\mathfrak{gl}_2(E)$, and $\Pi^{J\-\mathrm{lalg}}$ be the subspace of $\Pi^{J\-\mathrm{la}}$ generated by finite-dimensional $\prod_{\sigma\in J}\mathfrak{gl}_2(E)$-stable subspaces. For $J_1,J_2$ two disjoint subsets of embeddings of $L$ into $E$, let $\Pi^{J_1\-\mathrm{la},J_2\-\mathrm{lalg}}$ be the subspace of $\Pi^{(J_1\cup J_2)\-\mathrm{la}}$ generated by finite-dimensional $\prod_{\sigma\in J_2}\mathfrak{gl}_2(E)$-stable subspaces. In particular, we get $\GL_2(L)$-equivariant inclusions $\Pi^{J_1\-\mathrm{la},J_2\-\mathrm{lalg}}\subset \Pi^{(J_1\cup J_2)\-\mathrm{la}}\subset \Pi^{\lan}$. For $\sigma\in \Sigma$ an embedding, we denote by $\sigma^c:=\Sigma\bs\{\sigma\}$ the set of other embeddings. Here we note that if $\Pi$ is an admissible Banach representation of $\GL_2(L)$, then $\Pi^{\lan}\neq 0$ by \cite{ST03}, but $\Pi^{J\-\lan}$ could be zero if $J\neq\Sigma$. (An easy example is that when $J=\varnothing$ then $\Pi^{\varnothing\-\lan}$ is the smooth vectors in $\Pi$.)

Now let $\rho$ be a $2$-dimensional continuous $E$-linear representation of $\Gal_F$. We can associate to $\rho$ a unitary Banach representation of $\GL_2(L)$ over $E$:
\begin{align*}
    \Pi(\rho):=\Hom_{E[\Gal_F]}(\rho,\tilde{H}^1(K^v,E)).
\end{align*}
The following is our main result. (We will discuss its applications to the $p$-adic local Langlands program in \S 1.3.) For simplicity, we assume there is no extra multiplicities coming from the tame part. In particular, $\Pi(\rho)^{\lalg}$ is irreducible as $\GL_2(L)$-representations.
\begin{theorem}\label{1}
Let $\rho$ be a $2$-dimensional continuous $E$-linear absolutely irreducible representation of $\Gal_F$. Let $\sigma\in \Sigma$ be an embedding. Suppose that 
\begin{enumerate}[(1)]
    \item $\Pi(\rho)^{\sigma\-\lan,\sigma^c\-\lalg}\neq 0$,
    \item $\rho|_{\Gal_L}$ is $\sigma$-de Rham of regular $\sigma$-Hodge--Tate weights.
\end{enumerate}
Then $\rho$ arises from a classical automorphic form.
\end{theorem}

For the condition (1),  conjecturally this is equivalent to $\rho$ is $\sigma^c$-de Rham of different $\sigma^c$-Hodge--Tate weights. In fact, we make the following conjecture. We are grateful to Yiwen Ding for insightful discussions on this conjecture.
\begin{conj}\label{conj:sigmaalgsigmadR}
    Let $E$ be a finite extension of $L$  containing all embeddings of $L$ into $\overline{\Q}_p$ and 
		$$\rho:\mathrm{Gal}_F\ra \mathrm{GL}_2(E)$$ be a two-dimensional continuous absolutely irreducible representation of $\mathrm{Gal}_F$. Assume that $\Pi(\rho)\neq 0$. Then for any subset $J$ of embeddings of $L$ into $E$,  $\Pi(\rho)^{J^c\textup{-la}, J\textup{-lalg}}\neq 0$ if and only if  $\rho$ is $J$-de Rham of different $J$-Hodge--Tate weight.
\end{conj}

If this is true, combining this with Theorem \ref{1} one can prove the classicality for every two-dimensional continuous absolutely irreducible representation of $\mathrm{Gal}_F$ appearing the completed cohomology which is de Rham of different Hodge--Tate weights. For this conjecture, we can prove the ``only if" direction.

\begin{theorem}\label{2}
Let $\rho:\mathrm{Gal}_F\rightarrow \mathrm{GL}_2(E)$ be a two-dimensional continuous absolutely irreducible representation of $\mathrm{Gal}_F$. Suppose that $\Pi(\rho)^{\sigma^c\-\lan , \sigma\text{-}\mathrm{lalg}}\neq 0$ for some embedding $\sigma\in \Sigma$. Then $\rho|_{\mathrm{Gal}_L}$ is $\sigma$-de Rham of different $\sigma$-Hodge--Tate weights.
\end{theorem}

\subsection{Overview of the proof of main results}
Our proof is motivated by Pan's work. The starting point is to identify locally analytic vectors in the completed cohomology group $\tilde{H}^1(K^v,E)^{\lan}$ as the cohomology of the sheaf of locally analytic sections on the completed unitary Shimura curve. Let $C$ be the completion of an algebraic closure of $\Q_p$. Fix an embedding $\iota:L\to C$. Let $\calX_K$ be the adic space associated to $X_K\times_{F}C$, where $C$ is a viewed as an $F$-algebra via $F\to F_w\isom L\ov{\iota}\to C$. By the work of Scholze \cite{Sch15}, there is a perfectoid space $\calX_{K^v}$ over $C$ such that $\calX_{K^v}\sim\ilim_{K_v\subset \GL_2(L)}\calX_{K^vK_v}$, with Hodge--Tate period map $\pi_{\HT}:\calX_{K^v}\to \fl$, where $\fl$ is the adic space associated to $\bbP^1_L\times_{L,\iota}C$. The primitive comparison theorem \cite{Sch13} gives
\begin{align*}
    \tilde{H}^1(K^v,C)\cong H^1(\mathcal{X}_{K^v},\mathcal{O}_{\mathcal{X}_{K^v}})\cong H^1(\mathscr{F}\ell,\pi_{\mathrm{HT}*}\mathcal{O}_{\mathcal{X}_{K^v}}).
\end{align*}
Let $\mathcal{O}_{K^v}:=\pi_{\mathrm{HT}*}\mathcal{O}_{\mathcal{X}_{K^v}}$. Using the geometric Sen theory \cite{RC1}, we show that if $J$ is a subset of embeddings of $L$ into $C$ that contains $\iota$, then (Theorem \ref{thm:JlacommutesH1})
\begin{align*}
    \tilde{H}^1(K^v,E)^{J\-\lan}\cong H^1(\mathscr{F}\ell,\calO_{K^v}^{J\-\lan}).
\end{align*} 
Therefore, in order to study the locally $J$-analytic vectors in the completed cohomology group, it suffices to consider cohomology of the sheaf $\calO_{K^v}^{J\-\lan}$. It turns out that there is a tensor product decomposition of the sheaf
\begin{align*}
    \calO_{K^v}^{J\-\lan}\isom \hat\ox_{\calO_{K^v}^{\sm},\eta\in J}    \calO_{K^v}^{\eta\-\lan}
\end{align*}
where $\calO_{K^v}^{\sm}:=\calO_{K^v}^{\varnothing\-\lan}$ is the subsheaf of $\calO_{K^v}^{\lan}$ consisting of smooth sections. See Proposition \ref{prop:RVBiotala} for more information about the completed tensor product over $\calO_{K^v}^{\sm}$. Roughly speaking, it is given by taking completed tensor product over all finite levels $\calO_{K^vK_v}$ and then taking colimits on $K_v$.
\begin{itemize}
    \item When $\eta=\iota$, the structure of $\calO_{K^v}^{\iota\-\lan}$ is very similar to the  modular curve case considered in \cite{Pan22}.
    \item When $\eta\neq\iota$, as the $\eta$-Hodge filtration for the universal abelian scheme over $\calX_{K^v}$ degenerates, so that non-rigorously speaking the structure of $\calO_{K^v}^{\eta\-\lan}$ is very similar to a completed tensor product of $\calO_{K^v}^{\sm}$ with $\calC^{\eta\-\lan}(\frg,C)$ over $C$. Here, $\calC^{\eta\-\lan}(\frg,C)$ is the germ of locally $\eta$-analytic functions on $\GL_2(L)$ and the completed tensor product over $C$ is computed for LB spaces.
\end{itemize}
The geometric Sen theory \cite{RC1, Pil22, camargo2024locallyanalyticcompletedcohomology,boxer2025modularitytheoremsabeliansurfaces} also produces an action $\theta_{\frh}$ of the Cartan subalgebra $\frh=\left(\begin{matrix}*&0\\0&*\end{matrix}  \right)\subset \Lie\GL_2(L)\ox_{L,\iota}C$ on $\calO_{K^v}^{\lan}$. After fixing an integral infinitesimal character, we can decompose the sheaf with respect to this $\theta_{\frh}$-action. For example, we have
\begin{align*}
    \calO_{K^v}^{\lan,\tilde{\chi}_0}\isom \calO_{K^v}^{\lan,(0,0)}\oplus \calO_{K^v}^{\lan,(-1,1)},
\end{align*}
where $\tilde{\chi}_0$ is the trivial infinitesimal character, and $\calO_{K^v}^{\lan,(a,b)}$ denotes the isotypic space for the weight $(a,b)\in \bbZ^2$ of the $\theta_{\frh}$-action. It turns out that there is an intertwining operator 
\begin{align*}
    I:\calO_{K^v}^{\lan,(0,0)}\to \calO_{K^v}^{\lan,(-1,1)}(1),
\end{align*}
where $(1)$ denotes the Tate twist. On $\calO_{K^v}^{\iota\-\lan,(0,0)}$, the construction of $I$ is basically the same as in \cite{PanII}, which are compositions of differential operators along unitary Shimura curves of finite level and differential operators along the flag variety. Using our detailed description of $\calO_{K^v}^{\iota^c\-\lan}$, we can also extend the definition of these differential operators, and hence the intertwining operator, to the whole locally $\bbQ_p$-analytic part $\calO_{K^v}^{\lan,(0,0)}$. Here we note that these constructions can be naturally generalized to the case of regular algebraic weights.

As in \cite{PanII}, we also give a $p$-adic Hodge-theoretic interpretation of $I$ in terms of a geometric version of the Fontaine operator \cite{fontaine2004arithmetique}. A key new point in our proof is that in order to compare the intertwining operator with the Fontaine operator, we need to extend the lifting of $\calO_{K^v}^{\sm,\Gal_L\-\sm}$ (subsheaf of $\calO_{K^v}^{\sm}$ consisting of $\Gal_L$-smooth sections) in $\calO\bbB_{\dR}^+$ to a lifting of $\calO_{K^v}^{\iota^c\-\lan,\Gal_L\-\sm}$ in $\calO\bbB_{\dR}^+$, which is done naturally by the relative de Rham comparison theorem for the universal abelian scheme over $\{\calX_{K^vK_v,\bar L}\}_{K_v\subset \GL_2(L)}$.

Once we have identify the intertwining operator with the Fontaine operator, together with some calculation on the cohomology of the intertwining operator, we can prove Theorem \ref{1}. For Theorem \ref{2}, as $\calO_{K^v}^{\iota^c\-\lan}$ is killed by the intertwining operator from our construction, the Fontaine operator acts trivially on it. This gives our partially de Rham result.

\subsection{Applications to the $p$-adic local Langlands program}
There are also some applications to the $p$-adic local Langlands program for $\GL_2(L)$. 
\subsubsection*{Structure of $\Pi(\rho)^{\sigma\-\lan,\sigma^c\-\lalg}$}
Our first application is to study the structure of $\Pi(\rho)^{\sigma\-\lan,\sigma^c\-\lalg}$ in detail, when $\rho$ arises from a classical automorphic form. As a byproduct, we also show that in this case, $\Pi(\rho)^{\sigma\-\lan,\sigma^c\-\lalg}$ only depends on the local Galois representation $\rho_w:=\rho|_{\Gal_L}$.

For simplicity we assume $\rho_w$ is de Rham of parallel Hodge--Tate weights $0,1$. Here our convention is that the Hodge--Tate weights of the $p$-adic cyclotomic character is $-1$. The case with general regular Hodge--Tate weights can be treated similarly. Starting from $\rho$, it restricts to a local Galois representation $\rho_w$, and then we can associate a Weil--Deligne representation $r_w$ with coefficients in $C$. By the local Langlands correspondence, to $r_w$ we can attach a smooth irreducible representation $\pi_v$ of $\GL_2(L)$ over $C$. Fix an embedding $u:E\to C$ such that $u\comp \sigma=\iota$ as embeddings of $L$ into $C$, and we define 
\begin{align*}
    \pi_\sigma(\rho):=\Pi(\rho)^{\sigma\-\lan}\hat\ox_{E,u}C.
\end{align*}
It turns out that all these structural results we described below about $\pi_\sigma(\rho)$ follows a similar pattern as in the $\GL_2(\bbQ_p)$-case. 

When $\pi_v$ is supercuspidal, we can realize $\pi_\sigma(\rho)$ on the supersingular locus of the unitary Shimura curves, which is related to coverings of Drinfeld's upper half plane (via flip-flopping). Let $\Omega=\bbP^1\bs \bbP^1(L)$ be the Drinfeld's upper half plane, which is viewed as an adic space over $C$ where $L$ acts via $\iota$. It has a tower of \'etale coverings $\{\calM_{\Dr,n}\}_{n\ge 0}$, such that $\calM_{\Dr,n}$ is a $D_L^\times/(1+\varpi^n\calO_{D_L})$-torsor over $\Omega$, where $D_L$ is the non-split quaternion algebra over $L$ and $\varpi$ is a fixed uniformizer of $L$. Let $\tau_v$ be the smooth finite dimensional irreducible representation of $D_L^\times$ over $C$ associated to $\pi_v$ via the classical Jacquet--Langlands correspondence. Up to unramified twists, we may assume $\varpi$ (viewed as an element in the center of $D_L^\times$) acts trivially on $\tau_v$. Then:
\begin{theorem}
There is a $\GL_2(L)$-equivariant isomorphism:
\begin{align*}
    0\to \pi_v\to (\tau_v\ox_C\dlim_n H^1_c(\calM_{\Dr,n},\calO_{\calM_{\Dr,n}}))^{D_{L}^\times}\to \pi_\sigma(\rho)\to 0.
\end{align*}
\end{theorem}
Recall that we have an isomorphism $(\tau_v\ox_C \dlim_n H^1_{\dR,c}(\calM_{\Dr,n}))^{D_L^\times}\isom \pi_v\ox_C r_w$ by \cite{CDN20}. Then the first map in above is induced by the $\sigma$-Hodge filtration of $\rho|_{\Gal_L}$ together with the natural map $H^1_{\dR,c}(\calM_{\Dr,n})\to H^1_c(\calM_{\Dr,n},\calO_{\calM_{\Dr,n}})$. 
\medskip

When $\pi_v$ is a principal series, we can realize $\pi_\sigma(\rho)$ on the ordinary locus of the unitary Shimura curves, which is related to the (non-perfect) Igusa curves. Let $\calX_{K^v,n}^{\text{multi}}$ be a component of the ordinary locus of $\calX_{K^v,n}:=\calX_{K^v(1+\varpi^n \Mat_2(\calO_L))}$, such that they are mapped to the point $\infty=[0:1]\in\fl$ via the infinite level Hodge--Tate period map $\pi_{\HT}$. By restricting the de Rham complex $\calO_{K^v}^{\sm}\to \Omega^{1,\sm}_{K^v}$ (where $\Omega^{1,\sm}_{K^v}:=\pi_{\HT,*}(\dlim_{K_v}\pi_{K_v}^{-1}\Omega^1_{\calX_{K^vK_v}})$ with $\pi_{K_v}:\calX_{K^v}\to \calX_{K^vK_v}$ the natural projection map) to an overconvergent neighbourhood of $\calX_{K^v,n}^{\text{multi}}$, we get an overconvergent $F$-isocrystal and its cohomology calculates the rigid cohomology of the Igusa curves $\Ig_{K^v,n}$ (which is an irreducible component of the special fiber of a suitable integral model of $\calX_{K^v,n}$). We denote this cohomology group by $H^i_{\rig}(\Ig_{K^v,n})$, $i\ge 0$. After taking colimits on $n$, the cohomology group $H^i_{\rig}(\Ig_{K^v,\infty}):= \dlim_n H^i_{\rig}(\Ig_{K^v,n})$ carries a natural smooth admissible action of $\bar B(L)$ and the Hecke algebra $\mathbb{T}^S$. Let $\lambda:\bbT^S\to C$ be the Hecke eigenvalue given by the values of $\rho$ on the Frobenii.
\begin{theorem}
There is a $\GL_2(L)$-equivariant isomorphism:
\begin{align*}
    0\to \pi_v\to \left( \Ind^{\GL_2(L)}_{\bar B(L)}H^1_{\rig}(\Ig_{K^v,\infty})[\lambda] \right)^{\iota\-\lan}\to \pi_\sigma(\rho)\to 0.
\end{align*}
\end{theorem}
Here, $(-)^{\iota\-\lan}$ denotes the locally $\iota$-analytic induction. Also, after semisimplification for the $\bar B(L)$-action, the cohomology group $H^1_{\rig}(\Ig_{K^v,\infty})[\lambda]$ (as a $\bar B(L)$-representation) is exactly given by the Jacquet module $J_{\bar N(L)}(\pi_v)$ of $\pi_v$, as $\bar B(L)$-representations. For the proof, we use \cite[Theorem V.5.4]{HT01}, \cite[Corollary 19]{Joh16}. For simplicity we assume the $\bar B(L)$-action on $J_{\bar N(L)}(\pi_v)$ is already semisimple. Then $\left( \Ind^{\GL_2(L)}_{\bar B(L)}H^1_{\rig}(\Ig_{K^v,\infty})[\lambda] \right)^{\sm}\isom \pi_v^{\oplus 2}$ and the first map also comes from the position of the $\sigma$-Hodge filtration of $\rho|_{\Gal_L}$. Here $(-)^{\sm}$ denotes the smooth induction functor.
\medskip

Finally, when $\pi_v$ is a (twist of) Steinberg representation, both the ordinary locus and the supersingular locus enter into the picture. In this case, the Jacquet--Langlands transfer of $\pi_v$ is a character, and the Jacquet module of $\pi_v$ is the modulus character of $\bar B(L)$ up to twist. Then:
\begin{theorem}
There is a $\GL_2(L)$-equivariant isomorphism:
\begin{align*}
    0\to \St_2^\infty\to [H^1_c(\Omega,\calO_\Omega)  -(\Ind^{\GL_2(L)}_{B(L)}|\cdot|\ox |\cdot|^{-1})^{\iota\-\lan}]\to \pi_\sigma(\rho)\ox \chi\to 0,
\end{align*}
where $\chi$ is some character and $\St_2^\infty$ is the smooth Steinberg representation for $\GL_2(L)$.
\end{theorem}
Here, by some calculation of extension groups, the middle term $H^1_c(\Omega,\calO_\Omega)  -(\Ind^{\GL_2(L)}_{B(L)}|\cdot|\ox |\cdot|^{-1})^{\iota\-\lan}$ is the unique non-split extension of $(\Ind^{\GL_2(L)}_{B(L)}|\cdot|\ox |\cdot|^{-1})^{\iota\-\lan}$ by $H^1_c(\Omega,\calO_\Omega)$.

In fact, all these results can be formulated in a uniform way. Recall that $\calO_{K^v}^{\lan,(0,0)}$ is a subsheaf of $\calO_{K^v}^{\lan}$ such that the action of $\theta_{\frh}$ is trivial. Let $\calO_{K^v}^{\iota\-\lan,(0,0)}\subset \calO_{K^v}^{\lan,(0,0)}$ be the subsheaf consisting of locally $\iota$-analytic sections. It contains $\calO_{K^v}^{\sm}$ as a subsheaf. Then we can suitably extend the differential operator on $\calO_{K^v}^{\sm}$ to $\calO_{K^v}^{\iota\-\lan,(0,0)}$, resulting a de Rham complex 
\begin{align*}
    \DR:=[\calO_{K^v}^{\iota\-\lan,(0,0)}\to \calO_{K^v}^{\iota\-\lan,(0,0)}\ox_{\calO_{K^v}^{\sm}}\Omega^{1,\sm}_{K^v}].
\end{align*}
Let $\DR^{\sm}:=[\calO_{K^v}^{\sm}\to \Omega^{1,\sm}_{K^v}]$ and we denote $\tilde{\pi}:=\bbH^1(\DR)[\lambda]$. Then we have the following.
\begin{theorem}\label{thm:introlocality}
There is a $\GL_2(L)$-equivariant isomorphism:
\begin{align*}
    0\to \pi_v\to \tilde{\pi}\to \pi_\sigma(\rho)\to 0.
\end{align*}
\end{theorem}
Here, the first map comes from 
\begin{align*}
    \pi_v\isom H^0(\Omega^{1,\sm}_{K^v})[\lambda]\to \bbH^1(\DR^{\sm})[\lambda]\to \bbH^1(\DR)[\lambda]=\tilde{\pi}
\end{align*}
with $H^0(\Omega^{1,\sm}_{K^v})[\lambda]\to \bbH^1(\DR^{\sm})$ induced by the Hodge filtration $[0\to \Omega^{1,\sm}_{K^v}]\to \DR^{\sm}$. See Theorem \ref{thm:kerI2} for another presentation of $\pi_\sigma(\rho)$. We can show that $\bbH^1(\DR)[\lambda]$ only depends on $\pi_v$. Therefore, from the above presentation, we see $\pi_\sigma(\rho)$ is determined by $\pi_v$ and the $\sigma$-Hodge filtration of $\rho|_{\Gal_L}$. This gives the following locality result.
\begin{theorem}[Theorem \ref{thm:locality}]
Let $\rho$ be a $2$-dimensional $E$-linear continuous representation of $\Gal_F$, which arises from a classical automorphic form on $G$. Then the $\GL_2(L)$-representation $\pi_\sigma(\rho)$ only depends on the local Galois representation $\rho|_{\Gal_L}$. 
\end{theorem}
Therefore, in this case, we can rename $\pi_\sigma(\rho)$ by $\pi_\sigma(\rho_w)$. From now on until the end of the introduction, we always assume that $\rho$ arises from a classical automorphic form. We also note that the above results can be generalized to the case of arbitary regular algebraic weights.

\subsubsection*{Admissibility of coherent cohomology of Drinfeld towers of dimension $1$ over $L$}
Our second application is that we obtain admissiblilty of certain locally analytic representation of $\GL_2(L)$ defined using the coherent cohomology of Drinfeld towers of dimension $1$ over $L$. Recall that by our definition, $\pi_\sigma(\rho)$ is the base change of $\Pi(\rho)^{\sigma\-\lan}$ from $E$ to $C$. When $\pi_v$ is supercuspidal, we present $\pi_\sigma(\rho)$ as a quotient of $\tilde{\pi}$ by a smooth admissible representation $\pi_v$. As $\tilde{\pi}$ is a locally analytic $\GL_2(L)$-representation defined using coherent cohomology of Drinfeld tower, using the $E$-structure of the Drinfeld tower $\calM_{\Dr,n}$ (after quotient by the action of $\varpi$), we can find a locally analytic $\GL_2(L)$-representation $\tilde{\pi}_E$ over $E$, such that $\tilde{\pi}_E\hat\ox_{E,u}C\isom \tilde{\pi}$. Then we have the following result.
\begin{theorem}[Theorem \ref{thm:admissible}]\label{admissibility}
The locally $\sigma$-analytic representation $\tilde{\pi}_E$ is admissible.
\end{theorem}
We note that the above result also holds for compactly supported cohomology groups of general equivariant line bundles on Drinfeld tower (via translation functors). In fact, the idea is simple once we established the result Theorem \ref{thm:introlocality}. As $\Pi(\rho)^{\sigma\-\lan}$ is a closed subrepresentation of the completed cohomology of the unitary Shimura curve, it is admissible. Hence after handling those $E$-structures, we can deduce the admissibility of $\tilde{\pi}_E$. See \cite[Remark 7.3.5]{PanII} for a related discussion. We also note that in \cite{Patel_Schmidt_Strauch_2019, ardakov2023globalsectionsequivariantline}, they give some coadmissibility results for global sections of some equivariant line bundles on the first covering of $\Omega$, based on explicit calculation of some equivariant $\wideparen\calD$-modules.

\subsubsection*{Wall-crossing and the Breuil's locally analytic $\Ext^1$-conjecture for $\GL_2(L)$}
Our final application is to prove the Breuil's locally analytic $\Ext^1$-conjecture for $\GL_2(L)$ when the Weil--Deligne representation $r_w$ is irreducible. Motivated by the work \cite{Din24}, in \cite{WConDrinfeld} we use the wall-crossing technique to study global sections of equivariant line bundles on Drinfeld tower and prove the Breuil's locally analytic $\Ext^1$-conjecture for $\GL_2(L)$ in the setting of pro-\'etale cohomology of Drinfeld tower \cite{CDN20}. Now we apply the same method in this setting. Recall that the wall-crossing functor for locally analytic representations is the natural generalization of the classical wall-crossing functor for representations of semisimple Lie algebras \cite{JLS22,Din24}. In particular, we can consider the wall-crossing of the sheaf $\calO_{K^v}^{\iota\-\lan,(0,0)}$, and also the de Rham complex $\DR=[\calO_{K^v}^{\iota\-\lan,(0,0)}\to \calO_{K^v}^{\iota\-\lan,(0,0)}\ox_{\calO_{K^v}^{\sm}}\Omega^{1,\sm}_{K^v}]$ associated to it. Let $\Theta_s\bbH^1(\DR)$ denote the wall-crossing of the first hypercohomology group of $\DR$. It also carries an action of $\bbT^S$ as the action of $\GL_2(L)$ and $\bbT^S$ commute. Then we have the following result.
\begin{theorem}[Theorem \ref{thm:WCdeRham}]
Let $\lambda:\bbT^S\to C$ be the Hecke eigenvalue associated to $\rho$. There is a $\GL_2(L)$-equivariant exact sequence 
\begin{align*}
    0\to \bbH^1(\DR_c)[\lambda]\to \Theta_s\bbH^1(\DR)[\lambda]\to \bbH^1(\DR)[\lambda]\to 0.
\end{align*}
\end{theorem}
Here, $\DR_c$ is the quotient of $\DR$ by $\DR^{\sm}$. This also gives us some structural results about the wall-crossing of $\pi_\sigma(\rho_w)$, see \S \ref{ss:wc} for more details. 

Next, we can use the above theorem to get information of some extension groups. From now on we assume $r_w$ is irreducible. In particular, $\pi_v$ is supercuspidal. Write $\pi_{c}:=\bbH^1(\DR_c)[\lambda]$, which is a locally $\iota$-analytic representation of $\GL_2(L)$. Let $r_{w,E}$, $\pi_{v,E}$, $\pi_{c,E}$ be $E$-structures of $r_w$, $\pi_v$, $\pi_{c}$ respectively. From Theorem \ref{admissibility} we deduce $\pi_{c,E}$ is also admissible. Define 
\begin{align*}
    \Ext^1_{\GL_2(L)}(\pi_{c,E},\pi_{v,E})\isom \Ext^1_{D(G,E)}(\pi_{v,E}',\pi_{c,E}')
\end{align*}
where $D(G,E)$ is the $E$-valued locally $\sigma$-analytic distribution algebra over $\GL_2(L)$, and $(-)'$ denotes the strong dual. The following theorem (a general form was conjectured by Breuil \cite{Bre19}) establishes a link between this extension group and the Weil--Deligne representation $r_{w,E}$, which enables us to describe the extension class inside $\Pi(\rho)^{\sigma\-\lan}$. When $L=\bbQ_p$, this theorem was proved by Yiwen Ding \cite{Din22} using $p$-adic local Langlands correspondence for $\GL_2(\bbQ_p)$ and explicit calculation of $(\vphi,\Gamma)$-modules.
\begin{theorem}[Theorem \ref{thm:MAIN}]\label{thm:introext1}
There exists an isomorphism of $E$-vector spaces
\[
    r_{w,E}\aisom \Ext^1_{\GL_2(L)}(\pi_{c,E},\pi_{v,E})
\]
which only depends on the Weil--Deligne representation $r_{w,E}$ and a choice of $\sigma$. For any $2$-dimensional $E$-linear continuous representation $\rho_w:\Gal(\bar L/L)\to \GL_2(E)$ satisfies the following conditions:
\begin{enumerate}[(i)]
    \item $\rho_w$ is de Rham of parallel Hodge--Tate weight $0,1$ and the underlying Weil--Deligne representation is isomorphic to $r_{w,E}$,
    \item there exists a $2$-dimensional continuous  $E$-linear representation $\rho:\Gal(\bar F/F)\to \GL_2(E)$ arising from a classical automorphic form, such that $\rho|_{\Gal(\bar L/L)}\isom \rho_w$,
\end{enumerate}
under the natural isomorphisms
\[
    D_{\dR,\sigma}(\rho_w)\aisom r_{w,E}\aisom \Ext^1_{\GL_2(L)}(\pi_{c,E},\pi_{v,E}),
\]
it sends $\Fil^1D_{\dR,\sigma}(\rho_w)$ to the $E$-line generated by the extension class given by $\Pi(\rho)^{\sigma\-\lan}$.
\end{theorem}
Here the assumption that $\rho_w$ is of parallel Hodge--Tate weight $0,1$ is not crucial in the proof, only due to the lack of Hyodo--Kato cohomology with non-trivial coefficients. The proof is basically the same as \cite{WConDrinfeld}, but we also need the isomorphism between compactly supported de Rham cohomology of Lubin--Tate towers and Drinfeld towers to establish the relation between cohomology of Drinfeld tower and the completed cohomology of unitary Shimura curves (whose supersingular locus is uniformized by the Lubin--Tate tower). 

Finally, as an application of the above result, we show that there is a locally $\bbQ_p$-analytic subrepresentation of $\Pi(\rho)^{\lan}$ which determines $\rho_w$. Define 
\begin{align*}
    \Pi_1(\rho_w):=\bigoplus_{\pi_{v,E},\sigma\in \Sigma}\Pi(\rho)^{\sigma\-\lan}\subset \Pi(\rho)^{\lan},
\end{align*}
where the amalgamated sum is over $\pi_{v,E}$, and varies in $\sigma\in \Sigma$. 
\begin{theorem}[Theorem \ref{thm:gal}]
Suppose $\rho_w$ satisfies the condition in Theorem \ref{thm:introext1}, then the $E$-linear
isomorphism class of the locally $\bbQ_p$-analytic $\GL_2(L)$-representation $\Pi_1(\rho_w)$ determines $\rho_w$.
\end{theorem}

\subsubsection*{Organization of the paper}
The structure of this paper is organized as follows. In \S \ref{ladef} we recall some facts about locally analytic vectors. The main difference with \cite{Pan22} is that we need  to deal with the $L$-analytic group instead of $\Q_p$.

In \S \ref{uni} we introduce basic properties of unitary Shimura curves. We also study the local structure of the sheaf of locally analytic sections on the unitary Shimura curve at infinite level. 

In \S \ref{intertw} we construct the differential operators $d$ and $\overline{d}$, and study their basic properties on both ordinary locus and supersingular locus.
We explicitly compute the cohomology of the intertwining operator and deduce its classicality. The calculation also shed some lights on the structure of $\Pi(\rho)$ for some Galois representation $\rho$ which is de Rham at $p$. We also study translations and wall-crossings of the cohomology of the intertwining operators. 

In \S \ref{padic} we first prove the $p$-adic Hodge-theoretic interpretation of the intertwining operator acting on $\mathcal{O}^{\text{la}}_{K^v}$. Combining this with the previous computations, we prove our main results.

In \S \ref{p-LLC} we discuss various applications to the $p$-adic local Langlands program for $\GL_2(L)$.

\subsubsection*{Notations and conventions}
Let $L$ be a finite extension of $\bbQ_p$ (this is the base field), and let $E$ be a finite extension of $\bbQ_p$ (this is the coefficient field) such that $[L:\bbQ_p]=|\Hom_{\bbQ_p}(L,E)|$. We always assume that $E$ is sufficiently large. Let $L_0\subset L$ be the maximal subextension unramified over $\bbQ_p$. Let $L_0^{\ur}$ be the maximal unramified extension of $L_0$, and let $\breve{L}_0$ be the completion of $L_0^{\ur}$. We always view $L_0$ as a subfield of $L$ via the natural embedding. Fix an algebraic closure $\bar L$ of $L$.  Let $\text{W}_L$ be the Weil group of $L$, and let $\WD_L$ be the Weil--Deligne group of $L$. Let $C:=\hat{\bar L}$ be the completion of $\bar L$. Let $B_{\dR}$ be the field of de Rham periods constructed from $\bar L$. Let $B_{\cris}\subset B_{\st}\subset B_{\dR}$ be the rings of crystalline (resp. semi-stable) periods, where we fix a uniformizer $\varpi$ of $L$, and we fix a branch of the $p$-adic logarithm $\log_\varpi$ such that $\log_\varpi(\varpi)=0$. Let $\N$ denote the set of non-negative integers.

Our convention is that the Hodge--Tate weight of the $p$-adic cyclotomic character of is $-1$.

For $\pi_1,\dots,\pi_n$ a list of locally $\bbQ_p$-analytic representations of $\GL_2(L)$, we write $\pi_1-\pi_2-\cdots-\pi_n$ to denote a successive extension of $\pi_1,\dots,\pi_n$ such that for each $k$, $\pi_k-\pi_{k+1}-\cdots-\pi_n$ is a non-split extension of $\pi_{k+1}-\cdots-\pi_n$ by $\pi_k$.

\subsubsection*{Acknowledgments}
We are deeply grateful to Liang Xiao and Yiwen Ding, our respective academic advisors, for their invaluable guidance, critical feedback, and sustained support during this research program. We also thank Yuanyang Jiang, Ruochuan Liu and Lue Pan for their constructive discussions and suggestions.

\section{Locally analytic functions on locally $L$-analytic groups}\label{ladef}
In this section, we define locally $J$-analytic vectors for continuous representations of some locally $L$-analytic groups, introduce the corresponding derived functor, and discuss their basic properties.
\subsection{Locally $L$-analytic groups and locally $J$-analytic functions}
Let $E$ be a sufficiently large $p$-adically complete extension of $\bbQ_p$ such that $\#\Hom_{\bbQ_p\-\alg}(L,E)=[L:\bbQ_p]$, which will be the coefficient field for representations. Put $\Sigma:=\Hom_{\bbQ_p\-\alg}(L,E)$. In this subsection we recall some basic facts about locally $L$-analytic groups and define locally $J$-analytic functions on them, where $J\subset \Sigma$ is a subset of embeddings. Our main references are \cite{Ser64, Sch11, Eme17}.

Let $G$ be a locally $L$-analytic group of $L$-dimension $N$. Let $\frg$ be the Lie algebra of $G$, a Lie algebra over $L$ of dimension $N$. Let $\frg_0$ be an $\calO_L$-Lie subalgebra of $\frg$ such that $L\ox_{\calO_L}\frg_0=\frg$, and $[\frg_0,\frg_0]\subset \alpha\frg_0$ with a $p$-adically sufficiently small element $\alpha\in \mathcal{O}_L$. We then use the Baker--Campbell--Hausdorff formula to define a group structure on $\frg_0$, making it a locally $L$-analytic group, together with an open embedding of locally $L$-analytic groups 
\[
    \exp_L:\frg_0\to G
\] 
such that $\Lie(\exp_L):\Lie(\frg_0)\to\frg$ is the identity map. For each $n\in\bbZ_{\ge 0}$, set $\frg_n=p^n\frg_0$, and let $G_n$ denote the image of $\frg_n$ under $\exp_L$. Then $\{G_n\}_{n\ge 0}$ forms a cofinal sequence of open subgroups of $G$. We will refer to such a sequence of open compact  locally $L$-analytic subgroups $\{G_n\}_{n\ge 0}$ as a sequence of standard open compact subgroups of $G$.

Fix a basis $X_1,\dots,X_N$  of $\frg_0$ over $\calO_L$. Define 
\[
    \log_L:=\exp_L^{-1}:G_0\to\frg_0\isom\calO_L^N.
\] 
For any $n\in\bbZ_{\ge 0}$, $\log_L$ restricts to an isomorphism of locally $L$-analytic manifolds from $G_n$ to $p^n\calO_L^N$, and we refer to these as the canonical charts of $G_n$. Moreover, the formal group law on $\frg_n$ not only makes them into locally $L$-analytic groups but also rigid analytic groups over $L$. Let $\bbG_n$ be the rigid analytic group over $L$ defined by the group law on $G_n$; then we have $\bbG_n(L)=G_n$. By abuse of notation, we use the underlying $L$-rational points $G_n$ to denote the rigid analytic group $\bbG_n$ over $L$. 

We now define certain analytic functions on $G$. First, by restriction of scalars, we view $G_n$ as a rigid analytic group over $\bbQ_p$, which we denote by $\Res_{L/\bbQ_p}G_n$. Fix a basis of $\calO_L$ over $\bbZ_p$, say $\alpha_1,\dots,\alpha_d$. Then, if we let $\ul x=(x_1,\dots,x_{dN})$ denote the coordinates on $\frg_0$ with respect to the $\bbZ_p$-basis $(X_i\alpha_j)_{i=1,\dots,N,j=1,\dots,d}$ of $\frg_0$, we have:
\[
    \calC^{\rig}(\Res_{L/\bbQ_p}G_n,E)=\{f=\sum_{k\in\bbN^{dN}} a_{\ul k}\ul x^{\ul k}:a_{\ul k}\in E,a_{\ul k}p^{n\ul k}\to 0\}
\]
where $\ul x^{\ul k}=x_1^{k_1}x_2^{k_2}\cdots x_{dN}^{k_{dN}}$ for $\ul k=(k_1,k_2,\dots,k_{dN})$. This is a Banach algebra over $E$ with norm $||f||_{G_n}:=\sup_{\ul k\in\bbN^{dN}}|a_{\ul k}p^{n\ul k}|$. As usual, we define the space of locally analytic functions on $\Res_{L/\bbQ_p}G$ to be 
\[
    \calC^{\lan}(\Res_{L/\bbQ_p}G,E):=\dlim_{\calP}\prod_{P\in\calP}\calC^{\rig}(\Res_{L/\bbQ_p}P,E),
\]
where $\calP$ is a partition of $G$ by cosets of $G_n$ for $n\ge 0$ and we define $\Res_{L/\bbQ_p}P$ for $P\in\calP$ using the rigid structure of $\bbG_n$ over $L$. Equipped with the product norm, this is a Hausdorff LB-space \cite[Definition 1.1.16]{Eme17} over $E$. The topology is independent of choice of the system $\{G_n\}_{n\ge 0}$.

Now we want to construct another coordinate system on $G$ so that the $L$-linear structure on $\frg$ enters in. Let $J\subset\Sigma$ be a subset of embeddings. Let $\ul y=(y_1,\dots,y_N)$ denote the $L$-coordinates on $\frg_0$ with respect to the $\calO_L$-basis $X_1,\dots,X_N$ of $\frg_0$, we define 
\[
    \calC^{J\-\an}(G_n,E):=\{f=\sum_{\ul k_{\ul\eta}\in(\bbN^N)^J}c_{\ul k_{\ul \eta}}\ul y^{\ul k_{\ul \eta}}:c_{\ul k_{\ul \eta}}\in E,c_{\ul k_{\ul \eta}}p^{n\ul k_{\ul \eta}}\to 0\},
\]
where $\ul y^{\ul k_{\ul\eta}}=\prod\limits_{\eta\in J}\prod_{i=1}^N\eta(y_i)^{k_{i,\eta}}$. It is a Banach algebra with norm defined by 
\[
    ||f||_{G_n,J}:=\sup_{\ul k_{\ul \eta}\in(\bbN^N)^J}|c_{\ul k_{\ul \eta}}p^{n\ul k_{\ul \eta}}|.
\]
Similarly, we define the space of locally $J$-analytic functions on $G$ to be 
\[
    \calC^{J\-\lan}(G,E):=\dlim_{\calP}\prod_{P\in\calP}\calC^{J\-\an}(P,E)
\]
which is a Hausdorff LB-space over $E$. As usual, the topology is independent of the choice of the system $\{G_n\}_{n\ge 0}$.

We note that if $J=\varnothing$, then for each $n\in\bbZ_{\ge 0}$, $\calC^{\varnothing\-\an}(G_n,E)$ consists of constant functions on $G_n$, and $\calC^{\varnothing\-\lan}(G,E)$ consists of locally constant functions on $G$, i.e., 
\[
    \calC^{\varnothing\-\lan}(G,E)=\calC^{\infty}(G,E).
\]

\begin{proposition}\label{prop:lafunc}
There exists a $G$-equivariant isomorphism of LB-spaces over $E$:
\[
    \calC^{\Sigma\-\lan}(G,E)\isom\calC^{\lan}(\Res_{L/\bbQ_p}G,E).
\]
\begin{proof}
Note that 
\begin{align*}
    L\otimes_{\Q_p}E&\ra \prod_{\eta\in \Sigma}E\\
    a\otimes b&\mapsto (\eta(a)b)_\eta
\end{align*}
is an isomorphism, so there is an invertible matrix representing the change of basis from the coordinate system $\{x_i\}_{i=1,\dots,dN}$ to the coordinate system  $\{\eta(y_i)\}_{i=1,\dots,N,\eta\in\Sigma}$. The proposition is clear from this.

\end{proof}
\end{proposition}
This shows that the set of functions $\calC^{\Sigma\-\lan}(G,E)$ we defined by hand, is the set of all $E$-valued locally analytic functions on $\Res_{L/\bbQ_p}G$.

Let $J\subset \Sigma$ be a subset of embeddings of $L$ into $E$. Let $\calC^{J\-\lan}(\frg,E)=\dlim_{n}\calC^{J\-\an}(G_n,E)$ denote the algebra of the germ of locally $J$-analytic functions on $\frg$ at $1\in G$, where $\{G_n\}_{n\ge 0}$ is a sequence of standard open subgroups of $G$. 

\begin{proposition}\label{prop:Jgerm}
We have the following tensor product decomposition: 
\begin{align*}
    \calC^{J\-\lan}(\frg,E)\isom \dlim_n \hat\ox_{E,\eta\in J}\calC^{\eta\-\an}(G_n,E).
\end{align*}
Here the completed tensor product refers to the projective tensor product defined in \cite[$\S$17 B]{Sch02}.
\begin{proof}
By definition, $\calC^{J\-\lan}(\frg,E)\isom \dlim_n\calC^{J\-\an}(G_n,E)$, and for each $n$, we have an explicit description
\begin{align*}
\calC^{J\-\an}(G_n,E)=\{f=\sum_{\ul k_{\ul\eta}\in(\bbN^N)^J}c_{\ul k_{\ul \eta}}\ul y^{\ul k_{\ul \eta}}:c_{\ul k_{\ul \eta}}\in E,c_{\ul k_{\ul \eta}}p^{n\ul k_{\ul \eta}}\to 0\}.
\end{align*}
This shows 
\begin{align*}
    \calC^{J\-\an}(G_n,E)\isom \hat\ox_{E,\eta\in J}\calC^{\eta\-\an}(G_n,E)
\end{align*}
and the result follows by taking colimits on $n$.
\end{proof}
\end{proposition}
By \cite[Proposition 1.1.32]{Eme17}, the completed tensor product commutes with colimits with injective and compact transition maps. Hence we also write 
\begin{align*}
    \calC^{J\-\lan}(\frg,E)\isom \hat\ox_{E,\eta\in J}\calC^{\eta\-\lan}(\frg,E).
\end{align*}

\subsection{Locally $J$-analytic vectors}
Let $G$ be a locally $L$-analytic subgroup with Lie algebra $\frg$. Let $G_0\subset G$ be a compact open subgroup of $G$ as before, and let $\{G_n\}_{n\ge 0}$ be a cofinal sequence of open subgroups of $G_0$. 

Let $V$ be a vector space over $E$ such that $\frg$ acts $E$-linearly on $V$. Then we have a natural morphism of $\bbQ_p$-Lie algebras $\frg\to\End_E(V)$, which $E$-linearly extends to a map $E\ox_{\bbQ_p}\frg\to \End_E(V)$. Now, as $E$ contains all Galois conjugations of $L$, we have a natural decomposition 
\[
    E\ox_{\bbQ_p}\frg=:\frg_\Sigma=\prod_{\eta\in\Sigma}\frg_\eta
\]
over $E\ox_{\bbQ_p}L\aisom \prod_{\eta\in \Sigma} E_{\eta},1\ox a\mapsto (\eta(a))_\eta$. Write $1=\sum_\eta e_\eta$ with $e_\eta\in E\ox_{\bbQ_p}L$ and $e_\eta e_{\eta'}=\delta_{\eta\eta'}e_\eta$ such that $E_\eta=e_\eta(E\ox_{\bbQ_p}L)$, then we see that $\frg_\eta=e_\eta(E\ox_{\bbQ_p}\frg)$, and this 
is a decomposition of Lie algebras over $E\ox_{\bbQ_p}L$. In other words, we have $[\frg_\eta,\frg_{\eta'}]=0$ if $\eta\neq\eta'\in\Sigma$. Indeed, for any $X,Y\in\frg_\Sigma$, we have $[e_\eta X,e_\iota Y]=e_\eta e_\iota[X,Y]=0$ as the Lie bracket is $E\ox_{\bbQ_p}L$-linear. For any subset $J\subset \Sigma$, set $\frg_J:=\prod_{\eta\in J}\frg_\eta$, which is a Lie algebra over $E\ox_{\bbQ_p}L$, where $E\ox_{\bbQ_p}L$ acts on each $\frg_\eta$ via the projection defined by $e_\eta$.

\begin{proposition}\label{prop:123}
For any $J_1,J_2\subset \Sigma$, we have
\[
    \calC^{J_1\-\an}(G_n,E)^{\frg_{J_2}}=\calC^{J_1\bs J_2\-\an}(G_n,E),
\]
where the action of $\frg_{J_2}$ on $\calC^{J_1\-\an}(G_n,E)$ is induced from the action of $\frg$ on $\calC^{\Sigma\-\an}(G,E)$. In particular, we see that for any subset $J\subset\Sigma$, 
\[
    \calC^{\Sigma\-\an}(G_n,E)^{\frg_{\Sigma\bs J}}=\calC^{J\-\an}(G_n,E).
\]
\begin{proof}
From our description of $J$-analytic functions on $G_n$, we see that the action of $\frg_\Sigma$ on $\calC^{\eta\-\an}(G_n,E)$ factors through $\frg_\eta$. This shows that if we take $J_1$-analytic functions on $G_n$ killed by $\frg_{J_2}$, only the terms along $J_1\bs J_2$ survive. More precisely, by Proposition \ref{prop:Jgerm}, we have $\calC^{J_1\-\an}(G_n,E)\isom \hat\ox_{E,\eta\in J_1}\calC^{\eta\-\an}(G_n,E)$. As $\calC^{\eta\-\an}(G_n,E)^{\frg_{J_2}}=0$ if $\eta\in J_2$ and $\frg_{J_2}$ acts trivially on $\calC^{\eta\-\an}(G_n,E)$ if $\eta\not\in J_2$, we see 
\begin{align*}
    \calC^{J_1\-\an}(G_n,E)^{\frg_{J_2}}\isom (\hat\ox_{E,\eta\in J_1}\calC^{\eta\-\an}(G_n,E))^{\frg_{J_2}}\isom \hat\ox_{E,\eta\in J_1\bs J_2}\calC^{\eta\-\an}(G_n,E)\isom \calC^{J_1\bs J_2\-\an}(G_n,E).
\end{align*}
\end{proof}
\end{proposition}

Now we define $J$-analytic vectors in continuous representations of $G$.
\begin{definition}
Let $V$ be a Banach space over $E$ equipped with a continuous representation of $G$. For any $J\subset \Sigma$, we say a vector $v\in V$ is \emph{$(G_n,J)$-analytic}, if the orbit map $g\mapsto g.v$ is a $V$-valued $J$-analytic function on $G_n$. Let $V^{(G_n,J)\-\an}\subset V$ denote the subspace of $V$ consisting of $(G_n,J)$-analytic vectors. Then we have the following basic properties: 
\begin{itemize}
    \item $V^{(G_n,J)\-\an}\isom (\calC^{J\-\an}(G_n,E)\hat\ox_E V)^{G_n}$, where $G_n$ acts as $g.f:=gf(g^{-1}.)$. This follows from a similar argument of \cite[Prop 3.2.9]{Eme17}. Note that we equip $V^{(G_n,J)\-\an}$ with the norm induced from $\calC^{J\-\an}(G_n,V)$, which makes it a Banach space over $E$. We denote this norm by $||\cdot||_{G_n,J}$.
    \item Let $V^{G_n\-\an}$ denote the usual $G_n$-analytic vectors in $V$ for the locally $\bbQ_p$-analytic group $\Res_{L/\bbQ_p}G_n$, then $V^{G_n\-\an}=V^{(G_n,\Sigma)\-\an}$. 
    \item $V^{(G_n,J)\-\an}=(V^{(G_n,\Sigma)\-\an})^{\frg_{\Sigma\bs J}}$. Indeed, the action of $\frg$ on $V^{(G_n,\Sigma)\-\an}$ is induced from the right translation on $\calC^{\Sigma\-\an}(G_n,E)$ inside $(\calC^{\Sigma\-\an}(G_n,E)\hat\ox_E V)^{G_n}$. Hence by Proposition \ref{prop:123},
    \[
        (V^{(G_n,\Sigma)\-\an})^{\frg_{\Sigma\bs J}}\isom (\calC^{J\-\an}(G_n,E)\hat\ox_E V)^{G_n}\isom V^{(G_n,J)\-\an}.
    \] 
    \item The inclusion map $V^{(G_n,J)\-\an}\to V$ is $G_n$-equivariant. To see this, it suffices to show that if $\frg_J.v=0$, then for any $g\in G_n$, $\frg_J.gv=0$. The latter is equivalent to $g^{-1}\frg_Jg.v=0$. However, as $\frg_J$ is a Lie ideal of $\frg$, we have $g^{-1}\frg_Jg\subset\frg_J$, and thus $g^{-1}\frg_Jg.v=0$.
\end{itemize}
\end{definition}

\begin{definition}
Let $V$ be a Banach space over $E$ equipped with a continuous representation of $G$. For any $J\subset \Sigma$, we say that a vector $v\in V$ is \emph{locally $J$-analytic} if the orbit map $g\mapsto g.v$ is a $V$-valued $J$-analytic function on $G_n$ for some $n\in \N$. Let $V^{J\-\lan}\subset V$ denote the subspace of $V$ consisting of locally $J$-analytic vectors. Then: 
\begin{itemize}
    \item $V^{J\-\lan}=\dlim_n V^{(G_n,J)\-\an}$. This is an LB-space over $E$.
    \item $V^{\Sigma\-\lan}$ consists of locally analytic vectors of $V$ under the action of $\Res_{L/\bbQ_p}G$, i.e., $V^{\Sigma\-\lan}=V^{\lan}$ with the usual definition.
    \item $V^{J\-\lan}=(V^{\Sigma\-\lan})^{\frg_{\Sigma\bs J}\-\lan}$.
    \item The inclusion map $V^{J\-\lan}\to V$ is $G$-equivariant.
\end{itemize}
\end{definition}

\begin{remark}
The above definition can be  generalized to Hausdorff LB-spaces as in \cite[2.1.6]{PanII}, and we will have similar results.
\end{remark}

We need some lemmas on locally $J$-analytic vectors.
\begin{lemma}
For any integer $n\in\bbZ_{\ge 0}$ and any subset $J\subset \Sigma$, the action (either left translation or right translation) of $G_{n+1}$ on $\calC^{J\-\an}(G_n,E)^{\circ}/p$ is trivial, where $\calC^{J\-\an}(G_n,E)^{\circ}$ is the unit ball in $\calC^{J\-\an}(G_n,E)$.
\begin{proof}
We may assume that $n=0$. Let $f\in \calC^{J\-\an}(G_0,E)^{\circ}$. For any $x\in G_0$, $y\in G_1$, we can write $xy=x+py'$ for some $y'\in G_0$. Hence $f(xy)=f(x+py')\equiv f(x)\mod p$.
\end{proof}
\end{lemma}

\begin{lemma}\label{lem:GNnormusualnorm}
Let $n\in\bbZ_{\ge 0}$ and $J\subset \Sigma$ be a subset. Let $V$ be a Banach representation of $G_n$ with norm given by $||\cdot||_V$. If $v\in V^{(G_n,J)\-\an}$, then 
\begin{enumerate}[(i)]
    \item $v\in V^{(G_{n+1},J)\-\an}$.
    \item $||v||_{G_{n+1},J}\le ||v||_{G_n,J}$.
    \item $||v||_{G_m,J}=||v||_V$ for $m$ sufficiently large.
\end{enumerate}
\begin{proof}
The first two statements are easy to verify. Now we prove the last one. As $v$ is $(G_n,J)$-analytic, the orbit function $o_v(g)$ can be written as a $V$-valued rigid analytic function on $G_n$. In particular, the norm of the coefficients of $o_v(g)$ are bounded above. Since $o_v(e)=v$, the constant term of $o_v(g)$ is $v$. Therefore, if we take a sufficiently large $m$, which amounts to shrinking the disk sufficiently small, then $||v||_{G_m,J}=\sup_{\ul k_{\ul \eta}\in(\bbN^{N})^J}||v_{\ul k_{\ul \eta}}p^{m\ul k_{\ul \eta}}||=||v_{\ul 0_{\ul \eta}}||=||v||$, as desired.
\end{proof}
\end{lemma}

\begin{lemma}\label{lem:Dercont}
Let $V$ be a Banach representation of $G_0$. Let $D\in\Lie G$ and $n\ge 1$. Then there exists a constant $C_{D,J,n}$ such that $||D(x)||_{G_n,J}\le C_{D,J,n}||x||_{G_n,J}$ for any $x\in V^{(G_n,J)\-\an}$.
\begin{proof}
Note that if $v\in V^{(G_n,J)\-\an}$, then $||v||_{G_n,J}=||v||_{G_n,\Sigma}$. Hence this theorem follows from \cite[Lemme 2.6]{BC16}.
\end{proof}
\end{lemma}

Next, we define derived functors for taking locally $J$-analytic vectors.
\begin{definition}
Let $V$ be a Banach space over $E$ equipped with a continuous action of $G$. Define
\[
    R^iJ\-\lan(V):=\dlim_n H^i_{\cont}(G_n,V\hat\ox_E \calC^{J\-\an}(G_n,E)).
\]
Here we equip $V\hat\ox_E \calC^{J\-\an}(G_n,E)$ with the diagonal $G_n$-action with $G_n$ acts on  $\calC^{J\-\an}(G_n,E)$ via the left translation. We also put 
\begin{align*}
    RJ\-\lan(V):=\dlim_n R\Gamma_{\cont}(G_n,V\hat\ox_E \calC^{J\-\an}(G_n,E)).
\end{align*}
Here, we use the solid formalism described in \cite[\S 2.2]{boxer2025modularitytheoremsabeliansurfaces}.
\end{definition}
By definition, we see that $R^0J\-\lan(V)=V^{J\-\lan}$.

\begin{remark}
In some special cases, we have the following results:
\begin{itemize}
    \item If $J=\Sigma$, we recover the definition given in \cite[Definition 2.2.1]{Pan22}.
    \item If $J=\varnothing$, then $\calC^{\varnothing\-\an}(G_n,E)\isom E$ are trivial representations of $G_n$, so that 
    \[
        R^i\varnothing\-\lan(V)=\dlim_n H^i_{\cont}(G_n,V)=H^i_{\st}(G,V)
    \]
    is the stable cohomology used in \cite[Definition 1.1.5]{Eme06}, as $\{G_n\}_{n\ge 0}$ form a cofinal system of open compact subgroups of $G$.
\end{itemize}
\end{remark}

One can define the notion of $J\-\lan$-acyclic and strongly $J\-\lan$-acyclic as in \cite[Def 2.2.1]{Pan22}, and we will have similar results as in \cite[Lem 2.2.2]{Pan22} by a similar proof. However, we note that admissible Banach representations are not $J\-\lan$-acyclic in general for $J\neq \Sigma$. But in some special cases, we still have the acyclicity results.
\begin{proposition}
Let $J\subset\Sigma$ be a subset. Then $\calC^{\cont}(G,E)$ is $J\-\lan$-acyclic.
\begin{proof}
Recall that if $V$ is a continuous Banach representation of $G$, then there is a $G$-equivariant isomorphism 
\begin{align*}
    \calC^{\cont}(G,E)\hat\ox_E{V}\isom \calC^{\cont}(G,E)\hat\ox_E{V},
\end{align*}
where $G$ acts on the LHS by the diagonal action, and $G$ acts on the RHS by acting only on the first factor (but not on $V$). Explicitly, via the identification $\calC^{\cont}(G,E)\hat\ox_E{V}\isom \calC^{\cont}(G,V)$, this isomorphism is given by $f\mapsto gf(g)$. Therefore, 
\begin{align*}
    H^i_{\cont}(G_n,\calC^{\cont}(G,E)\hat\ox_E \calC^{J\-\an}(G_n,E))\isom H^i_{\cont}(G_n,\calC^{\cont}(G,E))\hat\ox_E \calC^{J\-\an}(G_n,E).
\end{align*}
As $H^i_{\cont}(G_n,\calC^{\cont}(G,E))=0$ for $i\ge 1$, we deduce that $R^iJ\-\lan(\calC^{\cont}(G,E))=0$ for $i\ge 1$. Therefore, $\calC^{\cont}(G,E)$ is $J\-\lan$-acyclic. 
\end{proof}
\end{proposition}

\begin{corollary}\label{cor:Jlala}
Let $J\subset\Sigma$ be a subset. Let $V$ be an admissible Banach representation of $G$ over $E$. Then we have 
\[
    RJ\-\lan(V)=R\Gamma(\frg_{\Sigma\bs J},R\lan(V)).
\]
\begin{proof}
This follows from the same proof of \cite[Thm 1.1.13]{Eme06}, where he proved the case when $J=\varnothing$. In fact, by \cite[Lemma 1.1.1]{Eme06}, since $V$ is admissible, $V$ has a resolution by a complex each of whose terms is a finite direct sum of $\calC^{\cont}(G,E)$. Since we know that $\calC^{\cont}(G,E)$ is $J$-la-cyclic, we can just carry out the same argument.
\end{proof}
\end{corollary}

\section{The unitary Shimura curve at infinite level}\label{uni}
This section is organized as follows. First, we recall the definition of unitary Shimura curves of finite and infinite level. Second, we apply geometric Sen theory to study the pro-\'etale morphisms from infinite level unitary Shimura curves to finite level unitary Shimura curves. Finally, we provide a detailed description of locally analytic sections on the unitary Shimura curves of infinite level and discuss the relation to the completed cohomology groups of the unitary Shimura curves.

\subsection{Unitary Shimura curves}\label{sh}
We recall the definition of unitary Shimura curves. Our main references are \cite{Car83,HT01,NY14,Din17}. Let $L$ be a finite extension of $\bbQ_p$. This will be our base field. Let $\calO_L$ be the ring of integers of $L$ and let $\varpi$ be a uniformizer of $L$. Let $k$ be the residue field of $L$. Suppose that $L/\bbQ_p$ has ramification index $e$ and inertia degree $f$, and put $q=p^f$. Choose a totally real number field $F^+$ of degree $d$, with primes $v=v_1,v_2,\dots,v_r$ lying over $p$, such that the completion of $F^+$ at $v$ is isomorphic to $L$. We fix an isomorphism $F_v^+\isom L$ from now on. Let $\varpi_i$ be a uniformizer of $F^+_{v_i}$, and assume that $\varpi_1$ maps to $\varpi\in L$ under the fixed isomorphism. Choose an imaginary quadratic field $\calE$ in which $p$ splits as $p=uu^c$, where $c$ denotes complex conjugation. Here we fix a prime $u$ of $\calE$ over $p$, which amounts to fixing an embedding of $\calE$ into $\bbC$. Let $F = \calE F^+$, which is a CM field. Let $w_i,i=1,\dots,r$ denote the unique primes of $F$ above $u$ and $v_i$. Set $w=w_1$.

Let $D$ be a quaternion algebra over $F^+$, which is split at $v$ and exactly one  archimedean place $\infty$. The norm map $N_{F/F^+}$ defines a morphism of tori $N_{F/F^+}:\Res_{F/\bbQ}(\bbG_m)\to \Res_{F^+/\bbQ}(\bbG_m)$. From $D$, we define an algebraic group $G$ over $\bbQ$, whose $R$-point for any $\bbQ$-algebra $R$ is given by  
\begin{small}
\begin{align*}
    G(R)=\{(g,z)\in (D\ox_{\bbQ}R)^\times\times(F\ox_{\bbQ}R)^\times:N_{D/F^+}(g)=N_{F/F^+}(z),N_{D/F^+}(g)\cdot N_{F/F^+}(z)\in R^\times\},
\end{align*}
\end{small}where $N_{D/F^+}$ denotes the reduced norm map for $D$. Then the map $(g,z)\mapsto N_{D/F^+}(g)\cdot N_{F/F^+}(z)$ defines a morphism of algebraic groups $\nu:G\to \bbG_{m}$. This is the same as the method used in \cite[2.1.2]{Car83}, and it was taken from \cite{NY14}.

We can view $G$ as a unitary similitude group. Set $B:=D\ox_{F^+}F$, which is a quaternion algebra over $F$. Then, we can choose appropriate local Hasse invariants of $D$ away from $v,\infty$ so that its base change to $F$, i.e., $B$, satisfies the following properties:
\begin{itemize}
    \item the opposite algebra $B^{\op}$ is isomorphic to $B\ox_{F,c}F$.
    \item $B$ is split at $w$.
    \item If $x$ is a place of $F$ whose restriction to $F^+$ does not split in $F/F^+$, then $B_x$ is split.
    \item the number of finite places of $F^+$ above which $B$ is ramified is congruent to $d-1$ modulo $2$.
\end{itemize}
Choose an involution $*$ of the second kind on $B$. Let $V = B$ and consider it as a $B\ox_F B^{\op}$-module. Fix $\beta\in B$ with $\beta^*=-\beta$, and define an alternating $*$-Hermitian pairing $V\times V\to\bbQ$ by
\[
    (x,y)=\tr_{F/\bbQ}\tr_{B/F}(x\beta y^*),
\]
where $\tr_{B/F}$ is the reduced norm map for $B/F$. Define another involution $^\#$ of the second kind on $V$ by $x^\#=\beta x^*\beta^{-1}$. Then for any $\bbQ$-algebra $R$, we also have
\[
    G(R)=\{(g,\lambda)\in (B^{\op}\ox_{\bbQ}R)^\times \times R^\times:gg^\#=\lambda\}.
\]
The map $(g,\lambda)\mapsto \lambda$ defines a homomorphism $\nu:G\to\bbG_m$, which coincides with the map $\nu$ before. If $x$ is a prime in $\bbQ$ which splits as $x=yy^c$ in $\calE$, then $y$ induces an isomorphism 
\[
    G(\bbQ_x)\isom (B_y^{\op})^\times\times \bbQ_x^\times.
\]
In particular, $u$ induces an
isomorphism
\[
    G(\bbQ_p)\isom (B_u^{\op})^\times\times\bbQ_p^\times\isom\bbQ_p^\times\times\prod_{i=1}^r(B_{w_i}^{\op})^\times.
\]
One can choose such $\beta$ so that the following holds:
\begin{itemize}
    \item if $x$ is a prime in $\bbQ$ that does not split in $\calE$, then $G$ is quasi-split at $x$;
    \item the Hermitian pairing $(-,-)$ on $V\ox_{\bbQ}\bbR$ has invariants $(1,1)$ at one embedding $F^+\inj \bbR$ and $(0, 2)$ at all other embeddings.
\end{itemize}

For each $i=1,\dots, r$, fix a maximal order $\Lambda_i=\calO_{B,w_i}\subset B_{w_i}$. The pairing $(-,-)$ induces a perfect duality between $V_{w_i}:=V\ox_F F_{w_i}$ and $V_{w_i^c}$. We will then identify $V_{w_i^c}$ with $ V_{w_i}^{\vee}$. Let $\Lambda_i^{\vee}\subset V_{w_i^c}$ denote the dual of $\Lambda_i$. We get a $\bbZ_p$-lattice
\[
    \Lambda=\bigoplus_{i=1}^r \Lambda_i\oplus\bigoplus_{i=1}^r\Lambda_i^{\vee}\subset V\ox_{\bbQ}\bbQ_p
\]
and $(-,-)$ restricts to a perfect pairing $\Lambda\times\Lambda\to\bbZ_p$. Note that this decomposition of $\Lambda$ is with respect to the decomposition of $F\ox_{\bbQ}\bbQ_p$, namely $F\ox_{\bbQ}\bbQ_p$ acts on $\Lambda_i$ via $F_{w_i}$, and on $\Lambda_i^{\vee}$ via $F_{w_i^c}$. We fix an isomorphism $\calO_{B,w}\isom\Mat_2(\calO_L)$, and compose it with the transpose map to get an isomorphism $\calO_{B,w}^{\op}\isom\Mat_2(\calO_L)$. If $\epsilon\in \Mat_2(\calO_{L})$ is the idempotent which has $1$ in the $(1,1)$-entry and $0$ everywhere else, $\epsilon$ induces an isomorphism
\[
    \Lambda_{11}:=\epsilon\calO_{B,w}\isom(\calO^2_L)^{\vee}.
\]
Here, the action of an element $g\in \Mat_2(\calO_{F,w})\isom(\calO_{B_w}^{\op})$ is via right multiplication by $g^t$. Finally, we define $\calO_B$ as the unique maximal $\bbZ_{(p)}$-order in $B$ which localizes to $\calO_{B,w_i}$ for each $i$ and satisfies $\calO_B^*=\calO_B$.

Let $K\subset G(\bbA^\infty)$ be a sufficiently small open compact subgroup. An open compact subgroup $K\subset G(\bbA^\infty)$ is called \emph{sufficiently small} if there exists a prime $x$ of $\bbQ$ such that the projection of $K$ in $G(\bbQ_x)$ contains no nontrivial finite order subgroup. Recall that $G(\bbQ_p)\isom \bbQ_p^\times \times\prod_{i=1}^r (B_{w_i}^{\op})^\times$. Write $G(\bbA^\infty)=G_vG^v$, with $G_v=(B_w^{\op})^\times\isom \GL_2(L)$, and $G^v= G(\bbA^{\infty,p})\cdot\bbQ_p^\times\cdot \prod_{i=2}^r (B_{w_i}^{\op})^\times$. For each $i=2,\dots,r$, fix $m_2,\dots,m_r\in\bbZ_{\ge 0}$, and set $K_{v_i}=1+\varpi_i^{m_i}\calO_{B,w_i}^{\op}\subset (\calO_{B,w_i}^{\op})^\times$ if $m_i\ge 1$, and $K_{v_i}=(\calO_{B,w_i}^{\op})^\times$ if $m_i=0$. Let $ K^p\subset G(\bbA^{\infty,p})$ be a fixed compact open subgroup. Set $K^v:= K^p\cdot\bbZ_p^\times\cdot K_{v_2}\cdots K_{v_r}$. This will be our tame level subgroup. For  a sufficiently small open compact subgroup $K\subset G(\bbA^\infty)$, define a contravariant functor $X_K$ from locally Noetherian $F$-schemes to sets as follows. If $S$ is a connected locally Noetherian $F$-scheme and $s$ is a geometric point of $S$, we define $X_K(S,s)$ to be the set of equivalence classes of $(r + 4)$-tuples $(A, \lambda, i , \bar \eta^p, \bar \eta_i )$ where
\begin{itemize}
    \item $A$ is an abelian scheme over $S$ of dimension $4d$;
    \item $\lambda:A\to A^{\vee}$ is a polarization;
    \item $i:B\inj \End_S(A)\ox_{\bbZ}\bbQ$ is an injective homomorphism such that $(A, i )$ is compatible (see below) and $\lambda\comp i(b)=i(b^*)^{\vee}\comp\lambda$ for all $b\in B$;
    \item $\bar \eta^p$ is a $\pi_1(S,s)$-invariant $ K^v$-orbit of isomorphisms of $B\ox_{\bbQ}\bbA^{\infty,p}$-modules $\eta^p:B\ox_{\bbQ}\bbA^{\infty,p}\to V^pA_s$ which takes the standard pairing $(-,-)$ on $V$ to a $(\bbA^{\infty,p})^\times$-multiple of the $\lambda$-Weil pairing on $V^pA_s$;
    \item $\bar\eta_1$ is $\pi_1(S,s)$-invariant $K_{v_1}$-orbit of isomorphisms $\eta_{1}:\Lambda_{11}\ox_{\bbZ_p}\bbQ_p\to \epsilon V_{w_1}A_s$ of $L$-vector spaces;
    \item for $i = 2, . . . , r$ , $\eta_i$ is a $\pi_1(S,s)$-invariant $K_{v_i}$-orbit of isomorphisms of $B_{w_i}$-modules $\eta_i: \Lambda_i\ox_{\bbZ_p}\bbQ_p\to V_{w_i}A_s$.
\end{itemize}
Here two $(r +4)$-tuples $(A, \lambda, i , \bar\eta^p, \bar \eta_i )$ and $(A', \lambda', i' , (\bar\eta^p)', (\bar \eta_i)' )$ are equivalent if there is an isogeny $\alpha:A\to A'$ which takes $\lambda$ to a $\bbQ^\times$-multiple of $\lambda'$, takes $i$ to $i'$, and takes $\bar \eta$ to $\bar \eta'$. Here the compatibility means the following: The map $i$ induces an action of $\calE$ on $\Lie A$, and we let $\Lie^+ A$ (resp. $\Lie^-A$) denote the summand of $\Lie A$ where $\calE$ acts in the same way as via the structure morphism $\calE\to \calO_S$ (resp. as the conjugation of the structure morphism $\calE\ov{c}\to \calE\to\calO_S$). So that $\Lie A=\Lie^+A\oplus \Lie ^-A$. We say that $(A, i )$ is \emph{compatible} if $\Lie^+ A$ has rank $2$ over $\calO_S$ and the action of $F^+$ on $\Lie^+ A$ via $i$ agrees with the action via the structure morphism $F^+\to \calO_S$.

The set $X_K(S, s)$ is canonically independent of the choice of $s$, giving $X_K$ on connected $S$, and one extends to disconnected $S$ via $X_K(\coprod_{i\in I} S_i):=\prod_{i\in I} X_K(S_i)$. This functor is representable by a smooth projective $F$-scheme of pure dimension $1$, which we will also denote by $X_K$. If $\calA$ is the universal abelian scheme over $X_K$, we let $\calG_{\calA}:=\epsilon\calA[\varpi^\infty]$ denote the Barsotti--Tate $\calO_L$-module, and we call $\calG:=\calG_{\calA}$ the associated Barsotti--Tate $\calO_L$-module to $\calA$. 

Let $\iota:L\to C$ be a fixed $\bbQ_p$-embedding of $L$ into $C$. Define $\calX_K$ as the adic space associated to $X_K\times_{F}C$, where we view $C$ as an $F$-algebra via $F\to F_w\isom L\ov{\iota}\to C$. For open compact subgroups $K_v' \subset K_v$ within $G_v = \GL_2(L)$, there exists a natural étale Galois covering $\calX_{K^vK_v'} \to \calX_{K^vK_v}$. Taking the inverse image along these covering maps yields a perfectoid space at the infinite level.
\begin{theorem}(\cite[Theorem I.6]{Sch15}, \cite[Theorem 3.3.3]{JLH21})
For any sufficiently small open compact subgroup $K^{v}\subset G^v$, there exists a unique perfectoid space $\calX_{K^{v}}$ over $C$, such that 
\[
    \calX_{K^{v}}\sim\ilim_{K_{v}\subset\GL_2(L)}\calX_{K^{v}K_{v}}
\]
where $K_{v}$ runs through all open compact subgroups of $\GL_2(L)$. Moreover, $\calX_{K^v}$ is a pro-\'etale $K_v$-torsor over $\calX_{K^vK_v}$.
\end{theorem}
In order to work on $\calX_{K^v}$, we find that the more direct approach used in \cite{JLH21} is convenient. 

\subsection{Automorphic sheaves on unitary Shimura curves}
We recall the definition of certain automorphic sheaves on the unitary Shimura curves. Let us first introduce some notation. Let $\bar L$ be an algebraic closure of $L$ with the canonical injection given by $\iota$. Let $C$ be the $p$-adic completion of $\bar L$. Later we will need to decompose an $\calO_B\ox_{\bbZ} C$-module with respect to the decomposition of $\calO_B\ox_{\bbZ} C$. Consider an $\calO_B\ox_{\bbZ}C$-module $M$. Firstly, $M$ decomposes with respect to the $\calO_F\ox_{\bbZ}\bbQ_p$-action: 
\[
    M\isom \prod_{v|p}M_v^+\oplus M_v^-,
\]
where $M_v^+:=M\ox_{\calO_B\ox_{\bbZ}\bbQ_p}\calO_{B_{w}}$. Secondly, define $\epsilon=\epsilon_1$, $\epsilon_2\in \calO_{B,w}$ corresponding to $\left( \begin{matrix} 1&0\\0&0\end{matrix} \right)$ and $\left( \begin{matrix} 0&0\\0&1\end{matrix} \right)$, respectively, under the fixed isomorphism $\calO_{B,w}\isom\Mat_2(\calO_L)$. We further decompose $M_v^{+}$ with respect to the idempotents $\epsilon_1,\epsilon_2$ in $\calO_{B,w}\isom \Mat_2(L)$:
\[
    M_v^+=\epsilon_1 M_v^+\oplus \epsilon_2 M_v^+=M_v^{+,1}\oplus M_v^{+,2}
\]
and an analogous decomposition holds for $M_v^-$. Finally, each $M_v^{+,1}$ is an $L\ox_{\bbQ_p}C$-module, and therefore decomposes as 
\[
    M_v^{+,1}=\bigoplus_{\eta\in\Hom_{\bbQ_p\-\alg}(L,C)}M_{v,\eta}^{+,1}.
\]

\subsubsection*{De Rham cohomology of universal abelian varieties}
Fix a sufficiently small open compact subgroup $K_v\subset \GL_2(L)$, and set  $\calX=\calX_{K^vK_v}$. Let $\pi:\calA\to\calX$ be the universal abelian variety over $\calX$, and let $\Lie\calA$ denote the Lie algebra of $\calA/\calX$, which is a locally free $\calO_{\calX}$-module carrying an $\calO_B$-action. Therefore, we can decompose $\Lie\calA$ with respect to the decomposition of $\calO_B\ox_{\bbZ}C$. For each $\eta\in\Hom_{\bbQ_p\-\alg}(L,C)$, Let $(\Lie\calA)^{+}_{v,\eta}$ be the direct summand of $\Lie\calA$ such that $\calO_B\ox_{\bbZ}\calO_{\calX}$ acts via $\calO_{B,w}\ox_{L,\eta}\calO_{\calX}$. Since $\calX$ is viewed as an adic space over $L$ via the embedding $\iota$, the compatibility condition implies that $(\Lie\calA)^{+,1}_{v,\iota}$ is locally free of rank $1$, while $(\Lie\calA)^{+,1}_{v,\eta}=0$ for $\eta\neq\iota$. By duality, $(R^1\pi_*\calO_{\calA})^{+,1}_{v,\iota}$ is locally free of rank $1$, while $(R^1\pi_*\calO_{\calA})^{+,1}_{v,\eta}$ is locally free of rank $2$ for $\eta\neq\iota$. See \cite[Rmk 3.3.1]{Din17} for more details. Note that $(\Lie\calA)^{+,1}_{v}$ coincides with the Lie algebra of the universal Barsotti--Tate $\calO_L$-module $\calG_{\calA}$.

Next, we recall the structure of the de Rham cohomology of the universal abelian variety. Let $R^1\pi_*\Omega^{\bullet}_{\calA/\calX}$ be the first relative de Rham cohomology of $\pi:\calA\to \calX$. Since it carries an action of $\calO_B\ox_{\bbZ}C$, we can also decompose $R^1\pi_*\Omega^{\bullet}_{\calA/\calX}$ with respect to this action. It turns out that $(R^1\pi_*\Omega^{\bullet}_{\calA/\calX})^{+,1}_{v,\eta}$ is a locally free $\calO_{\calX}$-module of rank $2$ for each $\eta\in \Hom_{\bbQ_p\-\alg}(L,C)$, and there is an exact sequence of locally free $\calO_{\calX}$-modules
\[
    0\to (\pi_*\Omega^1_{\calA/\calX})^{+,1}_{v,\eta} \to (R^1\pi_*\Omega^{\bullet}_{\calA/\calX})^{+,1}_{v,\eta}\to (R^1\pi_*\calO_{\calA})^{+,1}_{v,\eta}\to 0.
\]
This defines the Hodge filtration on $D_{\calX,\eta}:=(R^1\pi_*\Omega^{\bullet}_{\calA/\calX})^{+,1}_{v,\eta}$, given by $\Fil^0D_{\calX,\eta}=D_{\calX,\eta}$, $\Fil^1D_{\calX,\eta}=(\pi_*\Omega^1_{\calA/\calX})^{+,1}_{v,\eta}$ and $\Fil^2 D_{\calX,\eta}=0$. The compatibility condition ensures that when $\eta=\iota$, the Hodge filtration on $D_{\calX,\eta}$ is non-trivial, while for the other cases the Hodge filtration is trivial, i.e, $\Fil^1 D_{\calX,\eta}=0$. In the non-trivial case, we set $\omega_{\calX}:=\Fil^1D_{\calX,\eta}$. Then $\gr^0D=(\Lie\calA^{\vee})^{+,1}_{v,\iota}$ is canonically isomorphic to $\omega_{\calX}^{-1}\ox \wedge^2 (R^1\pi_*\Omega^{\bullet}_{\calA/\calX})^{+,1}_{v,\iota}$. Note that $\wedge^2 (R^1\pi_*\Omega^{\bullet}_{\calA/\calX})^{+,1}_{v,\iota}$ is a trivial line bundle, but with a non-trivial $G(\bbA^\infty)$-action. 


Let $\nabla_{\calX}$ be the Gauss--Manin connection on $R^1\pi_*\Omega_{\calA/\calX}^{\bullet}$. As this connection is functorial, it preserves the decomposition of $R^1\pi_*\Omega_{\calA/\calX}^{\bullet}$ by the action of $\calO_B\ox_{\bbZ}C$. Specifically, it defines an integrable connection on each $D_{\calX,\eta}$, satisfying the Griffith transversality property. For $D_{\calX,\iota}$, the Kodaira--Spencer isomorphism states that the composition 
\[
    \Fil^1D_{\calX,\iota}\ov{\nabla}\to \Fil^0D_{\calX,\iota}\ox_{\calO_\calX}\Omega^1_{\calX/C}\to \gr^0D_{\calX,\iota}\ox_{\calO_{\calX}}\Omega^1_{\calX/C}
\]
is an isomorphism.

Finally, we introduce some notation relevant to the formulation of de Rham torsors, which will later be more compatible with normalizations on local Shimura varieties. For $\eta\in \Hom_{\bbQ_p\-\alg}(L,C)$ and integers $a\ge b$, define 
\[
    D_{\calX,\eta}^{(a,b)}:=\Sym^{a-b}_{\calO_{\calX}}D_{\calX,\eta}^*\ox_{\calO_{\calX}} {\det}^b_{\calO_{\calX}}D_{\calX,\eta}^*
\]
where $D_{\calX,\eta}^*$ is the dual of $D_{\calX,\eta}$. In particular, $D_{\calX,\eta}^{(1,0)}=D_{\calX,\eta}^*$ and $D_{\calX,\eta}^{(1,1)}=\det D_{\calX,\eta}^*$. Each $D_{\calX,\eta}^{(a,b)}$ is a filtered vector bundle with integrable connection of rank $a-b+1$. When $\eta\neq\iota$, the Hodge filtration on $D_{\calX,\eta}^{(a,b)}$ is trivial. However, the Hodge filtration on $D_{\calX,\iota}^{(a,b)}$ is non-trivial and is generated by 
\[
    0\to \omega^{(0,1)}_{\calX}\to D_{\calX,\iota}^{(1,0)}\to \omega_{\calX}^{(1,0)}\to 0
\]
where $\omega_{\calX}^{(0,1)}:=\Fil^1 D_{\calX,\iota}^{(1,0)}$ and $\omega_{\calX}^{(1,0)}:=\gr^0D_{\calX,\iota}^{(1,0)}$. For $a,b$ integers, define 
\[
    \omega_{\calX}^{(a,b)}:=(\omega_{\calX}^{(1,0)})^{\ox a}\ox_{\calO_{\calX}}(\omega_{\calX}^{(0,1)})^{\ox b}.
\]
The graded pieces of $D_{\calX,\iota}^{(a,b)}$ with respect to the Hodge filtration are of the form $\omega_{\calX}^{(a,b)}$ for some $(a,b)\in\bbZ^2$. The Kodaira--Spencer isomorphism is 
\[
    \omega_{\calX}^{(0,1)}\aisom \omega_{\calX}^{(1,0)}\ox \Omega_{\calX}^1
\]
and we have $\Omega_{\calX/C}^1\isom \omega_{\calX}^{(-1,1)}$.

\subsubsection*{\'Etale cohomology of universal abelian varieties}
Let $\pi:\calA\to \calX$ be the universal abelian variety, and let $\calG_{\calA}=\epsilon\calA[\varpi^\infty]$ be the associated Barsotti--Tate $\calO_L$-module. We denote $\calG=\calG_{\calA}$ as the universal Barsotti--Tate $\calO_L$-module over $\calX$. Then $R^1\pi_*\bbQ_p$ is a local system on $\calX_{K}$ endowed with an $\calO_B\ox_{\bbZ}\bbQ_p$-action. Let $(R^1\pi_*\bbQ_p)^{+,1}_{v}$ be the local system on $\calX$ such that the action of $\calO_B\ox_{\bbZ}\bbQ_p$ factors through $\epsilon\Mat_2(\calO_{B,w})$. At geometric points $\bar s$ of $\calX$, we identify the fiber $(R^1\pi_*\bbQ_p)^{+,1}_{v,\bar s}$ with the dual of the rational Tate module of the Barsotti--Tate $\calO_L$-module $\calG_{\calA,\bar s}$ at $\bar s$. From the moduli problem we know that $(R^1\pi_*\bbQ_p)^{+,1}_{v}$ is naturally isomorphic to the local system associated to the standard representation of $\GL_2(L)$ on $L^{\oplus 2}$. See \cite[Proposition 3.2.2]{Din17} for more details. Note that $R^1\pi_*\bbQ_p$ extends  naturally to the local system $R^1\pi_*\hat\bbQ_p$ on $\calX_{\pro\et}$, with $(R^1\pi_*\bbQ_p)^{+,1}_{v}$ extending to a local system $(R^1\pi_*\hat\bbQ_p)^{+,1}_{v}$. Note that the local system $R^1\pi_*\hat\bbQ_p$ is trivialized on $\calX_{K^v}$.

The relative de Rham comparison theorem relates the \'etale cohomology of $\pi:\calA\to\calX$ and the de Rham cohomology of $\pi:\calA\to \calX$. Let $\hat\calO_{\calX}$ be the completed structure sheaf of the pro-\'etale site on $\calX$. It gives an exact sequence of sheaves on $\calX_{\pro\et}$: 
\[
    0\to R^1\pi_*\calO_{\calA}\ox_{\calO_{\calX}}\hat\calO_{\calX}(1)\to R^1\pi_*\hat\bbZ_p\ox_{\bbZ_p}\hat\calO_{\calX}(1)\to \pi_*\Omega^1_{\calA/\calX}\ox_{\calO_{\calX}}\hat\calO_{\calX}\to 0.
\]
Here $(1)$ denotes the Tate twist. Note that by functoriality of the relative $p$-adic de Rham comparison theorem, this sequence is $\calO_B\ox_{\bbZ}\hat \calO_{\calX}$-equivariant. Therefore, it decomposes with respect to the action of $\calO_B\ox_{\bbZ}\hat \calO_{\calX}$. For $\eta\neq\iota$, taking the $(-)^{+,1}_{v,\eta}$ component yields an isomorphism 
\[
    (R^1\pi_*\hat C)^{+,1}_{v,\eta}\ox_{\hat C}\hat\calO_{\calX}\isom D_{\calX,\eta}\ox_{\calO_{\calX}}\hat\calO_{\calX}
\]
as $\Fil^1 D_{\calX,\eta}=0$. Here $\hat C=\hat\calO_C[\frac{1}{p}]$ and $\hat\calO_C=\ilim_n\calO_C/p^n\calO_C$ as pro-\'etale sheaves on $\calX$. Besides, if we take $(-)^{+,1}_{v,\iota}$, we get an exact sequence 
\[
    0\to \gr^0D_{\calX,\eta} \ox_{\calO_{\calX}}\hat\calO_{\calX}(1)\to (R^1\pi_*\hat C)^{+,1}_{v,\iota}(1)\ox_C \hat\calO_{\calX}\to \Fil^1 D_{\calX,\eta}\ox_{\calO_{\calX}}\hat\calO_{\calX}\to 0.
\]
For each $\eta\in \Hom_{\bbQ_p\-\alg}(L,C)$, the local system $(R^1\pi_*\hat C)^{+,1}_{v,\eta}$ is identified with the local system associated to $V_\eta$ via the $K_v$-torsor $\calX_{K^v}\to \calX$, where $V_{\eta}$ is the standard representation of $\GL_2(L)$ composed with the embedding $\eta:L\inj C$, referred as the \emph{$\eta$-standard representation} of $\GL_2(L)$.

The position of $\gr^0D_{\calX,\eta} \ox_{\calO_{\calX}}\hat\calO_{\calX}$ in $(R^1\pi_*\hat C)^{+,1}_{v,\iota}\ox_C \hat\calO_{\calX}$ gives the Hodge--Tate period map 
\[
    \pi_{\HT}:\calX_{K^v}\to\fl,
\]
where $\fl$ is the adic space associated to $\bbP^1_L\times_{L,\iota}C$. Our definition of $\pi_{\HT}$ agrees with that in \cite{JLH21}. One can also take the original definition of the Hodge--Tate period map as in \cite{Sch15}, as in our case the Hodge filtration for embeddings $\eta\neq\iota$ is trivial. The basic properties of $\pi_{\HT}$ are summarized as follows.
\begin{theorem}\label{thm:piHT}(\cite[Theorem III.1.2 and Theorem III.3.18]{Sch15})
The Hodge--Tate period map $\pi_{\HT}:\calX_{K^v}\to \fl$ is $\GL_2(L)\times\Gal_L$-equivariant. It is also equivariant under the Hecke operators away from $v$, where $\fl$ carries the trivial Hecke action. Moreover, let $\ffrb$ be the set of finite intersections of rational subsets $U_1=\{[x_1:x_2]:|x_1|\ge |x_2|\}$, $U_2=\{[x_1:x_2]:|x_1|\le |x_2|\}$ in $\fl$. Then $\ffrb$ is a basis of open affinoid subsets of $\fl$ stable under finite intersections and each $U\in\ffrb$ has the following properties:
\begin{enumerate}[(i)]
    \item its preimage $V:=\pi_{\HT}^{-1}(U)$ is affinoid perfectoid;
    \item $V$ is the preimage of an affinoid subset $V_{K_v}\subset\calX_{K^vK_v}$ for some sufficiently small open subgroup $K_v\subset\GL_2(L)$;
    \item The natural map $\dlim_{K_v}H^0(V_{K_v},\calO_{\calX_{K^vK_v}})\to H^0(V, \calO_{\calX_{K^v}})$ has dense image.
\end{enumerate}
\end{theorem}

Finally, the following important result holds: up to twist, the Hodge--Tate exact sequence on $\calX_{\pro\et}$ coincides with the Faltings extension for $\calX$. More precisely, let 
\[
    0\to\hat\calO_{\calX}(1)\to\gr^1\calO\bbB_{\dR,\calX}^+\ov{\nabla}\to \hat\calO_{\calX}\ox_{\calO_{\calX}}\Omega^1_{\calX/C}\to 0
\]
be the $\gr^1$-piece of the $\bbB_{\dR}^+$-Poincar\'e lemma \cite[Corollary 6.13]{Sch13}. On the other hand, tensoring the Hodge--Tate exact sequence by $(\gr^0 D_{\calX,\iota})^{-1}$ gives 
\[
    0\to \hat\calO_{\calX}(1)\to (R^1\pi_*\hat C)^{+,1}_{v,\iota}(1)\ox_C(\gr^0 D_{\calX,\iota})^{-1}\ox_{\calO_{\calX}}\hat\calO_{\calX}\to\hat\calO_{\calX}\ox \Omega^1_{\calX/C}\to 0
\]
where $\Fil^1 D_{\calX,\eta}\ox(\gr^0D_{\calX,\eta})^{-1}$ is identified with $\Omega^1_{\calX/C}$ via the Kodaira--Spencer isomorphism. A similar proof for \cite[Theorem 4.2.2]{Pan22} yields the following. 
\begin{theorem}\label{thm:FEHT}
There exists a canonical isomorphism between $(R^1\pi_*\hat C)^{+,1}_{v,\iota}(1)\ox_C(\gr^0 D_{\calX,\iota})^{-1}$ and $\gr^1\calO\bbB_{\dR}^+$ such that, as extensions of $\calO_{\calX}\ox \Omega^1_{\calX/C}$ by $\hat\calO_{\calX}(1)$, they differ by $-1$. 
\end{theorem}

\subsubsection*{Automorphic sheaves on the unitary Shimura variety at infinite level}
We now consider the pullback of automorphic sheaves from $\calX=\calX_{K^vK_v}$ to $\calX_{K^v}$. Instead of the pro-\'etale site on $\calX$, we work on the analytic site on $\calX_{K^v}$.

We first fix some notation. Let $\calO_{\calX_{K^v}}$ be the restriction of $\hat\calO_{\calX}$ to the analytic site on $\calX_{K^v}$. Define $\calO_{K^v}:=\pi_{\HT,*}\calO_{\calX_{K^v}}$. Let $\eta\in \Hom_{\bbQ_p\-\alg}(L,C)$ be an embedding. As $(R^1\pi_*\hat C)^{+,1}_{v,\eta}$ is the local system associated to the $\eta$-standard representation $V_{\eta}$ by the $K_v$-torsor $\calX_{K^v}\to \calX$, it becomes a trivial local system on $\calX_{K^v}$. We denote it by $V_{\eta}$ by abuse of notation. In general, for $a\ge b$ integers, define 
\[
    V_{\eta}^{(a,b)}:=\Sym^{a-b}V_\eta\ox_C{\det}^bV_{\eta}.
\]
This is an $\eta$-algebraic representation of $\GL_2(L)$ over $C$, and every $\eta$-algebraic representation of $\GL_2(L)$ over $C$ are of this form. Moreover, let $V$ be an irreducible algebraic representation of $\GL_2(L)$ over $C$. Then there exists integers $(a_\eta,b_\eta)_\eta$ such that 
\[
    V\isom\ox_\eta V_\eta^{(a_\eta,b_\eta)}.
\]

Let $D_{\calX,\eta}^{(a,b)}$ be the filtered vector bundle with integrable connection on $\calX$ we constructed in the previous subsection. Define $D_{\calX_{K^v},\eta}^{(a,b),\sm}:=\dlim_{K_v'\subset K_v}\pi_{K_v'}^{-1}D_{\calX,\eta}^{(a,b)}$, where $\pi_{K_v}:\calX_{K^vK_v'}\to \calX_{K^vK_v}$ is the natural projection. Let $\calO_{\calX_{K^v}}^{\sm}=\dlim_{K_v'\subset K_v}\pi_{K_v'}^{-1}\calO_{\calX}$ denote the set of smooth vectors in $\calO_{\calX_{K^v}}$. Then  $D_{\calX_{K^v},\eta}^{(a,b)}:=D_{\calX_{K^v},\eta}^{(a,b),\sm}\ox_{\calO_{\calX_{K^v}}^{\sm}}\calO_{\calX_{K^v}}$ is just the pullback of $D_{\calX,\eta}^{(a,b)}$ to $\calX_{K^v}$ as coherent sheaves. Moreover, $D_{\calX_{K^v},\eta}^{(a,b),\sm}$ is the set of smooth vectors in $D_{\calX_{K^v},\eta}^{(a,b)}$ for the $\GL_2(L)$-action. Let $D^{(a,b)}_{K^v,\eta}:=\pi_{\HT,*}D_{\calX_{K^v,\eta}}^{(a,b)}$. Let $\omega_{\calX_{K^v}}^{(a,b)}$ denote the pullback of $\omega_{\calX}^{(a,b)}$ to $\calX_{K^v}$ as coherent sheaves, and let $\omega_{K^v}^{(a,b)}:=\pi_{\HT,*}\omega_{\calX_{K^v}}^{(a,b)}$. Let $\omega_{K^v}^{(a,b),\sm}$ be the set of smooth vectors in $\omega_{K^v}^{(a,b)}$, and we have $\omega_{K^v}^{(a,b)}=\omega_{K^v}^{(a,b),\sm}\ox_{\calO_{K^v}^{\sm}}\calO_{K^v}$. Since $\pi_{\HT,*}$ is exact, we have exact sequences
\begin{align*}
    &V_\eta^{(1,0)}\ox_C \calO_{{K^v}}\isom D_{{K^v},\eta}^{(0,-1)},\\
    0\to \omega_{K^v}^{(0,-1)}\to &V_{\iota}^{(1,0)}\ox_C\calO_{K^v}\to \omega_{K^v}^{(-1,0)}(-1)\to 0.
\end{align*}

We introduce some notation related to the flag variety. Let $a,b$ be integers, and we define $\omega_{\fl}^{(a,b)}:=\calO_{\fl}(a-b)\ox {\det}^b$ where $\calO_{\fl}(1)$ is the usual Serre twist on projective spaces. The Euler sequence on $\fl$ is (up to twist)
\[
    0\to \omega_{\fl}^{(0,1)}\to V^{(1,0)}_\iota\ox_C \calO_{\fl}\to \omega_{\fl}^{(1,0)}\to 0,
\]
and we have $\Omega_{\fl}^{1}\isom \omega_{\fl}^{(-1,1)}$, $\det\ox_C\calO\isom \omega_{\fl}^{(1,1)}$. The pullback (as coherent sheaves) of the Euler exact sequence on $\fl$ is the Hodge--Tate exact sequence, which identifies 
\[
    \pi_{\HT}^*\omega_{\fl}^{(1,0)}\isom \omega_{\calX_{K^v}}^{(-1,0)}(-1)\quad 
    \mathrm{and}\quad \pi_{\HT}^*\omega_{\fl}^{(0,1)}\isom \omega_{\calX_{K^v}}^{(0,-1)}.
\]
Therefore, $\pi_{\HT}^*\omega_{\fl}^{(a,b)}\isom\omega_{\calX_{K^v}}^{(-a,-b)}(-a)$. We note that these isomorphisms are equivariant for the $\GL_2(L)$-action, $\Gal_L$-action, and Hecke action away from $p$.
\begin{remark}
Our normalizations on the \'etale cohomology and de Rham cohomology of the universal abelian variety are chosen so that to compatible with normalizations on the local setting. Roughly speaking, we require the \emph{\'etale cohomology}  and the \emph{de Rham homology} of the universal abelian variety to be the standard representations, and the indices on the Hodge filtrations on them are chosen so that they are compatible with usual equivariant vector bundles on the flag variety $\fl$ which we use the Borel subalgebra consisting of lower triangular matrices to define it. We are grateful to Zhixiang Wu for helpful conversations on normalizations.
\end{remark}

\subsection{Geometric Sen theory and functor $\VB$}
We recall some results from geometric Sen theory \cite{Pan22}, \cite{RC1}, along with the functor $\VB$ introduced in \cite{Pil22} and furthur developed in \cite[\S 4.5]{boxer2025modularitytheoremsabeliansurfaces}. These results will enable us to decomplete $\calO_{K^v}$ and analyze the structure of this huge sheaf. Although one can directly use the results from  \cite{RC1}, for further applications we will reproduce the explicit calculation in \S 4.3 of \cite{Pan22}.

First, we recall the definition of the geometric Sen operator. Our local setting is as follows. Let $X=\Spa(A,A^+)$ be a one-dimensional smooth affinoid adic space over $\Spa(C,\calO_C)$ such that there exists a toric chart $c:X\to \bbT^1$. We fix this toric chart. Let $\tilde{X}=\Spa(B,B^+)$ be a perfectoid affinoid adic space over $C$, such that $\tilde{X}$ is a pro-\'etale Galois covering of $X$ with Galois group $G$, a compact locally $\bbQ_p$-analytic group. For $G_0\subset G$ an open subgroup, let $X_{G_0}:=\Spa(B^{G_0},(B^+)^{G_0})$. Let $\bbT^1_\infty\sim \ilim_n\bbT^1_n$ with $\bbT^1_n=\Spa(C\langle T^{\pm 1/p^n}\rangle,\calO_C\langle T^{\pm 1/p^n}\rangle)$ be the usual pro-\'etale Galois covering of $\bbT^1$ with Galois group $\Gamma\isom\bbZ_p$ obtained by taking the $p^\infty$-th power of roots of the coordinate functions on $\bbT^1$. Then we can pullback this $\bbZ_p$-tower over $\bbT^1$ via the toric chart $c:X\to\bbT^1$, which we denote by $X_\infty=\Spa(A_\infty,A_\infty^+)$. Finally, let $\tilde{X}_\infty$ be the fiber product of $X_\infty$ and $\tilde{X}$ over $X$. Then $\tilde{X}_\infty=\Spa(B_\infty,B_\infty^+)$ is an affinoid perfectoid space over $\Spa(C,\calO_C)$. Let $X_{G_0,n}=\Spa(B_{G_0,n},B_{G_0,n}^+)$ be the fiber product of $X_{G_0}$ and $X_n$ over $X$.

We first work with finite dimensional locally analytic representation of $G$ over $C$. The reason we work with $C$ is that we need to consider various embeddings of $L$. Using arguments similar to those in \cite[Proposition 3.2.11]{Pan22}, we get the following. 
\begin{proposition}\label{prop:decomplete}
Let $V$ be a finite-dimensional locally $\bbQ_p$-analytic representation of $G$ over $C$, with a $G$-stable lattice $T\subset V$. Then there exists a sufficiently small open subgroup $G_0\subset G$, and a sufficiently large integer $n(G_0)\in\bbZ_{\ge 0}$, such that for any $n\ge n(G_0)$, there exists a unique free $B_{G_0,n}^+$-submodule $D_{G_0,n}^+(T)$ of rank $\dim_C V$ stable under $G\times\Gamma$, such that:
\begin{itemize}
    \item The natural morphism $(B_\infty)^{\circ}\ox_{B_{G_0,n}^+}D_{G_0,n}^+(T)\to (B_\infty)^{\circ}\ox_{\calO_C}T$ is an isomorphism.
    \item $\Gamma$ acts locally analytically on $D_{G_0,n}^+(T)$.
\end{itemize}
Moreover, set $D_{G_0,n}(V):=D_{G_0,n}^+(T)[\frac{1}{p}]$, we have 
\begin{itemize}
    \item $D_{G_0,n}(V)=(B_\infty\ox_C V)^{G_0,p^n\Gamma\-\an}$, and the natural map 
    \[
        B_\infty\ox_{B_{G_0,n}}(B_\infty\ox_{C}V)^{G_0,p^n\Gamma\-\an}\isom B_\infty\ox_{C}V
    \]
    is an isomorphism.
\end{itemize}
\end{proposition}
We note that in the proof we need the following lemma to ensure that there always exists a unitary lattice inside a finite dimensional locally analytic representations.
\begin{lemma}
Let $V$ be a finite-dimensional locally analytic representation of $G$ over $C$. Then there exists a $G$-stable $\calO_C$-lattice $T\subset V$.
\begin{proof}
It suffices to show that there exists a norm $||\cdot||$ on $V$ such that the action is unitary (then one can take $T$ to be the unit ball). Let $||\cdot||$ be a norm on $V$. As $\frg$ acts as linear operators on $V$, the action is automatically continuous. Let $X_1,\dots,X_n$ be a $\bbQ_p$-basis of $\frg$. Then after replacing $X_i$ by $p^NX_i$, we may assume that $X_1,\dots,X_n$ is a $\bbZ_p$-basis of $\frg_0$ with $[\frg_0,\frg_0]\subset a\frg_0$ where $a$ is sufficiently small, and the operator norm $||X_i||\le p^{-2}$. Now we integrate $\frg_0$ to an open compact subgroup of $G$, say $G_0$. Then for any $g\in G_0$, we have $||g v||\le ||v||$ for any $v\in V$. Finally, as $G$ is compact, set $||v||':=\frac{1}{[G:G_0]}\sum_{h\in G/G_0}||hv||$, we see that the action of $G$ on $(V,||\cdot||')$ is unitary.
\end{proof}
\end{lemma}

As the action of $\Gamma$ on $D_{G_0,n}(V)$ is locally analytic, the Lie algebra $\Lie \Gamma$ acts on $D_{G_0,n}(V)$. Via the natural isomorphism 
\[
    B_\infty\ox_{B_{G_0,n}}D_{G_0,n}(V)\isom B_\infty\ox_{C}V.
\]
It extends $B_\infty$-linearly to an action
\[
    \phi_V:\Lie\Gamma\to \End_{B_\infty}(B_\infty\ox_C V).
\]
The action $\phi_V$ called the \emph{Sen operator} of $V$. We summarize some basic properties of the Sen operator $\phi_V$.
\begin{theorem}\label{thm:simpson}
For each finite-dimensional locally analytic representation of $G$ over $C$, there exists a unique $B_\infty$-linear action of $\Lie\Gamma$ on $B_\infty\ox_C V$
\[
    \phi_V:\Lie\Gamma\to \End_{B_\infty}(B_\infty\ox _C V),
\]
such that
\begin{enumerate}[(1)]
    \item this action extends the natural action of $\Lie\Gamma$ on the $G$-smooth, $\Gamma$-locally analytic vectors in $B_\infty\ox_C V$;
    \item $\phi_V$ commutes with the action of $G\times\Gamma$;
    \item $\phi_V$ is functorial in $V$, i.e., suppose $\psi:V\to W$ is a morphism of locally analytic representations of $G$, then $\id\ox\psi:B_\infty\ox_C V\to B_\infty\ox_C W$ intertwines $\phi_V$ and $\phi_W$;
    \item $\phi_V$ is compatible with tensor products, i.e., suppose $V_1$, $V_2$ are two finite dimensional locally analytic representations of $G$, then $\phi_{V_1\ox_C V_2}=\phi_{V_1}\ox \id+\id\ox\phi_{V_2}$.
\end{enumerate}
\end{theorem}
This is analogous to \cite[Proposition 3.2.14]{Pan22}, but one needs to replace the coefficient field by $C$ in \cite[Proposition 3.2.11]{Pan22}.

Next, there is a universal Sen operator, which determines all the Sen operators above and essentially depends only on the covering $\tilde{X}/X$.
\begin{theorem}\label{thm:universalsimpson}
Let $\tilde{X}\to X$ be a perfectoid pro-\'etale Galois cover with Galois group $G$ as before. Then there exists a morphism of Lie algebras:
\[
    \phi_{\tilde{X}/X}:\Lie\Gamma\to B\ox_{\bbQ_p}\Lie G
\] 
such that $\phi_{\tilde{X}/X}$ is universal, i.e., for any finite dimensional locally analytic representation $V$ of $G$ over $C$, let $\phi_V:\Lie\Gamma\to\End_{B_\infty}(B_\infty\ox_C V)$ denote the map constructed in Theorem \ref{thm:simpson}, and let $a_V:B_\infty\ox_{\bbQ_p}\Lie G\to\End_{B_\infty}(B_\infty\ox_C V)$ be the induced $B_\infty$-linear Lie algebra representation on $B_\infty\ox_{C}V$. Then $a_V\comp \phi_{\tilde{X}/X}=\phi_V$. In particular, if $C\ox_{\bbQ_p}\Lie G\to \End_CV$ is injective, then $\phi_V$ determines $\phi_{\tilde{X}/X}$.
\begin{proof}
This is \cite[Corollary 3.3.4]{Pan22} and \cite[Corollary 3.3.6]{Pan22}. Roughly speaking, we apply Theorem \ref{thm:simpson} to $\calC^{\an}(G_0,C)$, which can be approximated by finite dimensional representations. This gives the map $\phi_{\tilde{X}/X}$. For any finite dimensional locally analytic representation $V$, we may assume that it is $G_0$-analytic, and use the matrix coefficient map to relate $V$ with $\calC^{\an}(G_0,C)$.
\end{proof}
\end{theorem}

Now we calculate the the universal geometric Sen operator for the pro-\'etale Galois cover $\calX_{K^v}\to \calX_{K^vK_v}$. The idea is the same as \cite[\S 4.2]{Pan22}. Namely, we calculate the geometric Sen operator for the algebraic representation $\ox_\eta V_\eta^{(1,0)}$ using the Hodge--Tate exact sequence and its identification with the Faltings extension, which then determines the universal geometric Sen operator as the representation is faithful. Let $\frg_{\iota}:=\gl_2(L)\ox_L C$. Let $\frg_\iota^{0}:=\frg_\iota\ox_C\calO_{\fl}$ with $\frb_{\iota}^0\subset \frg_{\iota}^0$ the universal Borel algebra and $\frn_{\iota}^0$ the universal nilpotent algebra.  The following result is a special case of \cite[Theorem 5.2.5]{camargo2024locallyanalyticcompletedcohomology}.
\begin{theorem}\label{thm:geomSen}
Let $U\in\ffrb$ be sufficiently small, with $V=\pi_{\HT}^{-1}(U)=\textup{Spa}(B,B^+)$ and $V_{K_v}\subset\calX_{K^vK_v}$ such that $V=\pi_{K_v}^{-1}(V_{K_v})$. Then the universal geometric Sen operator for the pro-\'etale covering $V\to V_{K_v}$ is given by the action of a generator of $\frn^0_{\iota}|_U$ via the pullback along $\pi_{\textup{HT}}$.
\begin{proof}
By Theorem \ref{thm:universalsimpson}, it suffices to calculate the geometric Sen operator for the algebraic representation $V:=\ox_\eta V_\eta^{(1,0)}$ as $V$ is a faithful representation of $\GL_2(L)$. As $\phi_V=\sum_{\eta} \phi_{V_\eta}$, it suffices to calculate the geometric Sen operator for each $V_\eta$. For $\eta\neq\iota$, we have $V_\eta\ox_C\calO_{K^v}=D_{K^v,\eta}^{*,\sm}\ox_{\calO^{\sm}_{K^v}}\calO_{K^v}$. Choose a toric chart $c:V_{K_v}\ra \bbT^1$ and consider the associated perfectoid toric chart as above. Then we have
\begin{align*}
    D_{G_0,n}(V_\eta)&=(V_\eta\otimes B_\infty)^{G_0,p^n\Gamma\textup{-an}}\\
    &=(D_{\mathcal{X}_{K^v},\eta}^{*,\sm}\ox_{\calO^{\sm}_{\mathcal{X}_{K^v}}}\calO_{\mathcal{X}_{K^v}})(V_\infty)^{G_0,p^n\Gamma\textup{-an}}\\
    &=D_{\mathcal{X}_{K^v},\eta}^{*,\sm}(V_{G_0})\ox_{\calO^{\sm}_{\mathcal{X}_{K^v}}(V_{G_0})}\calO_{\mathcal{X}_{K^v}}(V_{G_0,\infty})^{p^n\Gamma\textup{-an}}.
\end{align*}
As $\Gamma$ acts smoothly on the latter term, the action of $\Lie\Gamma$ on $D_{G_0,n}(V_\eta)$ is trivial. Hence $\phi_{V_\eta}=0$ for each $\eta\neq\iota$. To compute the geometric Sen operator for $V_{\iota}$, we know that up to a twist by a line bundle, $V_\eta\ox_C\calO_{K^v}$ is isomorphic to $\gr^1\calO\bbB_{\dR}^+$, which is an extension of $\calO_{K^v}\ox_{\calO^{\sm}_{K^v}}\Omega^{1,\sm}_{K^v}$ by $\calO_{K^v}$. A local computation (e.g. \S 2 of \cite{DLLZ}) using toric charts shows that the action of $\Gamma$ on $\gr^1\calO\bbB_{\dR}^+$ is unipotent. Therefore, taking the derivative of this action, we get a nilpotent map on $V_\eta\ox_C\calO_{K^v}$ which induces an isomorphism between the two graded pieces. By the definition of the exact Hodge--Tate sequence, we see that this map is given by the universal nilpotent Lie algebra on $\fl$.
\end{proof}
\end{theorem}

\begin{remark}
It is not surprising to see that the image of the universal geometric Sen operator $\Lie\Gamma\to B\ox_{\bbQ_p}\frg$ belongs to $B\ox_{L,\iota}\frg$. Indeed, the pro-\'etale vector bundle $V_\eta\ox_C\calO_{K^v}$ is actually an \'etale vector bundle $D_{K^v,\eta}^{*,\sm}\ox_{\calO_{K^v}^{\sm}}\calO_{K^v}$ for $\eta\neq\iota$.
\end{remark}

\begin{remark}
If one defines a notion of locally $\iota$-analytic covering as in \cite[3.5]{Pan22} by replacing locally analytic by locally $\iota$-analytic, then one can show in the same way that $\calX_{K^v}\to \calX_{K^vK_v}$ is a locally $\iota$-analytic covering.
\end{remark}

\subsubsection*{The functor $\VB$}
Next we recall the definition of $\VB$ in \cite{Pil22, boxer2025modularitytheoremsabeliansurfaces} and slightly generalize it to our setting. For simplicity let $G=\GL_2(L)$ and let $\frg$ be the Lie algebra of $\GL_2(L)$. Let $\Coh_{\frg}(\fl)$ denote the category of $\frg$-equivariant coherent sheaves on $\fl$. For $\calF\in\Coh_{\frg}(\fl)$, the nilpotent Lie algebra $\frn^0_{\iota}$ acts on $\calF$ and the we may consider the derived functor associated to the functor $\calF\mapsto \calF^{\frn^0_\iota}$,  denoted by $R\Gamma(\frn^0_\iota,\calF)$. Explicitly, $R\Gamma(\frn^0_\iota,\calF)$ is given by the complex $\calF\ov{\frn_{\iota}^0}\to \calF\ox(\frn^0_{\iota})^{\vee}$. Let $\Coh_{\frg}(\fl)^{\frn^0_\iota}$ denote the full subcategory of $\Coh_\frg(\fl)$ consisting of objects killed by $\frn^0$. For example, $\calO_{\fl}$ is fixed by the $\frn^0_{\iota}$-action so that $\calO_{\fl}\in \Coh_{\frg}(\fl)^{\frn^0_\iota}$.

For $\calF\in\Coh_{\frg}(\fl)$, define $\VB(\calF)$ to be the sheaf on $\fl$ given by 
\[
    \VB(\calF)(U):=\dlim_{K_v}H^0_{\cont}(K_v,\calF(U)\ox_{\calO_{\fl}(U)}\calO_{K^v}(U)),
\]
for every affinoid open subset $U\subset \fl$. From the definition we see that $\VB(\calF)$ is an $\calO^{\sm}_{K^v}$-module. We summarize basic properties of the functor $\VB$ as follows.
\begin{theorem}\label{thm:VB}
\begin{enumerate}[(i)]
    \item Let $\calF\in \Coh_{\frg}(\fl)$. Then $\VB(\calF)$ is an $\calO^{\sm}_{K^v}$-module, locally free of finite rank.
    \item The functor $\VB$ admits natural right derived functor $R^i\VB: \Coh_{\frg}(\fl)\ra \calO^{\sm}_{K^v}\text{-mod}$, and for any $\calF\in \Coh_{\frg}(\fl)$, we have 
    \[
        R^i\VB(\calF)=\VB(H^i(\frn^0_{\iota},\calF)).
    \]
    \item Let $\calF\in \Coh_\frg(\fl)^{\frn^0_\iota}$. Then the natural map 
    \[
        \VB(\calF)\ox_{\calO^{\sm}_{K^v}}\calO_{K^v}\to \calF\ox_{\calO_{\fl}}\calO_{K^v}.
    \]
    is an isomorphism. Conversely, if $\calF\in \Coh_{\frg}(\fl)$ and there exists an $\calO_{K^v}^{\sm}$-module $\calM$ equipped with a smooth action of $K_v$, such that 
    \[
        \calM\ox_{\calO^{\sm}_{K^v}}\calO_{K^v}\isom \calF\ox_{\calO_{\fl}}\calO_{K^v},
    \]
    then $\calM\isom \VB(\calF)$.
    \item $\VB$ is an exact tensor functor on $\Coh_\frg(\fl)^{\frn^0_\iota}$. More precisely, if
    \[
        0\to \calF_1\to \calF_2\to \calF_3\to 0
    \]
    is an exact sequence in $\Coh_{\frg}(\fl)^{\frn^0_{\iota}}$. Then 
    \[
        0\to \VB(\calF_1)\to \VB(\calF_2)\to \VB(\calF_3)\to 0
    \]
    is exact. 
    \item $\VB$ is a tensor functor on $\Coh_\frg(\fl)^{\frn^0_\iota}$ that sends $\calO_{\fl}$ to $\calO_{K^v}^{\sm}$. More precisely, if $\calF,\calG\in \Coh_{\frg}(\fl)^{\frn_\iota^0}$, then $\VB(\calF\ox_{\calO_{\fl}}\calG)=\VB(\calF)\ox_{\calO^{\sm}_{K^v}}\VB(\calG)$.
    \item For $U\subset \fl$ affinoid open, there exists a compact open subgroup $K_v\subset \GL_2(L)$ and an $\calO_{K_v}|_U$-module denoted by $\VB_{K_v}(\calF)$, such that 
    \begin{align*}
        \VB(\calF)|_U=\VB_{K_v}(\calF)\ox_{\calO_{K_v}|_U}\calO_{K^v}^{\sm}|_U
    \end{align*}
    where $\calO_{K_v}=\pi_{\HT,*}\pi_{K_v}^{-1}\calO_{\calX_{K^vK_v}}$.
\end{enumerate}
\end{theorem}
Since we already know that the universal geometric Sen operator for $\calX_{K^v}\to \calX_{K^vK_v}$ is $\frn_{\iota}^0$, the proof is almost the same as in \cite[Theorem 3.3]{Pil22}. The fact that $\VB$ is a tensor functor follows from (iii).

As some examples, we compute the value of $\VB$ in specific cases using the Hodge--Tate exact sequence.
\begin{proposition}\label{prop:RVBcalc}
\begin{enumerate}[(i)]
    \item $\VB(\omega_{\fl}^{(a,b)})=\omega_{K^v}^{(-a,-b),\sm}(-a)$ and $R^1\VB(\omega_{\fl}^{(a,b)})=\omega_{K^v}^{(-a-1,-b+1),\sm}(-a-1)$.
    \item $\VB(V_\iota^{(a,b)}\ox_C\calO_{\fl})=\omega_{K^v}^{(-b,-a),\sm}(-b)$ and  $R^1\VB(V_{\iota}^{(a,b)}\ox_C\calO_{\fl})=\omega_{K^v}^{(-a-1,-b+1),\sm}(-a-1)$.
    \item For $\eta\neq\iota$, $\VB(V_\eta^{(a,b)}\ox_C\calO_{\fl})=D_{K^v,\eta}^{(-b,-a),\sm}$ and $R^1\VB(V_\eta^{(a,b)}\ox_C\calO_{\fl})=D_{K^v,\eta}^{(-b,-a),\sm}\ox_{\calO^{\sm}_{K^v}}\Omega^{1,\sm}_{K^v}(-1)$.
\end{enumerate}
\begin{proof}
This is essentially \cite[Corollaire 3.30]{Pil22}. More specifically,
\begin{itemize}
    \item For $R\VB(V_\eta^{(a,b)}\ox_C\calO_{\fl})$,  this follows from $V_\eta^{(1,0)}\ox_C\calO_{K^v}\isom D_{K^v}^{(0,-1)}$ for $\eta\neq\iota$.
    \item For $R\VB(\omega_{\fl}^{(a,b)})$ with $(a,b)\in\bbZ^2$, this follows from $\pi_{\HT,*}\pi_{\HT}^*\omega_{\fl}^{(a,b)}\isom \omega_{K^v}^{(-a,-b)}(-a)$ and Theorem \ref{thm:VB}(ii).
    \item For $R\VB(V_\iota^{(a,b)}\ox_C\calO_{\fl})$, we first consider the case $(a,b)=(1,0)$. Consider the exact sequence on $\fl$:
    \[
        0\to \omega_{\fl}^{(0,1)}\to V^{(1,0)}\ox_C\calO_{\fl}\to \omega_{\fl}^{(1,0)}\to 0.
    \]
    Recall that $\frn^0_{\iota}$ maps $\omega_{\fl}^{(1,0)}$ isomorphically to $\omega_{\fl}^{(0,1)}$.  Thus by Theorem \ref{thm:VB}(ii), if we apply $R\VB$ to this exact sequence, we get $$\VB(\omega_{\fl}^{(0,1)})\aisom \VB(V^{(1,0)}\ox_C\calO_{\fl}),\quad  R^1\VB(V^{(1,0)}\ox_C\calO_{\fl})\aisom R^1\VB(\omega_{\fl}^{(1,0)}),$$ and the connection map $\VB(\omega_{\fl}^{(1,0)})\to R^1\VB(\omega_{\fl}^{(0,1)})$ is an isomorphism. The general case follows by taking symmetric powers and determinants.
\end{itemize}
\end{proof}
\end{proposition}

We can also evaluate $\VB$ on $\mathfrak{g}$-equivariant quasi-coherent sheaves on $\fl$. When the quasi-coherent sheaf is the $p$-adic completion of colimit of some coherent sheaves,
as we have an integral version of the decompletion method in Proposition \ref{prop:decomplete}, some results in Theorem \ref{thm:VB} are still true. Similar results are already proved in \cite[3.3]{Pan22}, \cite[Proposition 4.2]{Pil22}.

Let $\eta\in \mathrm{Hom}(L,C)$ be a field embedding. For each $n\ge 0$, consider the $G_n$-action on the sheaf $\calC^{\eta\-\an}(G_n,C)\hat\ox_C\calO_{\fl}$ given by the left translation on $\calC^{\eta\-\an}(G_n,C)$ and the natural action on $\calO_{\fl}$. This action makes $\calC^{\eta\-\an}(G_n,C)\hat\ox_C\calO_{\fl}$ a $G_n$-equivariant sheaf on $\fl$. Here, for $U$ a rational open subset of $\fl$, $(\calC^{\eta\-\an}(G_n,C)\hat\ox_C\calO_{\fl})(U)$ is defined as $\calC^{\eta\-\an}(G_n,C)\hat\ox_C\calO_{\fl}(U)$. We equip it with the $G_n$-action given by
\begin{align*}
    g.f(-)=gf(g^{-1}-g),\quad\textrm{for }g\in G_n \textrm{ and }f\in \calC^{\eta\-\an}(G_n,C)\hat\ox_{C}\calO_{\fl}.
\end{align*} 
We can take the Lie algebra action of this $G_n$-action, which induces an $\calO_{\fl}$-linear action of $\frn_{\iota}^0$ on $\calC^{\eta\-\an}(G_n,C)\hat\ox_C\calO_{\fl}$. Let $\calC^{\eta,n\-\an}:=\Gamma(\frn^0_{\iota}, \calC^{\eta\-\an}(G_n,C)\hat\ox_C\calO_{\fl})$. Finally, define 
\begin{align*}
    \calC^{\eta\-\lan}:=\dlim_n\calC^{\eta,n\-\an}=\dlim_n\Gamma(\frn^0_{\iota}, \calC^{\eta\-\an}(G_n,C)\hat\ox_C\calO_{\fl}).
\end{align*}
Note that if $\eta\neq \iota$, then $\frn_{\iota}^0$ acts trivially on $\calC^{\eta\-\an}(G_n,C)\hat\ox_C\calO_{\fl}$  and  $\calC^{\eta\-\lan}=\dlim_n\calC^{\eta\-\an}(G_n,C)\hat\ox_C\calO_{\fl}=\calC^{\eta\-\lan}(\frg,C)\hat\ox\calO_{\fl}$. Using a similar proof as \cite[Theorem 4.1]{Pil22}, we get the following result.
\begin{proposition}\label{prop:RVBiotala}
We have $\VB(\calC^{\eta\-\lan}(\frg,C)\hat\ox_C\calO_{\fl})\isom\VB(\calC^{\eta\-\lan})\isom \calO_{K^v}^{\eta\-\lan}$ and there is a canonical isomorphism
\[
    \calO_{K^v}^{\eta\-\lan}\hat\ox_{\calO^{\sm}_{K^v}}\calO_{K^v}\aisom \calC^{\eta\-\lan}\hat\ox_{\calO_{\fl}}\calO_{K^v}.
\]
\end{proposition}
See the proof for the definition of the complete tensor product over $\calO_{K^v}^{\sm}$.
\begin{proof}
We briefly explain some key steps of the proof. First of all, we claim that there exists a sufficiently small open compact subgroup $K_n\subset\GL_2(L)$ and a sheaf
\begin{align*}
    \VB_{K_n}(\calC^{\eta\-\an}(G_n,C)\hat\ox\calO_{\fl})
\end{align*}
on $\fl$ such that 
\begin{align*}
    \VB_{K_n}(\calC^{\eta\-\an}(G_n,C)\hat\ox_C\calO_{\fl})\hat\ox_{\calO_{K_n}}\calO_{K^v}\isom \calC^{\eta,n\-\an}\hat\ox_{\calO_{\fl}}\calO_{K^v}.
\end{align*}
Here the completed tensor product over $\calO_{K_n}$ is defined by taking $p$-adic completion of the usual tensor product. More precisely, for $U\in\ffrb$ a sufficiently small open affinoid subset, such that $\pi_{\HT}^{-1}(U)=V$, write $V=\Spa(B,B^+)$ with the pro-\'etale $\Gamma$-torsor $V_\infty=\Spa(B_\infty,B_\infty^+)$ over $V$. Let $V_{n,m}\subset \calC^{\eta\-\alg}(G_n,C)$ be the subspace consisting of degree $\le m$ functions, which is stable under the $G_n$-action. The Gauss norm on $G_n$ defines a $G_n$-invariant lattice $T_{n,m}$ inside $V_{n,m}$. By Proposition \ref{prop:decomplete}, there exists a sufficiently small open subgroup $K_n\subset G_n$, such that for sufficiently large integer $n'$, there exists a unique free $B_{K_n,n'}^+$-submodule $D_{K_n,n'}^+(T_n)$ of rank $\dim_C V_{n,m}$ stable under $G_n\times\Gamma$ such that 
\begin{align*}
    (B_\infty)^\circ\ox_{B_{K_n,n'}^+}D_{K_n,n'}^+(T_{n,m})\isom (B_\infty)^\circ\ox_{\calO_C}T_{n,m}
\end{align*}
and this isomorphism is equivariant for the $G_n\times \Gamma$-action. By taking direct limits on $m$ and taking $p$-adic completions, we get an isomorphism 
\begin{align*}
    (B_\infty)^\circ\hat \ox_{B_{K_n,n'}^+}D_{K_n,n'}^+(T_n)\isom (B_\infty)^\circ\hat \ox_{\calO_C}T_n.
\end{align*}
where $T_n$ is the unit ball inside $\calC^{\eta\-\an}(G_n,C)=:V_n$. After inverting $p$, we get an isomorphism
\begin{align*}
    B_\infty\hat \ox_{B_{K_n,n'}}D_{K_n,n'}(V_n)\isom B_\infty\hat\ox_{C}V_n.
\end{align*}
From the construction of $\VB$, we know 
\begin{align*}
    \VB_{K_n}(\calC^{\eta\-\an}(G_n,C)\hat\ox_C\calO_{\fl})(U)\isom H^0(\Gamma,D_{K_n,n'}(V_n))=H^0(\Gamma,H^0(\Lie\Gamma,D_{K_n,n'}(V_n))).
\end{align*}
Then after taking $\Lie\Gamma$-invariant for the Sen operator as in Theorem \ref{thm:simpson} and taking $\Gamma$-invariant, as by Theorem \ref{thm:geomSen} the action of the Sen operator comes from the pull back of $\frn^0_{\iota}$, we get 
\begin{align*}
    B\hat\ox_{B_{K_n}}\VB_{K_n}(\calC^{\eta\-\an}(G_n,C)\hat\ox_C\calO_{\fl})(U)\isom B\hat\ox_{\calO_{\fl}(U)} H^0(\frn^0_{\iota},\calC^{\eta\-\an}(G_n,C)\hat\ox_C\calO_{\fl}(U)).
\end{align*}
This proves the claim.
Taking direct limit, we get
$$\VB(\calC^{\eta\-\lan}(\frg,C)\hat\ox_C\calO_{\fl})\hat\ox_{\calO^{\sm}_{K^v}}\calO_{K^v}\aisom \calC^{\eta\-\lan}\hat\ox_{\calO_{\fl}}\calO_{K^v}.$$
This also implies that $\VB(\calC^{\eta\-\lan}(\frg,C)\hat\ox_C\calO_{\fl})$ is invariant under $\frn^0_\iota$, thus $\VB(\calC^{\eta\-\lan}(\frg,C)\hat\ox_C\calO_{\fl})\isom\VB(\calC^{\eta\-\lan})$. Finally, as
\begin{align*}
     \calO_{K^v}^{\eta\-\lan}&=\dlim_n H^0_{\cont}(K_n,\calC^{\eta\-\an}(G_n,C)\hat\ox_C\calO_{K^v} )\\
    &=\dlim_n \VB(\calC^{\eta\-\an}(G_n,C)\hat\ox_C\calO_{\fl})\\
    &=\VB(\calC^{\eta\-\lan}(\frg,C)\hat\ox_C\calO_{\fl})
\end{align*}
we get all the desired results.
\end{proof}

\begin{proposition}\label{prop:R1VB}
For any $\eta\in\Hom_{\bbQ_p\-\alg}(L,C)$, we have 
\begin{align*}
    R^i\VB(\calC^{\eta\-\lan}(\frg,C)\hat\ox_C\calO_{\fl})\isom R^i\eta\-\lan(\calO_{K^v}).
\end{align*}
Moreover,
\begin{itemize}
    \item $R^1\iota\-\lan\calO_{K^v}=0$,
    \item $R^1\eta\-\lan\calO_{K^v}\neq 0$ for $\eta\neq\iota$.
\end{itemize}
\begin{proof}
Using the definition of $\VB$, we deduce 
\begin{align*}
    R^i\VB(\calC^{\eta\-\lan}(\frg,C)\hat\ox_C\calO_{\fl})\isom \dlim_{K_v}H^i_{\cont}(K_v,\calC^{\eta\-\lan}(K_v,C)\hat\ox_C\calO_{K^v})\isom R^i\eta\-\lan(\calO_{K^v}).
\end{align*}
Here we note that if $K$ is a compact $p$-adic Lie group, and $\{M_i\}_{i\in I}$ is a filtered direct sequence, then $H^q_{\cont}(K,\dlim_i M_i)\isom \dlim_i H^q_{\cont}(K, M_i)$. Indeed, if we write $H^q_{\cont}(K,\dlim_i M_i)$ in terms of cochains, it suffices to prove that $\calC^{\cont}(K,\dlim_i M_i)\isom \dlim_i \calC^{\cont}(K, M_i)$. This follows from the fact that $K$ is topologically finitely generated.

By the locally analytic Poincar\'e lemma applied to $\calC^{\iota\-\lan}(\frg,C)$, we know that 
\[
    H^1(\frn^0_{\iota},\calC^{\iota\-\lan}(\frg,C)\hat\ox_C\calO_{\fl})=0,
\]
which implies that $R^1\iota\-\lan\calO_{K^v}=0$. Next, when $\eta\neq\iota$, by Theorem \ref{thm:VB}, we have
\[
    R^1\VB(\calC^{\eta\-\lan}(\frg,C)\hat\ox_C\calO_{\fl})=\VB(R^1\Gamma(\frn^0_{\iota},\calC^{\eta\-\lan}(\frg,C)\hat\ox_C\calO_{\fl})).
\]
As $\frn^0_{\iota}$ acts trivially on $\calC^{\eta\-\lan}(\frg,C)\hat\ox_C\calO_{\fl}$, we deduce that
\begin{align*}
\VB(R^1\Gamma(\frn^0_{\iota},\calC^{\eta\-\lan}(\frg,C)\hat\ox_C\calO_{\fl}))\isom \VB(\calC^{\eta\-\lan}(\frg,C)\hat\ox_C\Omega^1_{\fl})
\end{align*}
which shows that 
\begin{align*}
    R^1\eta\-\lan\calO_{K^v}\isom \calO_{K^v}^{\eta\-\lan}\ox_{\calO_{K^v}^{\sm}}\Omega_{K^v}^{1,\sm}\neq 0.
\end{align*}
\end{proof}
\end{proposition}

\begin{proposition}\label{prop:tensordecomposition}
For any subset of embeddings $J\subset \Hom_{\bbQ_p\-\alg}(L,C)$, there exists a natural decomposition 
\[
    \calO_{K^v}^{J\-\lan}\isom \hat\ox_{\calO^{\sm}_{K^v},\eta\in J}\calO_{K^v}^{\eta\-\lan}.
\]
\begin{proof}
By Proposition \ref{prop:Jgerm}, there is an isomorphism $\calC^{J\-\lan}(\frg,C)\isom \hat\ox_{C,\eta\in J}\calC^{\eta\-\lan}(\frg,C)$. Given two embeddings $\eta,\eta'\in \Hom_{\bbQ_p\-\alg}(L,C)$, we know there is an isomorphism 
\begin{align*}
    \calO_{K^v}\hat\ox_{\calO_{K^v}^{\sm}}\VB(\calC^{\eta\-\lan}(\frg,C)\hat\ox_C\calO_{\fl})\isom\calC^{\eta\-\lan}(\frg,C)\hat\ox_C\calO_{\fl} 
\end{align*}
and similarly for $\eta'$. This shows 
\begin{align*}
    \calO_{K^v}\hat\ox_{\calO_{K^v}^{\sm}}\VB(\calC^{\eta\-\lan}(\frg,C)\hat\ox_C\calO_{\fl})\hat\ox_{\calO_{K^v}^{\sm}}\VB(\calC^{\eta'\-\lan}(\frg,C)\hat\ox_C\calO_{\fl})\isom\calC^{\eta\-\lan}(\frg,C)\hat\ox_C\calC^{\eta'\-\lan}(\frg,C)\hat\ox_C\calO_{K^v}.
\end{align*}
Taking $K_n$-invariants (for the diagonal action) on the right hand side and taking colimits on $n$, we obtain 
\begin{align*}
    \VB(\calC^{\eta\-\lan}(\frg,C)\hat\ox_C\calO_{\fl})\hat\ox_{\calO_{K^v}^{\sm}}\VB(\calC^{\eta'\-\lan}(\frg,C)\hat\ox_C\calO_{\fl})\isom \VB(\calC^{\eta\-\lan}(\frg,C)\hat\ox_C\calC^{\eta'\-\lan}(\frg,C)\hat\ox_C\calO_{\fl}).
\end{align*}
Then the result follows by induction on embeddings in $J$, and it is a topological isomorphism by the open mapping theorem \cite[Theorem 1.1.17]{Eme17}.
\end{proof}
\end{proposition}

\subsection{A local description of $\calO_{K^v}^{\Sigma\-\lan}$}\label{locdescr}
Next, we perform some local calculation to study the structure of $\calO_{K^v}^{\lan}$ as an $\calO_{K^v}^{\sm}$-module. Since $\calO_{K^v}^{\lan}$ decomposes as a completed tensor product (see Proposition \ref{prop:tensordecomposition}) of $\calO_{K^v}^{\eta\-\lan}$ over $\calO^{\sm}_{K^v}$ for each $\eta\in\Hom_{\bbQ_p\-\alg}(L,C)$, it suffices to study the subsheaves $\calO_{K^v}^{\eta\-\lan}$. There are two cases, depending on whether $\eta\neq\iota$ or $\eta=\iota$. Roughly speaking, when $\eta=\iota$, the structure of the sheaf $\calO_{K^v}^{\iota\-\lan}$ follows a similar pattern to that of the sheaf $\calO_{K^p}^{\lan}$ in \cite[4.2.6]{Pan22}. However, when $\eta\neq\iota$, the sheaf $\calO_{K^v}^{\eta\-\lan}$ behaves like a complete tensor product of $\calO_{K^v}^{\sm}$ with $\calC^{\eta\-\lan}(\frg,C)$, the germs of locally $\eta$-analytic function at $1$.

\subsubsection*{A local description of $\calO_{K^v}^{\iota\-\lan}$}
We first consider the subsheaf $\calO_{K^v}^{\iota\-\lalg}\subset\calO_{K^v}^{\iota\-\lan}$ consisting of locally algebraic sections for the $\GL_2(L)$-action. Let $\calC^{\iota\-\alg}(\GL_2(L),C)$ be the ring of $C$-valued $\iota$-algebraic functions on $\GL_2(L)$. We have an isomorphism of $\GL_2(L)\times \GL_2(L)$-modules (given by left and right translations)
\begin{align}\label{123}
    \calC^{\iota\-\alg}(\GL_2(L),C)=\bigoplus_{(a,b)\in\bbZ^2,a\ge b}(V_\iota^{(a,b)})^*\ox V_\iota^{(a,b)}.
\end{align}
Let $\calC^{\iota\-\alg}(G_n,C)$ be the subring of $\calC^{\iota\-\an}(G_n,C)$ consisting of functions that come from restrictions of functions in $\calC^{\iota\-\alg}(\GL_2(L),C)$. Let $\calC^{\iota\-\lalg}(\frg,C)\subset \calC^{\iota\-\lan}(\frg,C)$ be the algebra of germs of locally $\eta$-algebraic function at $1$, defined as 
\[
    \calC^{\iota\-\lalg}(\frg,C)=\dlim_n\calC^{\iota\-\alg}(G_n,C).
\]
Define $\calC^{\iota\-\lalg}:=H^0(\frn^0_{\iota},\calC^{\iota\-\lalg}(\frg,C)\ox_C\calO_{\fl})$ as a filtered colimit of coherent sheaf on $\fl$, annihilated by $\frn^0_\iota$. Here each term $\calC^{\iota\-\alg}(G_n,C)$ is equipped with the natural topology on finite dimensional $C$-vector spaces and $\calC^{\iota\-\lalg}(\frg,C)$ is equipped with the direct limit topology. Since
\[  
    \VB(\calC^{\iota\-\lalg}(\frg,C)\ox_C\calO_{\fl})\isom\calO_{K^v}^{\iota\-\lalg}
\]
by definition of the locally algebraic sections \cite[Corollary 4.2.7]{Eme17} and (\ref{123}), we obtain an isomorphism 
\[
    \calO_{K^v}^{\iota\-\lalg}\ox_{\calO^{\sm}_{K^v}}\calO_{K^v}\aisom \calC^{\iota\-\lalg}\ox_{\calO_{\fl}}\calO_{K^v}
\]
compatible with $\GL_2(L)$-actions as in the locally analytic case. As $H^0(\frn^0_{\iota},V^{(a,b)}_{\iota}\ox_C\calO_{\fl})\isom \omega_{\fl}^{(b,a)}$, we have 
\[
    \calC^{\iota\-\lalg}\ox_{\calO_{\fl}}\calO_{K^v}\isom\bigoplus _{(a,b)\in\bbZ^2,a\ge b}V_\iota^{(a,b)}\ox_C\omega_{K^v}^{(a,b)}(a)
\]
Taking $G_n$-invariants and taking direct limits on $n$, we get 
\[
    \calO_{K^v}^{\iota\-\lalg}\isom \bigoplus _{(a,b)\in\bbZ^2,a\ge b}V^{(a,b)}_{\iota}\ox_C\omega_{K^v}^{(a,b),\sm}(a).
\]
The self tensor product on $\calO_{K^v}^{\iota\-\lalg}$ over $\calO_{K^v}^{\sm}$ makes it an $\calO_{K^v}^{\sm}$-algebra. The $\calO_{K^v}^{\sm}$-algebra $\calO_{K^v}^{\iota\-\lalg}$ is generated by the two generators
\begin{align*}
    \calO_{K^v}^{V^{(1,0)}_\iota\-\lalg}\isom V^{(1,0)}_{\iota}\ox\omega_{K^v}^{(1,0),\sm}(1),\quad \calO_{K^v}^{V^{(1,1)}_\iota\-\lalg}\isom V^{(1,1)}_{\iota}\ox\omega_{K^v}^{(1,1),\sm}(1).
\end{align*}
Here, for $V$ an $\eta$-algebraic representation of $\GL_2(L)$ over $C$, we define $\calO_{K^v}^{V\-\lalg}\subset \calO_{K^v}^{\eta\-\lan}$ as the image of the natural inclusion 
\begin{align*}
    \Hom_{\gl_2(L)\ox_{L,\eta}C}(V,\calO_{K^v}^{\eta\-\lan})\ox_C V\hookrightarrow \calO_{K^v}^{V\-\lalg}.
\end{align*}
More precisely, for $U\in\ffrb$, let $s$ be a local basis of $\omega_{K^v}^{(-1,0),\sm}(U)$ with $e_{1,\iota},e_{2,\iota}$ a basis of $V^{(1,0)}_{\iota}(1)$, then $\frac{e_{1,\iota}}{s}$ and $\frac{e_{2,\iota}}{s}$ are generators of $V^{(1,0)}(1)\ox\omega_{K^v}^{(1,0),\sm}(U)=\calO_{K^v}^{V^{(1,0)}_{\iota}\-\lalg}(U)$. The underlying $\calO_{K^v}^{\sm}$-module structure on $\calO_{K^v}^{V^{(1,1)}_{\iota}\-\lalg}=V^{(1,1)}_{\iota}\ox_C\omega_{K^v}^{(1,1),\sm}(1)$ is trivial, and we let $\mathrm{t}$ to be a generator of this sheaf. The $\GL_2(\bbQ_p)$-case analogue of this element is the element $\rm t$ described in \cite[3.2.1]{PanII}. Then we have 
\begin{align*}
    \calO_{K^v}^{\iota\-\lalg}(U)\isom \calO_{K^v}^{\sm}(U)[\frac{e_{1,\iota}}{s},\frac{e_{2,\iota}}{s},\mathrm{t}].
\end{align*}


In order to describe $\calO_{K^v}^{\iota\-\lan}$, we also need sections coming from the flag variety.
\begin{proposition}
The image of the natural map $\calO_{\fl}\to\calO_{K^{v}}$ lands in $\calO_{K^{v}}^{\iota\-\lan}$.
\begin{proof}
Firstly, as in \cite[4.2.6]{Pan22} we know that the natural map $\calO_{\fl}\to\calO_{K^{v}}$ lands in $\calO_{K^{v}}^{\Sigma\-\lan}$. By construction the $\GL_2(L)$ action on $\fl$ is given by the natural action of $\GL_2(C)$ via the embedding $\iota:\GL_2(L)\inj\GL_2(C)$. We see that the image of $\calO_{\fl}$ in $\calO_{K^v}^{\Sigma\-\lan}$ is killed by $\frg_{\Sigma\bs \{\iota\}}$. Therefore, $\calO_{\fl}$ is contained in $\calO_{K^v}^{\iota\-\lan}$.
\end{proof}
\end{proposition}

Let $\mathrm{t} \in \calO_{K^v}^{V^{(1,1)}_{\iota}\-\lalg}$, $e_{1,\iota},e_{2,\iota}\in \omega_{K^v}^{(-1,0),V^{(1,0)}_{\iota}\-\lalg}$ as defined before. Let $U\in\ffrb$  such that $e_{1,\iota}$ does not vanish on $U$.
\begin{itemize}
    \item Let $e_{1,\iota,n}\in \omega_{K_n}^{(-1,0)}(U)$ be such that $||e_{1,\iota}-e_{1,\iota,n}||\le p^{-n}$, so that $||\frac{e_{1,\iota}}{e_{1,\iota,n}}-1||\le p^{-n}$.
    \item Let $x=\frac{e_{2,\iota}}{e_{1,\iota}}\in\calO^{\iota\-\lan}_{K^v}(U)$, and let $x_n\in\calO_{K_n}(U)$ be such that $||x-x_n||\le p^{-n}$.
    \item Let $\mathrm{t}_n\in\calO_{K_n}(U)$ be such that $||\mathrm{t}-\mathrm{t}_n||\le p^{-n}$.
\end{itemize}

\begin{theorem}\label{thm:Oiotalalocal}
Let $U\in\ffrb$ such that $e_1$ does not vanish on $U$. Let $s\in\calO_{K^v}^{\iota\-\lan}(U)$, then there exists a sufficiently large $n$ such that
\[
    s=\sum_{i,j,k\ge 0}c_{i,j,k}(x-x_n)^i(\log(\frac{e_{1,\iota}}{e_{1,n}}))^j(\log(\frac{\rm t}{\mathrm{t}_n}))^k
\]
with $c_{i,j,k}\in\calO_{K_n}(U)$ and $p^{(i+j+k)n}c_{i,j,k}\to 0$ as $i+j+k\to \infty$. 
\begin{proof}
As the flag variety in our case is also the analytification of $\bbP^1$, we can provide a description of the section in $\calO_{K^v}^{\iota\-\lan}(U)$ using the same method in \cite[Theorem 4.3.9]{Pan22}.
\end{proof}
\end{theorem}

\begin{proposition}\label{thm:Oiotalalocal1}
Let $U\in\ffrb$ be a sufficiently small open subset. The subspace consisting of locally $\iota$-algebraic sections $\calO_{K^v}^{\iota\-\lalg}(U)$ is dense in $\calO_{K^v}^{\iota\-\lan}(U)$ in the LB topology on $\calO_{K^v}^{\iota\-\lan}(U)$.
\begin{proof}
We may assume that $e_{1,\iota}$ does not vanish on $U$. Let $s\in\calO_{K^v}^{\iota\-\lan}(U)$ be a locally $\iota$-analytic section, which we may assume of the form 
\[
    s=\sum_{i,j,k\ge 0}c_{i,j,k}(x-x_n)^i(\log(\frac{e_{1,\iota}}{e_{1,n}}))^j(\log(\frac{\mathrm{t}}{\mathrm{t}_n}))^k.
\]
As 
\begin{align*}
    \log(\frac{e_{1,\iota}}{e_{1,n}})=\sum_{m= 1}^\infty (-1)^{m-1}\frac{1}{m}(\frac{e_{1,\iota}}{e_{1,n}}-1)^m
\end{align*}
and $e_{1,\iota},e_{1,n}\in\calO_{K^v}^{\sm}(U)$, one can approximate $\log(\frac{e_{1,\iota}}{e_{1,n}})$ by $K_n$-analytic, locally $\iota$-algebraic sections by truncating the summation. Similarly $\log(\frac{\mathrm{t}}{\mathrm{t}_n})$ can be approximated by $K_n$-analytic, locally $\iota$-algebraic sections. Finally, recall that $x=\frac{e_{2,\iota}}{e_{1,\iota}}$, and $e_{1,\iota,n}\in \omega_{K^v}^{(-1,0),\sm}(U)$ is a locally $\iota$-algebraic sections such that $e_{1,\iota}/e_{1,\iota,n}\in \calO_{K^v}^{\iota\-\lalg}(U)$ is sufficiently close to $1$. Then we can write 
\[
    x=\frac{e_{2,\iota}}{e_{1,\iota}}=\frac{e_{2,\iota}/e_{1,\iota,n}}{e_{1,\iota}/e_{1,\iota,n}}=\frac{\frac{e_{2,\iota}}{e_{1,\iota,n}}}{1+(\frac{e_{1,\iota}}{e_{1,\iota,n}}-1)}=\frac{e_{2,\iota}}{e_{1,\iota,n}}\sum_{i\ge 0}(1-\frac{e_{1,\iota}}{e_{1,\iota,n}})^i,
\]
which enables us to approximate $x$ using $K_n$-analytic, locally $\iota$-algebraic sections by truncating the summation. Here $\frac{e_{2,\iota}}{e_{1,\iota,n}}\in \calO_{K^v}^{\iota\-\lalg}(U)$ as $e_{2,\iota}$ is locally $\iota$-algebraic and $e_{1,\iota,n}$ is smooth for the $K_n$-action. Thus we can approximate $s$ using $K_n$-analytic, locally $\iota$-algebraic sections for some fixed $n$, which applies that $\calO_{K^v}^{\iota\-\lalg}(U)$ is dense in $\calO_{K^v}^{\iota\-\lan}(U)$.
\end{proof}
\end{proposition}

\subsubsection*{A local description of $\calO_{K^v}^{\eta\-\lan}$ for $\eta\neq\iota$}
We consider the subsheaf $\calO^{\eta\-\lalg}_{K^v}\subset\calO_{K^v}^{\eta\-\lan}$ consisting of locally $\eta$-algebraic sections. Using the functor $\VB$, and note that $\frn_\iota^0$ acts trivially on $\calC^{\eta\-\lalg}(\frg,C)\hat\ox_C\calO_{\fl}$, we have a natural isomorphism 
\[
    \calC^{\eta\-\lalg}(\frg,C)\hat\ox_C\calO_{K^v}\isom \calO_{K^v}^{\eta\-\lalg}\hat\ox_{\calO_{K^v}^{\sm}}\calO_{K^v}.
\]
By Proposition \ref{prop:RVBcalc} we deduce that 
\[
    \calO_{K^v}^{\eta\-\lalg}\isom \bigoplus_{(a,b)\in\bbZ^2,a\ge b}V^{(a,b)}_{\eta}\ox_C D_{K^v,\eta}^{(a,b),\sm}
\]

Similar to the $\eta=\iota$ case, we have the following result.
\begin{theorem}\label{thm:Oetalalocal}
Let $U\in\ffrb$ be a sufficiently small open subset and $s\in\calO_{K^v}^{\eta\-\lan}(U)$ for some $\eta\neq \iota$, then there exists a sufficiently large $n$ and $e_1',\dots e_4'\in \calO_{K^v}^{\eta\-\lalg}(U)$ with $||e_i'||=1$ such that 
\[
    s=\sum_{i,j,k,l\ge 0}c_{i,j,k,l}e_1'^ie_2'^je_3'^ke_4'^l
\]
with $c_{i,j,k,l}\in\calO_{K_n}(U)$ and $c_{i,j,k,l}\to 0$ as $i+j+k+l\to \infty$. 
\end{theorem}

    \begin{proof}
From the proof of Theorem \ref{prop:RVBiotala},
we know that there exists a compact open subgroup $K_n\subset\GL_2(L)$, and an $\calO_{K_n}$-module $\VB_{K_n}(\calC^{\eta\-\an}(G_n,C)\hat\ox_C\calO_{K^v})$, such that 
\begin{align*}
    \VB_{K_n}(\calC^{\eta\-\an}(G_n,C)\hat\ox_C\calO_{K^v})\hat\ox_{\calO_{K_n}}\calO_{K^v}\aisom\calC^{\eta\-\an}(G_n,C)\hat\ox_C\calO_{K^v}.
\end{align*}
Let $e_{i}\in \calC^{\eta\-\alg}(G_n,C)$ with $i=1,2,3,4$ be topological generators of the $p$-adically complete $C$-algebra $\calC^{\eta\-\an}(G_n,C)$ with norm $1$. For example, one can choose $\{e_i\}$ be the function on $G_n$ sending $\begin{pmatrix}
    a& b\\c&d
\end{pmatrix}$ to $\frac{a-1}{p^n},\frac{b}{p^n},\frac{c}{p^n},\frac{d-1}{p^n}$ respectively. Then $\calC^{\eta\-\an}(G_n,C)=C\langle e_1,e_2,e_3,e_4\rangle$. Suppose that $e_i\ox 1$ (where $1\in \calO_{K^v}(U)$ is the constant function on $U$) corresponds to $\sum_{j} s_{i,j}\ox f_{i,j}$ under the above isomorphism, where $s_{i,j}\in \VB_{K_n}(\calC^{\eta\-\an}(G_n,C)\hat\ox_C\calO_{K^v})(U)$ and $f_{i,j}\in \calO_{K^v}(U)$. Note that since $\{e_i\}_{i=1,\dots,4}$ generates a finite-dimensional representation of $G_n$, it corresponds to a finite rank $\mathcal{O}_{K_n}(U)$-module under the isomorphism, thus these summations are of finite term. Choose $f_{i,j,m}\in\calO_{K^v}^{\sm}(U)$ such that $f_{i,j,m}$ is sufficiently close to $f_{i,j}$ as $\calO_{K^v}^{\sm}(U)$ is dense in $\calO_{K^v}(U)$. We may assume that these $f_{i,j,m}$ lie in $\calO_{K_n}(U)$ by shrinking $K_n$. Then $\sum_{j} s_{i,j}\ox f_{i,j,m}$ is sufficiently close to $\sum_{j} s_{i,j}\ox f_{i,j}$. Let $e_{i}'\in \calC^{\eta\-\an}(G_n,C)\hat\ox_C\calO_{K^v}(U)$ be such that they correspond to $\sum_{j} s_{i,j}\ox f_{i,j,m}$ via the above isomorphism. As $e_i'$ is sufficiently close to $e_i$ for $i=1,2,3,4$, we have 
\begin{align*}
    \calC^{\eta\-\an}(G_n,C)\hat\ox_C \calO_{K^v}=C\langle e_1,e_2,e_3,e_4\rangle \hat\ox_C \calO_{K^v}=C\langle e_1',e_2',e_3',e_4'\rangle \hat\ox_C \calO_{K^v}.
\end{align*} 
Thus
\begin{align*}
    \calO_{K^v}|_U\langle e_1',e_2',e_3',e_4'\rangle\isom \VB_{K_n}(\calC^{\eta\-\an}(G_n,C)\hat\ox_C\calO_{K^v})|_U\hat \ox_{\calO_{K_n}|_U}\calO_{K^v}|_U.
\end{align*}
By construction $e_i'$ is fixed by the diagonal action of $K_n$ on $\calC^{\eta\-\an}(G_n,C)\hat\ox_C\calO_{K^v}(U)$ (since this diagonal action corresponds to the action of $K_n$ on the $\calO_{K^v}$ component of $\VB_{K_n}(\calC^{\eta\-\an}(G_n,C)\hat\ox_C\calO_{K^v})$), thus taking $K_n$-invariants on both sides we get 
\begin{align*}
    \calO_{K_n}(U)\langle e_1',e_2',e_3',e_4'\rangle\isom \VB_{K_n}(\calC^{\eta\-\an}(G_n,C)\hat\ox_C\calO_{K^v})(U).
\end{align*}
As the direct limit over $n$ of the right hand side is $\calO_{K^v}^{\eta\-\lan}(U)$ and from the construction we know that $e_i'$ is locally algebraic, via the isomorphism 
\begin{align*}
    \VB_{K_n}(\calC^{\eta\-\alg}(G_n,C)\hat\ox_C\calO_{K^v})\hat\ox_{\calO_{K_n}}\calO_{K^v}\aisom\calC^{\eta\-\alg}(G_n,C)\hat\ox_C\calO_{K^v}
\end{align*}
From this we get the desired result.
\end{proof}

\begin{proposition}\label{prop:Oetala}
Let $U\in\ffrb$ be a sufficiently small open subset. The subspace consisting of locally $\eta$-algebraic sections $\calO_{K^v}^{\eta\-\lalg}(U)$ is dense in $\calO_{K^v}^{\eta\-\lan}(U)$ in the LB topology on $\calO_{K^v}^{\eta\-\lan}(U)$.
\begin{proof}
From the above theorem, we know that for any $s\in \calO_{K^v}^{\eta\-\lan}(U)$, there exists a sufficiently large $n$ such that $s$ can be approximated by $K_n$-$\eta$-algebraic vectors. Thus $\calO_{K^v}^{\eta\-\lalg}(U)$ is dense in $\calO_{K^v}^{\eta\-\lan}(U)$.
\end{proof}
\end{proposition}

Combining Propositions \ref{thm:Oiotalalocal1}, \ref{prop:Oetala}, and \ref{prop:tensordecomposition}, we get the following result.
\begin{corollary}
For any sufficiently small $U\in\ffrb$, $\calO_{K^v}^{\lalg}(U)$ is dense in $\calO_{K^v}^{\lan}(U)$. 
\end{corollary}

\begin{remark}
Recall the work of Berger-Colmez \cite{BC16}, where for any $p$-adic Lie extension $K_\infty/K$ with $K$ a finite extension of $\bbQ_p$ and Galois group given by a $p$-adic Lie group $\Gamma$, they showed that $\hat K_\infty^{\lan}$ is given by convergent power series of $d-1$ variables if we base change $\hat K_\infty^{\lan}$ to $C$, where $d$ is the dimension of the $p$-adic Lie group $\Gamma$.

In our geometric situation, the dimension of $\calX_{K^vK_v}$ is $1$ and we have the pro-\'etale Galois covering $\calX_{K^v}\to \calX_{K^vK_v}$ with Galois group $K_v$, which is a $p$-adic Lie group of dimension $4d$ with $d=[L:\bbQ_p]$. From the above calculation we see that every element in $\calO_{K^v}^{\lan}$ locally can be given by a power series with $(4d-1)$-variables over $\calO_{K^v}^{\sm}$, and the only non-trivial relation on $\calO_{K^v}^{\lan}$ is given by the geometric Sen operator along $\calX_{K^v}\to \calX_{K^vK_v}$. 
\end{remark}

There are various actions on the sheaf $\calO_{K^v}^{\lan}$. We study their relations, and obtain similar results as in \cite[Theorem 5.1.8]{Pan22}, \cite[Proposition 6.5.4]{PanII}. As $\frn_{\iota}^0$ acts trivially on $\calO_{K^v}^{\lan}$, it induces an action of $\frh^0_{\iota}=\frh_{\iota}\ox_C\calO_{\fl}$ on $\calO_{K^v}^{\lan}$ with $\frh^0_{\iota}=\frh_{\iota}\ox_C\calO_{\fl}$ the universal Cartan algebra. Therefore, we get a natural action of $\frh_{\iota}$ on $\calO_{K^v}^{\lan}$, which we denote by $\theta_{\frh}$. We note that this action of $\frh_{\iota}$ commutes with the $\GL_2(L)$-action. 

Besides, there is an action of the arithmetic Sen operator on $\calO_{K^v}^{\lan}$. We follow the treatment in \cite[Definition 5.1.5]{Pan22}, \cite[3.3]{PanII}. Let $U\in\ffrb$ be an affinoid open subset of $\fl$ such that $e_1$ does not vanish on $U$. Then there exists a sufficiently large finite extension $M$ of $\bbQ_p$ inside $C$, such that the smooth sections $x_n,e_{1,n},\mathrm{t}_n$ in Theorem \ref{thm:Oiotalalocal} can be defined over $M$, and $e_i'$ for $\eta\neq\iota$ and $i=1,\dots, 4$ in Theorem \ref{thm:Oetalalocal} can also be defined over $M$. Besides, the $\Gal_L$-action on $x$ is trivial, and $\Gal_L$ acts on $e_1$ and $\rm t$ through the cyclotomic character. As $U\in\ffrb$, $\pi_{\HT}^{-1}(U)$ is the preimage of an affinoid open subset $V_{K_v}$ under the projection map $\pi_{K_v}:\calX_{K^v}\to \calX_{K^vK_v}$. We assume $V_{K_v}$ is defined over $M$. We may assume $M$ contains $\mu_{p^2}$ and define $M_\infty=\cup_n M(\mu_{p^n})$. Then $\Gal(M_\infty/M)\isom \bbZ_p$. Define 
\begin{align*}
    A_{\iota,M,n}:=\{f=\sum_{i,j,k\ge 0}c_{i,j,k}(x-x_n)^i(\log(\frac{e_{1,\iota}}{e_{1,n}}))^j(\log(\frac{\rm t}{\mathrm{t}_n}))^k:f\in \calO_{K^v}^{\iota\-\lan}(U),c_{i,j,k}\text{ is fixed by }\Gal(\bar L/M)\}.
\end{align*}
This is a $\Gal(\bar L/M)$-stable subspace of $\calO_{K^v}^{\iota\-\lan}(U)$, and $\dlim_n A_{\iota,M,n}\hat\ox_{M}C\isom \calO_{K^v}^{\iota\-\lan}(U)$. Moreover, $A_{\iota,M,n}$ is fixed by  $\Gal(\bar L/M_\infty)$, and the action of $\Gal(M_\infty/M)$ on $A_{\iota,M,n}$ is analytic. This induces a Sen operator $1\in \bbZ_p\isom \Lie(\Gal(M_\infty/M))$ on $\calO_{K^v}^{\iota\-\lan}(U)$. Besides, for $\eta\neq\iota$, define 
\begin{align*}
    A_{\eta,M,n}:=\{f=\sum_{i,j,k,l\ge 0}a_{ijkl}e_{1}'^ie_{2}'^je_{3}'^ke_{4}'^l, f\in \calO_{K^v}^{\eta\-\lan}(U),a_{ijkl}\text{ is fixed by }\Gal(\bar L/M)\}.
\end{align*}
This is also a $\Gal(\bar L/M)$-stable subspace of $\calO_{K^v}^{\eta\-\lan}(U)$, and $\dlim_n A_{\eta,M,n}\hat\ox_{M}C\isom \calO_{K^v}^{\eta\-\lan}(U)$. Moreover, $A_{\eta,M,n}$ is fixed by  $\Gal(\bar L/M_\infty)$, and the action of $\Gal(M_\infty/M)$ on $A_{\eta,M,n}$ is analytic. This induces a Sen operator $1\in \bbZ_p\isom \Lie(\Gal(M_\infty/M))$ on $\calO_{K^v}^{\eta\-\lan}(U)$. Finally, the Sen operator on $\calO_{K^v}^{\eta\-\lan}(U)$ for $\eta\in\Hom_{\bbQ_p\-\alg}(L,C)$ defines a Sen operator on $\calO_{K^v}^{\lan}(U)$ by Proposition \ref{prop:tensordecomposition}. Clearly this operator is compatible when we vary $U\in\ffrb$, and is independent of choice of $M$. We call this the arithmetic Sen operator on $\calO_{K^v}^{\lan}$.

\begin{corollary}\label{cor:hSen}
$\theta_{\frh}\left( \begin{matrix} 1&0\\0&0 \end{matrix} \right)$ is the arithmetic Sen operator on $\calO_{K^v}^{\lan}$. The arithmetic Sen operator acts trivially on $\calO_{K^v}^{\iota^c\-\lan}$.
\begin{proof}
As $\theta_{\frh}\left( \begin{matrix} 1&0\\0&0 \end{matrix} \right)$ and the arithmetic Sen operator acts continuously on $\calO_{K^v}^{\lan}$, it suffices to check this result on $\calO_{K^v}^{\lalg}$, which is dense in $\calO_{K^v}^{\lan}$. For $\Hom_{\bbQ_p\-\alg}(L,C)\bs\{\iota\}$-algebraic sections, the operator on both sides acts trivially so that the equation holds trivially. For $\iota$-algebraic sections, this is similar to \cite[Thm 5.1.8]{Pan22}. More precisely, as both sides are differential operators of order $1$, it suffices to check this on generators $V_{\iota}^{(1,0)}\ox \omega_{K^v}^{(1,0),\sm}(1)$, $V_{\iota}^{(1,1)}\ox\omega_{K^v}^{(1,1),\sm}(1)$, and $V_{\eta}^{(1,0)}\ox D_{\eta}^{(1,0),\sm}$.
\end{proof}
\end{corollary}

Let $\chi:\frh_{\iota}\to C$ be a weight of $\frh_\iota$, which is given by $\chi(\left( \begin{matrix} a&0\\0&d\end{matrix} \right))=n_1a+n_2d$ for some $n_1,n_2\in C$ and $a,d\in C$. We write $\chi=(n_1,n_2)$. In this paper we will mainly focus on the case where $n_1$, $n_2$ are all integers. Let $\calO_{K^v}^{\lan,{\chi}}$ denote the isotypic part inside $\calO_{K^v}^{\lan}$ where the $\theta_{\frh}$-action is given by $\chi$. As the $\theta_{\frh}$-action on $\calO_{K^v}^{\iota^c\-\lan}$ is trivial, from the isomorphism in Proposition \ref{prop:tensordecomposition}, we get a natural decomposition
\[
    \calO_{K^v}^{\lan,\chi}\isom \calO_{K^v}^{\iota\-\lan,\chi}\hat\ox_{\calO_{K^v}^{\sm}}\calO_{K^v}^{\iota^c\-\lan}.
\]

\begin{corollary}\label{cor:Oiotalachi}
Let $U\in\ffrb$ such that $e_{1,\iota}$ is invertible on $U$. Let $\chi=(n_1,n_2)$ be a character with $n_1,n_2\in \bbZ$. For $s\in \calO_{K^v}^{\iota\-\lan,\chi}(U)$, there exists a sufficiently large $n$ such that 
\begin{align*}
    s = \mathrm{t}^{n_2}e_{1,\iota}'^{n_1-n_2}\sum_{i=0}^{+\infty}c_i(x-x_n)^i
\end{align*}
for any $i\ge 0$, $c_i\in H^0(U,\omega_{K^v}^{(n_1-n_2,0),\sm})$, and $c_i p^{ni}\to 0$ as $i\to \infty$. Here we put $e_{1,\iota}':=e_{1,\iota}(1)$.

In particular, the image of $\omega_{K^v}^{(n_1,n_2),\sm}(U)\ox_C\omega_{\fl}^{(n_1,n_2)}(n_1)(U)$ is dense in $\calO_{K^v}^{\iota\-\lan,\chi}(U)$ for $U\in\ffrb$.
\begin{proof}
For the first claim, this follows from Theorem \ref{thm:Oiotalalocal}, and explicit actions of $\theta_{\frh}(\left( \begin{matrix}     a&0\\0&d \end{matrix} \right))$ on $\mathrm{t}$, $e_{1,\iota}$, $e_{2,\iota}$. See \cite[Theorem 3.2.2]{PanII} for the modular curve case.

For the second claim, after twist by $\mathrm{t}$, we may assume that $n_1=0$. From the explicit description of sections in $\calO_{K^v}^{\iota\-\lan,\chi}$ as above, we see 
\begin{align*}
    \{e_{1,\iota}^{n_2}\sum_{i=0}^{N}c_ix^i, c_i\in H^0(U,\omega_{K^v}^{(-n_2,0),\sm})\}
\end{align*}
generates a dense subspace of $\calO_{K^v}^{\iota\-\lan,\chi}(U)$. One also obtain similar descriptions when $e_{2,\iota}$ is invertible on $U$. From this we deduce the second claim.
\end{proof}
\end{corollary}

Let $\calZ(U(\frg_{\iota}))$ be the center of the universal enveloping algebra of $\frg_\iota$, which is a $C$-algebra generated by $z=\left( \begin{matrix}1&0\\0&1\end{matrix} \right)$ and the Casimir operator $\Omega=u^+u^-+u^-u^++\frac{1}{2}h^2\in Z(U(\frg_{\iota}))$.
\begin{corollary}\label{cor:infSen}
Let $\Theta_{\Sen}$ denote the arithmetic Sen operator on $\calO_{K^v}^{\lan}$. Then as operators on $\calO_{K^v}^{\lan}$, we have 
\[
    (2\Theta_{\Sen}-z+1)^2-1=2\Omega.
\]
\begin{proof}
    This follows from Corollary \ref{cor:hSen} and the relation between the horizontal action $\theta_\frh$ and the infinitesimal action, as in \cite[Corollary 4.2.8]{Pan22}.
\end{proof}

\end{corollary}

Let $\tilde{\chi}_k$ denote the infinitesimal character for $\Sym^k V_\iota^*$. Using the relations of the arithmetic Sen operator and the actions of $\calZ(U(\frg_{\iota}))$ Corollary \ref{cor:hSen}, we get the following results. 
\begin{corollary}\label{cor:infdecomp}
There is a natural $\GL_2(L)$ and Galois equivariant decomposition 
\[
    \calO_{K^v}^{\lan,\tilde{\chi}_k}= \calO_{K^v}^{\lan,(0,-k)}\oplus \calO_{K^v}^{\lan,(-k-1,1)}.
\]
\end{corollary}

We will also need certain twists of $\calO_{K^v}^{\lan,(a,b)}$, for some $(a,b)\in\bbZ^2$. The following treatment is parallel to \cite[3.1.2]{PanII}.
\begin{proposition}\label{prop:twist}
Let $(a,b),(a',b')\in\bbZ^2$ be integers. We have 
\begin{enumerate}[(i)]
    \item $\calO_{K^v}^{\lan,(a,b)}\ox_{\calO_{K^v}^{\sm}}\omega_{K^v}^{(a',b'),\sm}\isom \omega_{K^v}^{(a',b'),\lan,(a,b)}$.
    \item $\calO_{K^v}^{\lan,(a,b)}\ox_{\calO_{\fl}}\omega_{\fl}^{(a',b')}\isom \omega_{K^v}^{(-a',-b'),\lan,(a+b',b+a')}(-a')$.
\end{enumerate}
where $\omega_{K^v}^{(a',b'),\lan,(a,b)}$ is the subsheaf of $\omega_{K^v}^{(a',b'),\lan}$ such that $\theta_\frh$ acts via  the character $(a,b)$.
\begin{proof}
For the first isomorphism, this follows from $\omega_{K^v}^{(a',b')}\isom \omega_{K^v}^{(a',b'),\sm}\ox_{\calO_{K^v}^{\sm}}\calO_{K^v}$, and $\theta_{\frh}(\frh_{\iota})$ acts trivially on $\omega_{K^v}^{(a',b'),\sm}$. For the second isomorphism, we note that the $\GL_2(L)$-action on $\omega_{\fl}^{(a',b')}$ is locally $\iota$-analytic, and $\pi_{\HT,*}\pi_{\HT}^{*}\omega_{\fl}^{(a',b')}=\omega_{K^v}^{(-a',-b')}(-a')$ and $\theta_{\frh}(\frh_{\iota})$ acts on $\omega_{\fl}^{(a',b')}$ via the weight $(a',b')$.
\end{proof}
\end{proposition}

\subsection{Completed cohomology of unitary Shimura curves}
Firstly we recall the definition of the completed cohomology of unitary Shimura curves introduced by Emerton \cite{Eme06}. Let $(G,X)$ be the Shimura datum such that for $K\subset G(\bbA^\infty)$ sufficiently small,
\[
    X_K(\bbC)=G(\bbQ)\bs (G(\bbA^\infty)\times X/K)
\]
is the set of complex points of the unitary Shimura curve $X_K$. Let $K^v\subset G^v$ be a sufficiently small open compact subgroup as before. Set 
\[
    \tilde{H}^i(K^v,\bbZ/p^n):=\dlim_{K_v\subset \GL_2(L)}H^i(\calX_{K^vK_v}(\bbC),\bbZ/p^n)
\]
and define the \emph{completed cohomology} of the unitary Shimura curves of tame level $K^v$ to be
\[
    \tilde{H}^i(K^v,\bbZ_p):=\ilim_n\dlim_{K_v\subset \GL_2(L)}H^i(X_{K^vK_v}(\bbC),\bbZ/p^n).
\]
The completed cohomology group carries many natural actions.
\begin{itemize}
    \item It has a natural action of $\GL_2(L)$. Moreover, as $\calX_{K^vK_v}$ is a curve, $\tilde{H}^1(K^v,\bbZ_p)$ is torsion-free. Set 
    \[
        \tilde{H}^1(K^v,\bbQ_p):=\tilde{H}^1(K^v,\bbZ_p)[\frac{1}{p}].
    \]
    This is an admissible unitary Banach representation of $\GL_2(L)$, with $\tilde{H}^1(K^v,\bbZ_p)$ the unit ball. 
    \item As $X_{K^vK_v}$ is defined over $F$, we can equip $\tilde{H}^i(K^v,\bbZ_p)$ with a natural action of $\Gal_F$.
    \item Let $\bbT(K^vK_v)\subset \End_{\bbZ_{p}}(H^1(X_{K^vK_v}(\bbC),\bbZ_{p}))$ be the usual unramified Hecke algebra, given by the image of the abstract Hecke algebra $\bbT^S:=\bbZ[K^S\bs G(\bbA^{\infty,S})/K^S]$, where $S$ is a finite set of primes containing $p$, such that $K^v=K^{S}K^v_S$ with $K^S\subset G(\bbA^{\infty,S})$ a product of hyperspecial subgroups. Then we define the big Hecke algebra of tame level $K^v$ to be
    \[
        \bbT(K^v):=\ilim_{K_v}\bbT(K^vK_v).
    \]
    It acts faithfully on $\tilde{H}^1(K^v,\bbZ_p)$.
\end{itemize}
These actions on $\tilde{H}^1(K^v,\bbZ_p)$ all commute with each other, and satisfy the following relations.
\begin{itemize}
    \item The Eichler--Shimura relation. More precisely, let $\lambda:\bbT(K^v)\to E$ be a $\bbZ_p$-algebra homomorphism. By \cite{carayol86, carayol1994formes, Che}, we can associate a representation $\rho:\Gal_F\to \GL_2(E)$ when $\rho$ is absolutely irreducible, and we have 
    \begin{align*}
        \Hom_{\Gal_F}(\rho,\tilde{H}^1(K^v,E))=\Hom_{\Gal_F}(\rho,\tilde{H}^1(K^v,E)[\lambda]).
    \end{align*}
    By the main result of \cite{BLR}, we have 
    \begin{align*}
        \tilde{H}^1(K^v,E)[\lambda]\isom\rho\ox_E\Hom_{\Gal_F}(\rho,\tilde{H}^1(K^v,E)).
    \end{align*}
    In the modular curve case, this was discussed in \cite[6.1.1]{Pan22}.
    \item The infinitesimal character and Hodge--Tate weights. Suppose that $\rho|_{\Gal_L}$ is of Hodge--Tate weight $\{a_{\sigma},b_{\sigma}\}_{\sigma\in\Sigma}$ with $a_\sigma\ge b_\sigma$, by \cite[Theorem 1.4]{DPSinf}, the $\GL_2(L)$-representation $\Hom_{\Gal_F}(\rho,\tilde{H}^1(K^v,E)^{\lan})$ have the same infinitesimal character as the Verma module $U(\gl_2\ox_{\bbQ_p}E)\ox_{U(\frb\ox_{\bbQ_p}E)}\mu$ with $\mu=(a_\sigma-1,b_\sigma)_{\sigma\in\Sigma}$.
\end{itemize}

We note that by \cite[Corollary 6.1.7]{camargo2024locallyanalyticcompletedcohomology}, for any $i\ge 0$, there is a natural isomorphism 
\[
    \tilde{H}^i(K^v,\bbQ_p)\isom H^i_{\pro\et}(\calX_{K^v},\bbQ_p).
\]
This isomorphism is equivariant for the $\Gal_L$-action. 

\begin{theorem}\label{11}
There is a $\GL_2(L)\times \Gal_L$-equivariant and Hecke equivariant isomorphism
\begin{align*}
    \tilde{H}^1(K^v,C)\isom H^1(\fl,\calO_{K^v}).
\end{align*}
\end{theorem}
Here, the action of $\Gal_L$ on $\calO_{K^v}$ is induced by the $L$-structure of $\calX_{K^vK_v}$, with $K_v\subset\GL_2(L)$ open compact subgroups.
\begin{proof}
The proof is similar to \cite[Theorem 4.4.6]{Pan22}, which follows from the primitive comparison theorem \cite[Theorem 1.3]{Sch13} and the ``affiness" of $\pi_{\HT}$.
\end{proof}

Moreover, we have the following theorem.
\begin{theorem}\label{thm:JlacommutesH1}
Let $J\subset\Hom_{\bbQ_p\-\alg}(L,C)$ be a subset of embeddings containing $\iota$. Let $V$ be an algebraic representation of $\GL_2(L)$ such that $\gl_2(L)\ox_{L,\eta}C$ acts trivially on $V$ for any $\eta\in J$. We have a $\GL_2(L)\times\Gal_L$-equivariant and Hecke equivariant isomorphism
\begin{align*}
    H^1(\fl,\calO_{K^v})^{J\-\lan,V\-\lalg}=H^1(\fl,\calO_{K^v}^{J\-\lan,V\-\lalg}).
\end{align*}
\begin{proof}
Let $\{G_n\}_{n\ge 0}$ be a sequence of standard open compact subgroup of $\GL_2(L)$. For any $C$-Banach space $M$ with a continuous action of $\GL_2(L)$, we have
\begin{align*}
    M^{J\-\lan,V\-\lalg}
    &\cong \dlim_n ((M\hat \otimes \calC^{J\-\an}(G_n,C))\otimes V^*)^{G_n}\otimes V\\
    &\cong \dlim_n ((M\hat \otimes \calC^{J\-\an}(G_n,C))\otimes V^*\otimes V)^{G_n},
\end{align*}
where we equip $V^*$ with the natural action of $G_n$ and $V$ with the trivial action.
As the local systems associated to the algebraic representations of $\GL_2(L)$ on $\calX_{K^vK_v}$ is trivial on $\calX_{K^v}$ (which is quasi-compact and quasi-separated), and $\calC^{J\-\an}(G_n,C)$ are completions of algebraic representations of $\GL_2(L)$ for $n\ge 0$, by the primitive comparison theorem or \cite[Theorem 4.4.6]{Pan22}
we get 
\begin{align*}
    R\Gamma_{\pro\et}(\calX_{K^v},C)\hat\ox_C \calC^{J\-\an}(G_n,C)\ox V^*\ox V&\isom R\Gamma_{\pro\et}(\calX_{K^v},\calO_{\calX_{K^v}})\hat\ox_C\calC^{J\-\an}(G_n,C)\ox_C V^*\ox V \\
    &\isom R\Gamma_{\pro\et}(\calX_{K^v},\calC^{J\-\an}(G_n,C)\ox_C V^*\ox V\hat\ox_C \calO_{\calX_{K^v}})\\
    &\cong R\Gamma_{\an}(\calX_{K^v},\calC^{J\-\an}(G_n,C)\ox_C V^*\ox V\hat\ox_C \calO_{\calX_{K^v}})\\
    &\cong R\Gamma_{\an}(\fl,\calC^{J\-\an}(G_n,C)\ox_C V^*\ox V\hat\ox_C \calO_{{K^v}}).
\end{align*}
where we equip $V^*$ with the natural action of $G_n$ and $V$ with the trivial action, the second step follows from the fact that completed tensor product with Banach space is exact (see \cite[Theorem 1]{gruson1966theorie}), the third step follows from the almost acyclicity for affinoid perfectoid spaces and the last step follows from $R\pi_{\HT,*}\calO_{\calX_{K^v}}=\calO_{K^v}$. Here $R\Gamma_{\an}(\fl,-)$ denotes the cohomology on the analytic site of $\fl$. Therefore, taking group cohomology we get
\begin{align*}
    R\Gamma_{\cont}(G_n, R\Gamma_{\pro\et}(\calX_{K^v},C)\hat\ox_C \calC^{J\-\an}(G_n,C)\ox_C V^*\ox V)\\\isom R\Gamma_{\cont}(G_n, R\Gamma_{\an}(\fl,\calC^{J\-\an}(G_n,C)\ox_C V^*\ox V\hat\ox_C \calO_{{K^v}})).
\end{align*}
We compute $\dlim_n$ on both sides. Firstly,
\begin{align*}
    &\dlim_n R\Gamma_{\cont}(G_n, R\Gamma_{\pro\et}(\calX_{K^v},C)\hat\ox_C \calC^{J\-\an}(G_n,C)\ox_C V^*\ox_C V)\\\cong &RJ\-\lan (R\Gamma_{\pro\et}(\calX_{K^v},C)\ox_C V^*\ox_C V).
\end{align*}
As
$
    H^i_{\pro\et}(\calX_{K^v},\bbQ_p)\isom \tilde{H}^i(K^v,\bbQ_p)
$ is an admissible $\GL_2(L)$-representation, by Corollary \ref{cor:Jlala} and \cite[Theorem 2.2.3]{Pan22} 
\begin{align*}
    &RJ\-\lan (R\Gamma_{\pro\et}(\calX_{K^v},C)\ox_C V^*\otimes V)\\
    \cong& R\Gamma(\frg_{{J^c}},R\lan (R\Gamma_{\pro\et}(\calX_{K^v},C))\ox_C V^*\otimes V)\\
    \cong& R\Gamma(\frg_{{J^c}}, R\Gamma_{\pro\et}(\calX_{K^v},C)^{\lan}\ox_C V^*\otimes V),
\end{align*}
where $J^c=\Hom_{\bbQ_p\-\alg}(L,C)\bs J$, and $\frg_{J^c}:=\ox_{\eta\in J^c}\gl_2(L)\ox_{L,\eta}C$.
Secondly, by Lemma \ref{Cech} together with Cech computations, we deduce
\begin{align*}
    &\dlim_n R\Gamma_{\cont}(G_n,R\Gamma_{\an}(\fl,\calC^{J\-\an}(G_n,C)\hat\ox_C \calO_{{K^v}}\ox_C V^*\ox_C V))\\
    \isom &\dlim_n R\Gamma_{\an}(\fl,R\Gamma_{\cont}(G_n,\calC^{J\-\an}(G_n,C)\hat\ox_C \calO_{{K^v}}\ox_C V^*\ox_C V))\\\cong&R\Gamma_{\an}(\fl,\dlim_n R\VB(\calC^{J\-\an}(G_n,C)\hat\ox_C\calO_{\fl}\ox_C V^*\ox_C V))\\
    \cong&R\Gamma_{\an}(\fl,\calO_{K^v}^{J\-\lan,V\-\lalg}),
\end{align*}
where the last step follows by Proposition \ref{prop:RVBiotala} (which only stated for $J=\iota$, but the result holds for general $J$ contains $\iota$ and the proof dose not change). From these two parts, we get an isomorphism 
\begin{align*}
    R\Gamma(\frg_{{J^c}}, R\Gamma_{\pro\et}(\calX_{K^v},C)^{\lan}\ox_C V^*\ox_C V)\isom R\Gamma_{\an}(\fl,\calO_{K^v}^{J\-\lan,V\-\lalg}).
\end{align*}
Taking $H^1$, we get a natural map 
\begin{align*}
    H^1(\fl,\calO_{K^v}^{J\-\lan,V\-\lalg})\to H^0(\frg_{{J^c}},H^1_{\pro\et}(\calX_{K^v},C)^{\lan}\ox_C V^*\ox_C V)\cong H^1_{\pro\et}(\calX_{K^v},C)^{J\-\lan,V\-\lalg},
\end{align*}
whose kernel and cokernel are controlled by $H^i(\frg_{{J^c}},H^0_{\pro\et}(\calX_{K^v},C)^{\lan}\ox_C V^*\otimes V)$ with $i=1,2$. As the action of $\GL_2(L)$ on $H^0_{\pro\et}(\calX_{K^v},C)\isom \tilde{H}^0(K^v,C)$ factors through determinants, and the action of the Lie algebra of the center of $\GL_2(L)$ on $\tilde{H}^0(K^v,C)$ has no higher cohomology, by Kunneth formula we see that  
\begin{align*}
    H^1(\frg_{{J^c}}, H^0_{\pro\et}(\calX_{K^v},C)\ox_C V^*\otimes V) = H^2(\frg_{{J^c}},H^0_{\pro\et}(\calX_{K^v},C)\ox_C V^*\otimes V) = 0
\end{align*}
because
\begin{align*}
    H^1(\frg_{{J^c}}, \calC^{\lan}(\calO_L,C)\ox_C V^*\otimes V) = H^2(\frg_{{J^c}},\calC^{\lan}(\calO_L,C)\ox_C V^*\otimes V) = 0.
\end{align*}
Therefore, we deduce 
\begin{align*}
    H^1(\fl,\calO_{K^v}^{J\-\lan,V\-\lalg})\isom H^1_{\pro\et}(\calX_{K^v},C)^{J\-\lan,V\-\lalg}\isom \tilde H^1(K^v,C)^{J\-\lan,V\-\lalg}\isom H^1(\fl,\calO_{K^v})^{J\-\lan,V\-\lalg}
\end{align*}
as desired.
\end{proof}
\end{theorem}

We also show $H^1(\fl,\calO_{K^v}^{J\-\lan,J^c\-\lalg})$ can be computed using Cech complexes with coverings in $\ffrb$. This amounts to prove:
\begin{lemma}\label{Cech}
Let $J\subset\Hom_{\bbQ_p\-\alg}(L,C)$ be a set of embeddings such that $\iota\in J$. Then the Cech cohomology group $\check H^j(U,\calO_{K^v}^{J\-\lan,J^c\-\lalg})=0$ for any $U\in\ffrb$ and $j\ge 1$.
\begin{proof}
Let $U\in\ffrb$ and let $\ffru\subset\ffrb$ be a cover of $U$. As $\pi_{\HT}^{-1}(U)$ is affinoid perfectoid for any $U\in\ffrb$, the augmented Cech complex $\calO_{K^v}(U)\to \check{C}(\ffru,\calO_{K^v})$ is acyclic. From Theorem \ref{thm:VB}(iv), we deduce the augmented Cech complex
\begin{align*}
    \calO_{K^v}(U)^{J\-\lan,J^c\-\lalg}\to \check{C}(\ffru,\calO_{K^v})^{J\-\lan,J^c\-\lalg}=\check{C}(\ffru,\calO_{K^v}^{J\-\lan,J^c\-\lalg})
\end{align*}
is also acyclic. This shows $\check H^j(U,\calO_{K^v}^{J\-\lan,J^c\-\lalg})=0$ for $j\ge 1$. (Alternatively one can also use similar arguments as in \cite[Proposition 3.2.11]{PanII} together with the power series expansion description of $\calO_{K^v}^{J\-\lan,J^c\-\lalg}$.)
\end{proof}
\end{lemma}

\begin{theorem}\label{thm:JlalgcommutesH1}
Let $V_{\iota}=V^{(a,b)}_{\iota}$ with $a,b\in\bbZ^2$ be an $\iota$-algebraic representation of $\GL_2(L)$. There is a $\GL_2(L)\times \Gal_L$-equivariant and Hecke equivariant exact sequence 
\begin{align*}
    0\to &H^1(\fl,\omega_{K^v}^{(-b,-a),\sm}\ox_{\calO_{K^v}^{\sm}}\calO_{K^v}^{\iota^c\-\lan})(-b)\\
    &\to \Hom_{\frg_{\iota}}(V_{\iota}^{(a,b)},H^1(\fl,\calO_{K^v}^{\lan}))\to H^0(\fl,\omega_{K^v}^{(-a-1,-b+1),\sm}\ox_{\calO_{K^v}^{\sm}}\calO_{K^v}^{\iota^c\-\lan})(-a-1)\to 0
\end{align*}
\begin{proof}
Our strategy is similar to Theorem \ref{thm:JlacommutesH1}. Let $\{G_n\}_{n\ge 0}$ be a sequence of standard open compact subgroup of $\GL_2(L)$. First of all, by the primitive comparison theorem, we get 
\begin{align*}
    &R\Gamma_{\pro\et}(\calX_{K^v},C)\hat\ox_C\calC^{\iota^c\-\an}(G_n,C)\ox_C V_{\iota}^*\\
    \isom & R\Gamma_{\pro\et}(\calX_{K^v},\calC^{\iota^c\-\an}(G_n,C)\ox_C V_{\iota}^*\hat\ox_C\calO_{\calX_{K^v}})\\
    \isom & R\Gamma_{\an}(\calX_{K^v},\calC^{\iota^c\-\an}(G_n,C)\ox_C V_{\iota}^*\hat\ox_C\calO_{\calX_{K^v}}).
\end{align*}
Therefore, 
\begin{align*}
    &\dlim_n R\Gamma_{\cont}(G_n,R\Gamma_{\pro\et}(\calX_{K^v},C)\hat\ox_C C^{\iota^c\-\an}(G_n,C)\ox_C V_{\iota}^*))\\
    \isom & \dlim_n R\Gamma_{\an}(\calX_{K^v},R\Gamma_{\cont}(G_n,\calC^{\iota^c\-\an}(G_n,C)\hat\ox_C\calO_{\calX_{K^v}}\ox_C V_{\iota}^*).
\end{align*}
Firstly, 
\begin{align*}
    &\dlim_n R\Gamma_{\an}(\calX_{K^v},R\Gamma_{\cont}(G_n,\calC^{\iota^c\-\an}(G_n,C)\ox_C V_{\iota}^*\hat\ox_C \calO_{\calX_{K^v}}))\\=&R\Gamma_{\an}(\fl,\dlim_n R\VB(\calC^{\iota^c\-\an}(G_n,C)\hat\ox_C\calO_{\fl}\ox_C V_{\iota}^*))\\
    =&R\Gamma_{\an}(\fl, R\VB(\calC^{\iota^c\-\lan}(\frg,C)\hat\ox_C\calO_{\fl}\ox_C V_{\iota}^*)).
\end{align*}
As 
\begin{align*}
    H^0(R\VB(\calC^{\iota^c\-\lan}(\frg,C)\hat\ox_C\calO_{\fl}\ox_C V_{\iota}^*))&\isom \omega_{K^v}^{(-b,-a),\sm}\ox_{\calO_{K^v}^{\sm}}\calO_{K^v}^{\iota^c\-\lan}(-b)\\
    H^1(R\VB(\calC^{\iota^c\-\lan}(\frg,C)\hat\ox_C\calO_{\fl}\ox_C V_{\iota}^*))&\isom 
    \omega_{K^v}^{(-a-1,-b+1),\sm}\ox_{\calO_{K^v}^{\sm}}\calO_{K^v}^{\iota^c\-\lan}(-a-1)
\end{align*}
and 
\begin{align*}
    H^2(\fl,\omega_{K^v}^{(-b,-a),\sm}\ox_{\calO_{K^v}^{\sm}}\calO_{K^v}^{\iota^c\-\lan}(-b))=0,
\end{align*}
we see that there is an exact sequence 
\begin{align*}
    0\to &H^1(\fl,\omega_{K^v}^{(-b,-a),\sm}\ox_{\calO_{K^v}^{\sm}}\calO_{K^v}^{\iota^c\-\lan})(-b)\\
    &\to H^1(R\Gamma_{\an}(\fl, R\VB(\calC^{\iota^c\-\lan}(\frg,C)\hat\ox_C\calO_{\fl}\ox_C V_{\iota}^*)))\\
    &\to H^0(\fl,\omega_{K^v}^{(-a-1,-b+1),\sm}\ox_{\calO_{K^v}^{\sm}}\calO_{K^v}^{\iota^c\-\lan})(-a-1)\to 0.
\end{align*}
Next, 
\begin{align*}
    &\dlim_n R\Gamma_{\cont}(G_n, R\Gamma_{\pro\et}(\calX_{K^v},C)\hat\ox_C \calC^{\iota^c\-\an}(G_n,C)\ox_C V^*_{\iota}))\\=&R\iota^c\-\lan (R\Gamma_{\pro\et}(\calX_{K^v},C)\ox_C V_{\iota}^*)
    \\=&R\Gamma(\frg_{\iota},R\lan(R\Gamma_{\pro\et}(\calX_{K^v},C)\ox_C V_{\iota}^*)).
\end{align*}
Using similar methods of Theorem \ref{thm:JlacommutesH1}, we can show that 
\begin{align*}
    H^1(R\Gamma(\frg_{\iota},R\lan(R\Gamma_{\pro\et}(\calX_{K^v},C)\ox_C V_{\iota}^*)))    \isom \Hom_{\frg_{\iota}}(V_{\iota},H^1(\fl,\calO_{K^v}^{\lan})).
\end{align*}
From this we deduce the result.
\end{proof}
\end{theorem}

\begin{remark}
We sketch a more direct and clearer proof of the following result:
\begin{align*}
    \tilde{H}^1(K^v,C)^{\iota\-\lan}_{\frm}\isom H^1(\fl,\calO_{K^v}^{\iota\-\lan})_{\frm}.
\end{align*}
where $(-)_{\frm}$ denotes localization at a non-Eisenstein maximal ideal of the Hecke algebra $\bbT^S$. By Theorem \ref{thm:Oetalalocal}, we know that $R\Gamma(\frg_{\iota^c},\calO^{\lan}_{K^v})=\calO^{\iota\-\lan}_{K^v}[0]$. Hence the Chevalley-Eilenberg complex $\calO^{\lan}_{K^v}\ox \wedge^\bullet\frg_{\iota^c}^*$ gives a resolution of $\calO^{\iota\-\lan}_{K^v}$. As $\tilde H^0(K^v,C)^{\lan}_{\frm}=0$, we see $H^1(\fl,-)_{\frm}$ of the Chevalley-Eilenberg complex $\calO^{\lan}_{K^v}\ox \wedge^\bullet\frg_{\iota^c}^*$ remains a resolution of $H^1(\fl,\calO_{K^v}^{\iota\-\lan})_{\frm}$. From this we deduce that 
\begin{align*}
    H^1(\fl,\calO_{K^v}^{\iota\-\lan})_{\frm}=H^1(\fl,\calO_{K^v}^{\lan})^{\frg_{\iota^c}}_{\frm}
\end{align*}
which is as we want.
\end{remark}

\section{Intertwining operator}\label{intertw}
Recall that in Corolloary \ref{cor:infdecomp} we have a natural decomposition of the $\tilde{\chi}_k$-isotypic part in $\calO_{K^v}^{\lan}$, specifically
\[
    \calO_{K^v}^{\lan,\tilde{\chi}_k}= \calO_{K^v}^{\lan,(0,-k)}\oplus \calO_{K^v}^{\lan,(-k-1,1)}.
\]
Here $\tilde{\chi}_k$ is the infinitesimal character of the \emph{dual} of the $k$-th symmetric power of the standard representation. In this section, we aim to construct a map
\[
    I:\calO_{K^v}^{\lan,(0,-k)}\to\calO_{K^v}^{\lan,(-k-1,1)}(k+1),
\]
which is equivariant for the $\GL_2(L)$-action, the Galois action, and the Hecke action. We will call this the intertwining operator, in the same spirit as \cite[Section 4.4]{PanII}. This intertwining operator also restricts to a natural map 
\begin{align*}
    I:\calO_{K^v}^{\iota\-\lan,\iota^c\-\lalg,(0,-k)}\to\calO_{K^v}^{\iota\-\lan,\iota^c\-\lalg,(-k-1,1)}(k+1),
\end{align*}
where $\calO_{K^v}^{\iota\-\lan,\iota^c\-\lalg}\subset \calO_{K^v}^{\lan}$ is the subsheaf consisting of locally $\iota$-analytic, locally $\iota^c$-algebraic sections. We will first define some differential operators on $\calO_{K^v}^{\lan,(0,-k)}$. The intertwining operator will be the composition of some of them. The behavior of these differential operators varies across different Newton strata of unitary Shimura curves. We will recall some descriptions of these strata and study the differential operators on each of the strata.

\subsection{Differential operators}\label{int}
Let $k\ge 0$ be an integer. We will construct two differential operators on $\calO_{K^v}^{\lan,(0,-k)}$. Roughly speaking, we reproduce the constructions in \cite[Section 4]{PanII}, but replace $\calO_{K^p}^{\sm}$ by $\calO_{K^v}^{\iota^c\-\lan}$ everywhere. Heuristically, $\calO_{K^v}^{\iota^c\-\lan}$ consists of sections that are ``smooth'' under the locally $L$-analytic action of $\GL_2(L)$.
\begin{proposition}\label{prop:deta}
Suppose $\eta\neq\iota$. There exists a $\GL_2(L)$-equivariant continuous map 
\[
    d_\eta:\calO_{K^v}^{\eta\-\lan}\to \calO_{K^v}^{\eta\-\lan}\ox_{\calO_{K^v}^{\sm}}\Omega_{K^v}^{1,\sm}
\]
which is the unique extension of the Gauss--Manin connection on $\calO_{K^v}^{\eta\-\lalg}$.
\begin{proof}
We use the power series description of $\calO_{K^v}^{\eta\-\lan}$ as described in Theorem \ref{thm:Oetalalocal}. For $U\in\ffrb$ sufficiently small and $s\in \calO_{K^v}^{\eta\-\lan}(U)$, there exists sections $e_1',e_2',e_3',e_4'\in \calO_{K^v}^{\lalg}$ for $i=1,\dots,4$, such that $s=\sum_{i,j,k,l}c_{ijkl}e_1'^ie_2'^je_3'^k e_4'^l$ where $c_{ijkl}\in \calO_{K_n}(U)$ for some sufficiently large $n$ and $c_{ijkl}\rightarrow 0$ when $i+j+k+l\rightarrow\infty$. Define $d_\eta(s)=\lim_{n}\nabla(\sum_{i+j+k+l\le n}c_{ijkl}e_1'^ie_2'^je_3'^k e_4'^l)$ where $\nabla$ is the Gauss--Manin connection on $\calO_{K^v}^{\eta\-\lalg}$. As $d:\calO_{K_n}(U)\to \Omega_{K_n}^{1,\sm}(U)$ is a continuous operator, by the Leibniz rule we know that $d_\eta(s)$ is a well-defined section in $(\calO_{K^v}^{\eta\-\lan}\ox_{\calO_{K^v}^{\sm}}\Omega_{K^v}^{1,\sm})(U)$, which gives us the differential operator 
\[
    d_\eta:\calO_{K^v}^{\eta\-\lan}\to \calO_{K^v}^{\eta\-\lan}\ox_{\calO_{K^v}^{\sm}}\Omega_{K^v}^{1,\sm}.
\]
Finally, from the construction we see that $d_\eta$ is uniquely determined by its restriction on the locally $\eta$-algebraic part.
\end{proof}
\end{proposition}

\begin{corollary}\label{cor:diotac}
There exists a $\GL_2(L)$-equivariant continuous map 
\[
    d_{\iota^c}:\calO_{K^v}^{\iota^c\-\lan}\to \calO_{K^v}^{\iota^c\-\lan}\ox_{\calO_{K^v}^{\sm}}\Omega_{K^v}^{1,\sm}
\]
which is the unique extension of the Gauss--Manin connection on $\calO_{K^v}^{\iota^c\-\lalg}$.
\begin{proof}
By Proposition \ref{prop:tensordecomposition}, we know that 
\begin{align*}
    \calO_{K^v}^{\iota^c\-\lan}\isom \hat\ox_{\calO_{K^v}^{\sm},\eta\neq\iota}\calO_{K^v}^{\eta\-\lan}.
\end{align*}
Then we can define $d_{\iota^c}$ on $\calO_{K^v}^{\iota^c\-\lan}$ using the Leibniz rule for $d_{\eta}$ on each $\calO_{K^v}^{\eta\-\lan}$ in Proposition \ref{prop:deta}.
\end{proof}
\end{corollary}

\begin{theorem}\label{thm:d}
There exists a unique continuous operator 
\[
    d^{k+1}:\calO_{K^v}^{\lan,(0,-k)}\to \calO_{K^v}^{\lan,(0,-k)}\ox_{\calO_{K^v}^{\sm}}(\Omega_{K^v}^{1,\sm})^{\ox k+1}
\]
such that it is determined by the following properties:
\begin{enumerate}[(i)]
    \item $d^{k+1}$ is $\calO_{\fl}$-linear.
    \item The restriction of $d^{k+1}$ on $\calO_{K^v}^{\lalg,(0,-k)}$ is given by the Gauss--Manin connection.
\end{enumerate}
Moreover, it restricts to a continuous operator 
\[
    d^{k+1}:\calO_{K^v}^{\iota\-\lan,\iota^c\-\lalg,(0,-k)}\to \calO_{K^v}^{\iota\-\lan,\iota^c\-\lalg,(0,-k)}\ox_{\calO_{K^v}^{\sm}}(\Omega_{K^v}^{1,\sm})^{\ox k+1}.
\]
\begin{proof}
First we handle the case where $k=0$. For $s\ox t\in \calO_{K^v}^{\lan,(0,0)}\isom \calO_{K^v}^{\iota\-\lan,(0,0)}\hat\ox_{\calO_{K^v}^{\sm}}\calO_{K^v}^{\iota^c\-\lan}$, put 
\begin{align*}
    d^1(s\ox t):=d_{\iota}(s)\ox t+s\ox d_{\iota^c}(t)
\end{align*}
where $d_{\iota}:\calO_{K^v}^{\iota\-\lan,(0,0)}\to \calO_{K^v}^{\iota\-\lan,(0,0)}\ox_{\calO_{K^v}^{\sm}}\Omega_{K^v}^{1,\sm}$ is defined using similar argument in \cite[Theorem 4.1.4]{PanII}, and $d_{\iota^c}$ is the differential operator defined in Corollary \ref{cor:diotac}.

Now we handle the general case. We may $\calO_{K^v}^{\lan,(0,0)}$-linearly extend the Hodge filtration on $D_{K^v}^{(k,0),\sm}$ to 
\begin{align*}
    D_{K^v}^{(k,0),\sm}\ox_{\calO_{K^v}^{\sm}}\calO_{K^v}^{\lan,(0,0)}.
\end{align*}
Besides, the differential operator $d^1$ on $\calO_{K^v}^{\lan,(0,0)}$ and the Gauss--Manin connection on $D_{K^v}^{(k,0),\sm}$ induces an operator
\begin{align*}
    \nabla:D_{K^v}^{(k,0),\sm}\ox_{\calO_{K^v}^{\sm}}\calO_{K^v}^{\lan,(0,0)}\to D_{K^v}^{(k,0),\sm}\ox_{\calO_{K^v}^{\sm}}\calO_{K^v}^{\lan,(0,0)}\ox_{\calO_{K^v}^{\sm}}\Omega_{K^v}^{1,\sm}.
\end{align*}
The graded pieces of $D_{K^v}^{(k,0),\sm}\ox_{\calO_{K^v}^{\sm}}\calO_{K^v}^{\lan,(0,0)}$ are 
\begin{align*}
    \omega_{K^v}^{(k-i,i),\sm}\ox_{\calO_{K^v}^{\sm}}\calO_{K^v}^{\lan,(0,0)},i=0,1,\dots,k
\end{align*}
with $\omega_{K^v}^{(k,0),\sm}\ox_{\calO_{K^v}^{\sm}}\calO_{K^v}^{\lan,(0,0)}$ being the quotient. Similarly, the graded pieces of
\begin{align*}
    D_{K^v}^{(k,0),\sm}\ox_{\calO_{K^v}^{\sm}}\calO_{K^v}^{\lan,(0,0)}\ox_{\calO_{K^v}^{\sm}}\Omega_{K^v}^{1,\sm}
\end{align*}
are 
\begin{align*}
    \omega_{K^v}^{(k-i-1,i+1),\sm}\ox_{\calO_{K^v}^{\sm}}\calO_{K^v}^{\lan,(0,0)},i=0,1,\dots,k
\end{align*}
with $\omega_{K^v}^{(-1,k+1),\sm}\ox_{\calO_{K^v}^{\sm}}\calO_{K^v}^{\lan,(0,0)}$ being the subobject. The operator $\nabla$ induces isomorphisms between other graded pieces, hence induces a morphism 
\begin{align*}
    \omega_{K^v}^{(k,0),\sm}\ox_{\calO_{K^v}^{\sm}}\calO_{K^v}^{\lan,(0,0)}\to \omega_{K^v}^{(-1,k+1),\sm}\ox_{\calO_{K^v}^{\sm}}\calO_{K^v}^{\lan,(0,0)}.
\end{align*}
By construction we know that the above map is $\calO_{\fl}$-linear. In particular, we first twist the above connection by $\omega_{K^v}^{(-k,-k),\sm}$ and then by Proposition \ref{prop:twist}, applying $-\ox_{\calO_{\fl}}\omega_{\fl}^{(0,-k)}$, we get a morphism 
\begin{align*}
    \calO_{K^v}^{\lan,(0,-k)}\to \calO_{K^v}^{\lan,(0,-k)} \ox_{\calO_{K^v}^{\sm}}(\Omega_{K^v}^{1,\sm})^{\ox k+1}.
\end{align*}
One directly checks that this gives the desired differential operator $d^{k+1}$ on $\calO_{K^v}^{\lan,(0,-k)}$, and restricts to a differential operator on $\calO_{K^v}^{\iota\-\lan,\iota^c\-\lalg,(0,-k)}$.

For uniqueness, by Corollary \ref{cor:Oiotalachi} and Proposition \ref{prop:Oetala}, $\calO_{K^v}^{\lalg,\chi}$ generates a dense $\calO_{\fl}$ submodule in $\calO_{K^v}^{\lan,\chi}$, so that $d^{k+1}$ is determined by the two properties. 
\end{proof}
\end{theorem}

\begin{theorem}\label{thm:dbar}
There exists up to scalar a unique continuous operator
\[
    \bar d^{k+1}:\calO_{K^v}^{\lan,(0,-k)}\to \calO_{K^v}^{\lan,(0,-k)}\ox_{\calO_{\fl}}(\Omega^1_{\fl})^{\ox k+1}
\]
such that it is determined by the following three properties:
\begin{enumerate}[(i)]
    \item $\bar d^{k+1}$ is $\calO_{K^v}^{\iota^c\-\lan}$-linear.
    \item There exists a nonzero constant $c\in\bbQ^\times$ such that $\bar d^{k+1}(s)=c(u^+)^{k+1}(s)\ox (dx)^{k+1}$ for any $s\in\calO_{K^v}^{\lan,(0,-k)}$ defined over an open subset of $\fl$ where $x$ has no poles.
    \item $\bar d^{k+1}$ commutes with the $\textup{GL}_2(L)$ action and the Hecke actions away from $p$. 
\end{enumerate}
Moreover, it is surjective with kernel $\calO_{K^v}^{\iota\-\lalg,\iota^c\-\lan,(0,-k)}$.
\begin{proof}
Let $\chi=(0,-k)$. We first define $\bar d^{k+1}$ on $\calO_{K^v}^{\iota\-\lan,\chi}$. As the structure of $\calO_{K^v}^{\iota\-\lan}$ is similar to that of $\calO_{K^p}^{\lan,\chi}$ studied in \cite{PanII}, we may use similar methods in \cite[Theorem 4.2.7]{PanII} to define $\bar d^{k+1}$ on $\calO_{K^v}^{\iota\-\lan,\chi}$. Then we just $\calO_{K^v}^{\iota^c\-\lan}$-linearly extend the definition of $\bar d^{k+1}$ on $\calO_{K^v}^{\iota\-\lan,\chi}$ to $\calO_{K^v}^{\lan,\chi}$. For uniqueness, note that it is determined by the properties (ii) and (iii).
For the image and kernel, as it is $\calO_{K^v}^{\iota^c\-\lan}$-linear, they are determined by its restriction on $\calO_{K^v}^{\iota\-\lan,\chi}$, and then we can use similar methods in \cite[Proposition 4.2.9]{PanII} to deduce the result.
\end{proof}
\end{theorem}

As $d^{k+1}$ is $\calO_{\fl}$-linear, we may twist $d^{k+1}$ by $(\Omega^1_{\fl})^{\ox k+1}$ to get a differential operator 
\[
    d'^{k+1}:\calO_{K^v}^{\lan,(0,-k)}\ox_{\calO_{\fl}}(\Omega^1_{\fl})^{\ox k+1}\to (\Omega^{1,\sm}_{K^v})^{\ox k+1}\ox_{\calO_{K^v}^{\sm}}\calO_{K^v}^{\lan,(0,-k)}\ox_{\calO_{\fl}}(\Omega^1_{\fl})^{\ox k+1}.
\]
Besides, as $\bar d^{k+1}$ is $\calO_{K^v}^{\sm}$-linear, we may twist $\bar d^{k+1}$ by $(\Omega^{1,\sm}_{K^v})^{\ox k+1}$ to get a differential operator 
\[
    \bar d'^{k+1}:\calO_{K^v}^{\lan,(0,-k)}\ox_{\calO_{K^v}^{\sm}}(\Omega^{1,\sm}_{K^v})^{\ox k+1}\to (\Omega^1_{\fl})^{\ox k+1} \ox_{\calO_{\fl}} \calO_{K^v}^{\lan,(0,-k)}\ox_{\calO_{K^v}^{\sm}}(\Omega^{1,\sm}_{K^v})^{\ox k+1}.
\]
As $\Omega_{K^v}^{1,\sm}=\omega_{K^v}^{(-1,1),\sm}$ and $\Omega_{\fl}^{1}=\omega_{\fl}^{(-1,1)}$, we have a natural identification
\[
    (\Omega^1_{\fl})^{\ox k+1} \ox_{\calO_{\fl}} \calO_{K^v}^{\lan,(0,-k)}\ox_{\calO_{K^v}^{\sm}}(\Omega^{1,\sm}_{K^v})^{\ox k+1}\isom \calO_{K^v}^{\lan,(-k-1,1)}(k+1).
\]
From the local expansion of $\calO_{K^v}^{\lan,\chi}$, we see that the following diagram 
$$
\begin{tikzcd}
    \calO_{K^v}^{\lan,(0,-k)} \arrow[r, "d^{k+1}"] \arrow[d, "\bar d^{k+1}"] & \calO_{K^v}^{\lan,(0,-k)}\ox_{\calO_{K^v}^{\sm}}(\Omega^{1,\sm}_{K^v})^{\ox k+1} \arrow[d, "\bar d'^{k+1}"] \\
    \calO_{K^v}^{\lan,(0,-k)}\ox_{\calO_{\fl}}(\Omega^1_{\fl})^{\ox k+1} \arrow[r, "d'^{k+1}"]           & (\Omega^1_{\fl})^{\ox k+1} \ox_{\calO_{\fl}} \calO_{K^v}^{\lan,(0,-k)}\ox_{\calO_{K^v}^{\sm}}(\Omega^{1,\sm}_{K^v})^{\ox k+1}\isom \calO_{K^v}^{\lan,(-k-1,1)}(k+1)       
\end{tikzcd}
$$
is commutative.

\begin{definition}\label{def:twistofd}
Define the intertwining operator 
\[
    I_k:\calO_{K^v}^{\lan,(0,-k)}\to \calO_{K^v}^{\lan,(-k-1,1)}(k+1)
\]
as the composition $\bar d'^{k+1}\comp d^{k+1}=d'^{k+1}\comp\bar d^{k+1}$. 
\end{definition}

We need both the above intertwining operator and its locally $\iota^c$-algebraic version.
\begin{proposition}
The intertwining operator $I_k$ restricts to an operator 
\[
    I_{k}^{\iota^c\-\lalg}:\calO_{K^v}^{\iota\-\lan,\iota^c\-\lalg,(0,-k)}\to \calO_{K^v}^{\iota\-\lan,\iota^c\-\lalg,(-k-1,1)}(k+1).
\]
\begin{proof}
This follows directly from the constructions of $d^{k+1}$ and $\bar d^{k+1}$, which are compatible with taking locally $\iota^c$-algebraic sections.
\end{proof}
\end{proposition}

\subsection{The ordinary locus}
First we study the properties of differential operators defined in the previous section on the ordinary locus. The main references we use for the geometry of the ordinary locus of unitary Shimura curves are \cite{Car83,HT01,Joh16}.

First we recall the definition of integral models of $X_{K}$. Fix a sufficiently small compact open subgroup $ K^p\subset G(\bbA^{\infty,p})$ and a tuple $\ul m=(m_1,\dots,m_r)\in\bbZ_{\ge 0}^r$. The moduli functor $\ffrx_{ K^p,\ul m}$ is defined as follows. Let $S$ be a connected locally Noetherian $\calO_L$-scheme and $s$ be a geometric point of $S$, then $\ffrx_{ K^p,\ul m}(S,s)$ is the set of equivalence classes of $(r + 4)$-tuples $(A,\lambda,i,\bar\eta^p,\alpha_i)$ where
\begin{itemize}
    \item $A/S$ is an abelian scheme of dimension $4d$;
    \item $\lambda:A\to A^{\vee}$ is a prime-to-$p$ polarization;
    \item $i : \calO_B \to \End(A) \ox_{\bbZ} \bbZ_{(p)}$ is a homomorphism such that $(A, i )$ is compatible and $\lambda\comp i(b)=i(b^*)^{\vee}\comp\lambda$ for all $b\in\calO_B$;
    \item $\bar\eta^p$ is a $\pi_1(S,s)$-invariant $U^p$-orbit of isomorphisms of $B\ox_{\bbQ}\bbA^{\infty,p}$-modules $\eta^p:V\ox_{\bbQ}\bbA^{\infty,p}\to V^p A_s$ which take the pairing $(-,-)$ on $V\ox_{\bbQ}\bbA^{\infty,p}$ to a $(\bbA^{\infty,p})^\times$-multiple of the $\lambda$-Weil pairing on $V_pA^s$.
    \item $\alpha_1:\varpi^{-m_1}\Lambda_{11}/\Lambda_{11}=\epsilon\varpi^{-m_1}\Lambda_1/\Lambda_1\to \calG_A[\varpi^{m_1}]=\epsilon\calA[\varpi^{m_1}]$ is a Drinfeld $\varpi^{m_1}$-level structure.
    \item for $i=2,\dots,r$, $\alpha_i:(\varpi_i^{-m_i}\Lambda_i/\Lambda_i)_S\to A[\varpi_i^{m_i}]$ is an isomorphism of $S$-group schemes with $\calO_B$-actions.
\end{itemize}
Two $(r+4)$-tuples $(A,\lambda,i,\bar\eta^p,\alpha_i)$ and $(A',\lambda',i',(\bar\eta^p)',\alpha_i')$ are equivalent if there is a prime-to-$p$ isogeny $\delta:A\to A'$ and a $\gamma\in\bbZ_{(p)}^\times$ such that $\delta$ carries $\lambda$ to $\gamma\lambda'$, $i$ to $i'$, $\bar\eta^p$ to $(\bar\eta^p)'$ and $\alpha_i$ to $\alpha_i'$. $\ffrx_{ K^p,\ul m}(S,s)$ is canonically independent of the choice of $s$, and we get a functor on all locally Noetherian $\calO_L$-schemes by requiring that $\ffrx_{ K^p,\ul m}(\sqcup S_i)=\prod_i\ffrx_{ K^p,\ul m}(S_i)$. This functor is representable by a projective scheme of pure dimension $1$ over $\calO_L$. $\ffrx_{ K^p,\ul m}$ is smooth over $\calO_L$ when $m_1=0$. If $\ul m'\ge \ul m$ (pointwisely), then the natural map $\ffrx_{K^p,\ul m'}\to\ffrx_{ K^p,\ul m}$ is finite and flat; moreover it is \'etale if $m_1'=m_1$. Note that if $\ul m=(m_1,\dots,m_r)$ and $K_{v_i}=1+\varpi_i^{m_i}\calO_{B,w_i}^{\op}\subset G_p^v$, then for the level subgroup $K= K^p\cdot\bbZ_p^\times\cdot\prod_{i=1}^r K_{v_i}$, $X_K$ is canonically isomorphic to the generic fiber of $\ffrx_{ K^p,\ul m}$.

Fix $m_2,\dots m_r\in \N$. Let $K^v=K= K^p\cdot\bbZ_p^\times\cdot\prod_{i\ge 2}^r K_{v_i}$ with $K_{v_i}=1+\varpi_i^{m_i}\calO_{B,w_i}^{\op}\subset G_p^v$, and we define $\ffrx_{K^v,m}$ to be the adic space associated to $\ffrx_{ K^p,\ul m}\times_{\calO_L,\iota}\calO_C$, where $\ul m=(m,m_2,\dots,m_r)$. Let $\bar X_{K^v,m}$ be the special fiber of $\ffrx_{K^v,m}$. Over $\bar X_{K^v,m}$, we have a universal abelian scheme $\bar A$ and the associated Barsotti--Tate $\calO_L$-module $\calG$. There is a Newton stratification on $\bar X_{K^v,m}$, given by 
\[
    \bar X_{K^v,m}=\bar X_{K^v,m}^{\ord}\sqcup \bar X_{K^v,m}^{\ss}
\]
where $\calG^{\et}$ has $\calO_L$-height $1$ on $\bar X_{K^v,m}^{\ord}$, and $\calG^{\et}$ is zero on $\bar X_{K^v,m}^{\ss}$, where $\calG^{\et}$ is the maximal \'etale quotient of $\calG$. Moreover, $\bar X_{K^v,m}^{\ord}$ is an open subscheme of $\bar X_{K^v,m}$, which is smooth of dimension $1$, and $\bar X_{K^v,m}^{\ss}$ is a closed subscheme of $\bar X_{K^v,m}$, which is of dimension $0$. 

In this section, we mainly handle the ordinary part $\bar X_{K^v,m}^{\ord}$, and leave the supersingular part $\bar X_{K^v,m}^{\ss}$ later. By \cite[9.4.3]{Car83}, we know that the irreducible components of $\bar X_{K^v,m}$ are parametrized by $\bbP^1(\calO_L/\varpi^m)$, or equivalently one-dimensional $\calO_L$-linear quotients $(\calO_L/\varpi^m)^{\oplus 2}\surj \calO_L/\varpi^m$. Each component is isomorphic to the Igusa curve of level $m$. We recall the definition of Igusa curves. Let $s$ be a geometric point of $\bar X_{K^v,m}^{\ord}$. Then $\calG_s\isom \LT_L\times(L/\calO_L)$ as Barsotti--Tate $\calO_L$-modules, where $\LT_L$ denotes the Lubin--Tate formal group for $L$. Define $\Ig_{K^v,m}/\bar X_{K^v,0}^{\ord}$ to be the moduli space for isomorphisms 
\[
    j: \varpi^{-m}\calO_L/\calO_L\aisom\calG^{\et}[\varpi^m].
\]
Then $\Ig_{K^v,m}$ is represented by an affine smooth curve of dimension $1$. From \cite[Page 124]{HT01} we know that the ordinary locus in each irreducible component of $\bar X_{K^v,m}$ is isomorphic to $\Ig_{K^v,m}$ up to powers of relative Frobenius. Suppose that the canonical subgroup corresponds to the second factor in $(\calO_L/\varpi^m)^{\oplus 2}$ under the level structure. Define the \emph{canonical locus} $\calX_{K^v,m,c}\subset\calX_{K^v,m}$ to be the tubular neighborhood of the ordinary locus such that the reduction lies in the irreducible component of $\bar X_{K^v,m}$ indexed by the ``canonical'' submodule $(0,*)\subset (\calO_L/\varpi^m)^2$. In particular, $\bar X_{K^v,m}$ is stable under the $\bar B(\calO_L)$-action, where $\bar B$ is the Borel subgroup consisting of lower triangular matrices.

Recall that for $(a_{\iota},b_{\iota})\in\bbZ^2$ we have an invertible $\calO_{K^v}^{\sm}$-module $\omega_{K^v}^{(a_{\iota},b_{\iota}),\sm}$. Besides, for any $\eta\in\Hom_{\bbQ_p\-\alg}(L,C)$ and $(a_{\eta},b_{\eta})\in\bbZ^2$ with $a_{\eta}\ge b_{\eta}$, we have a filtered integrable connection $D_{K^v,\eta}^{(a_{\eta},b_{\eta}),\sm}$ which is a locally free $\calO_{K^v}^{\sm}$-module of rank $a_\eta-b_\eta+1$. When $\eta\neq\iota$, the Hodge filtration on $D_{K^v,\eta}^{(a_{\eta},b_{\eta}),\sm}$ is trivial, and $\omega_{K^v}^{(b_{\iota},a_{\iota}),\sm}$ (resp. $\omega_{K^v}^{(a_{\iota},b_{\iota}),\sm}$) is the minimal (resp. maximal) non-trivial graded piece of the Hodge filtration of $D_{K^v,\iota}^{(a_{\iota},b_{\iota}),\sm}$.

Let
\begin{align*}
    W=\{(a_\eta,b_\eta)_{\eta}:\text{ for any }\eta\in\Hom_{\bbQ_p\-\alg}(L,C),(a_\eta,b_\eta)\in\bbZ^2, \text{and }a_\eta\ge b_\eta\text{ when }\eta\neq\iota\}
\end{align*}
be the set of integral weights for modular forms on unitary Shimura curves. We define an action of the Weyl group of $\GL_2$ on $W$ given by $s.((a_\iota,b_{\iota}),(a_{\eta},b_{\eta})_{\eta\neq\iota})=((b_\iota,a_{\iota}),(a_{\eta},b_{\eta})_{\eta\neq\iota})$ for $s$ the non-trivial reflection in the Weyl group. We similarly define the dot action of the Weyl group on $W$ as usual. For any $w\in W$, define 
\[
    \omega^{w,\sm}_{K^v}:=\omega_{K^v}^{(a_\iota,b_\iota),\sm}\ox_{\calO_{K^v}^{\sm}}\bigotimes_{\eta\neq\iota} D_{K^v,\eta}^{(a_\eta,b_\eta),\sm}
\]
which is a locally free $\calO_{K^v}^{\sm}$-module of rank $\prod_{\eta\neq\iota}(a_\eta-b_\eta+1)$. Let $W_+\subset W$ be the subset such that $a_{\iota}\ge b_{\iota}$. For $w=(a_{\eta},b_{\eta})_{\eta\in\Sigma}\in W_+$, the integrable connection on
\[
    D_{K^v}^{w,\sm}:=D_{K^v,\iota}^{(a_\iota,b_\iota),\sm}\ox_{\calO_{K^v}^{\sm}}\bigotimes_{\eta\neq\iota} D_{K^v,\eta}^{(a_\eta,b_\eta),\sm}
\]
together with the Kodaira--Spencer isomorphism defines the $\theta$-operator \cite[Lemme 3.3.6]{Din17}:
\[
    \theta:\omega^{w,\sm}_{K^v}\to \omega_{K^v}^{w,\sm}\ox_{\calO_{K^v}^{\sm}}(\Omega_{K^v}^{1,\sm})^{\ox k+1}
\]
with $k=a_\iota-b_{\iota}$. When $w=((0,-k)_{\iota},(0,0)_{\eta\neq\iota})$, $\omega^{w,\sm}_{K^v}=\omega_{K^v}^{(0,-k),\sm}$ and the $\theta$-operator is the usual $\theta$-operator on automorphic line bundles $\omega_{K^v}^{(0,-k),\sm}\to \omega_{K^v}^{(0,-k),\sm}\ox(\Omega_{K^v}^{1,\sm})^{\ox k+1}$ (up to twist by determinants). 

We define overconvergent modular forms on unitary Shimura curves following \cite[Section 9]{Kas04}, \cite[Chapitre 4]{Din17}. Let $K_{m}\subset \GL_2(L)$ be the level subgroup such that $\calX_{K^vK_{m}}$ is canonically isomorphic to the generic fiber of $\ffrx_{K^v, m}$.  

Let $B\subset \GL_2$ be the algebraic subgroup consisting of upper triangular matrices. 
\begin{definition}
Let $w\in W$ be an integral weight.
\begin{itemize}
    \item Define the \emph{space of modular forms of weight $w$ and tame level $K^v$} as 
    \[
        M_w(K^v):= H^0(\fl,\omega_{K^v}^{w,\sm}).
    \]
    Using similar arguments of \cite[Lemma 5.3.5]{Pan22}, we see that this space coincides with the usual space of modular forms on unitary Shimura curves of weight $w$.
    \item Define the \emph{space of overconvergent modular forms of weight $w$ and level $K^vK_m$} as
    \[
        M_w^{\dagger}(K^vK_m):=\dlim_{\calX_{K^vK_m,c}\subset U}H^0(U,\omega_{K^vK_{m}}^{w}),
    \]
    where $U$ runs over strict neighborhood of $\calX_{K^vK_m,c}$. Note that our definition is slightly different from the one in \cite{Kas04} as we only work with the canonical Igusa component. 
    \item Define the space of overconvergent modular forms of weight $w$ and tame level $K^v$ as
    \[
        M_w^{\dagger}(K^v):=\dlim_{ m} M_w^{\dagger}(K^vK_{ m}).
    \]
\end{itemize}
\end{definition}
Note that the natural action of $\bar B(\calO_L/\varpi^m)\subset\GL_2(\calO_L/\varpi^m)$ on $\bar X_{K^v,m}$ preserves the Igusa component for the canonical index. This gives a $\bar B(\calO_L)$-action on $M_w^{\dagger}(K^v)$.

\begin{lemma}\label{lem:overconvergentmodularforms}
We have a $\bar B(\calO_L)$-equivariant isomorphism
\[
    \dlim_{U\ni\infty}\omega_{K^v}^{w,\sm}(U)=M_w^{\dagger}(K^v)
\]
and the restriction map $H^0(\fl,\omega_{K^v}^{w,\sm})\to \dlim_{U\ni\infty}\omega_{K^v}^{w,\sm}(U)$ is identified with the natural inclusion map $M_w(K^v)\inj M_w^{\dagger}(K^v)$.
\begin{proof}
This follows from similar arguments in \cite[Lem 5.3.7]{Pan22}. Roughly speaking, the strict neighborhoods of the canonical locus are cofinal to open neighborhoods of $\infty\in\fl$ via the Hodge--Tate period map \cite[Corollary 3.2.5]{JLH21}.
\end{proof}
\end{lemma}
In particular, using the above isomorphism, we can extend the $\bar B(\calO_L)$-action on $M_w^{\dagger}(K^v)$ to a $\bar B(L)$-action. Using Lemma \ref{lem:overconvergentmodularforms}, we can calculate the cohomology of $d^{k+1}$ on open neighborhoods of $\bbP^1(L)$ in terms of overconvergent modular forms. For simplicity we first restrict the differential operator $d^{k+1}$ to $\calO_{K^v}^{\iota\-\lan,(\ox_{\eta\neq\iota}V_{\eta}^{(a_\eta,b_\eta)})\-\lalg,\chi}$ with $\chi=(0,-k)$ for some $k\ge 0$. Recall that $\calO_{K^v}^{\iota\-\lan,(\ox_{\eta\neq\iota}V_{\eta}^{(a_\eta,b_\eta)})\-\lalg,\chi}$ has a dense subspace given by $(\ox_{\eta\neq\iota}V_{\eta}^{(a_\eta,b_\eta)})\ox (\ox_{\eta\neq\iota}D_{\eta}^{(a_\eta,b_\eta)})\ox \omega_{K^v}^{(0,-k),\sm}\ox \omega_{\fl}^{(0,-k)}$. For $\calF$ a sheaf on $\fl$, we let $\calF_\infty$ to denote the stalk of $\calF$ at $\infty$. By similar methods in \cite[Lemma 5.1.2, 5.1.3]{PanII}, we can identify $$(d^{k+1})_\infty:(\calO_{K^v}^{\iota\-\lan,(\ox_{\eta\neq\iota}V_{\eta}^{(a_\eta,b_\eta)})\-\lalg,\chi})_\infty\to(\calO_{K^v}^{\iota\-\lan,(\ox_{\eta\neq\iota}V_{\eta}^{(a_\eta,b_\eta)})\-\lalg,\chi}\ox_{\calO_{K^v}^{\sm}}(\Omega^{1,\sm}_{K^v})^{\ox k+1})_\infty$$ with the map 
\begin{align*}
    1\ox\theta^{k+1}:(\ox_{\eta\neq\iota}V_{\eta}^{(a_\eta,b_\eta)})\ox_C(\omega_{\fl}^{(0,-k)})_\infty\hat\ox_C M_{w}^{\dagger}(K^v)\to (\ox_{\eta\neq\iota}V_{\eta}^{(a_\eta,b_\eta)})\ox_C (\omega_{\fl}^{(0,-k)})_\infty\hat\ox_C M_{s\cdot w}^{\dagger}(K^v)
\end{align*}
induced from $\theta^{k+1}:M_{w}^{\dagger}(K^v)\to M_{s\cdot w}^{\dagger}(K^v)$ with $w=((0,-k)_{\iota},(a_\eta,b_\eta)_{\eta\neq\iota})\in W$ and $s\cdot w=((-k-1,1)_{\iota},(a_\eta,b_\eta)_{\eta\neq\iota})\in W$

The cohomology of $\theta^{k+1}$ has a natural identification with the rigid cohomology of the Igusa curve. Indeed, by \cite{Kas04, Din17}, for any $w\in W_+$, $D_{K^vK_{m}}^w$ defines an overconvergent $F$-isocrystal on $\Ig_{K^v,m}$, which we denote by $D_{K^vK_{m}}^{w,\dagger}$. By the definition of rigid cohomology, we have a natural isomorphism
\begin{align*}
    H^i_{\rig}(\Ig_{K^v,m},D_{K^vK_{m}}^{w,\dagger})=\bbH^i(M_{w}^{\dagger}(K^vK_{m})\ov{\theta^{k+1}}\to M_{s\cdot w}^{\dagger}(K^vK_{m})).
\end{align*}
By finiteness results on rigid cohomology \cite[Theorem 1.2.1]{Kedlayafinitenessrigidcohomology}, we know that $H^i_{\rig}(\Ig_{K^v,m},D_{K^vK_{m}}^{w,\dagger})$ is a finite-dimensional $C$-vector space. Define 
\[
    H^i_{\rig}(\Ig_{K^v},D_{K^v}^{w,\sm,\dagger}):=\dlim_m H^i_{\rig}(\Ig_{K^v,m},D_{K^vK_{m}}^{w,\dagger}).
\]
Note that $H^i_{\rig}(\Ig_{K^v},D_{K^v}^{w,\sm,\dagger})$ is a smooth admissible $\bar B(L)$-representation as  $H^i_{\rig}(\Ig_{K^v,m},D_{K^vK_{m}}^{w,\dagger})$ is finite dimensional.

Similar to \cite[5.1.9]{PanII}, let $\calH^i_{\ord}(K^v,w)$ denote the smooth geometric induction of $H^i_{\rig}(\Ig_{K^v},D_{K^v}^{w,\sm,\dagger})$ from $\infty$ to $\bbP^1(L)\subset\fl$. Explicitly, let $U$ be an open subset of $\bbP^1(L)$. Let $\pi:\GL_2(L)\to \bbP^1(L)$ denote the projection. Then $\calH^i_{\ord}(K^v,w)(U)$ is the set of smooth functions
\begin{align*}
    f:\pi^{-1}(U)\to H^i_{\rig}(\Ig_{K^v},D_{K^v}^{w,\sm,\dagger})
\end{align*}
such that $f(bg)=b.f(g)$ for $b\in \bar B(L)$ and $g\in \GL_2(L)$. In particular, we have 
\[
    H^0(\bbP^1(L),\calH^i_{\ord}(K^v,w))=(\Ind^{\GL_2(L)}_{\bar B(L)}H^i_{\rig}(\Ig_{K^v},D_{K^v}^{w,\sm,\dagger}))^\infty
\]
here $(\Ind^{\GL_2(L)}_{\bar B(L)}-)^\infty$ denotes the smooth induction. Let $w=((0,-k)_{\iota},(a_\eta,b_\eta)_{\eta\neq\iota})$. By an analogue of \cite[Corollary 5.1.11, Corollary 5.1.12]{PanII} we have the following result.
\begin{proposition}\label{prop:dord} Then
\begin{align*}
    (\ker  d^{k+1})|_{\bbP^1(L)}\isom (\ox_{\eta\neq\iota}V_{\eta}^{(a_\eta,b_\eta)})\ox_C\omega_{\fl}^{(0,-k)}|_{\bbP^1(L)}\ox_C\calH^0_{\ord}(K^v,w)\\
    (\coker  d^{k+1})|_{\bbP^1(L)}\isom (\ox_{\eta\neq\iota}V_{\eta}^{(a_\eta,b_\eta)})\ox_C\omega_{\fl}^{(0,-k)}|_{\bbP^1(L)}\ox_C\calH^1_{\ord}(K^v,w).
\end{align*}
and 
\begin{align*}
    H^0(\bbP^1(L),\ker d^{k+1})\isom (\Ind^{\GL_2(L)}_{\bar B(L)}H^0_{\rig}(\Ig_{K^v},D_{K^v}^{w,\sm,\dagger})z_2^{-k})^{\iota\-\lan}\\
    H^0(\bbP^1(L),\coker d^{k+1})\isom (\Ind^{\GL_2(L)}_{\bar B(L)}H^1_{\rig}(\Ig_{K^v},D_{K^v}^{w,\sm,\dagger})z_2^{-k})^{\iota\-\lan}.
\end{align*}
where $z_i:\bar B(L)\to C$ is the character sending $\left( \begin{matrix} a_1&0\\b&a_2 \end{matrix} \right)$ to $\iota(a_i)$ for $i=1,2$.
\end{proposition}

\begin{remark}
To see that this is an induction for the locally-$\iota$ analytic vectors, note that $(\ox_{\eta\neq\iota}V_{\eta}^{(a_\eta,b_\eta)}\ox _C\calO_{K^v})^{\iota\-\lan}\isom \calO_{K^v}^{\iota\-\lan}\ox_{\calO_{K^v}^{\sm}}\Hom_{\frg}((\ox_{\eta\neq\iota}V_{\eta}^{(a_\eta,b_\eta)})^*,\calO_{K^v}^{\lan})$.
\end{remark}

As the differential operator $d'^{k+1}$ is a twist of $d^{k+1}$ by $(\Omega_{\fl}^1)^{\ox k+1}$, we have similarly 
\begin{align*}
    (\ker  d'^{k+1})|_{\bbP^1(L)}\isom (\ox_{\eta\neq\iota}V_{\eta}^{(a_\eta,b_\eta)})\ox_C\omega_{\fl}^{(-k-1,1)}|_{\bbP^1(L)}\ox_C\calH^0_{\ord}(K^v,w)\\
    (\coker  d'^{k+1})|_{\bbP^1(L)}\isom (\ox_{\eta\neq\iota}V_{\eta}^{(a_\eta,b_\eta)})\ox_C\omega_{\fl}^{(-k-1,1)}|_{\bbP^1(L)}\ox_C\calH^1_{\ord}(K^v,w).
\end{align*}
and 
\begin{align*}
    H^0(\bbP^1(L),\ker d'^{k+1})\isom (\Ind^{\GL_2(L)}_{\bar B(L)}H^0_{\rig}(\Ig_{K^v},D_{K^v}^{w,\sm,\dagger})z_1^{-k-1}z_2)^{\iota\-\lan}\\
    H^0(\bbP^1(L),\coker d'^{k+1})\isom (\Ind^{\GL_2(L)}_{\bar B(L)}H^1_{\rig}(\Ig_{K^v},D_{K^v}^{w,\sm,\dagger})z_1^{-k-1}z_2)^{\iota\-\lan}.
\end{align*}

Finally, we want to mention the relation between the rigid cohomology of Igusa curves and de Rham cohomology of unitary Shimura curves, see for example \cite{HT01},\cite{Joh16}. Note that in \cite{Joh16} he assumed that the totally real number field $F^+=\bbQ$, but the reason he added this assumption is to ensure certain newton strata are affine, which is satisfied in our situation by \cite[Proposition 2.2.7]{Joh16}, \cite{MR3989256}. We rewrite their results in our situation.
\begin{theorem}\label{thm:IgJacquet}
We have an equality of virtual $\bar B(L)\times \bbT^S$-representations
\begin{align*}
    &J_{ \bar N(L)}(\dlim_{m}H^1_{\dR}(\calX_{K^vK_m},D_{K^vK_m}^{w})-\dlim_{m}H^0_{\dR}(\calX_{K^vK_m},D_{K^vK_m}^{w})-\dlim_{m}H^2_{\dR}(\calX_{K^vK_m},D_{K^vK_m}^{w}))\\
    = &2(\dlim_{m}H^1_{\rig}(\Ig_{K^v,m},D_{K^vK_m}^{w,\dagger})-\dlim_{m}H^0_{\rig}(\Ig_{K^v,m},D_{K^vK_m}^{w,\dagger}))
\end{align*}
where $J_{\bar N(L)}$ denote the smooth Jacquet functor for $\bar N(L)\subset \bar B(L)$ the unipotent radical.
\begin{proof}
In \cite[Corollary 19]{Joh16}, the author proves this theorem in the dual case, i.e., for compactly supported rigid  cohomology of $\Ig_{K^v,m}$ with coefficient in $D_{K^vK_m}^{w,\dagger}$. Subsequently, Poincaré duality for rigid cohomology \cite{Kedlayafinitenessrigidcohomology} can be employed to deduce our case. Finally, we note that if $\pi$ is an admissible smooth $\GL_2(L)$-representation, then $J_{\bar N(L)}(\pi)=(J_{N(L)}(\pi^*))^*$. This is a direct consequence of the second adjoint theorem for smooth representations.
\end{proof}
\end{theorem}

\subsection{The supersingular locus}
\subsubsection{The Lubin--Tate space at infinite level}
In this subsection, we recall some facts on $1$-dimensional Lubin--Tate spaces at the infinite level. The main references we use are \cite{Wei16, SW13}.

Let us fix some notation. Let $\iota:L\inj C$ be the distinguished embedding we fixed before. Let $\breve{L}$ be the completion of the maximal unramified extension of $L$ inside $C$. Let $D_L$ denote the unique non-split quaternion algebra over $L$. Also, fix an isomorphism $\Lie(\GL_2(L))\ox_{\bbQ_p}C\isom\Lie(D_L^\times)\ox_{\bbQ_p}C$. If we put $\frg_\iota:=\Lie(\GL_2(L))\ox_{L,\iota}C$ and $\check \frg_\iota:=\Lie(D_L^\times)\ox_{L,\iota}C$, then the above identification induces an isomorphism $\frg_\iota\isom \check\frg_\iota$.

Let $H_0$ be a one-dimensional formal $\calO_L$-module over $\bar\bbF_q$ of $\calO_L$-height $2$ (hence of usual height $2[L:\bbQ_p]$), which is unique up to quasi-isogeny. Let $\Nilp_{\calO_{\breve{L}}}$ be the category of $\calO_{\breve{L}}$-algebras such that $\varpi$ is nilpotent. Let $\ffrm_{\LT,0}$ be the functor from $\Nilp_{\calO_{\breve{L}}}$ to the category of sets, assigning an $\calO_{\breve{L}}$-algebra $R$ to the set of equivalence classes of pairs $(G,\rho)$, where 
\begin{itemize}
    \item $G$ is a compatible $p$-divisible $\calO_L$-module over $R$.
    \item $\rho:H_0\ox_{\bar\bbF_q}R/\pi\dasharrow G\ox_R R/\pi$ is an $\calO_L$-equivariant quasi-isogeny.
\end{itemize}
Two pairs $(G,\rho)$ and $(G',\rho')$ are equivalent if the quasi-isogeny $\beta'\comp\beta^{-1}:G\ox_R R/\pi\dasharrow G'\ox_R R/\pi$ lifts to an isomorphism $\tilde{\beta'\comp\beta^{-1}}:G\aisom G'$ over $R$. Here the property of being compatible means that the natural $\calO_L$-action on $\Lie(G)$ induced by the $\calO_L$-action on $G$, coincides with the $\calO_L$-action induced by the natural action of $R$ on $\Lie(G)$.

By \cite{RZ96}, this functor is represented by a formal scheme over $\Spf(\calO_{\breve{L}})$. It is also equipped with a $D_L^\times$-action via the identification $D_L^\times\isom \mathrm{QIsog}_{\calO_L}(H_0)$. Lift $\iota:L\inj C$ to a $\bbQ_p$-embedding $\breve{\iota}:\breve{L}\inj C$. We denote by $\calM_{\LT,0}$ the base change of the adic generic fiber of $\ffrm_{\LT,0}$ from $\breve L$ to $C$ via $\breve{\iota}$. It will be clear from below that the choice of embedding $\breve{\iota}:\breve{L}\inj C$ does not affect our results. 

Let $(\calG,\rho)$ be the universal deformation of $H_0$ on $\ffrm_{\LT,0}$. The $\varpi$-adic Tate module of $\calG$ defines a $\bbZ_p$-local system of rank $2[L:\bbQ_p]$ on the \'etale site of $\calM_{\LT,0}$, whose dual will be denoted by $V_{\LT}$. For any $n\ge 0$, by considering trivializations of $V_{\LT}/\varpi^nV_{\LT}$, we get an \'etale Galois covering $\calM_{\LT,n}$ of $\calM_{\LT,0}$ with Galois group $\GL_2(\calO_L/\varpi^n\calO_L)$. By \cite{SW13}, there exists a perfectoid space $\calM_{\LT,\infty}$ over $C$ such that 
\[
    \calM_{\LT,\infty}\sim\ilim_n\calM_{\LT,n}.
\]
There is a natural continuous right action of $\GL_2(\calO_L)$ on $\calM_{\LT,\infty}$, which can be extended to $\GL_2(L)$. Note that the action of $\GL_2(L)$ here differs from the one in \cite{SW13} by $g\mapsto (g^{-1})^t$.

There are two important maps on $\calM_{\LT,\infty}$, namely the Gross--Hopkins period map and the Hodge--Tate period map. The Gross--Hopkins period map comes from Grothendieck--Messing theory on deformation of Barsotti--Tate groups, and the Hodge--Tate period map comes from Hodge--Tate comparison maps.

We recall the definition of the Gross--Hopkins period map. Recall that $(G,\rho)$ is the universal Barsotti--Tate group on $\ffrm_{\LT,0}$. Then its covariant Dieudonne crystal defines a rank $2d$ vector bundle $M(\calG)$ over $\calM_{\LT,0}$, equipped with an $\calO_L$-action and an integrable connection $\nabla_{\LT,0}$, which is the dual of the Gauss--Manin connection. In fact, $M(\calG)$ is equipped with a natural trivialization. Let $M(H_0)$ denote the covariant Dieudonne module of $H_0$, which is a free $W(\bar\bbF_p)$-module of rank $2d$, equipped with an $\calO_L$-action. As $\End_{\calO_L}(H_0)=\calO_{D_L}$, we see that $M(H_0)[\frac{1}{p}]$ is the standard representation of $D_L^\times$ on $L^{\oplus 2}\ox_{\bbQ_p}W(\bar\bbF_p)[\frac{1}{p}]$. Then the universal quasi-isogeny $\rho$ on $\ffrm_{\LT,0}$ induces a natural trivialization 
\[
    M(\calG)\isom M(H_0)\ox_{W(\bar\bbF_p)}\calO_{\calM_{\LT,0}}
\]
under which $\nabla_{\LT,0}$ is identified with the standard connection on $\calO_{\calM_{\LT}}$. Note that this isomorphism is $\calO_L$-equivariant, and the map $W(\bar\bbF_p)\to\calO_{\calM_{\LT,0}}$ is given by the restriction of $\breve{\iota}$. We rewrite this trivialization as 
\[
    M(\calG)\isom \bigoplus_\eta M(\calG)_\eta =\bigoplus_\eta W_\eta\ox_C \calO_{\calM_{\LT,0}},
\]
where $W_\eta$ is the $\eta$-standard representation of $D_L^\times$. The Lie algebra of $\calG$ defines a line bundle on $\calM_{\LT,0}$, whose dual will be denoted by $\omega_{\LT,0}$. By the Grothendieck--Messing theory, $(\calG,\rho)$ gives rise to a surjection 
\[
    M(H_0)\ox_{W(\bar\bbF_p)}\calO_{\calM_{\LT}}\isom M(\calG)\to(\omega_{\LT,0})^{-1}.
\]
By functoriality, we know that this surjection is $\calO_L\ox_{\bbZ_p}C$-equivariant. Let 
\[
    W_\iota\ox_C\calO_{\calM_{\LT,0}}=M(\calG)_\iota\to(\omega_{\LT,0})^{-1}
\]
be the direct summand of the above morphism, where the action of $\calO_L\ox_{\bbZ_p}C$ factors through $L\ox_{L,\iota}C$. Note that by the compatible condition on $\calG$ we know that the $\calO_L$-action on $\Lie(\calG)$ already factors through the embedding $\iota:L\inj C$. It turns out that the kernel of $M(\calG)_\iota\to(\omega_{\LT,0})^{-1}$ is a line bundle on $\calM_{\LT}$. This induces the so-called Gross--Hopkins period map \cite{GH94}
\[
    \pi_{\GM}:\calM_{\LT,0}\to\fl_{\GM}
\]
where $\fl_{\GM}$ is the adic space over $C$ associated to the flag variety parametrizing $1$-dimensional quotients of the $2$-dimensional $C$-vector space $W_\iota$. This map is $D_L^\times$-equivariant, if $\fl_{\GM}$ is identified with (the base change to $C$ of) the Brauer-Severi variety for $D_L$. Also note that the $D_L$-action on $\fl_{\GM}$ factors through the splitting $D_L^\times\inj\GL_2(C)$ induced by $\iota$. By Grothendieck--Messing theory, $\pi_{\GM}$ is surjective and \'etale, and admits local sections. Let $\omega_{\fl,\GM}$ denote the tautological ample line bundle on $\fl_{\GM}$, then $\pi_{\GM}^{*}(\omega_{\fl,\GM})=(\omega_{\LT,0})^{-1}$. 

Let $D(\calG)_\iota$ be the dual of $M(\calG)_\iota$. It is a vector bundle of rank $2$ on $\calM_{\LT,0}$, equipped with an integrable connection $\nabla_{\LT,0}$. The dual of the surjection $M(\calG)_\iota\surj (\omega_{\LT,0})^{-1}$ induces a Hodge filtration on $D(\calG)_\iota$, given by $\Fil^{0}D(\calG)_\iota=D(\calG)_\iota$, $\Fil^1D(\calG)_\iota=\omega_{\LT,0}$ and $\Fil^2D(\calG)_\iota=0$. We have the Kodaira--Spencer isomorphism 
\[
    \KS:\Fil^1D(\calG)_\iota\ov{\nabla_{\LT}}\to D(\calG)_\iota\ox_{\calO_{\calM_{\LT,0}}}\Omega_{\calM_{\LT,0}}^1\to\gr^0 D(\calG)_\iota\ox_{\calO_{\calM_{\LT,0}}}\Omega_{\calM_{\LT,0}}^1.
\]
This gives $\Omega_{\calM_{\LT,0}}^1\isom\omega_{\LT,0}^2\ox{\det}_{\iota}^{-1}$.  

Note that the above construction is compatible by adding finite level structures on $\calM_{\LT,0}$. Let us rename $D(\calG)$ by $D_{\LT,0}$, and $D(\calG)_\eta$ by $D_{\LT,0,\eta}$. We may pull back them as coherent modules to $\calM_{\LT,n}$ via the canonical projection $\calM_{\LT,n}\to \calM_{\LT,0}$, given by $D_{\LT,n}=\bigoplus_\eta D_{\LT,n,\eta}$. The Hodge filtrations on $D_{\LT,n,\eta}$ are trivial if $\eta\neq \iota$, and the Hodge filtration on $D_{\LT,n,\iota}$ is given by $\omega_{\LT,n}$. Similarly we have the Kodaira--Spencer isomorphisms. The Gross--Hopkins period map on $\calM_{\LT,n}$ are defined to be the composition $\calM_{\LT,n}\to\calM_{\LT,0}\ov{\pi_{\GM}}\to \fl_{\GM}$.

We discuss the Hodge--Tate period map on $\calM_{\LT,\infty}$. Let $\omega_{\LT}$ be the pull-back as coherent sheaves of $\omega_{\LT,0}$ to $\calM_{\LT,\infty}$. Since the $\varpi$-adic Tate module of $\calG$ gets trivialized on $\calM_{\LT,\infty}$, the Hodge--Tate exact sequence induces a $\GL_2(L)$-equivariant surjection 
\[
    V_L\ox_{\bbQ_p}\calO_{\LT,\infty}\to\omega_{\LT,\infty}(-1),
\]
where $V_L$ is the standard representation of $\GL_2(L)$ over $L$, and $\calO_{\LT,\infty}$ is the completed structure sheaf on $\calM_{\LT,\infty}$. By functoriality, this map is $\calO_L\ox_{\bbZ_p}C$-equivariant. Taking the $\iota$-component in the decomposition, we get an exact sequence  
\[
    0\to \omega_{\LT,\infty}^{-1}\ox {\det} V_\iota\to V_\iota\ox_C\calO_{\LT,\infty}\to\omega_{\LT,\infty}(-1)\to 0,
\]
where $V_\iota$ is the $\iota$-standard representation of $\GL_2(L)$ over $C$. Here we used the identification that the \emph{dual} of the $\varpi$-adic Tate module is isomorphic to $V_L$. Note that by the compatibility condition, $L$ already acts on $\omega_{\LT,\infty}$ via $\iota$. This defines the Hodge--Tate period map 
\[
    \pi_{\LT,\HT}:\calM_{\LT,\infty}\to\fl
\]
which is $\GL_2(L)$-equivariant, with the image is exactly the Drinfeld upper half plane $\Omega\subset\fl$. Note that the action of $\GL_2(L)$ on $\fl$ factors through the embedding $\iota$. Also $\pi_{\HT,\LT}^*(\omega_{\fl})=\omega_{\LT,\infty}(-1)$ with $\omega_{\fl}$ the tautological ample line bundle on $\fl$.

Besides, if we take other components of $V_L\ox_{\bbQ_p}\calO_{\calM_{\LT,\infty}}$, we get a $\GL_2(L)\times D_L^\times$-equivariant isomorphism 
\[
    V_\eta\ox_C\calO_{\LT,\infty}=D_{\LT,\infty,\eta}=D_{\LT,\infty,\eta}^{\sm}\ox_{\calO_{\LT,\infty}^{\sm}}\calO_{\LT,\infty},
\]
where $D_{\LT,\infty,\eta}$ is the pullback of $D_{\LT,\infty,0,\eta}$ to $\calM_{\LT,\infty}$, and $D_{\LT,\infty,\eta}^{\sm}$ is the set of smooth vectors in $D_{\LT,\infty,\eta}$ for the $\GL_2(L)$-action, which is also identified with 
\[
    D_{\LT,\infty,\eta}^{\sm}=\dlim_n\pi_n^{-1}D_{\LT,n,\eta}.
\]
As $D_{\LT,\infty,\eta}^{\sm}$ is get trivialized by the moduli problem, we get a $\GL_2(L)\times D_L^\times$-equivariant isomorphism 
\[
    V_\eta\ox_C\calO_{\LT,\infty}\isom D_{\LT,\infty,\eta}\isom W_\eta^*\ox_C \calO_{\LT,\infty}.
\]

Recall that we have a $D_L^\times$-equivariant map $\pi_{\GM}:\calM_{\LT,0}\to\fl_{\GM}$. We denote by $\pi_{\LT,\GM}:\calM_{\LT,\infty}\to\fl_{\GM}$ the composite of the projection map $\calM_{\LT,\infty}\to\calM_{\LT,0}$ and $\pi_{\GM}:\calM_{\LT,0}\to\fl_{\GM}$. It is $\GL_2(L)\times D_L^\times$-equivariant with respect to the trivial action of $\GL_2(L)$ on $\fl_{\GM}$ and natural action of $D_L^\times$ (on $W_\iota$) on $\fl_{\GM}$. Note that the centers $L^\times\subset\GL_2(L)$ and $L^\times\subset D_L^\times$ act in the same way on $\calM_{\LT,\infty}$.

\subsubsection{The Drinfeld space at infinite level}
In this subsection, we recall some results on Drinfeld tower. For this, let us first recall the definition of special formal $\calO_{D_L}$-modules.
\begin{definition}
Let $S$ be a scheme over $\calO_{\breve{L}}$. A special formal $\calO_{D_L}$-module over $S$ is a pair $(X,\iota)$ consisting of a formal group $X$ over $S$ and a ring homomorphism $\iota:\calO_D\inj \End(X)$ such that 
\begin{itemize}
    \item The action of $\calO_L$ on $\Lie(X)$ induced by $\iota$ coincides with that induced by $\calO_L\inj\calO_{\breve{L}}\to\calO_S$.
    \item Using the action of $\calO_{L_2}\subset\calO_D$ on $\Lie(X)$ by $\iota$, where $L_2/L$ is the unramified extension of degree $2$ lying in $D_L$, we can decompose $\Lie(X)$ with respect to the decomposition of $\calO_{L_2}\ox_{\bbZ_p}\calO_S$. Namely, we have $\Lie(X)=\bigoplus _{i\in\bbZ/2\bbZ}\Lie^i X$, here the action of $\calO_{L_2}$ on $\Lie^i(X)$ coincides with that induced by $\calO_{L_2}\ov{\sigma^i}\to \calO_{L_2}\inj\calO_{\breve{L}}\to\calO_S$. We require that each $\Lie^i(X)$ be a locally free $\calO_S$-module of rank $1$. This is the Kottwitz condition.
\end{itemize}
\end{definition}

Let $H_1$ be a special formal $\calO_{D_L}$-module over $\bar\bbF_q$ of $\calO_L$-height $4$, which is unique up to quasi-isogeny. Note that $\mathrm{QIsog}_{\calO_{D_L}}(H_1)=\GL_2(L)$. Let $\ffrm_{\Dr,0}$ be the functor from $\Nilp_{\calO_{\breve{L}}}$ to the category of sets, which assigns an $\calO_{\breve{L}}$-algebra $R$ to the set of equivalence classes of pairs $(G,\rho)$, where 
\begin{itemize}
    \item $G$ is a $p$-divisible special formal $\calO_{D_L}$-module over $R$.
    \item $\rho:H_1\ox_{\bar\bbF_q}R/\pi\dasharrow G\ox_R R/\pi$ is an $\calO_{D_L}$-equivariant quasi-isogeny.
\end{itemize}
Two pairs $(G,\rho)$ and $(G',\rho')$ are equivalent if the quasi-isogeny $\beta'\comp\beta^{-1}:G\ox_R R/\pi\dasharrow G'\ox_R R/\pi$ lifts to an isomorphism $\tilde{\beta'\comp\beta^{-1}}:G\aisom G'$ over $R$. Note that the property of being special implies compatible. 

Drinfeld \cite{Dr76} showed the functor $\ffrm_{\Dr,0}$ is represented by a formal scheme over $\Spf \calO_{\breve{L}}$. It is equipped with a $\GL_2(L)$-action via the identification $\GL_2(L)\isom \mathrm{QIsog}_{\calO_{D_L}}(H_1)$. Let $\calM_{\Dr,0}$ denote the base change of the adic generic fiber from $\breve{L}$ to $C$ via $\breve{\iota}$, where $\breve{\iota}$ is a lifting of $\iota$. Let $(\calG',\rho')$ be the universal $\varpi$-divisible special formal $\calO_{D_L}$-module on $\calM_{\Dr,0}$. The $\varpi$-adic Tate module defines a $\bbZ_p$-local system of rank $4d$ equipped with an action of $D_L$ on the \'etale site of $\calM_{\Dr,0}$, whose dual will be denoted by $V_{\Dr}$. For $n\ge 0$, we get an \'etale Galois covering $\calM_{\Dr,n}$ of $\calM_{\Dr,0}$ with Galois group $(\calO_{D_L}/\varpi^n)^\times=\calO_{D_L}^\times/(1+\varpi^n\calO_{D_L})$, by considering $\calO_{D_L}$-equivariant trivializations of $V_{\Dr}/\varpi^n V_{\Dr}$. Note that this covering $\calM_{\Dr,n}\to\calM_{\Dr,0}$ is $\GL_2(L)$-equivariant. By \cite{SW13}, there exists a perfectoid space $\calM_{\Dr,\infty}$ over $C$ such that 
\[
    \calM_{\Dr,\infty}\sim\ilim_n\calM_{\Dr,n}.
\]
There is a natural continuous right action of $\calO_{D_L}^\times$ on $\calM_{\Dr,\infty}$, which can be extended naturally to an action of $D_L^\times$.

Again there are two important maps, the Gross--Hopkins period map and the Hodge--Tate period map on $\calM_{\Dr,\infty}$. 

Let us first recall the definition of Gross--Hopkins period morphism. Let $(\calG',\rho')$ be the universal special formal $\calO_{D_L}$-module over $\ffrm_{\Dr,0}$. Its covariant Dieudonne crystal defines a vector bundle $M(\calG')$ of rank $4d$ on $\calM_{\Dr,0}$ equipped with an integrable connection $\nabla_{\Dr,0}$, the dual of the Gauss--Manin connection. $M(\calG')$ also carries with an action of $\calO_{D_L}\ox_{\bbZ_p}C$. Let $M(H_1)$ denote the covariant Dieudonne module of $H_1$, which is a free $W(\bar\bbF_p)$-module of rank $4d$, equipped with an $\calO_{D_L}$-action. The universal quasi-isogeny $\rho'$ induces a natural trivialization $M(\calG')\isom M(H_1)\ox_{W(\bar\bbF_p)}\calO_{\calM_{\Dr}}$ under which $\nabla_{\Dr,0}$ is identified with the standard connection on $\calO_{\calM_{\Dr,0}}$. The Lie algebra of $\calG'$ defines a rank $2$ vector bundle $\Lie(\calG')$ over $\calM_{\Dr,0}$, which carries an $\calO_{D_L}\ox_{\bbZ_p} C$-action. By the Grothendieck--Messing theory, $(\calG',\rho')$ gives rise to a surjection 
\[
    M(H_1)\ox_{W(\bar\bbF_p)}\calO_{\calM_{\Dr}}\isom M(\calG')\to\Lie(\calG'),
\]
which is $\calO_{D_L}\ox_{\bbZ_p}C$-equivariant. Let 
\[
    M(\calG')_\iota\to \Lie(\calG')
\]
be the direct summand of the above surjection such that the action of $\calO_{D_L}\ox_{\bbZ_p}C$ factors through $\calO_{D_L}\ox_{L,\iota}C$. Note that by the special condition, the action of $\calO_{D_L}\ox_{\bbZ_p}C$ on $\Lie(\calG')$ already factors through $\calO_{D_L}\ox_{L,\iota}C$. It induces the so-called Gross--Hopkins period map 
\[
    \pi_{\GM}':\calM_{\Dr,\infty}\to\fl'_{\GM}
\]
where $\fl'_{\GM}$ is the adic space over $C$ associated to the flag variety parametrizing $2$-dimensional $D_L$-equivariant quotients of the $4$-dimensional $C$-vector space $(M(H_1)\ox_{W(\bar\bbF_p)}C)_\iota$. Note that this map is $\GL_2(L)$-equivariant. Also note that the $\GL_2(L)$-action on $\fl'_{\GM}$ factor through the embedding $\iota$. Let $W_\iota$ be the standard $\iota$-algebraic representation of $D_L^\times$ over $C$. Then everything above is $W_{\iota}$-isotypic. More precisely, let $D(\calG')$ be the dual of $M(\calG')$. Let $D_{\Dr,0,\iota}=\Hom_{D_L^\times}(W_{\iota},D(\calG')_\iota)$, which is a $\GL_2(L)$-equivariant vector bundle on $\calM_{\Dr,0}$ of rank $2$, with trivialization $D_{\Dr,0,\iota}\isom V_\iota^*\ox_C\calO_{\Dr,0}$. Similarly one defines $D_{\Dr,0,\eta}=\Hom_{D_L^\times}(W_\eta,D(\calG')_\eta)$, with trivialization $D_{\Dr,0,\eta}\isom V_\eta^*\ox_C\calO_{\Dr,0}$. Let $\omega_{\Dr,0}^{-1}=\Hom_{D_L^\times}(W_{\iota},\Lie(\calG'))$. The Lie algebra of the universal Barsotti--Tate group defines the Hodge filtration on $D_{\Dr,0}$. More precisely, the dual of the surjection $M(\calG)_\iota\surj \Lie(\calG')$ induces a Hodge filtration on $D_{\Dr,0,\iota}$, given by $\Fil^{0}D_{\Dr,0,\iota}=D_{\Dr,0,\iota}$, $\Fil^1D_{\Dr,0,\iota}=\omega_{\Dr,0}$ and $\Fil^2D_{\Dr,0,\iota}=0$. We have the Kodaira--Spencer isomorphism 
\[
    \KS:\Fil^1D_{\Dr,0,\iota}\ov{\nabla_{\Dr,0}}\to D_{\Dr,0,\iota}\ox_{\calO_{\calM_{\Dr,0}}}\Omega_{\calM_{\Dr,0}}^1\to \gr^0D_{\Dr,0,\iota}\ox_{\calO_{\calM_{\Dr,0}}}\Omega_{\calM_{\Dr,0}}^1.
\]
This gives $\Omega_{\calM_{\Dr}}^1\isom\omega_{\Dr,0}^2\ox{\det}_\iota^{-1}$. Note that the Hodge filtration on $D_{\Dr,0,\eta}$ for $\eta\neq\iota$ is trivial.

As for the Gross--Hopkins period morphism, Drinfeld shows that the image of $\pi_{\GM}'$ is exactly $\Omega$, the Drinfeld upper half plane. Moreover, if we let $\calM_{\Dr,0}^{(0)}$ denote the component of $\calM_{\Dr,0}$ such that the quasi-isogeny is an isomorphism, then $\pi_{\GM}'|_{\calM^{(0)}_{\Dr,0}}$ is an isomorphism onto $\Omega$. Also if we let $\omega_{\fl_{\GM}'}$ denote the tautological ample line bundle on $\fl_{\GM}'$, we have ${\pi'}_{\GM}^*(\omega_{\fl_{\GM}'})=\omega_{\Dr,0}^{-1}$.

Again the above constructions are compatible by adding finite level structures on $\calM_{\Dr,0}$. We may pull back the vector bundles on $\calM_{\Dr,0}$ defined before as coherent modules to $\calM_{\Dr,n}$ via the canonical projection $\calM_{\Dr,n}\to \calM_{\Dr,0}$, given by $D_{\Dr,n}=\bigoplus_\eta D_{\Dr,n,\eta}$. The Hodge filtration on $D_{\Dr,n,\eta}$ is trivial if $\eta\neq \iota$, and the Hodge filtration on $D_{\Dr,n,\iota}$ is given by $\omega_{\Dr,n}$. Similarly we have the Kodaira--Spencer isomorphisms. The Gross--Hopkins period map on $\calM_{\Dr,n}$ are defined to be the composition $\calM_{\Dr,n}\to\calM_{\Dr,0}\ov{\pi_{\GM}}\to \fl_{\GM}'$.

Now let us recall the definition of Hodge--Tate period map on $\calM_{\Dr,\infty}$. Let $\omega_{\Dr,\infty}$ be the pull-back (as coherent sheaves) of $\omega_{\Dr,0}$ to $\calM_{\Dr,\infty}$. Since the $\pi$-adic Tate module of $\calG'$ gets trivialized on $\calM_{\Dr,\infty}$, the Hodge--Tate exact sequence induces a $D_L^\times$-equivariant surjection 
\[
    \calO_{D_L}\ox_{\bbZ_p}\calO_{\Dr,\infty}\to\Lie(\calG')^{\vee}(-1)
\]
where $\calO_{\Dr,\infty}$ is the completed structure sheaf on $\calM_{\Dr,\infty}$. By functoriality, this map is $\calO_L\ox_{\bbZ_p}C$-equivariant. Taking the $\iota$-component in the decomposition, we get a surjection
\[
    \calO_{D_L}\ox_{L,\iota}\calO_{\Dr,\infty}\to\Lie(\calG')^{\vee}(-1).
\]
Note that by the compatibility condition, $L$ already acts on $\Lie(\calG')^{\vee}$ via $\iota$. Let $\fl'$ be the adic space over $C$ associated to the flag variety parametrizing $2$-dimensional $D_L$-equivariant quotients of the $4$-dimensional $C$-vector space $\calO_{D_L}\ox_{L,\iota}C$. Then the Hodge--Tate sequence defines the Hodge--Tate period map  
\[
    \pi_{\Dr,\HT}:\calM_{\Dr,\infty}\to\fl'
\]
which is $D_L^\times$-equivariant, and surjective. Note that the action of $D_L^\times$ on $\fl'$ factors through the embedding $\iota$. Also $\pi_{\Dr,\HT}^*(\omega_{\fl'})=\omega_{\Dr,\infty}(-1)$. Again, the above construction are $W_{\iota}$-isotypic, in the sense that if we take $W_\iota$-isotypic components, we get the Hodge--Tate exact sequence 
\[
    0\to \omega_{\Dr,\infty}^{-1}\ox {\det} W_\iota\to  W_\iota\ox_C \calO_{\Dr,\infty}\to \omega_{\Dr,\infty}(-1)\to 0,
\]
and an isomorphism 
\[
    W_\eta\ox_C\calO_{\Dr,\infty}\isom D_{\Dr,\infty,\eta}\isom D_{\Dr,\infty,\eta}^{\sm}\ox_{\calO_{\Dr,\infty}^{\sm}}\calO_{\Dr,\infty}\isom V_\eta^*\ox_C\calO_{\Dr,\infty}.
\]

Recall that we have a $\GL_2(L)$-equivariant map $\pi_{\GM}':\calM_{\Dr,0}\to\fl_{\GM}'$. We denote by $\pi_{\Dr,\GM}:\calM_{\Dr,\infty}\to\fl_{\GM}'$ the composite of the projection map $\calM_{\Dr,\infty}\to\calM_{\Dr,0}$ and $\pi_{\GM}':\calM_{\Dr,0}\to\fl_{\GM}'$. It is $\GL_2(L)\times D_L^\times$-equivariant with respect to the trivial action of $D_L^\times$ on $\fl_{\GM}'$ and natural action of $\GL_2(L)$ (on $V_\iota$) on $\fl_{\GM}'$. Note that the centers $L^\times\subset\GL_2(L)$ and $L^\times\subset D_L^\times$ act in the same way on $\calM_{\Dr,\infty}$.

\subsubsection{Locally analytic sections on Lubin--Tate spaces and Drinfeld spaces at infinite level}
In this subsection, we first recall a duality result due to Scholze-Weinstein \cite{SW13}, which relates Lubin--Tate spaces at infinite level, and Drinfeld spaces at infinite level. From this result, we can deduce some relations between the Lubin--Tate side and Drinfeld side. Let $\calO_{\LT,\infty}$ denote the completed structure sheaf on $\calM_{\LT,\infty}$ and $\calO_{\Dr,\infty}$ denote the completed structure sheaf on $\calM_{\Dr,\infty}$.

Recall the following theorem \cite[Proposition 7.2.2, Theorem 7.2.3]{SW13}.
\begin{theorem}\label{thm:SWduality}
There are natural $\GL_2(L)\times D_L^\times$-equivariant isomorphisms 
\[
    \calM_{\LT,\infty}\isom\calM_{\Dr,\infty}
\]
such that $\pi_{\LT,\HT}$ is naturally identified with $\pi_{\Dr,\GM}$ and $\pi_{\LT,\GM}$ is naturally identified with $\pi_{\Dr,\HT}$. This induces a $\GL_2(L)\times D_L^\times$-equivariant isomorphism 
\[
    \calO_{\LT,\infty}\isom \calO_{\Dr,\infty}.
\]
\end{theorem}

We fix some notation to be consistent with the global situation. Let 
\[
    0\to \omega_{\fl}^{(0,1)}\to V^{(1,0)}_\iota\ox\calO_{\fl}\to \omega_{\fl}^{(1,0)}\to 0
\]
be an exact sequence of $\GL_2(L)$-equivariant sheaves on $\fl$. 
\begin{itemize}
    \item Applying $\pi_{\LT,\HT}^*$, we get 
    \[
        0\to \omega_{\LT,\infty}^{(0,-1)}\to V^{(1,0)}_{\iota}\ox_C\calO_{\LT,\infty}\to \omega_{\LT,\infty}^{(-1,0)}(-1)\to 0.
    \]
    \item Applying $\pi_{\Dr,\GM}^*$, we get 
    \[
        0\to \omega_{\Dr,\infty}^{(0,1)}\to V^{(1,0)}_{\iota}\ox_C\calO_{\Dr,\infty}\to \omega_{\Dr,\infty}^{(1,0)}\to 0.
    \]
\end{itemize}

Let 
\[
    0\to \omega_{\fl'}^{(0,1)}\to W^{(1,0)}_{\iota}\ox\calO_{\fl'}\to \omega_{\fl'}^{(1,0)}\to 0
\]
be an exact sequence of $D_L^\times$-equivariant sheaves on on $\fl'$. 
\begin{itemize}
    \item Applying $\pi_{\Dr,\HT}^*$, we get 
    \[
        0\to \omega_{\Dr,\infty}^{(0,-1)}\to W^{(1,0)}_{\iota}\ox_C\calO_{\Dr,\infty}\to \omega_{\Dr,\infty}^{(-1,0)}(-1)\to 0.
    \]
    \item Applying $\pi_{\LT,\GM}^*$, we get 
    \[
        0\to \omega_{\LT,\infty}^{(0,1)}\to W^{(1,0)}_{\iota}\ox_C\calO_{\LT,\infty}\to \omega_{\LT,\infty}^{(1,0)}\to 0.
    \]
\end{itemize}
Moreover, for $(a,b)\in\bbZ^2$, we define 
\[
    \omega_{\LT,\infty}^{(a,b)}:=(\omega_{\LT,\infty}^{(1,0)})^{\ox (a-b)}\ox ({\det}V_{\iota}^{(1,0)})^{\ox b}\qquad \omega_{\Dr,\infty}^{(a,b)}:=(\omega_{\Dr,\infty}^{(1,0)})^{\ox (a-b)}\ox ({\det}W_{\iota}^{(1,0)})^{\ox b}.
\]

In particular, as $\pi_{\LT,\GM}$ is naturally identified with $\pi_{\Dr,\HT}$, we get the following corollary.
\begin{corollary}\label{lg}
Under the natural isomorphism $\calM_{\LT,\infty}\isom \calM_{\Dr,\infty}$, we have a $\GL_2(L)\times D_L^\times$-equivariant and Galois equivariant isomorphism 
\[
    \omega_{\LT,\infty}^{(-a,-b)}(-a)\isom\omega_{\Dr,\infty}^{(a,b)}.
\]
\end{corollary}
Besides, for any $\eta\neq\iota$ and $(a,b)\in\bbZ^2$ with $a\ge b$, we have 
\[
    V_\eta^{(a,b)}\ox_C\calO_{\LT,\infty}\isom (D_{\LT,\eta}^{(a,b)})^*\isom (W_\eta^{(a,b)})^*\ox_C\calO_{\LT,\infty},
\]
where $D_{\LT,\eta}^{(a,b)}:=\Sym^{a-b}D_{\LT,\eta}^*\ox_{\calO_{\LT,\infty}} {\det}^bD_{\LT,\eta}^*$. This shows that 
\[
    \calO_{\LT,\infty}^{V_\eta\-\lalg}\isom V_\eta\ox W_\eta\ox \calO_{\LT,\infty}^{\sm}.
\]
Similarly, if we define $D_{\Dr,\eta}^{(a,b)}:=\Sym^{a-b}D_{\Dr,\eta}^*\ox_{\calO_{\Dr,\infty}} {\det}^bD_{\Dr,\eta}^*$, we have a natural isomorphism 
\[
    W_\eta^{(a,b)}\ox_C\calO_{\Dr,\infty}\isom (D_{\Dr,\eta}^{(a,b)})^*\isom (V_\eta^{(a,b)})^*\ox_C\calO_{\Dr,\infty},
\]
and 
\[
    \calO_{\Dr,\infty}^{W_\eta\-\lalg}\isom W_\eta\ox V_\eta\ox \calO_{\Dr,\infty}^{\sm}.
\]

Let $\calM_{\LT,\infty}^{(0)}$ be the component of $\calM_{\LT,\infty}$ such that the quasi-isogenies in the moduli problem are of height $0$. Note that $\calM_{\LT,\infty}^{(0)}$ is $\GL_2(L)^0\times \calO_{D_L}^\times$-stable. Let $\pi_{\HT,\LT}^{(0)}:\calM_{\LT,\infty}^{(0)}\to \Omega$ be the restriction of $\pi_{\HT,\LT}$ to $\calM_{\LT,\infty}^{(0)}$ and $\Omega$. By embedding $\calM_{\LT,\infty}^{(0)}$ into unitary Shimura curves, we deduce that there exists a basis of open affinoid subsets $\ffrb$ of $\Omega$, such that for any $U\in\ffrb$, the preimage $\pi_{\HT,\LT}^{(0),-1}(U)$ is affinoid perfectoid and is also the preimage of some open affinoid subset of $\calM_{\LT,n}^{(0)}$. Let 
\begin{align*}
    \calO_{\LT}:=\pi_{\LT,\HT,*}^{(0)}\calO_{\calM_{\LT,\infty}^{(0)}}.
\end{align*}

Let $\calM_{\Dr,\infty}^{(0)}$ be the component of $\calM_{\Dr,\infty}$ such that the quasi-isogenies in the moduli problem are of height $0$. Let $\pi_{\Dr,\GM}^{(0)}:\calM_{\Dr,\infty}^{(0)}\to \Omega$ be the restriction of $\pi_{\Dr,\GM}$ to $\calM_{\Dr,\infty}^{(0)}$ and $\Omega$. By \cite[Proposition 4.2]{CDN20}, there exists a basis of open affinoid subsets $\ffrb$ of $\Omega$, such that for any $U\in\ffrb$, the preimage $\pi_{\Dr,\GM}^{(0),-1}(U)$ is affinoid perfectoid. Let 
\begin{align*}
    \calO_{\Dr}:=\pi_{\Dr,\GM,*}^{(0)}\calO_{\calM_{\Dr,\infty}^{(0)}}.
\end{align*}
By Theorem \ref{thm:SWduality}, there is a $\GL_2(L)^0\times\calO_{D_L}^\times$-equivariant isomorphism 
\begin{align*}
    \calO_{\LT}\isom\calO_{\Dr}.
\end{align*}

We write $\calO_{\LT}^{\lan}$ (resp. $\calO_{\Dr}^{\lan}$) for the set of locally analytic vectors in $\calO_{\LT}$ for the $\GL_2(L)$-action (resp. in $\calO_{\Dr}$ for the $D_L^\times$-action). Similarly we define $\calO_{\LT}^{\sm}\subset\calO_{\LT}^{\lalg}\subset\calO_{\LT}^{\lan}$ to be the set of smooth vectors, locally algebraic vectors for the $\GL_2(L)$-action, and $\calO_{\Dr}^{\sm}\subset\calO_{\Dr}^{\lalg}\subset\calO_{\Dr}^{\lan}$ to be the set of smooth vectors, locally algebraic vectors for the $D_L^\times$-action. Let $\omega_{\LT}^{(a,b)}:=\pi_{\LT,\HT,*}^{(0)}\omega_{\LT,\infty}^{(a,b)}|_{\calM_{\LT,\infty}^{(0)}}$ with $(a,b)\in\bbZ^2$, and we define $D_{\LT,\eta}^{(a,b)}$ with $\eta\neq\iota$ and $(a,b)\in\bbZ^2$, $a\ge b$, $\omega_{\Dr}^{(a,b)}$ with $(a,b)\in\bbZ^2$, $D_{\Dr,\eta}^{(a,b)}$ with $\eta\neq\iota$ and $(a,b)\in\bbZ^2$, $a\ge b$ similarly.

We study the local structure of $\calO_{\LT}^{\lan}$, using similar methods when we study locally analytic sections on the unitary Shimura curve of infinite level at $p$. Since the situation is very similar to that case, we omit the proof.
\begin{theorem}\label{thm:OLTla}
We have $\calO_{\LT}^{V_\iota^{(a,b)}\-\lalg}\isom V_{\iota}^{(a,b)}\ox \omega_{\LT}^{(a,b),\sm}$ and for $\eta\neq\iota$, $\calO_{\LT}^{V_\eta^{(a,b)}\-\lalg}\isom V_\eta^{(a,b)}\ox W_\eta^{(a,b)}\ox \calO_{\LT}^{\sm}$. There is a natural decomposition 
\[
    \calO_{\LT}^{\lan}\isom \hat \ox_{\calO_{\LT}^{\sm},\eta}\calO_{\LT}^{\eta\-\lan}.
\]
with $\calO_{\LT}^{\lalg}$ dense in $\calO_{\LT}^{\lan}$ in the LB topology. More precisely, for $U\in\ffrb$,
\begin{enumerate}[(i)]
    \item Let $b\in H^0(U,\omega_{\LT}^{(-1,0)})$ be a basis, $e_{1,\iota},e_{2,\iota}\in H^0(U,\omega_{\LT}^{(-1,0)}(-1))$ be images of a standard basis of $V_{\iota}^{(1,0)}$ in the Hodge--Tate filtration. Let $\mathrm{t}\in\calO_{\LT}^{V^{(1,1)}_{\iota}\-\lalg}(U)$ be a generator. Let $e_{1,\iota,n},e_{2,\iota,n},\mathrm{t}_{n}\in\calO_{\LT}^{\sm}(U)$ such that $\frac{e_{1,\iota}}{b}-e_{1,\iota,n},\frac{e_{2,\iota}}{b}-e_{2,\iota,n},\mathrm{t}-\mathrm{t}_n$ has norm smaller than $p^{-n}$. Then for any $s\in\calO_{\LT}^{\iota\-\lan}(U)$, there exists a sufficiently large $n$, such that 
    \[
        s=\sum_{i,j,k\ge 0}a_{ijk}(\frac{e_{1,\iota}}{b}-e_{1,\iota,n})^i(\frac{e_{2,\iota}}{b}-e_{2,\iota,n})^j(\mathrm{t}-\mathrm{t}_n)^k
    \]
    with $a_{ijk}\in\calO_{\LT}^{\sm}(U)$.
    \item Let $\eta$ be an embedding different from $\iota$. Let $e_{ij,\eta},i,j=1,2$ be a basis of $V_\eta\ox W_\eta$. Let $e_{ij,\eta,n}\in\calO_{\LT}^{\sm}$ such that $e_{ij,\eta}-e_{ij,\eta,n}$ has norm smaller than $p^{-n}$. Then for any $s\in\calO_{\LT}^{\eta\-\lan}(U)$, there exists a sufficiently large $n$, such that 
    \[
        s=\sum_{i,j,k,l\ge 0}a_{ijkl}(e_{11,\eta}-e_{11,\eta,n})^i(e_{12,\eta}-e_{12,\eta,n})^j(e_{21,\eta}-e_{21,\eta,n})^k(e_{22,\eta}-e_{22,\eta,n})^l
    \]
    with $a_{ijkl}\in\calO_{\LT}^{\sm}(U)$.
\end{enumerate}
\end{theorem}
Note that we may choose $\mathrm{t}\in\calO_{\LT}^{V_{\iota}^{(1,1)}\-\lalg}(U)\isom V_\iota^{(1,1)}\ox W_\iota^{(1,1)}\ox \calO_{\LT}^{\sm}(U)$ to be a generator of $V_\iota^{(1,1)}\ox W_\iota^{(1,1)}$. This element is also a generator for $\calO_{\Dr}^{W_\iota^{(1,1)}\-\lalg}$.

Dually, we have a similar description for $D_L^\times$-locally analytic sections on Drinfeld space at infinite level. 
\begin{theorem}\label{thm:ODrla}
We have $\calO_{\Dr}^{W_\iota^{(a,b)}\-\lalg}\isom W_{\iota}^{(a,b)}\ox \omega_{\Dr}^{(a,b),\sm}$ and for $\eta\neq\iota$, $\calO_{\Dr}^{W_\eta^{(a,b)}\-\lalg}\isom W_\eta^{(a,b)}\ox V_\eta^{(a,b)}\ox \calO_{\Dr}^{\sm}$. There is a natural decomposition 
\[
    \calO_{\Dr}^{\lan}\isom \hat \ox_{\calO_{\Dr}^{\sm},\eta}\calO_{\Dr}^{\eta\-\lan}.
\]
with $\calO_{\Dr}^{\lalg}$ dense in $\calO_{\Dr}^{\lan}$ in the LB topology. More precisely, let $U\in\ffrb$ be a sufficiently small affinoid perfectoid open subset. 
\begin{enumerate}[(i)]
    \item Let $c\in H^0(U,\omega_{\Dr}^{(-1,0)})$ be a basis, $f_{1,\iota},f_{2,\iota}\in H^0(U,\omega_{\Dr}^{(-1,0)}(-1))$ be images of a standard basis of $W_{\iota}^{(1,0)}$ in the Hodge--Tate filtration. Let $\mathrm{t}\in\calO_{\Dr}^{W^{(1,1)}_{\iota}\-\lalg}(U)$ be a generator. Let $f_{1,\iota,n},f_{2,\iota,n},\mathrm{t}_{n}'\in\calO_{\Dr}^{\sm}(U)$ such that $\frac{f_{1,\iota}}{c}-f_{1,\iota,n},\frac{f_{2,\iota}}{c}-e_{2,\iota,n},\mathrm{t}-\mathrm{t}_n'$ has norm smaller than $p^{-n}$. Then for any $s\in\calO_{\Dr}^{\iota\-\lan}(U)$, there exists a sufficiently large $n$, such that 
    \[
        s=\sum_{i,j,k\ge 0}a_{ijk}(\frac{f_{1,\iota}}{c}-f_{1,\iota,n})^i(\frac{f_{2,\iota}}{c}-f_{2,\iota,n})^j(\mathrm{t}-\mathrm{t}_n')^k
    \]
    with $a_{ijk}\in\calO_{\Dr}^{\sm}(U)$.
    \item Let $\eta$ be an embedding different from $\iota$. Let $f_{ij,\eta},i,j=1,2$ be a basis of $V_\eta\ox W_\eta$. Let $f_{ij,\eta,n}\in\calO_{\Dr}^{\sm}$ such that $f_{ij,\eta}-f_{ij,\eta,n}$ has norm smaller than $p^{-n}$. Then for any $s\in\calO_{\Dr}^{\eta\-\lan}(U)$, there exists a sufficiently large $n$, such that 
    \[
        s=\sum_{i,j,k,l\ge 0}a_{ijkl}(f_{11,\eta}-f_{11,\eta,n})^i(f_{12,\eta}-f_{12,\eta,n})^j(f_{21,\eta}-f_{21,\eta,n})^k(f_{22,\eta}-f_{22,\eta,n})^l
    \]
    with $a_{ijkl}\in\calO_{\Dr}^{\sm}(U)$.
\end{enumerate}
\end{theorem}

There are some close relations between $\calO_{\LT}^{\lan}$ and $\calO_{\Dr}^{\lan}$.
\begin{proposition}
Let $J\subset \Sigma$ be a set of embeddings such that $J\ni\iota$. Then the natural isomorphism $\calO_{\LT}\isom \calO_{\Dr}$ induces a natural $\GL_2(L)\times D_L^\times$-equivariant identification 
\[
    \calO_{\LT}^{J\-\lan}\isom \calO_{\Dr}^{J\-\lan}.
\]
The above isomorphism restricts to an isomorphism 
\[
    \calO_{\LT}^{\iota\-\lan,V_\eta^{(a,b)}\-\lalg}\isom \calO_{\Dr}^{\iota\-\lan,W_\eta^{(a,b)}\-\lalg}
\]
for $\eta\neq\iota$, $(a,b)\in\bbZ^2$ with $a\ge b$.
\begin{proof}
We first prove that $\calO_{\LT}^{\iota\-\lan}=\calO_{\Dr}^{\iota\-\lan}$. Roughly speaking, as $\calM_{\LT,n}^{(0)}$ is locally affinoid of finite type, it is easy to check that $\calO_{\LT}^{\sm}\subset \calO_{\Dr}^{\lan}$ by \cite[Lemma 2.3.4]{RC1}. As $\calM_{\LT,n}$ is an \'etale covering of $\fl'$ via the Gross--Hopkins period map, and the $D_L^\times$-action on $\fl'$ factors through $\iota$, we see that $\calO_{\LT}^{\sm}\subset \calO_{\Dr}^{\iota\-\lan}$. Then from the explicit power series expansions of $\calO_{\LT}^{\iota\-\lan}$ and $\calO_{\Dr}^{\iota\-\lan}$, we can directly check the equality $\calO_{\LT}^{\iota\-\lan}\isom \calO_{\Dr}^{\iota\-\lan}$. This part of proof is similar to \cite[Corollary 5.3.9]{PanII}.

Now we handle the locally $J\bs\{\iota\}$-algebraic part. For $\eta\neq\iota$, by Corollary \ref{lg} we have
\[
    \calO_{\LT}^{\eta\-\lalg}\isom \bigoplus_{(a,b)\in\bbZ^2,a\ge b}V_\eta^{(a,b)}\ox W_\eta^{(a,b)}\ox \calO_{\LT}^{\sm},
\]
and we have a similar description for $\calO_{\Dr}^{\eta\-\lalg}$. As we just show that $\calO_{\LT}^{\sm}\subset \calO_{\Dr}^{\iota\-\lan}$ and $\calO_{\Dr}^{\sm}\subset \calO_{\LT}^{\iota\-\lan}$, we have the isomorphism
\[
    \calO_{\LT}^{\iota\-\lan,J\bs\{\iota\}\-\lalg}\isom \calO_{\Dr}^{\iota\-\lan,J\bs\{\iota\}\-\lalg}
\]
for any set of embeddings $J$ contains $\iota$.

Next, we want to show $\calO_{\LT}^{\eta\-\lan}\subset \calO_{\Dr}^{\{\iota,\eta\}\-\lan}$. Let $G_n=1+p^n\Mat_2(\calO_L)\subset \GL_2(L)$ be an open compact subgroup. For $n$ sufficiently large, by Lemma \ref{lem:GNnormusualnorm}, the $G_n$-analytic norm of $e_{ij,\eta}-e_{ij,\eta,n}$ equals the usual norm. See Theorem \ref{thm:OLTla} for definition of $e_{ij,\eta}-e_{ij,\eta,n}$. Moreover, as the action of $D_L^\times$ on $e_{ij,\eta}-e_{ij,\eta,n}$ is locally analytic, for $m$ sufficiently large the $\check G_m$-analytic norm on $e_{ij,\eta}-e_{ij,\eta,n}$ is also the usual norm, where $\check G_m:=1+p^m\calO_{D_L}\subset D_L^\times$. From this, we deduce $\calO_{\LT}^{\eta\-\lan}\subset \calO_{\Dr}^{\{\iota,\eta\}\-\lan}$. Dually, we have $\calO_{\Dr}^{\eta\-\lan}\subset \calO_{\LT}^{\{\iota,\eta\}\-\lan}$. From this one deduces that $\calO_{\LT}^{J\-\lan}\isom \calO_{\Dr}^{J\-\lan}$ for any $J\ni\iota$.
\end{proof}
\end{proposition}

By the geometric Sen theory, we know that $\calO_{\LT}^{\lan}$ is killed by the universal nilpotent algebra on $\fl$. Hence it induces a horizontal action of $\frh_{\iota}$ on $\calO_{\LT}^{\lan}$ which we denote by $\theta_{\LT,\frh}$. Here $\frh_\iota\subset \frg_\iota$ is the Cartan subalgebra consists of diagonal matrices. Similarly, we have a horizontal action of $\check\frh_{\iota}$ on $\calO_{\Dr}^{\lan}$ which we denote by $\theta_{\Dr,\frh}$. Here, $\check\frh_{\iota}$ is the image of $\frh_\iota$ under the isomorphism $\frg_\iota\isom \check \frg_\iota$.
\begin{proposition}
The two actions $\theta_{\LT,\frh}$ and $\theta_{\Dr,\frh}$ on $\calO_{\LT}^{\lan}\isom \calO_{\Dr}^{\lan}$ are the same.

Moreover, under the natural isomorphism 
\[
    \omega_{\Dr}^{(a,b)}\isom \omega_{\LT}^{(-a,-b)}(-a)
\]
for some $(a,b)\in\bbZ^2$, we get an isomorphism 
\[
    \omega_{\Dr}^{(a,b),\lan,\theta_{\Dr,\frh}=(n_1,n_2)}\isom \omega_{\LT}^{(-a,-b),\lan,\theta_{\LT,\frh}=(n_1+a,n_2+b)}(-a)
\]
for any $(n_1,n_2)\in\bbZ^2$.
\begin{proof}
The proof is similar to \cite[Cor 5.3.13]{PanII}. For the claim on twists, first note that the isomorphism $\omega_{\Dr}^{(a,b)}\isom \omega_{\LT}^{(-a,-b)}(-a)$ remains true when we passing to locally analytic vectors for the $\GL_2(L)$-action and the $D_L^\times$-action. We have 
\[
    \omega_{\Dr}^{(a,b),\lan,\theta_{\Dr}=(n_1,n_2)}=\omega_{\Dr}^{(a,b),\sm}\ox_{\calO_{\Dr}^{\sm}}\calO_{\Dr}^{\lan,\theta_{\Dr}=(n_1,n_2)}=\omega_{\Dr}^{(a,b),\sm}\ox_{\calO_{\Dr}^{\sm}}\calO_{\LT}^{\lan,\theta_{\LT}=(n_1,n_2)}.
\]
As the Gross--Hopkins period map $\pi_{\Dr,\GM,0}:\calM_{\Dr,0}\to\Omega\subset\fl$ induces $\omega_{\Dr,0}^{(a,b)}=\omega_{\fl}^{(a,b)}$, and $\theta_{\LT}$ acts on $\omega_{\fl}^{(a,b)}$ via $(a,b)$, we see that $\omega_{\Dr}^{(a,b),\sm}\ox_{\calO_{\Dr}^{\sm}}\calO_{\LT}^{\lan,\theta_{\LT}=(n_1,n_2)}$ is $(n_1+a,n_2+b)$-isotypic for the $\theta_{\LT}$-action, hence the claim.
\end{proof}
\end{proposition}

Let $k\ge 0$ be an integer. We define differential operators on $\calO_{\LT}^{\lan,(0,-k)}\isom \calO_{\Dr}^{\lan,(0,-k)}$. The constructions are the same as Theorem \ref{thm:d} and Theorem \ref{thm:dbar}. See \cite[Theorem 5.2.15 and Remark 5.2.16]{PanII} in the $L=\bbQ_p$ case.
\begin{proposition}\label{d1}
There exists a unique continuous operator 
\[
    d^{k+1}_{\LT}:\calO_{\LT}^{\lan,(0,-k)}\to \calO_{\LT}^{\lan,(0,-k)}\ox_{\calO_{\LT}^{\sm}}(\Omega_{\LT}^{1,\sm})^{\ox k+1}
\]
such that it is determined by the following properties:
\begin{enumerate}[(i)]
    \item $d^{k+1}_{\LT}$ is $\calO_{\fl}$-linear.
    \item the restriction of $d^{k+1}_{\LT}$ on $\calO_{\LT}^{\lalg,(0,-k)}$ is given by the Gauss--Manin connection.
\end{enumerate}
\end{proposition}

\begin{proposition}\label{d2}
There exists a unique continuous operator
\[
    \bar d^{k+1}_{\LT}:\calO_{\LT}^{\lan,(0,-k)}\to \calO_{\LT}^{\lan,(0,-k)}\ox_{\calO_{\fl}}(\Omega^1_{\fl})^{\ox k+1}
\]
such that it is determined by the following properties:
\begin{enumerate}[(i)]
    \item $\bar d^{k+1}_{\LT}$ is $\calO_{\LT}^{\iota^c\-\lan}$-linear.
    \item There exists a non-zero constant $c\in\bbQ^\times$ such that $\bar d^{k+1}_{\LT}(s)=c(u^+)^{k+1}(s)\ox (dx)^{k+1}$ for any $s\in\calO_{K^v}^{\lan,\chi}$.
\end{enumerate}
Moreover, it is surjective with kernel $\calO_{\LT}^{\iota\-\lalg,\iota^c\-\lan,(0,-k)}$.
\end{proposition}
Here the surjectivity of $\bar d_{\LT}^{k+1}$ follows by a similar calculation of $\frn$-cohomology in \cite[Proposition 4.2.9, Theorem 5.2.16]{PanII}.

On the Drinfeld side, we have similar results. See \cite[Theorem 5.3.17 and Remark 5.3.18]{PanII} in the $L=\bbQ_p$ case.
\begin{proposition}\label{d3}
There exists a unique continuous operator 
\[
    d^{k+1}_{\Dr}:\calO_{\Dr}^{\lan,(0,-k)}\to \calO_{\Dr}^{\lan,(0,-k)}\ox_{\calO_{\Dr}^{\sm}}(\Omega_{\Dr}^{1,\sm})^{\ox k+1}
\]
such that it is determined by the following properties:
\begin{enumerate}[(i)]
    \item $d^{k+1}_{\Dr}$ is $\calO_{\fl'}$-linear.
    \item the restriction of $d^{k+1}_{\Dr}$ on $\calO_{\Dr}^{\lalg,(0,-k)}$ is given by the Gauss--Manin connection.
\end{enumerate}
\end{proposition}

\begin{proposition}\label{d4}
There exists a unique continuous operator
\[
    \bar d^{k+1}_{\Dr}:\calO_{\Dr}^{\lan,(0,-k)}\to \calO_{\Dr}^{\lan,(0,-k)}\ox_{\calO_{\fl'}}(\Omega^1_{\fl'})^{\ox k+1}
\]
such that it is determined by the following properties:
\begin{enumerate}[(i)]
    \item $\bar d^{k+1}_{\Dr}$ is $\calO_{\Dr}^{\iota^c\-\lan}$-linear.
    \item There exists a non-zero constant $c\in\bbQ^\times$ such that $\bar d^{k+1}_{\Dr}(s)=c(u^+)^{k+1}(s)\ox (dx)^{k+1}$ for any $s\in\calO_{K^v}^{\lan,(0,-k)}$.
\end{enumerate}
Moreover, it is surjective with kernel $\calO_{\Dr}^{\iota\-\lalg,\iota^c\-\lan,(0,-k)}$.
\end{proposition}
Again, the surjectivity of $\bar d_{\Dr}^{k+1}$ follows by a similar calculation of $\frn$-cohomology in \cite[Proposition 4.2.9, Theorem 5.3.17]{PanII}.

We note that we may also construct twists of these differential operators in the same spirit as \ref{def:twistofd}, so that to get $d_{\LT}'^{k+1}$, $\bar d_{\LT}'^{k+1}$, $d_{\Dr}'^{k+1}$, $\bar d_{\Dr}'^{k+1}$. Besides, under the natural identification $\calO_{\LT}^{\lan,\chi}\isom \calO_{\Dr}^{\lan,\chi}$, the differential operators we defined above have the following relations:
\begin{theorem}\label{d5}
$d_{\LT}^{k+1}$ equals $\bar d_{\Dr}^{k+1}$ up to a non-zero constant. $\bar d_{\LT}^{k+1}$ equals $d_{\Dr}^{k+1}$ up to a non-zero constant. 
\begin{proof}
In order to check $d_{\LT}^{k+1}=\bar d_{\Dr}^{k+1}$, we may check that $d_{\LT}^{k+1}|_{\calO_{\LT}^{\iota\-\lan,\iota^c\-\lalg,(0,-k)}}$ satisfies the properties of $\bar d_{\Dr}^{k+1}$ by continuity. As $d_{\LT}^{k+1}$ is $\calO_{\fl}$-linear, and $\calM_{\Dr,n}$ are \'etale coverings of $\calM_{\Dr,0}$, we can deduce that $d_{\LT}^{k+1}$ is $\calO_{\Dr}^{\iota^c\-\lalg}$-linear, and hence $\calO_{\Dr}^{\iota^c\-\lan}$-linear by continuity and Proposition \ref{prop:Oetala}. Since $\pi_{\HT,\GM}:\calM_{\LT,0}\to \fl'$ is an \'etale covering, it is also easy to see that up to constants, $d_{\LT}^{k+1}$ are given by $(\check u^+)^{k+1}$, where $\check u^+$ is the image of $u^+$ under the natural isomorphism $\frg_\iota\isom \check \frg_{\iota}$. This shows that $d_{\LT}^{k+1}$ satisfies the properties of $\bar d_{\Dr}^{k+1}$, so that by uniqueness of $\bar d_{\Dr}^{k+1}$ we deduce that $d_{\LT}^{k+1}=\bar d_{\Dr}^{k+1}$. Using similar methods, one can also show that $\bar d_{\LT}^{k+1}=d_{\Dr}^{k+1}$.
\end{proof}
\end{theorem}

\begin{corollary}
The differential maps $d_{\LT}^{k+1}$ and $d_{\Dr}^{k+1}$ are surjective, with kernel $\calO_{\Dr}^{\iota\-\lalg,\iota^c\-\lan,(0,-k)}$ and $\calO_{\LT}^{\iota\-\lalg,\iota^c\-\lan,(0,-k)}$ respectively.
\end{corollary}

\subsubsection{Supersingular locus and the $p$-adic uniformization theorem}\label{padicuniformization}
In this section, we compute the cohomology of the intertwining operator on the supersingular locus. Firstly, we recall the $p$-adic uniformization theorem for the supersingular locus in the unitary Shimura curve. We follow the treatment in \cite[Section 11]{Car83}.

Recall that $F^+$ is a totally real number field and there exists a place $v$ of $F^+$ such that $F^+_v\isom L$. Let $D$ be the quaternion algebra over $F^+$ we used before, i.e., $B=D\ox_{F^+}F$, and $D$ is split at $v,\infty$. In order to describe the supersingular locus of $\calX_{K^v}$, let $\bar D$ be the quaternion algebra over $F^+$ such that the invariants of $D$ and $\bar D$ agree at all places except $v,\infty$, and $\bar D$ is ramified at $v,\infty$.

Now, starting from $\bar D$, we can use a similar procedure to define an algebraic group $\bar G$ over $\bbQ$, whose $R$-point for any $\bbQ$-algebra $R$ is given by 
\[
    \bar G(R)=\{(g,z)\in (\bar D\ox_{\bbQ}R)^\times\times(\calE\ox_{\bbQ}R)^\times:N_{\bar D/F^+}(g)=N_{F/F^+}(z),N_{\bar D/F^+}(g)\cdot N_{F/F^+}(z)\in R^\times\},
\]
where $N_{\bar D/F^+}$ denote the reduced norm map for $\bar D$. Note that our $\bar G$ is $\bar G'$ used in \cite{Car83}. Set $\bar B:=\bar D\ox_{F^+}F$. From the discussion in \cite[11.5.1]{Car83}, we see that $\bar B$ also has an involution of the second kind, and positive. In particular, we can also interpret $\bar G$ as a unitary similitude group, and $\bar G(\bbQ_p)=\bbQ_p^\times\times D_L^\times\times\prod_{i=2}^r (B_{w_i}^{\op})^\times$. In other words, the difference between $G(\bbQ_p)$ and $\bar G(\bbQ_p)$ is that we switch the $\GL_2(L)$-factor to its inner form $D_L^\times$. If we decompose $\bar G(\bbA^\infty)=\bar G_v\bar G^v$ similar to $G(\bbA^\infty)=G_v G^v$, then we have $\bar G_v\isom D_L^\times$, and we have a natural identification $\bar G^v\isom G^v$. 

Let $(A,\lambda,\iota,\eta)$ be an $\bar\bbF_q$-point of $\bar X_{K^v,0}^{\ss}$, the supersingular locus in the special fiber of the unitary Shimura curve of spherical level, and forget the data of level structure to get $(A,\lambda,\iota)$. From the Honda-Tate theory, \cite[11.5.1]{Car83} shows that any two such pairs $(A,\lambda,\iota)$, $(A',\lambda',\iota')$ are quasi-isogenous, and the automorphism group of $(A,\lambda,\iota)$ is $\bar G(\bbQ)$. From this, one deduces that if we let $K^v\subset G^v\isom \bar G^v$ be a level subgroup away from $v$, then $\bar G(\bbA^\infty)$ acts transitively (\cite[11.5.2]{Car83}) on $\bar X_{K^v,0}^{\ss}(\bar\bbF_q)$, and we have an isomorphism  
\[
    \bar X_{K^v,0}^{\ss}(\bar\bbF_q)\isom \bar G(\bbQ)\bs \bar G(\bbA^\infty)/K^v\calO_{D_L}^\times.
\]
Therefore, by using the Serre-Tate theory, one deduces that 
\[
    \pi_{\HT}^{-1}(\Omega)\isom \bar G(\bbQ)\bs(\calM_{\LT,\infty}\times G^v/K^v).
\]
This isomorphism is $\GL_2(L)\times D_L^\times$-equivariant, functorial in $K^v$.

Let $X_{K^v}=\bar G(\bbQ)\bs (\bar G(\bbA^\infty)/K^v)$, which is isomorphic to $\sqcup_{i\in I}\calO_{D_L}^\times$ as an $\calO_{D_L}^\times$-space, where $I=\bar X^{\ss}_{K^v,0}$ is a finite set, as it is a discrete subset of a proper curve. Then we can reinterpret the uniformization as a $\GL_2(L)^0$-equivariant isomorphism 
\[
    \pi_{\HT}^{-1}(\Omega)\isom (\calM_{\LT,\infty}^{(0)}\times X_{K^v})/\calO_{D_L}^\times\isom \sqcup_{i\in I}\calM_{\LT,\infty}^{(0)},
\]
where $\GL_2(L)^0=\{g\in\GL_2(L):|\det(g)|=1\}$. This isomorphism is compatible with the Hodge--Tate period map on unitary Shimura curves and Lubin--Tate spaces at infinite level, induces a $\GL_2(L)^0$-equivariant isomorphism $\calO_{K^v}|_{\Omega}\isom\bigoplus_{i\in I}\calO_{\LT}$ of $\calO_{\Omega}$-modules. Moreover, this identification is compatible with horizontal Cartan actions $\theta_{\frh}$ and $\theta_{\LT,\frh}$, and differential operators $d^{k+1}|_{\Omega}$ and $d_{\LT}^{k+1}$. In particular, we have 
\begin{theorem}
Under the $\GL_2(L)^0$-equivariant and $\calO_{\Omega}$-linear isomorphism $\calO_{K^v}|_{\Omega}\isom \bigoplus_{i\in I}\calO_{\LT}$, we have $d^{k+1}|_{\Omega}=\bigoplus_{i\in I}d^{k+1}_{\LT}$. In particular,
\begin{align*}
    d^{k+1}:\calO_{K^v}^{\lan,(0,-k)}\to \calO_{K^v}^{\lan,(0,-k)}\ox_{\calO_{K^v}^{\sm}}(\Omega^{1,\sm}_{K^v})^{\ox k+1}
\end{align*}
is surjective on $\Omega$. Moreover, for any $\eta\neq\iota$, $(a,b)\in\bbZ^2$ with $a\ge b$, the map 
\begin{align*}
d^{k+1}:\calO_{K^v}^{\iota\-\lan,V_{\eta}^{(a,b)}\-\lalg,(0,-k)}\to \calO_{K^v}^{\lan,V_{\eta}^{(a,b)}\-\lalg,(0,-k)}\ox_{\calO_{K^v}^{\sm}}(\Omega^{1,\sm}_{K^v})^{\ox k+1}
\end{align*}
is surjective on $\Omega$.
\end{theorem}

In order to describe the kernel of $d^{k+1}$ on $\Omega$, we recall some notation of algebraic automorphic forms on $\bar G$. 
\begin{definition}\label{def:AutoonD}
Let $w=(a_\eta,b_\eta)_{\eta\in\Sigma}\in W_+$ be a weight. Let $W^w:=\ox_{\eta\in\Sigma} W_\eta^{(a_\eta,b_\eta)}$.
\begin{enumerate}[(1)]
    \item Let $\calA_{\bar G,w}$ denote the set of maps 
    \[
        f:\bar G(\bbQ)\bs \bar G(\bbA^\infty)\to W^w
    \]
    such that $f(dh_vh^v)=h_v^{-1}f(d)$ for $h_v\in K_v\subset \bar G_v= D_L^\times$ and $h^v\in K^v\subset \bar G^v$ for some open compact subgroup $K_v$, $K^v$. It admits a right translation action of $\bar G(\bbA^\infty)$ and an action of $D_L^\times$ via $(h_v.f)(d)=h_v.f(dh_v)$ for $h_v\in D_L^\times$. Note that both actions are smooth.
    \item Let $\calA^1_{\bar G,w}\subset\calA_{\bar G,w}$ denote the subset of maps which factor through the reduced norm map. This is zero if $w\neq 0$ up to twist by determinants.
    \item Let $\calA^c_{\bar G,w}:=\calA_{\bar G,w}/\calA_{\bar G,w}^1$. Note that the Hecke action induces a natural splitting $\calA_{\bar G,w}\isom \calA_{\bar G,w}^c\oplus \calA_{\bar G,w}^1$.
\end{enumerate}
\end{definition}

Now we describe the kernel of $d^{k+1}|_{\calO_{K^v}^{\iota\-\lan,V_\eta^{(a,b)}\-\lalg,(0,-k)}}$ on $\Omega$.
\begin{proposition}\label{prop:donss}
Let $d^{k+1}:\calO_{K^v}^{\iota\-\lan,V_\eta^{(a,b)}\-\lalg,(0,-k)}\to \calO_{K^v}^{\iota\-\lan,V_\eta^{(a,b)}\-\lalg,(0,-k)}\ox_{\calO_{K^v}^{\sm}}(\Omega_{K^v}^{1,\sm})^{\ox k+1}$. There is a Hecke and $\GL_2(L)^0$-equivariant isomorphism 
\[
    \ker d^{k+1}|_{\Omega}\isom (\calA^{K^v}_{\bar G,w}\ox_C \omega_{\Dr}^{(0,-k),\sm})^{\calO_{D_L}^\times}\ox_C V_\eta^{(a,b)},
\]
where $w=(0,-k)_\iota,(a,b)_\eta$. Similarly, 
\[
    \ker d'^{k+1}|_{\Omega}\isom (\calA^{K^v}_{\bar G,w}\ox_C \omega_{\Dr}^{(-k-1,1),\sm})^{\calO_{D_L}^\times}\ox_C V_\eta^{(a,b)}.
\]
\begin{proof}
This is the same as \cite[Proposition 5.4.12]{PanII}. 
\end{proof}
\end{proposition}

\subsection{Cohomology of the intertwining operator I}\label{cohint}
Let $\chi=(0,-k)$ with $k\ge 0$. In the rest of this section, we perform some preprocessing procedures to calculate the cohomology of the intertwining operator.
\[
    I_k^{\iota^c\-\lalg}:\calO_{K^v}^{\iota\-\lan,\iota^c\-\lalg,(0,-k)}\to \calO_{K^v}^{\iota\-\lan,\iota^c\-\lalg,(-k-1,1)}(k+1).
\]
For simplicity, we write $I=I_k^{\iota^c\-\lalg}$ in this section.

First of all, let us fix some notation and give a brief summary. As 
\begin{align*}
    \calO_{K^v}^{\iota\-\lan,\iota^c\-\lalg,(0,-k)}\isom \bigoplus_{(k_{1,\eta},k_{2,\eta})\in\bbZ^2,\eta\neq\iota}\calO_{K^v}^{\iota\-\lan,(\ox_{\eta\neq\iota}V_\eta^{(k_{1,\eta},k_{2,\eta})})\-\lalg,(0,-k)},
\end{align*}
it suffices to compute the cohomology of $I$ on each direct summand. By abuse of notations, we also let $I$ to denote the following intertwining operator
\[
    I:\calO_{K^v}^{\iota\-\lan,(\ox_{\eta\neq\iota}V_\eta^{(k_{1,\eta},k_{2,\eta})})\-\lalg,(0,-k)}\to\calO_{K^v}^{\iota\-\lan,(\ox_{\eta\neq\iota}V_\eta^{(k_{1,\eta},k_{2,\eta})})\-\lalg,(-k-1,1)}(k+1).
\]
Consider the following diagram with exact rows and columns: 
$$
\begin{tikzcd}
    &                                                 & 0 \arrow[d]                                                 & 0 \arrow[d]                                   &                                                    &   \\
    &                                                 & \ker \bar d \arrow[d] \arrow[r, "\theta"]                   & \ker \bar d' \arrow[d]                        &                                                    &   \\
0 \arrow[r] & \ker d \arrow[d, "\bar d|_{\ker d}"'] \arrow[r] & \calO \arrow[r, "d"] \arrow[d, "\bar d"'] \arrow[rd, "I" description] & \calO\ox\Omega \arrow[d, "\bar d'"] \arrow[r] & \coker d \arrow[r] \arrow[d, "\bar d|_{\coker d}"] & 0 \\
0 \arrow[r] & \ker d' \arrow[r]                               & \calO\ox\bar\Omega \arrow[r, "d'"'] \arrow[d]               & \calO^s \arrow[d] \arrow[r]                   & \coker d' \arrow[r]                                & 0 \\
    &                                                 & 0                                                           & 0                                             &                                                    &  
\end{tikzcd}
$$
where some of the terms in the above diagram is given by:
\begin{itemize}
    \item $\calO=\calO_{K^v}^{\iota\-\lan,(\ox_{\eta\neq\iota}V_\eta^{(k_{1,\eta},k_{2,\eta})})\-\lalg,(0,-k)}$,
    \item $\calO\ox\Omega=\calO_{K^v}^{\iota\-\lan,(\ox_{\eta\neq\iota}V_\eta^{(k_{1,\eta},k_{2,\eta})})\-\lalg,(0,-k)}\ox_{\calO_{ K^v}^{\sm}}(\Omega_{K^v}^{1,\sm})^{\ox k+1}$,
    \item $\calO\ox\bar\Omega=\calO_{K^v}^{\iota\-\lan,(\ox_{\eta\neq\iota}V_\eta^{(k_{1,\eta},k_{2,\eta})})\-\lalg,(0,-k)}\ox_{\calO_{\fl}}(\Omega^1_{\fl})^{\ox k+1}$,
    \item $\calO^s=\calO_{K^v}^{\iota\-\lan,(\ox_{\eta\neq\iota}V_\eta^{(k_{1,\eta},k_{2,\eta})})\-\lalg,(-k-1,1)}(k+1)$.
\end{itemize}
We have calculated the cohomologies of these differential operators. Let $w=(k_{1,\eta},k_{2,\eta})_{\eta\in\Sigma}$ with $w_{\iota}=(0,-k)$ and $w_\eta=(k_{1,\eta},k_{2,\eta})$. By Theorem \ref{thm:dbar}, Proposition \ref{prop:donss} and Proposition \ref{prop:dord}, we get the following results:
\begin{itemize}
    \item $\ker\bar d=V^{w}\ox_C\omega_{K^v}^{w,\sm}$ with $V^w=\ox_{\eta\in\Sigma}V_\eta^{w_\eta}$.
    \item $\ker \bar d'=V^{w}\ox_C\omega_{K^v}^{s\cdot w,\sm}$ with $(s\cdot w)_{\iota}=(-k-1,1)$ and $(s\cdot w)_{\eta}=w_{\eta}$.
    \item $\coker d=i_*(i^*\omega_{\fl}^{w}\ox\calH^1_{\ord}( K^v,w))$ with $\omega_{\fl}^w:=\omega_{\fl}^{w_\iota}\ox_C (\ox_{\eta\neq\iota}V_\eta^{w_\eta})$.
    \item $\coker d'=i_*(i^*\omega_{\fl}^{s\cdot w}\ox\calH^1_{\ord}( K^v,w))$.
    \item $\ker d$ fits into an exact sequence
    \[
        0\to j_!\ker d|_{\Omega}\to\ker d\to i_*\ker d|_{\bbP^1(L)}\to 0,
    \]
    where 
    \begin{itemize}
        \item $\ker d|_{\Omega}=(\calA^{K^v}_{\bar G, w}\ox_C\omega_{\Dr}^{w,\sm})^{\calO_{D_L}^\times}$. 
        \item $\ker d|_{\bbP^1(L)}=i_*(i^*\omega_{\fl}^{w}\ox\calH^0_{\ord}( K^v,w))$.
    \end{itemize}
    \item $\ker d'$ fits into an exact sequence
    \[
        0\to j_!\ker d'|_{\Omega}\to\ker d'\to i_*\ker d'|_{\bbP^1(L)}\to 0,
    \]
    where 
    \begin{itemize}
        \item $\ker d'|_{\Omega}=(\calA^{K^v}_{\bar G, w}\ox_C\omega_{\Dr}^{s\cdot w,\sm})^{\calO_{D_L}^\times}$.
        \item $\ker d'|_{\bbP^1(L)}=i_*(i^*\omega_{\fl}^{s\cdot w}\ox\calH^0_{\ord}( K^v,w))$.
    \end{itemize} 
\end{itemize}

From the construction of those differential operators, we see that the induced differential operators are given by:
\begin{itemize}
    \item $d|_{\ker \bar d}:\ker\bar d\to\ker \bar d'$ is given by twists of the usual $\theta$-operator on automorphic sheaves.
    \item $\bar d|_{\ker d|_{\Omega}}:\ker d|_{\Omega}\to\ker d'|_{\Omega}$ are given by differential operators on Drinfeld towers. Note that on the bottom level the differential operator can be seen as the restriction of the one on flag variety via $j:\Omega\inj\fl$.
    \item $\bar d|_{\coker d}:\coker d\to\coker d'$ and $\bar d|_{\ker d|_{\bbP^1(L)}}:\ker d|_{\bbP^1(L)}\to \ker d'|_{\bbP^1(L)}$ are given by differential operators on the flag variety $\fl$, or more precisely on $i:\bbP^1(L)\inj\fl$.
\end{itemize}

For simplicity, for $\calF$ a sheaf on $\fl$, we let $H^i(\calF)$ to denote the $i$-th cohomology group of $\calF$ on $\fl$. The following lemma presents some vanishing results we will use later.
\begin{lemma}\label{lem:v}
We have the following vanishing results:
\begin{enumerate}[(i)]
    \item If $w$ is not the highest weight of a one-dimensional algebraic representation of $\GL_2(L)$ over $C$, then $H^0(\ker \bar d)=0$, and $H^1(\ker \bar d')=0$. Otherwise, up to twist by determinants we may assume that $w=0$. Then $H^0(\ker \bar d)=M_0(K^v)$ and $H^1(\ker \bar d')=M_0(K^v)$.
    \item $H^i(\coker d)=0$ and $H^i(\coker d')=0$ for $i\ge 1$.
    \item $H^i(j_!\ker d|_{\Omega})=0$ and $H^i(j_!\ker d'|_{\Omega})=0$ for $i\neq 1$.
    \item We have exact sequences
    \begin{align*}
        &0\to H^0(\ker d)\to H^0(\ker d|_{\bbP^1(L)})\to H^1(j_!\ker d|_{\Omega})\to H^1(\ker d)\to 0,\\
        &0\to H^0(\ker d')\to H^0(\ker d'|_{\bbP^1(L)})\to H^1(j_!\ker d'|_{\Omega})\to H^1(\ker d')\to 0.
    \end{align*}
\end{enumerate}
\begin{proof}
For (i), as $\ker \bar d$ is a subsheaf of $\calO$, and $\calO$ is a subsheaf of $\calO^{\lan}_{K^v}$, we see that $H^0(\ker \bar d)$ is a subspace of $H^0(\calO^{\lan}_{K^v})$ consisting of locally analytic functions on an open subset of $L^\times$ by identifying $H^0(\calO^{\lan}_{K^v})$ with locally analytic vectors in the zero-th completed cohomology group. But as the $\GL_2(L)$-action on $H^0(\calO^{\lan}_{K^v})$ factors through the determinant, one easily deduce the desired results. By Theorem \ref{thm:dbar}, we know $\ker\bar d=V^{w}\ox_C\omega_{K^v}^{w ,\sm}(k_{1,\iota})$. In order to compute $H^*(\ker \bar d)$, we can use Serre duality on each curve $\calX_{K^vK_v}$ and take inductive limit on $K_v$. Using similar methods we can compute $H^*(\ker \bar d')$.

For (ii), note that $\coker d$ and $\coker d'$ are supported on $\bbP^1(L)$.

For (iii), it suffices to show that $H^i(j_!\omega_{\Dr}^{(a,0),D_L^\times\-\sm})=0$ for any $a\in\bbZ$ and $i\neq 1$. This follows from the Serre duality for $\calM_{\Dr,n}$, and $\calM_{\Dr,n}$ is a Stein rigid analytic curve.

For (iv), we use (ii), (iii).
\end{proof}
\end{lemma}

\begin{remark}
From the proof we see that if $w$ is not the highest weight of a one-dimensional algebraic representation of $\Res_{L/\bbQ_p}\GL_2$ over $C$, then $H^0(D_{K^v}^{w,\sm})=0$.
\end{remark}

Recall that for hypercohomology of a complex of sheaves on a site, there are two spectral sequence computing it, inducing the Hodge filtration and the conjugate filtration on the hypercohomology group. We will mainly use this to compute the cohomology of the differential operators we constructed before, together with various vanishing results in Lemma \ref{lem:v}.

Define $\DR$ to be the two-term complex $\calO\ov{d}\to\calO\ox\Omega$. 
\begin{proposition}\label{prop:DRcoh}
\begin{enumerate}[(i)]
    \item $\bbH^0(\DR)=H^0(\calO)$.
    \item We have an exact sequence $0\to H^0(\calO\ox\Omega)\to \bbH^1(\DR)\to H^1(\calO)\to H^1(\calO\ox \Omega)\to 0$.
    \item $\bbH^0(\DR)=H^0(\ker d)$.
    \item we have an exact sequence $0\to H^1(\ker d)\to \bbH^1(\DR)\to H^0(\coker d)\to 0$.
    \item $\bbH^2(\DR)=0$.
\end{enumerate}
\begin{proof}
For (i) and (ii), we use the cohomology of the triangle $\calO\ox\Omega[-1]\to \DR\to \calO\to $. As $\calO$ is a subsheaf of $\calO^{\lan}_{K^v}$, and $H^0(\fl,\calO_{K^v}^{\lan})\isom\tilde{H}^0(K^v,C)^{\lan}$, we see that the $\GL_2(L)$-action on $H^0(\calO)$ factors through the determinant map. Let $z\in \GL_2(L)$ be a scalar matrix. For any $f\in \calO_{K^v}^{\lan}$, we know that $\theta_{\frh}(z).f=z.f$. From this we deduce that $H^0(\calO)$ consists of some locally algebraic functions on an open subset of $L^\times$. From this one can deduce that the map $H^0(\calO)\to H^0(\calO\ox\Omega)$ is zero. Moreover, if $k\neq 0$, we have $H^0(\calO)=0$.

For (iii), (iv) and (v), we use the cohomology of the triangle $\ker d\to \DR\to \coker d[-1]\to$. Since $\fl$ is of dimension $1$, we have $H^2(\ker d)=0$ and $H^1(\coker d)=0$, and from this we deduce $\bbH^2(\DR)=0$.
\end{proof}
\end{proposition}

Define $\DR'$ to be the two-term complex $\calO\ox\bar \Omega\ov{d'}\to\calO^s$. 
\begin{proposition}\label{prop:DRcohprime}
\begin{enumerate}[(i)]
    \item $\bbH^0(\DR')=H^0(\calO\ox\bar \Omega)$.
    \item We have an exact sequence $0\to \bbH^1(\DR')\to H^1(\calO\ox\Omega)\to H^1(\calO^s)\to 0$.
    \item $\bbH^0(\DR')=H^0(\ker d')$.
    \item we have an exact sequence $0\to H^1(\ker d')\to \bbH^1(\DR')\to H^0(\coker d')\to 0$.
    \item $\bbH^2(\DR')=0$.
\end{enumerate}
\begin{proof}
The proof is similar to Proposition \ref{prop:DRcoh}.
\end{proof}
\end{proposition}

Define $\bar\DR$ to be the two-term complex $\calO\ov{d}\to\calO\ox\bar \Omega$.
\begin{proposition}\label{prop:DSJIJ}
For the cohomology of $\bar\DR$, we have
\begin{enumerate}[(i)]
    \item $\bbH^i(\bar\DR)=H^i(\ker \bar d)$ for $i\ge 0$.
    \item $\bbH^0(\bar \DR)=H^0(\calO)$.
    \item We have an exact sequence $0\to H^0(\calO\ox\bar\Omega)\to \bbH^1(\bar\DR)\to H^1(\calO)\to H^1(\calO\ox \bar\Omega)\to 0$.
\end{enumerate}
\begin{proof}
The proof is similar to Proposition  \ref{prop:DRcoh}. The key point is that $\bar d$ is surjective, so that $\bar \DR$ is quasi-isomorphic to $\ker\bar d[0]$.
\end{proof}
\end{proposition}

Define $\bar\DR'$ to be the two-term complex $\calO\ox\Omega\ov{\bar d'}\to\calO^s$.
\begin{proposition}\label{prop:DRcohpp}
For the cohomology of $\bar\DR'$, we have
\begin{enumerate}[(i)]
    \item $\bbH^i(\bar\DR')=H^i(\ker \bar d')$ for $i\ge 0$.
    \item $\bbH^0(\bar \DR')=H^0(\calO\ox\Omega)$.
    \item We have an exact sequence $0\to \bbH^1(\bar\DR')\to H^1(\calO\ox\bar\Omega)\to H^1(\calO^s)\to 0$.
\end{enumerate}
\begin{proof}
The proof is similar to Proposition \ref{prop:DSJIJ}. 
\end{proof}
\end{proposition}

Let $\DR^{\lalg}$ be the kernel of the $\DR\ov{\bar d}\to \DR'$. We can write it as the two term complex $\DR^{\lalg}=[\ker \bar d\ov{d}\to \ker \bar d']$, with the differential map given by the $\theta$-opeartor.
\begin{proposition}\label{prop:DRcohppp}
For the cohomology of $\DR^{\lalg}$, we have the following results:
\begin{enumerate}[(i)]
    \item $\bbH^0(\DR^{\lalg})=H^0(\ker \bar d)$.
    \item We have an exact sequence $0\to H^0(\ker \bar d')\to \bbH^1(\DR^{\lalg})\to H^1(\ker \bar d)\to 0$.
    \item We have an isomorphism $H^1(\ker \bar d')\aisom  H^2(\DR^{\lalg})$.
\end{enumerate}
\begin{proof}
We use the cohomology of the triangle $\ker \bar d'[-1]\to \DR^{\lalg}\to \ker \bar d$, and vanishing results on cohomology of $\omega_{K^v}^{w,\sm}$.
\end{proof}
\end{proposition}

Finally, we study the cohomology of $\DR^{\lalg}\to\DR\to\DR'\to$.
\begin{proposition}\label{prop:DRcohpppp}
We have an isomorphism $\bbH^0(\DR^{\lalg})=\bbH^0(\DR)$.

We have an exact sequence 
\[
    0\to \bbH^0(\DR')\to \bbH^1(\DR^{\lalg})\to\bbH^1(\DR)\to \bbH^1(\DR')\to \bbH^2(\DR^{\lalg})\to 0.
\]
\begin{proof}
For $\bbH^0(\DR^{\lalg})=\bbH^0(\DR)$, we have already shown that they are the same by direct calculation. Then the cohomology of the triangle $\DR^{\lalg}\to\DR\to\DR'\to$ will yields a long exact sequence, which stops at $\bbH^2(\DR)=0$.
\end{proof}
\end{proposition}

\subsection{Cohomology of the intertwining operator II - Hecke decompositions}
In this section, we study the Hecke action on the cohomology of $I:\calO\to\calO^s$ and calculate some Hecke isotypic spaces.

Recall that $\bbT^S$ is the abstract spherical Hecke algebra. Let $w=(k_{1,\eta},k_{2,\eta})_{\eta\in\Sigma}\in W$ be a weight. Recall that $M_w(K^v)=H^0(\omega_{K^v}^{w,\sm})$ denote the space of classical modular forms of tame level $K^v$ and weight $w$. We know that the Hecke action of $\bbT^S$ on $M_w(K^v)$ is semisimple, and there is a natural decomposition with respect to these Hecke eigenvalues. Let $\sigma_w^{K^v}$ be the union of the Hecke eigenvalues appear in $M_w(K^v)$ and $M_{s\cdot w}(K^v)$. One can use the Eichler-Shimura decomposition to check that the set of Hecke eigenvalues appear in $\bbH^i(\DR^{\lalg})$ for $i=0,1,2$ is $\sigma_{w}^{K^v}$. For $\lambda$ inside $\sigma_{w}^{K^v}$, let $\pi$ be the automorphic representation of $G(\bbA)$ over $C$ associated to $\lambda$. Then we have 
\[
    M_{w}(K^v)[\lambda]=(\pi^{v}_f)^{K^v}\ox_C\pi_v
\]
where $\pi_f$ is the non-archimedean part of the automorphic representation of $G(\bbA)$ associated to $\lambda$, and $\pi_v$ is a smooth admissible representation of $\GL_2(L)$, which is a local component of $\pi$ at $v$. Recall that $v$ is a place of $F^+$ over $p$ we fixed before. 

Recall that $\calA_{\bar G,w}^{K^v}$ denote the space of automorphic forms for $\bar G$, of tame level $K^v$ and weight $w$. By the Jacquet--Langlands correspondence, the Hecke action on $\calA_{\bar G,w}^{K^v}$ is semisimple, and the set of Hecke eigenvalues inside $\calA_{\bar G,w}^{K^v}$ is a subset of $\sigma^{K^v}_{w}$. We give a description of $\calA_{\bar G,w}^{K^v}$. Suppose that $\pi_v$ is a special series or a supercuspidal representation. Then $\pi_v$ can be transferred to a smooth irreducible representation $\tau_v$ of $D_L^\times$  by the local Jacquet--Langlands correspondence, and we have 
\[
    \calA^{K^v}_{\bar G,w}[\lambda]=(\pi_f^v)^{K^v}\ox_C\tau_v.
\]
Here the normalization is that $\pi_v$ and $\tau_v$ have the same central character.

Next, we want to analysis the Hecke action on $H^*_{\rig}(\Ig_{K^v},D_{K^v}^{w,\sm,\dagger})$. We first prove the following lemma which allows us to control the Hecke action on $H^1_{\rig}(\Ig_{K^v},D_{K^v}^{w,\sm,\dagger})$.
\begin{lemma}\label{lem:pvc000}
There exists a $\GL_2(L)^0$-equivariant and $\bbT^S$-equivariant exact sequence 
\begin{align*}
    &\bbH^1(\DR^{\lalg})\to (\Ind^{\GL_2(L)}_{\bar B(L)}H^1_{\rig}(\Ig_{K^v},D_{K^v}^{w,\sm,\dagger}))^\infty\ox _CV^w\\
    \to& (H^2_{\dR,c}(\calM_{\Dr,\infty}^{(0)})\ox_C\calA_{\bar G,w}^{K^v})^{\calO_{D_L}^\times}\ox_C V^w\to \bbH^2(\DR^{\lalg})\to 0.
\end{align*}
\begin{proof}
We know that we have a morphism of exact sequences 
\begin{center}
\small
\begin{tikzpicture}[descr/.style={fill=white,inner sep=1.5pt}]
    \matrix (m) [
        matrix of math nodes,
        row sep=2.5em,
        column sep=2.5em,
        text height=1.5ex, 
        text depth=0.25ex
    ]
    { 0 & H^1(\ker d) & \bbH^1(\DR) & H^0(\coker d) & 0  \\
      0 & H^1(\ker d') & \bbH^1(\DR') & H^0(\coker d') & 0 \\
    };
    \path[->,font=\scriptsize]
    (m-1-1) edge (m-1-2)
    (m-1-2) edge (m-1-3)
    (m-1-3) edge (m-1-4)
    (m-1-4) edge (m-1-5)
    (m-2-1) edge (m-2-2)
    (m-2-2) edge (m-2-3)
    (m-2-3) edge (m-2-4)
    (m-2-4) edge (m-2-5)
    ;
    \path[->,font=\scriptsize]
    (m-1-2) edge node[right] {$\bar d$} (m-2-2)
    (m-1-3) edge node[right] {$\bar d$} (m-2-3)
    (m-1-4) edge node[right] {$\bar d$} (m-2-4)
    ;
\end{tikzpicture}.
\end{center}
By the snake lemma, we get an exact sequence 
\begin{align*}
    \bbH^1(\DR^{\lalg})\to(\Ind^{\GL_2(L)}_{\bar B(L)}H^1_{\rig}(\Ig_{K^v},D_{K^v}^{w,\sm,\dagger}))^\infty\ox _CV^w\to M\to \bbH^2(\DR^{\lalg})\to 0,
\end{align*}
where
\begin{itemize}
    \item $M=\coker (H^1(\ker d)\to H^1(\ker d'))$,
    \item as $H^0(\coker d)\to H^0(\coker d')$ is surjective, so that the term after $\bbH^2(\DR^{\lalg})$ is zero.
\end{itemize}

To study the morphism $H^1(\ker d)\to H^1(\ker d')$, note that we have another morphism of exact sequences 
\begin{center}
\small
\begin{tikzpicture}[descr/.style={fill=white,inner sep=1.5pt}]
    \matrix (m) [
        matrix of math nodes,
        row sep=2.5em,
        column sep=2.5em,
        text height=1.5ex, 
        text depth=0.25ex
    ]
    { 0 & H^0(\ker d) & H^0(i_*i^*\ker d) & H^1(j_!j^*\ker d) & H^1(\ker d) & 0  \\
      0 & H^0(\ker d') & H^0(i_*i^*\ker d') & H^1(j_!j^*\ker d') & H^1(\ker d') & 0 \\
    };
    \path[->,font=\scriptsize]
    (m-1-1) edge (m-1-2)
    (m-1-2) edge (m-1-3)
    (m-1-3) edge (m-1-4)
    (m-1-4) edge (m-1-5)
    (m-1-5) edge (m-1-6)
    (m-2-1) edge (m-2-2)
    (m-2-2) edge (m-2-3)
    (m-2-3) edge (m-2-4)
    (m-2-4) edge (m-2-5)
    (m-2-5) edge (m-2-6)
    ;
    \path[->,font=\scriptsize]
    (m-1-2) edge node[right] {$\bar d$} (m-2-2)
    (m-1-3) edge node[right] {$\bar d$} (m-2-3)
    (m-1-4) edge node[right] {$\bar d$} (m-2-4)
    (m-1-5) edge node[right] {$\bar d$} (m-2-5)
    ;
\end{tikzpicture}
\end{center}
As $\bar d:H^0(i_*i^*\ker d)\to H^0(i_*i^*\ker d')$ is surjective, we see that $M$ is also the cokernel of $\bar d:H^1(j_!j^*\ker d)\to H^1(j_!j^*\ker d')$. From this one deduces the claim.
\end{proof}
\end{lemma}

\begin{lemma}
The Hecke eigenvalues appear in $H^i_{\rig}(\Ig_{K^v},D_{K^v}^{w,\sm,\dagger})$ for $i=0,1$ is a subset of $\sigma_{w}^{K^v}$.
\begin{proof}
Consider the exact sequence 
\begin{align*}
H^0(\ker d)\to H^0(i_*i^*\ker d)\to H^1(j_!j^*\ker d),
\end{align*}
where $H^0(\ker d)=H^0(\DR)=H^0(\DR^{\lalg})$, and the Hecke eigenvalues appear in $H^1(j_!j^*\ker d)$ is a subset of $\sigma_{w}^{K^v}$. From this we deduce the results.
\end{proof}
\end{lemma}

\begin{theorem}\label{thm:H0Ig}
Let $\lambda\in \sigma_{w}^{K^v}$ such that $\pi_v$ is generic. Then 
\begin{enumerate}[(i)]
    \item $H^0_{\rig}(\Ig_{K^v},D_{K^v}^{w,\sm,\dagger})[\lambda]=0$.
    \item $H^1_{\rig}(\Ig_{K^v},D_{K^v}^{w,\sm,\dagger})\tilde{[\lambda]}\isom H^1_{\rig}(\Ig_{K^v},D_{K^v}^{w,\sm,\dagger})[\lambda]$. Here $(-)\tilde{[\lambda]}$ denote the generalized Hecke eigenspace with eigenvalue $\lambda$.
\end{enumerate}
Moreover, we have an equality of $\bar B(L)\times \bbT^S$-representations
\begin{align*}
    H^1_{\rig}(\Ig_{K^v},D_{K^v}^{w,\sm,\dagger})[\lambda]^{\ss}=(\pi_f^v)^{K^v}\ox_C J_{\bar N(L)}(\pi_v)^{\ss},
\end{align*}
where $(-)^{\ss}$ denote the semi-simplification for the $\bar B(L)$-action.
\begin{proof}
First of all, we note that as we assume $\pi_v$ is generic, $\lambda$ will not appear in $\bbH^i(\DR^{\lalg})$ for $i=0,2$. Suppose that $\pi_v$ is supercuspidal. We have an exact sequence 
\begin{align*}
    &\bbH^1(\DR^{\lalg})[\lambda]\to (\Ind^{\GL_2(L)}_{\bar B(L)}H^1_{\rig}(\Ig_{K^v},D_{K^v}^{w,\sm,\dagger})\tilde{[\lambda]})^\infty\ox _CV^w\\
    \to& (H^2_{\dR,c}(\calM_{\Dr,\infty}^{(0)})\ox_C\calA_{\bar G,w}^{K^v}[\lambda])^{\calO_{D_L}^\times}\ox_C V^w\to 0. 
\end{align*}
As $\bbH^1(\DR^{\lalg})[\lambda]$ is supercuspidal, we see that the map
\begin{align*}
    \bbH^1(\DR^{\lalg})[\lambda]\to (\Ind^{\GL_2(L)}_{\bar B(L)}H^1_{\rig}(\Ig_{K^v},D_{K^v}^{w,\sm,\dagger})\tilde{[\lambda]})^\infty\ox _CV^w
\end{align*}
is zero. But as the $\GL_2(L)^0$-action on the last term factors through the determinant map, we see that the last two term must be zero. In particular, $H^1_{\rig}(\Ig_{K^v},D_{K^v}^{w ,\sm,\dagger})[\lambda]=0$. As the Jacquet module of $\bbH^1(\DR^{\lalg})[\lambda]$ is zero, by Theorem \ref{thm:IgJacquet}, we see that $H^0_{\rig}(\Ig_{K^v},D_{K^v}^{w ,\sm,\dagger})[\lambda]=0$.

Next we assume that $\pi_v$ is special. As $\calA_{\bar G,w }^{K^v}[\lambda]$ is the trivial $\calO_{D_L}^\times$-representation, we see that
\[
    (H^2_{\dR,c}(\calM_{\Dr,\infty}^{(0)})\ox_C\calA_{\bar G,w }^{K^v}[\lambda])^{\calO_{D_L}^\times}=H^2_{\dR,c}(\calM_{\Dr,0}^{(0)})
\]
is also trivial as a $\GL_2(L)^0$-representation. Also we know that $\bbH^1(\DR^{\lalg})[\lambda]$ is two copy of the smooth Steinberg representation of $\GL_2(L)$ twisted by a finite dimensional locally algebraic representation. This shows that only one copy of Steinberg representation will survive in the image of the map $\bbH^1(\DR^{\lalg})[\lambda]\to (\Ind^{\GL_2(L)}_{\bar B(L)}H^1_{\rig}(\Ig_{K^v},D_{K^v}^{w ,\sm,\dagger})\tilde{[\lambda]})^\infty\ox _CV^w$, and therefore 
\[
    H^1_{\rig}(\Ig_{K^v},D_{K^v}^{w ,\sm,\dagger})\tilde{[\lambda]}=|\cdot|^{-1}\ox |\cdot|
\]
as the trivial $\bar B(L)^0:=\bar B(L)\cap \GL_2(L)^0$-representation. By Theorem \ref{thm:IgJacquet}, we have an identity of virtual $T(L)$-representations
\[
    2|\cdot|^{-1}\ox |\cdot|=2(|\cdot|^{-1}\ox |\cdot|-H^0_{\rig}(\Ig_{K^v},D_{K^v}^{w ,\sm,\dagger})\tilde{[\lambda]}),
\]
and therefore $H^0_{\rig}(\Ig_{K^v},D_{K^v}^{w ,\sm,\dagger})\tilde{[\lambda]}=0$.

Now we assume that $\pi_v$ is a principal series. As $\bbH^1(\DR^{\lalg})[\lambda]$ is two copy of some principal series, and the $\GL_2(L)^0$-action on $H^2_{\dR,c}(\calM_{\Dr,\infty}^{(0)})$ factors through the determinant map, we see that $\dim H^1_{\rig}(\Ig_{K^v},D_{K^v}^{w ,\sm,\dagger})\tilde{[\lambda]}\le 2$. Combining with the identity in Theorem \ref{thm:IgJacquet}, we see that $H^0_{\rig}(\Ig_{K^v},D_{K^v}^{w,\sm,\dagger})\tilde{[\lambda]}=0$.

Finally, by Theorem \ref{thm:IgJacquet}, we deduce that the Hecke action on $H^1(\Ig_{K^v},D_{K^v}^{w ,\sm,\dagger})$ is semi-simple when $\pi_v$ is generic.
\end{proof}
\end{theorem}

\begin{corollary}
Let $\lambda\in \sigma_{w}^{K^v}$ be a Hecke eigenvalue such that $\pi_v$ is generic. We have the following results:
\begin{enumerate}[(i)]
    \item $H^1(\ker d)\tilde{[\lambda]}=H^1(j_!j^*\ker d)\tilde{[\lambda]}=H^1(j_!j^*\ker d)[\lambda]$.
    \item $H^1(\ker d')\tilde{[\lambda]}=H^1(j_!j^*\ker d')\tilde{[\lambda]}=H^1(j_!j^*\ker d')[\lambda]$.
    \item $\bbH^0(\DR')\tilde{[\lambda]}=0$.
\end{enumerate}
\end{corollary}

Combining with the results we proved before, we get the following commutative diagrams with exact rows and columns:
\begin{theorem}\label{thm:kerI1}
Let $\lambda\in \sigma_{w}^{K^v}$ be a Hecke eigenvalue such that $\pi_v$ is generic. Let $I^1$ denote the intertwining operator on the level of cohomology groups: 
\begin{align*}
    I^1:H^1(\calO)\to H^1(\calO^s).
\end{align*}
Then we have a commutative diagram with exact rows and columns:
\begin{center}
    \begin{tikzpicture}[descr/.style={fill=white,inner sep=1.5pt}]
        \matrix (m) [
            matrix of math nodes,
            row sep=1.5em,
            column sep=1.5em,
            text height=1.5ex, text depth=0.25ex
        ]
        { 
        H^0(\calO\ox\Omega)[\lambda] &\bbH^1(\DR^{\lalg})[\lambda]&\bbH^1(\bar \DR)[\lambda]  &  \\
        H^0(\calO\ox\Omega)[\lambda] &\bbH^1(\DR)\tilde{[\lambda]} & H^1(\calO)\tilde{[\lambda]} &H^1(\calO\ox\Omega)\tilde{[\lambda]} \\
        & \bbH^1(\DR')\tilde{[\lambda]} & H^1(\calO\ox\bar\Omega)\tilde{[\lambda]} & H^1(\calO^s)\tilde{[\lambda]} \\
        };
        \path[right hook->, font=\scriptsize]
        (m-1-1) edge   (m-1-2)
        (m-2-1) edge   (m-2-2)
        (m-3-2) edge   (m-3-3)
        ;
        \path[->, font=\scriptsize]
        (m-2-2) edge  (m-2-3)
        (m-1-2) edge  (m-2-2)
        (m-2-2) edge  (m-3-2)
        (m-1-3) edge  (m-2-3)
        (m-2-3) edge  (m-2-4)
        (m-3-3) edge  (m-3-4)
        (m-2-3) edge  (m-3-3)
        (m-2-4) edge  (m-3-4)
        (m-2-3) edge node[above] {$I$} (m-3-4)
        ;
        \path[->>, font=\scriptsize]
        (m-1-2) edge  (m-1-3)
        
        ;
        \draw[double distance = 1.5pt] 
        (m-1-1) -- (m-2-1)
        ;
    \end{tikzpicture},
\end{center}
with 
\begin{align*}
    H^0(\calO\ox\Omega)[\lambda]\isom (\pi_f^{K^v})\ox_C\pi_v.
\end{align*}
Therefore, $\ker I^1\tilde{[\lambda]}$ fits into the following two exact sequences: 
\begin{align*}
    0&\to H^0(\calO\ox\Omega)[\lambda]\to \bbH^1(\DR)\tilde{[\lambda]}\to  \ker I^1\tilde{[\lambda]}\to 0,\\
    0&\to H^0(\calO\ox\Omega)[\lambda]\to \ker I^1\tilde{[\lambda]}\to \bbH^1(\DR')\tilde{[\lambda]}\to 0.
\end{align*}
\begin{proof}
This is a direct consequence of the result that $H^0_{\rig}(\Ig_{K^v},D_{K^v}^{w ,\sm,\dagger})\tilde{[\lambda]}=0$ in this case by \ref{thm:H0Ig}, and results in Proposition \ref{prop:DRcoh}, \ref{prop:DRcohprime}, \ref{prop:DSJIJ} \ref{prop:DRcohpp}, \ref{prop:DRcohppp}, \ref{prop:DRcohpppp}.
\end{proof}
\end{theorem}

Recall that in Proposition \ref{prop:DRcoh}, \ref{prop:DRcohprime}, there are exact sequences 
\begin{align*}
    0\to H^1(\ker d)\to \bbH^1(\DR)\to H^0(\coker d)\to 0,\\
    0\to H^1(\ker d')\to \bbH^1(\DR')\to H^0(\coker d')\to 0.
\end{align*}
We denote the restriction of $\bar d:\bbH^1(\DR)\to \bbH^1(\DR')$ on $H^1(\ker d)$ as $\mathrm{SS}$, and the restriction of $\bar d$ on $H^0(\coker d)$ as $\mathrm{ORD}$.
\begin{theorem}\label{thm:kerI2}
Let $\lambda\in \sigma_{w}^{K^v}$ be a Hecke eigenvalue such that $\pi_v$ is generic. We have a morphism of exact sequences 
\begin{center}
\begin{tikzpicture}[descr/.style={fill=white,inner sep=1.5pt}]
    \matrix (m) [
        matrix of math nodes,
        row sep=2.5em,
        column sep=2.5em,
        text height=1.5ex, 
        text depth=0.25ex
    ]
    { 0 & H^1(\ker d)[\lambda] & \bbH^1(\DR)\tilde{[\lambda]} & H^0(\coker d)[\lambda] & 0  \\
      0 & H^1(\ker d')[\lambda] & \bbH^1(\DR')\tilde{[\lambda]} & \coker \mathrm{SS}[\lambda] & 0 \\
    };
    \path[->,font=\scriptsize]
    (m-1-1) edge (m-1-2)
    (m-1-2) edge (m-1-3)
    (m-1-3) edge (m-1-4)
    (m-1-4) edge (m-1-5)
    (m-2-1) edge (m-2-2)
    (m-2-2) edge (m-2-3)
    (m-2-3) edge (m-2-4)
    (m-2-4) edge (m-2-5)
    ;
    \path[->,font=\scriptsize]
    (m-1-2) edge node[right] {$\SS$} (m-2-2)
    (m-1-3) edge node[right] {$\bar d$} (m-2-3)
    (m-1-4) edge node[right] {$\ORD$} (m-2-4)
    ;
\end{tikzpicture}
\end{center}
which is equivariant for the $\GL_2(L)$-action, $\Gal_L$-action, and the Hecke action. Taking the associated long exact sequence, we get an exact sequence 
\[
    0\to \ker\SS[\lambda]\to \bbH^1(\DR^{\lalg})[\lambda]\to \ker \ORD[\lambda]\to \coker \SS[\lambda]\to 0,
\] 
where each term we introduced before is given as follows:
\begin{itemize}
    \item $\bbH^1(\DR^{\lalg})[\lambda]=(\pi_f^v)^{K^v}\ox_C\pi_v^{\oplus 2}\ox_CV^w$,
    \item $\ker \mathrm{SS}[\lambda]=(\pi_f^v)^{K^v}\ox_C(H^1_{\dR,c}(\calM_{\Dr,\infty}^{(0)})\ox_C\tau_v)^{\calO_{D_L}^\times}\ox_C V^w$,
    \item $\coker \mathrm{SS}[\lambda]=(\pi_f^v)^{K^v}\ox_C(H^2_{\dR,c}(\calM_{\Dr,\infty}^{(0)})\ox_C\tau_v)^{\calO_{D_L}^\times}\ox_C V^w$,
    \item $H^1(\ker d)[\lambda]=(\pi_f^v)^{K^v}\ox_C(H^1(j_!\omega_{\Dr}^{w,\sm})\ox_C\tau_v)^{\calO_{D_L}^\times}$,
    \item $H^1(\ker d')[\lambda]=(\pi_f^v)^{K^v}\ox_C(H^1(j_!\omega_{\Dr}^{s\cdot w,\sm})\ox_C\tau_v)^{\calO_{D_L}^\times}$,
    \item $\ker\mathrm{ORD}[\lambda]=(\pi_f^v)^{K^v}\ox_C (\Ind^{\GL_2(L)}_{\bar B(L)}H^1_{\rig}(\Ig_{K^v},D_{K^v}^{w ,\sm,\dagger})[\lambda])^\infty\ox_C V^w$,
    \item $H^0(\coker d)[\lambda]=(\pi_f^v)^{K^v}\ox_C(\Ind^{\GL_2(L)}_{\bar B(L)}H^1_{\rig}(\Ig_{K^v},D_{K^v}^{w ,\sm,\dagger})[\lambda]\cdot z_{2,\iota}^{-k})^{\iota\-\lan}\ox_C V^{w^{\iota}}$, where $V^{w^{\iota}}=\ox_{\eta\neq\iota}V_\eta^{w_\eta}$,
    \item $H^0(\coker d')[\lambda]=(\pi_f^v)^{K^v}\ox_C(\Ind^{\GL_2(L)}_{\bar B(L)}H^1_{\rig}(\Ig_{K^v},D_{K^v}^{w ,\sm,\dagger})[\lambda]\cdot z_{1,\iota}^{-k-1}z_{2,\iota})^{\iota\-\lan}\ox_C V^{w^{\iota}}$.    
\end{itemize}
\begin{proof}
This follows by applying the snake lemma to the morphism of short exact sequences induced by $\bar d$, as in the proof of Lemma \ref{lem:pvc000}. But now we have a clearer understanding of each term in the induced long exact sequence.
\end{proof}
\end{theorem}

So our situation follows the same pattern as in \cite[Section 5]{PanII}, except that we have to replace locally analytic for $\GL_2(\bbQ_p)$-action by locally $\iota$-analytic for the $\GL_2(L)$-action.

\begin{corollary}\label{classical}
There is a generalized eigenspace decomposition
\[
    \ker I\isom \bigoplus_{\lambda\in\sigma_{w}^{K^v}}\ker I\tilde{[\lambda]}.
\]
In particular, $\ker I$ is classical in the sense that the Hecke eigenvalues appear in $\ker I$ is classical.
\begin{proof}
This follows from Theorem \ref{thm:kerI1}, \ref{thm:kerI2}.
\end{proof}
\end{corollary}

We discuss the semisimplicity for the Hecke action on $\ker I^1$.
\begin{corollary}
Let $\lambda\in\sigma_{w}^{K^v}$ be a Hecke eigenvalue such that $\pi_v$ is generic. Then the Hecke actions on $\bbH^1(\DR)\tilde{[\lambda]}$, $\bbH^1(\DR')\tilde{[\lambda]}$ and $\ker I\tilde{[\lambda]}$ are semisimple.
\begin{proof}
As $\bbH^1(\DR')\tilde{[\lambda]}$ and $\ker I\tilde{[\lambda]}$ are quotients of $\bbH^1(\DR)\tilde{[\lambda]}$, it suffices to show that 
\begin{align*}
    \bbH^1(\DR)\tilde{[\lambda]}\isom \bbH^1(\DR)[\lambda].
\end{align*}
By discussions before, we know that when $\pi_v$ is supercuspidal, the Hecke action on $\bbH^1(\DR)\tilde{[\lambda]}$ coming from the Hecke action on automorphic forms of $\bar G$, which is semisimple.  When $\pi_v$ is a principal series, the Hecke action on $\bbH^1(\DR)\tilde{[\lambda]}$ coming from the Hecke action on rigid cohomology of Igusa varieties, which is semisimple by Theorem \ref{thm:H0Ig}. Finally, when $\pi_v$ is a special series, Proposition \ref{prop:semisimplespecial} will show that the Hecke action on $\bbH^1(\DR)\tilde{[\lambda]}$ is semisimple.
\end{proof}
\end{corollary}

Let $\lambda\in \sigma_{w}^{K^v}$ be a Hecke eigenvalue such that $\pi_v$ is generic. We will discuss some examples in detail. For simplicity, we assume our weight $w$ satisfies $w_\eta=(0,0)$ for all embeddings $\eta$, but all results below is valid for general weights.
\subsubsection*{Examples: $\pi_v$ is a principal series}
Assume that $\pi_v$ is a principal series. Write $\pi_v=(\Ind^{\GL_2(L)}_{\bar B(L)}\chi)^\infty$ for some smooth character $\chi$ of $T(L)$. In this case, we know that
\begin{itemize}
    \item $\bbH^1(\DR)[\lambda]=H^0(\coker d)[\lambda]=(\pi_f^v)^{K^v}\ox_C(\Ind^{\GL_2(L)}_{\bar B(L)}H^1_{\rig}(\Ig_{K^v},D_{K^v}^{w ,\sm,\dagger})[\lambda])^{\iota\-\lan}$.
    \item $\bbH^1(\DR')[\lambda]=H^0(\coker d')[\lambda]=(\pi_f^v)^{K^v}\ox_C(\Ind^{\GL_2(L)}_{\bar B(L)}H^1_{\rig}(\Ig_{K^v},D_{K^v}^{w ,\sm,\dagger})[\lambda]z_{1,\iota}^{-1}z_{2,\iota})^{\iota\-\lan}$.
\end{itemize}
From theorem \ref{thm:kerI1} and theorem \ref{thm:kerI2}, we deduce that, as $\GL_2(L)$-representations, $\ker I[\lambda]$ fits into exact sequences:
\begin{align*}
    0\to (\pi_f^v)^{K^v}\ox_C(\Ind^{\GL_2(L)}_{\bar B(L)}\chi)^\infty \to &\ker I[\lambda]\to (\pi_f^v)^{K^v}\ox_C(\Ind^{\GL_2(L)}_{\bar B(L)}H^1_{\rig}(\Ig_{K^v},D_{K^v}^{w ,\sm,\dagger})[\lambda] z_{1,\iota}^{-1}z_{2,\iota})^{\iota\-\lan}\to 0,\\
    0\to (\pi_f^v)^{K^v}\ox_C(\Ind^{\GL_2(L)}_{\bar B(L)}\chi)^\infty\to &(\pi_f^v)^{K^v}\ox_C(\Ind^{\GL_2(L)}_{\bar B(L)}H^1_{\rig}(\Ig_{K^v},D_{K^v}^{w ,\sm,\dagger})[\lambda])^{\iota\-\lan}\to \ker I[\lambda]\to 0.
\end{align*}

\subsubsection*{Examples: $\pi_v$ is a special series}
After twisting by a smooth character, we may assume that $\pi_v=\St_2^\infty$, the smooth Steinberg representation of $\GL_2(L)$. 

We first compute $\bbH^1(\DR)[\lambda]$ and show that the Hecke action on $\bbH^1(\DR)\tilde{[\lambda]}$ is semisimple. Let us introduce some notations on representations. Set $I=(\Ind^{\GL_2(L)}_{B(L)}1)^{\iota\-\an}$, with cosocle $I_s=(\Ind^{\GL_2(L)}_{B(L)}z_{1,\iota}z_{2,\iota}^{-1})^{\iota\-\an}$. Let $\St_2^{\iota\-\an}:=I/1$ be the locally $\iota$-analytic Steinberg representation over $C$. We also define $\tilde{I}=(\Ind^{\GL_2(L)}_{B(L)}|\cdot|\ox |\cdot|^{-1})^{\iota\-\an}$, with cosocle $\tilde{I}_s=(\Ind^{\GL_2(L)}_{B(L)}z_{1,\iota}z_{2,\iota}^{-1}|\cdot|\ox |\cdot|^{-1})^{\iota\-\an}$. Using the theory of Orlik-Strauch representations, we have exact sequences 
\begin{align*}
    &0\to (\Ind^{\GL_2(L)}_{B(L)}1)^\infty\to I\to I_s\to 0,\\
    &0\to (\Ind^{\GL_2(L)}_{B(L)}|\cdot|\ox |\cdot|^{-1})^\infty\to \tilde{I}\to \tilde{I}_s\to 0
\end{align*}
of $\GL_2(L)$-representations. 


\begin{proposition}\label{universalSteinberg}
\begin{align*}
    \dim_C\Ext^1_{\GL_2(L),\iota}(\tilde{I},\St_2^{\iota\-\an})=1
\end{align*}
where $\Ext^1_{\GL_2(L),\iota}(-,-)$ denote extensions being locally $\iota$-analytic.
\begin{proof}
Recall that by definition we have $\St_2^{\iota\-\an}=I/1$. As $\Ext^{q}_{\GL_2(L),\iota}(\tilde{I},I)=0$ for $q=1,2$ by \cite[Prop 8.15, Prop 8.18]{Koh11}, we see that 
\[
    \Ext^1_{\GL_2(L),\iota}(\tilde{I},\St_2^{\iota\-\an})\isom\Ext^2_{\GL_2(L),\iota}(\tilde{I},1).
\]
By Schneider-Teitelbaum duality  \cite[Corollary 4.4]{MR2133762}, the remark after \cite[Theorem 8.4]{MR2375595} for the locally $L$-analytic case, and the computation for the Schneider-Teitelbaum dual for locally $L$-analytic principal series \cite[Proposition 6.5]{MR2133762}, \cite[Theorem 6.5]{Koh11} and the proof of \cite[Theorem 8.18]{Koh11} for the (FIN) condition (in \cite[Section 6]{MR2133762}), we have
\[
    \Ext^2_{\GL_2(L),\iota}(\tilde{I},1)=\Ext^1_{\GL_2(L),\iota}(1,(\Ind^{\GL_2(L)}_{\bar B(L)}t_1^{-1}t_2)^{\iota\-\an})
\]
so that from the Schraen spectral sequence \cite[Cor 4.9]{Schr11}
\[
    \Ext^p_{T(L),\iota}(H_q^{\an}(\bar N(L),1),t_1^{-1}t_2)\implies\Ext_{\GL_2(L),\iota}^{p+q}(1,(\Ind^{\GL_2(L)}_{\bar B(L)}t_1^{-1}t_2)^{\iota\-\an})
\]
together with $H_0^{\an}(\bar N(L),1)=t_1t_2^{-1}$, $H_1^{\an}(\bar N(L),1)=t_1^{-1}t_2$, we get that 
\[
    \dim \Ext^1_{\GL_2(L),\iota}(L(\nu),(\Ind^{\GL_2(L)}_{\bar B(L)}t_1^{-1}t_2)^{\iota\-\an})=1
\]
as desired.
\end{proof}
\end{proposition}

\begin{remark}
The non-trivial extension class inside $\dim_E\Ext^1_{\GL_2(L),\iota}(\tilde{I},\St_2^{\iota\-\an})$ (up to scalar) is the ``universal'' extension described in \cite[Theorem 5.4.1(v)]{BQ24}.
\end{remark}


 
Now we begin to calculate $\bbH^1(\DR)[\lambda]$. 
\begin{proposition}
The exact sequence 
\[
    0\to \ker\mathrm{SS}[\lambda]\to H^1(\ker d)[\lambda]\ov{\mathrm{SS}}\to H^1(\ker d')[\lambda]\to \coker\mathrm{SS}[\lambda]\to 0
\]
is isomorphic to $(\pi_f^v)^{K^v}$ tensor with the $\GL_2(L)$-equivariant exact sequence
\[
    0\to \St_2^\infty\to \St_2^{\iota\-\an}\to [I_s-1]\to 1\to 0.
\]
\begin{proof}
Since $\pi_v'=\calA^{K^v}_{\bar G,w }[\lambda]$ is the trivial representation of $D_L^\times$, it is straightforward to verify that $H^1(\ker d)[\lambda]\ov{\mathrm{SS}}\to H^1(\ker d')[\lambda]$ is isomorphic to $H^1(j_!\calO_{\Omega})\ov{\bar d_{\Dr}}\to H^1(j_!\Omega_{\Omega}^1)$. Consequently, by using results from the calculation of de Rham cohomology of Drinfeld spaces, we deduce the desired result.
\end{proof}
\end{proposition}

\begin{proposition}
The exact sequence 
\[
    0\to \ker\ORD[\lambda]\to H^0(\coker d)[\lambda]\ov{\mathrm{ORD}}\to H^0(\coker d')[\lambda]\to 0
\]
is $\GL_2(L)$-equivariantly isomorphic to $(\pi_f^v)^{K^v}$ tensor with the $\GL_2(L)$-equivariant exact sequence
\[
    0\to (\Ind^{\GL_2(L)}_{B(L)}|\cdot|\ox |\cdot|^{-1})\to I\to I_s\to 0.
\]
\begin{proof}
This is because as $T(L)$-representations, we have $J_{N(L)}\pi_v=|\cdot|\ox |\cdot|^{-1}$.
\end{proof}
\end{proposition}

Next we want to determine the structure of $\bbH^1(\DR)\tilde{[\lambda]}$ as a $\GL_2(L)\times\bbT^S$-representation. If we omit the $\bbT^S$-action, we know that $\bbH^1(\DR)\tilde{[\lambda]}$ is an extension of $\tilde{I}$ by $\St_2^{\iota\-\an}$ up to some multiplicities.

\begin{corollary}
As $\GL_2(L)$-representations, 
$\bbH^1(\DR)\tilde{[\lambda]}$ is the unique non-split extension of $\tilde{I}$ by $\St_2^{\iota\-\an}$ up to some multiplicities.
\begin{proof}
For simplicity we assume $\dim_C (\pi_f^v)^{K^v}=1$. As $\dim \Ext^1_{\GL_2(L),\iota}(\tilde{I},\St_2^{\iota\-\an})=1$, $\bbH^1(\DR)\tilde{[\lambda]}$ is either isomorphic to $\St_2^{\iota\-\an}\oplus \tilde{I}$ or a non-split extension of $\tilde{I}$ by $\St_2^{\iota\-\an}$. Suppose that $\bbH^1(\DR)\tilde{[\lambda]}\isom \St_2^{\iota\-\an}\oplus \tilde{I}$. As $\bbH^1(\DR')\tilde{[\lambda]}\isom \bbH^1(\DR)\tilde{[\lambda]}/\bbH^1(\DR^{\lalg})[\lambda]$ with $\bbH^1(\DR^{\lalg})[\lambda]\isom (\St_2^\infty)^{\oplus 2}$, we see that $\bbH^1(\DR')\tilde{[\lambda]}=(I_s\oplus 1-\tilde{I}_s)$. But we know that up to twist $\bbH^1(\DR')\tilde{[\lambda]}$ can only be an extension of $\tilde I_s$ by $I_s-1$. This gives a contradiction.
\end{proof}
\end{corollary}

\begin{proposition}\label{prop:semisimplespecial}
The Hecke action of $\bbT^S$ on $\bbH^1(\DR)\tilde{[\lambda]}$ is semisimple. That is, 
\[
    \bbH^1(\DR)\tilde{[\lambda]}=\bbH^1(\DR)[\lambda]
\]
Consequently,
\begin{align*}
    \bbH^1(\DR)\tilde{[\lambda]}&=\bbH^1(\DR)[\lambda]=(\pi_f^v)^{K^v}\ox_C(\St_2^{\infty})^{\oplus 2}-I_s-1-\tilde{I}_s,\\
    \bbH^1(\DR')\tilde{[\lambda]}&=\bbH^1(\DR')[\lambda]=(\pi_f^v)^{K^v}\ox_C I_s-1-\tilde{I}_s,\\
    \ker I^1\tilde{[\lambda]}&=\ker I^1[\lambda]=(\pi_f^v)^{K^v}\ox_C \St_2^{\infty}-I_s-1-\tilde{I}_s.
\end{align*}
\begin{proof}
For simplicity we assume $\dim_C (\pi_f^v)^{K^v}=1$. As $\bbT^S$ is a commutative algebra, and we know that the Hecke action on $H^1(\ker d)\tilde{[\lambda]}\isom \St_2^{\iota\-\an}$ and $H^0(\coker d)\tilde{[\lambda]}\isom \tilde{I}$ is semisimple, it suffices to show that $\Hom_{\GL_2(L)}(\tilde{I},\St_2^{\iota\-\an})=0$. The only same Jordan-Holder factors of $\tilde{I},\St_2^{\iota\-\an}$ is $\St_2^{\infty}$. However, $\St_2^\infty$ is not a quotient of $(\Ind^{\GL_2(L)}_{B(L)}|\cdot|\ox |\cdot|^{-1})$. This implies that $\Hom_{\GL_2(L)}(\tilde{I},\St_2^{\iota\-\an})=0$, and hence the Hecke action of $\bbT^S$ on $\bbH^1(\DR)\tilde{[\lambda]}$ is semisimple. This directly implies that the Hecke action of $\bbT^S$ on $\bbH^1(\DR')\tilde{[\lambda]}$ and $\ker I\tilde{[\lambda]}$ is semisimple, as they are quotients of $\bbH^1(\DR)\tilde{[\lambda]}$.
\end{proof}
\end{proposition}

\subsubsection*{Examples: $\pi_v$ is supercuspidal}\label{cuspidal}
Assume that $\pi_v$ is supercuspidal. In this case, $J_{\bar N(L)}\pi_v=0$, and the $D_L^\times$ action on $\tau_v$ does not factors through the reduced norm map. We may assume that the subgroup $1+\varpi^n\calO_L$ acts trivially on $\tau_v$, and up to twist by unramified characters we assume that $\varpi\in L^\times\subset D_L^\times$ acts trivially on $\tau_v$.

Recall that there is a natural action of $\GL_2(L)^0\times\calO_{D_L}^\times$ on $\calM_{\LT,\infty}^{(0)}\isom \calM_{\Dr,\infty}^{(0)}$. Define a valuation map as follows:
\begin{align*}
    v:\GL_2(L)\times D_L^\times\to \bbZ,(g,\check g)\mapsto v_p(\det(g))+v_p(\mathrm{nrd}(\check g)),
\end{align*}
where $\mathrm{nrd}$ is the reduced norm on $D_L^\times$. Let $(\GL_2(L)\times D_L^\times)^0$ be the kernel of $v$, then $(\GL_2(L)\times D_L^\times)^0$ acts on $\calM_{\LT,\infty}^{(0)}\isom \calM_{\Dr,\infty}^{(0)}$. From this we deduce that $H^1(j_!\calO_{\Dr}^{\sm})$ carries actions of $(\GL_2(L)\times D_L^\times)^0$, and there is a natural way to equip $\bigoplus_{i\in\bbZ} H^1(j_!\calO_{\Dr}^{\sm})$ with an action of $\GL_2(L)\times D_L^\times$ via the identification 
\begin{align*}
    \bigoplus_{i\in\bbZ} H^1(j_!\calO_{\Dr}^{\sm})\isom \text{ind}^{\GL_2(L)\times D_L^\times}_{(\GL_2(L)\times D_L^\times)^0}H^1(j_!\calO_{\Dr}^{\sm})
\end{align*}
where $\text{ind}$ stands for the compact induction. Similarly we can equip $\bigoplus _{i\in\bbZ}H^1(j_!\Omega_{\Dr}^{1,\sm})$ and $\bigoplus _{i\in\bbZ}H^q_{\dR,c}(\calM_{\Dr,\infty}^{(0)})$ with $q=1,2$ a natural $\GL_2(L)\times D_L^\times$-action. Hence we have 
\begin{itemize}
    \item $\bbH^1(\DR)[\lambda]=H^1(\ker d)[\lambda]=(\pi_f^v)^{K^v}\ox_C(\bigoplus_{i\in\bbZ} H^1(j_!\calO_{\Dr}^{\sm})\ox_C\tau_v)^{D_L^\times}$.
    \item $\bbH^1(\DR')[\lambda]=H^1(\ker d')[\lambda]=(\pi_f^v)^{K^v}\ox_C(\bigoplus_{i\in\bbZ} H^1(j_!\Omega_{\Dr}^{1,\sm})\ox_C\tau_v)^{D_L^\times}$.
    \item $\ker \SS[\lambda]=(\pi_f^v)^{K^v}\ox_C(\bigoplus_{i\in\bbZ} H^1_{\dR,c}(\calM_{\Dr,\infty}^{(0)})\ox_C\tau_v)^{D_L^\times}$.
\end{itemize}
From the identification 
\begin{align*}
    \ker \SS[\lambda]=\bbH^1(\DR^{\lalg})[\lambda]
\end{align*}
and $\bbH^1(\DR^{\lalg})[\lambda]\isom \pi_v^{\oplus 2}$, we deduce that 
\begin{align*}
    (\bigoplus_{i\in\bbZ} H^1_{\dR,c}(\calM_{\Dr,\infty}^{(0)})\ox_C\tau_v)^{D_L^\times}=\pi_v^{\oplus 2}.
\end{align*}
Therefore, we have exact sequences
\begin{align*}
    0\to (\pi_f^v)^{K^v}\ox_C\pi_v\to &\ker I[\lambda]\to (\pi_f^v)^{K^v}\ox_C\pi_c  \to 0,\\
    0\to (\pi_f^v)^{K^v}\ox_C\pi_v\to &(\pi_f^v)^{K^v}\ox_C\tilde{\pi}\to \ker I[\lambda]\to 0,
\end{align*} 
where $\tilde\pi=(\bigoplus_{i\in\bbZ}H^1(j_!\calO_{\Dr}^{\sm})\ox_C\tau_v)^{D_L^\times}$ and $\pi_c=(\bigoplus_{i\in\bbZ}H^1(j_!\Omega_{\Dr}^{1,\sm})\ox_C\tau_v)^{D_L^\times}$.

\subsection{Translations of intertwining operators}
In this section, we study the translation of the intertwining operator. Traditionally, the translation functor is a useful tool in the study of representations of semisimple
Lie algebras. There exist recent works [JLS22, Din24] applying translation functors to study locally analytic
representations. We use similar methods as in \cite{WConDrinfeld} to study translations of the sheaves $\calO_{K^v}^{\lan,(a,b)}$, and translations of the intertwining operators.

Let $\iota:L\inj C$ be our fixed $\bbQ_p$-embedding as before. Let $\frg_\iota:=\gl_2(L)\ox_{L,\iota}C$, and $\calZ(\frg_\iota)$ be the center of $U(\frg_\iota)$. Let $\frb_\iota$ be the Borel subalgebra consisting of upper triangular matrices, and let $\frh_{\iota}$ be the Cartan subalgebra consisting of diagonal matrices. For an integral weight $\mu$ of $\frh_\iota$, we denote by $\tilde{\chi}_\mu$ the infinitesimal character of $\calZ(\frg_\iota)$ acting on $U(\frg_{\iota})\ox_{U(\frb_\iota)}\mu$. If $\mu$ is dominant, we let $V^{\mu}$ be the $\iota$-algebraic representation of $\GL_2(L)$ over $E$, with highest weight $\mu$ with respect to $\frb_\iota$. The center $\calZ(\frg_\iota)$ is generated as an $E$-algebra by $z=\left( \begin{matrix} 1&0\\0&1\end{matrix} \right)\in \frg_\iota$, and the Casimir operator $\Omega=u^+u^-+u^-u^++\frac{1}{2}h^2\in U(\frg_\iota)$, where $u^+=\left( \begin{matrix} 0&1\\0&0\end{matrix} \right)\in\frg_\iota$, $u^-=\left( \begin{matrix} 0&0\\1&0\end{matrix} \right)\in\frg_\iota$, and $h=\left( \begin{matrix} 1&0\\0&-1\end{matrix} \right)\in\frg_\iota$.

Let $\Lambda$ be the root lattice of $\GL_2$ and $\Lambda^{\vee}$ be the dual root lattice. We identify $\Lambda$ with $\bbZ^2$ such that the standard representation of $\GL_2$ has highest weight $(1,0)$. Let $\varrho\in\Lambda^{\vee}$ be the dual root such that $\varrho^{\vee}=(1,0)$. The Weyl group of $\GL_2$ consists of two elements. We denote the non-trivial one by $s$. The action of the Weyl group on the root lattice is given by $s(a,b)=(b,a)$, and we define the dot action to be $s\cdot (a,b)=(b-1,a+1)$. Let $\lambda\in\Lambda$ be a root. If $s\cdot\lambda\neq \lambda$, then we say $\lambda$ is generic. Otherwise, we say $\lambda$ is singular. Given an integral weight $\lambda=(a,b)\in\Lambda$, we say it is dominant if $a\ge b$, and antidominant if $a<b$.

Let $\Rep^{\iota\-\lan}_{\tilde{\chi}_\mu}(G)$ be the category of locally $\iota$-analytic representations on LB spaces over $C$, with generalized infinitesimal character $\tilde{\chi}_\mu$. We define the translation functor from weight $\mu$ to weight $\lambda$ as follows:
\begin{align*}
    T_\mu^\lambda:\Rep^{\iota\-\lan}_{\tilde{\chi}_\mu}(G)&\to \Rep^{\iota\-\lan}_{\tilde{\chi}_\lambda}(G)\\
    M&\mapsto (V^{\bar\nu}\ox_C M)\{\calZ(\frg_\iota)=\tilde{\chi}_\lambda\}
\end{align*}
where $\bar\nu$ is the unique dominant weight in the orbit of the usual action of the Weyl group on $\lambda-\mu$, and $\{-\}$ denotes the generalized eigenspace. By \cite[Lemma 2.4.6, Theorem 2.4.7]{JLS22}, $T_\mu^\lambda$ is an exact functor, with left and right adjoints $T_\lambda^\mu$. 

Let $\mu,\lambda\in\Lambda$ be two roots. Recall that if $\mu=(a,b)\in\bbZ^2$, then $\calZ(\frg_\iota)$ acts on $\calO_{K^v}^{\lan,\mu}$ via $\tilde{\chi}_\mu$. We define the translation of the sheaf $\calO_{K^v}^{\lan,\mu}$ as follows. Let $T_\mu^\lambda \calO_{K^v}^{\lan,\mu}$ be the sheaf on $\fl$, such that for every sufficiently small open affinoid subset $U\subset\fl$, we define $(T_\mu^\lambda \calO_{K^v}^{\lan,\mu})(U):=T_\mu^\lambda (\calO_{K^v}^{\lan,\mu}(U))$. Here we use the Lie algebra action of $\frg_\iota$ on $\calO_{K^v}^{\lan,\mu}(U)$ to define the translation functor. One directly checks that $T_\mu^\lambda \calO_{K^v}^{\lan,\mu}$ is a well-defined $\GL_2(L)$-equivariant sheaf on $\fl$.

Next, we describe $T_\mu^\lambda \calO_{K^v}^{\lan,\mu}$ using certain twists of $\calO_{K^v}^{\lan,(a,b)}$. We always put the trivial action of $\bbT^S$ and $C$-semilinear trivial action of $\Gal(\bar L/L)$ on $V^{(a,b)}$ when computing translations of $\calO_{K^v}^{\lan,\mu}$. We will mainly use the following property (Proposition \ref{prop:twist}) to do calculations: 
\begin{align*}
    \calO_{K^v}^{\lan,(a,b)}\ox_{\calO_{\fl}}\omega_{\fl}^{(a',b')}\isom \omega_{K^v}^{(-a',-b'),\lan,(a+b',b+a')}(-a')
\end{align*}
for $(a,b),(a',b')\in\bbZ^2$ integers.

\begin{proposition}\label{prop:translationOsmoxOfl1}
Let $\mu,\lambda\in\Lambda$ be roots. In the following two situations:
\begin{itemize}
    \item $\mu,\lambda$ are generic, and $\mu,\lambda$ are both dominant or $\mu,\lambda$ are both antidominant,
    \item $\mu$ is generic and $\lambda$ is singular,
\end{itemize}
we have $T_{\mu}^{\lambda}\calO_{K^v}^{\lan,\mu}\isom \omega_{K^v}^{\mu-\lambda,\lan,\lambda}(\langle\varrho,\mu-\lambda\rangle)$. Here $\langle-,-\rangle$ denotes the natural pairing between the roots and dual roots, and $(-)$ denotes the Tate twist.
\begin{proof}
This follows by a direct computation on flag variety. For example, we have an exact sequence of locally free $\calO_{\fl}$-modules: 
\[
    0\to \omega_{\fl}^{(0,1)}\to V^{(1,0)}\ox_C\calO_{\fl}\to \omega_{\fl}^{(1,0)}\to 0.
\]
After tensoring the above exact sequence with the $\calO_{\fl}$-module $\calO_{K^v}^{\lan,(0,0)}$ over $\calO_{\fl}$, and get an exact sequence 
\[
    0\to \omega_{K^v}^{(0,-1),\lan,(0,1)}\to V^{(1,0)}\ox_C \calO_{K^v}^{\lan,(0,0)}\to \omega_{K^v}^{(-1,0),\lan,(1,0)}(-1)\to 0.
\]
The exactness of the above exact sequence follows from the flatness of $\omega_{\fl}^{(1,0)}$ over $\calO_{\fl}$. As $\omega_{K^v}^{(-1,0),\lan,(1,0)}(-1)$ has infinitesimal character $\tilde{\chi}_{(1,0)}$ and $\omega_{K^v}^{(0,-1),\lan,(0,1)}$ has infinitesimal character $\tilde{\chi}_{(0,1)}$, we see that $T_{(0,0)}^{(1,0)}\calO_{K^v}^{\lan,(0,0)}\isom\omega_{K^v}^{(-1,0),\lan,(1,0)}(-1)$, and $T_{(0,0)}^{(0,1)}\calO_{K^v}^{\lan,(0,0)}\isom\omega_{K^v}^{(0,-1),\lan,(0,1)}$.

For the general case, one can argue as follows.  Applying $\Sym^a\ox {\det}^b$ to the inclusion map $\omega_{\fl}^{(0,1)}\subset V^{(1,0)}_{\iota}\ox_C\calO_{\fl}$, one get a filtration on $V^{(a,b)}_{\iota}\ox_C\calO_{\fl}$. We pick $(a,b)\in\bbZ^2$ so that $\bar\nu=(a,b)$, where $\bar\nu$ is the unique dominant weight in the orbit of the dot action of the Weyl group on $\lambda-\mu$. After tensoring $\calO_{K^v}^{\lan,\mu}$ with $V^{(a,b)}_{\iota}\ox_C\calO_{\fl}$ over $\calO_{\fl}$, we get a filtration on $V^{\bar\nu}\ox_C\calO_{K^v}^{\lan,\mu}$. The condition on $\mu$ and $\lambda$ implies that there is only one graded piece such that $\calZ(\frg_\iota)$ acts via $\tilde{\chi}_\lambda$. From this one can compute the translation of $\calO_{K^v}^{\lan,\mu}$.  
\end{proof}
\end{proposition}

Next we handle another case.

\begin{proposition}\label{prop:122}
Suppose $\mu$ is singular with $\lambda$ generic and dominant. Then $T_{\mu}^\lambda\calO_{K^v}^{\lan,\mu}$ fits into an exact sequence of $\GL_2(L)$-equivariant sheaves on $\fl$:
\[
    0\to \omega_{K^v}^{\mu-s\cdot\lambda,\lan,s\cdot\lambda}(\langle\varrho,(\mu-s\cdot \lambda)\rangle)\to T_{\mu}^\lambda\calO_{K^v}^{\lan,\mu}\to \omega_{K^v}^{\mu-\lambda,\lan,\lambda}(\langle\varrho,\mu-\lambda\rangle)\to 0.
\]
\begin{proof}
This also follows by a direct computation on the flag variety. For example, we have an exact sequence of locally free $\calO_{\fl}$-modules: 
\[
    0\to \omega_{\fl}^{(0,1)}\to V^{(1,0)}\ox_C\calO_{\fl}\to \omega_{\fl}^{(1,0)}\to 0.
\]
Twisting the above exact sequence of $\calO_{\fl}$-modules by $\omega_{\fl}^{(-1,0)}$, one gets an exact sequence of locally free $\calO_{\fl}$-modules:
\[
    0\to \omega_{\fl}^{(-1,1)}\to V^{(1,0)}\ox_C\omega_{\fl}^{(-1,0)}\to \calO_{\fl}\to 0,
\]
which is also $\GL_2(L)$-equivariant. The line bundle $\omega_{\fl}^{(-1,1)}$ is naturally identified with $\Omega_{\fl}^1$. Tensoring with $\calO_{K^v}^{\lan,(0,0)}$ over $\calO_{\fl}$, one get an exact sequence 
\[
    0\to \calO_{K^v}^{\lan,(0,0)}\ox_{\calO_{\fl}}\Omega_{\fl}^1\to V^{(1,0)}\ox_C\omega_{K^v}^{(1,0),\lan,(-1,0)}(1)\to \calO_{K^v}^{\lan,(0,0)}\to 0.
\]
The action of the center $\calZ(\frg_\iota)$ on $\calO_{K^v}^{\lan,(0,0)}\ox_{\calO_{\fl}}\Omega_{\fl}^1$ and  $\calO_{K^v}^{\lan,(0,0)}$ are both given by the character $\tilde{\chi}_{(0,0)}$. From this we deduce that $T_{(-1,0)}^{(0,0)}\omega_{K^v}^{(1,0),\lan,(-1,0)}(1)$ fits into an exact sequence of $\GL_2(L)$-equivariant sheaves:
\[
    0\to \calO_{K^v}^{\lan,(0,0)}\ox_{\calO_{\fl}}\Omega_{\fl}^1\to T_{(-1,0)}^{(0,0)}\omega_{K^v}^{(1,0),\lan,(-1,0)}(1)\to \calO_{K^v}^{\lan,(0,0)}\to 0.
\]

For the general case, one also tensoring $\calO_{K^v}^{\lan,\mu}$ with $V^{\bar\nu}$ over $C$ so that to get a filtration on $V^{\bar\nu}\ox_C\calO_{K^v}^{\lan,\mu}$, just as in the proof of Proposition \ref{prop:translationOsmoxOfl1}. The condition on $\mu$ and $\lambda$ implies that there are exactly two graded pieces has infinitesimal character $\tilde{\chi}_\lambda$. From this we deduce the result.
\end{proof}
\end{proposition}

\begin{proposition}\label{prop:connectionmap1}
Suppose $\mu$ is singular with $\lambda$ generic and dominant. The action of the Casimir operator $\Omega \in\calZ(\frg_\iota)$ on $T_{\mu}^{\lambda}\calO_{K^v}^{\lan,\mu}$ is not-semisimple, and induces a $\GL_2(L)$-equivariant map 
\[
    \omega_{K^v}^{\mu-\lambda,\lan,\lambda}(\langle\varrho,\mu-\lambda\rangle)\to \omega_{K^v}^{\mu-s\cdot\lambda,\lan,s\cdot\lambda}(\langle\varrho,(\mu-s\cdot \lambda)\rangle)
\]
which is essentially given by differentiate sections along $\fl$, and is twist of $\bar d$. This is the map $\bar d$ constructed in Theorem \ref{thm:dbar} up to twist.
\begin{proof}

We can use the similar calculation as in the proof of \cite[Proposition 2.3.3]{WConDrinfeld}.

\end{proof}
\end{proposition}

As $T_\lambda^\mu$ is biadjoint to $T_\mu^{\lambda}$, it is natural to study the associate unit and counit maps.
\begin{proposition}
Let $\mu,\lambda\in\Lambda$. Suppose that $\mu$ and $\lambda$ are either both generic or both singular. Then the natural maps $\calO_{K^v}^{\lan,\mu}\to T_{\lambda}^\mu T_{\mu}^\lambda\calO_{K^v}^{\lan,\mu}$ and $T_{\lambda}^\mu T_{\mu}^\lambda\calO_{K^v}^{\lan,\mu}\to \calO_{K^v}^{\lan,\mu}$ are isomorphisms of $\GL_2(L)$-equvariant sheaves.
\begin{proof}
As $\calO_{K^v}^{\lan,\mu}\isom\calO_{K^v}^{\iota\-\lan,\mu}\hat\ox_{\calO_{K^v}^{\sm}}\calO_{K^v}^{\iota^c\-\lan}$, and $T_{\mu}^\lambda(\calO_{K^v}^{\lan,\mu})\isom T_\mu^\lambda(\calO_{K^v}^{\iota\-\lan,\mu})\hat\ox_{\calO_{K^v}^{\sm}}\calO_{K^v}^{\iota^c\-\lan}$, we can use \cite[Theorem 3.2.1]{JLS22} which says that the composition $T_\mu^\lambda\comp T_\lambda^\mu$ is isomorphic to the identity functor. One can also compute the result directly.
\end{proof}
\end{proposition}

\begin{proposition}\label{prop:wallcrossingofOsmoxOfl}
Suppose that $\mu$ is generic and $\lambda$ is singular. Then:
\begin{enumerate}[(i)]
    \item $T_{\lambda}^\mu T_{\mu}^\lambda\calO_{K^v}^{\lan,\mu}$ is an extension of $\calO_{K^v}^{\lan,\mu}$ by $\calO_{K^v}^{\lan,\mu}\ox_{\calO_{\fl}}(\Omega_{\fl}^1)^{\ox k+1}$ with $k=a-b$ if $\mu=(a,b)$.
    \item The natural map $\calO_{K^v}^{\lan,\mu}\to T_{\lambda}^\mu T_{\mu}^\lambda\calO_{K^v}^{\lan,\mu}$ is given by the composition $\bar d:\calO_{K^v}^{\lan,\mu}\to \calO_{K^v}^{\lan,\mu}\ox_{\calO_{\fl}}(\Omega_{\fl}^1)^{\ox k+1}$ and the natural inclusion $\calO_{K^v}^{\lan,\mu}\ox_{\calO_{\fl}}(\Omega_{\fl}^1)^{\ox k}\inj T_{\lambda}^\mu T_{\mu}^\lambda\calO_{K^v}^{\lan,\mu}$.
    \item The natural map $T_{\lambda}^\mu T_{\mu}^\lambda\calO_{K^v}^{\lan,\mu}\to \calO_{K^v}^{\lan,\mu}$ is surjective and induces an isomorphism $(T_{\lambda}^\mu T_{\mu}^\lambda\calO_{K^v}^{\lan,\mu})_{\chi_{\mu}}\to \calO_{K^v}^{\lan,\mu}$ where $(T_{\lambda}^\mu T_{\mu}^\lambda\calO_{K^v}^{\lan,\mu})_{\tilde \chi_{\mu}}$ denote the co-semisimple part of the infinitesimal action. 
    \item The infinitesimal semisimple part $(T_{\lambda}^\mu T_{\mu}^\lambda\calO_{K^v}^{\lan,\mu})^{\tilde \chi_{\mu}}$ is an extension of $\ker \bar d$ by $\calO_{K^v}^{\lan,\mu}\ox_{\calO_{\fl}}(\Omega_{\fl}^1)^{\ox k+1}$.
\end{enumerate}
\begin{proof}
For simplicity, we only calculate $T_{(-1,0)}^{(0,0)}T_{(0,0)}^{(-1,0)}\calO_{K^v}^{\lan,(0,0)}$ and the general case follows a similar pattern. This is also a special case of Proposition \ref{prop:122}. Firstly, $T_{(0,0)}^{(-1,0)}\calO_{K^v}^{\lan,(0,0)}\isom (\calO_{K^v}^{\lan,(0,0)}\ox_C V^{(0,-1)})\{\calZ(\frg_\iota)=\tilde{\chi}_{(-1,0)}\}=\omega_{K^v}^{(1,0),\lan,(-1,0)}(1)$. Translating back, we have
\[
    T_{(-1,0)}^{(0,0)}T_{(0,0)}^{(-1,0)}\calO_{K^v}^{\lan,(0,0)}=(\omega_{K^v}^{(1,0),\lan,(-1,0)}(1))\{\calZ(\frg_\iota)=\tilde{\chi}_{(0,0)}\},
\]
which is an extension of $\calO^{\lan,(0,0)}_{K^v}\ox_{\calO_{\fl}}\Omega_{\fl}^1$ by $\calO_{K^v}^{\lan,(0,0)}$, and is obtained by tensoring $\omega_{K^v}^{(1,0),\lan,(-1,0)}(1)$ with the exact sequence 
\[
    0\to \omega_{\fl}^{(0,1)}\to V^{(1,0)}\ox_C\calO_{\fl}\to \omega_{\fl}^{(1,0)}\to 0
\]
over $\calO_{\fl}$. 

The rest of the properties of the natural maps given by wall-crossing follow by a direct computation on flag variety. One can use the similar proof as in \cite[\S 2.3]{WConDrinfeld}.
\end{proof}
\end{proposition}

In the setting of translation functors in the BGG category $\mathcal{O}$, for $\lambda$ dominant regular and $\mu$ singular, $T_{\mu}^{\lambda}T_{\lambda}^{\mu}$ is called the wall-crossing functor and denoted by $\Theta_s$. See \cite[Chapter 7]{Hum94}. We sometimes also denote $\Theta_s(-)$ as the wall-crossing functor for locally analytic representations. For example, $\Theta_s(\mathcal{O}_{K^v}^{\lan,(0,0)})$ indicates $T_{(-1,0)}^{(0,0)}T_{(0,0)}^{(-1,0)}\mathcal{O}_{K^v}^{\lan,(0,0)}$.

As translation functors are exact functors, we can show that it commutes with taking cohomology.
\begin{proposition}\label{prop:translationcommuteswithcohomology}
Let $\mu,\lambda\in\Lambda$ be roots. Then $T_{\mu}^{\lambda}H^i(\fl,\calO_{K^v}^{\lan,\mu})\isom H^i(\fl,T_{\mu}^{\lambda}\calO_{K^v}^{\lan,\mu})$ for $i\ge 0$.
\begin{proof}
By Lemma \ref{Cech} and similar results in \cite[Lemma 5.1.2(1)]{Pan22}, we see $H^i(\fl,\calO_{K^v}^{\lan,\mu})$ can be computed using \v{C}ech complex. As $T_{\mu}^\lambda(-)$ is an exact functor, from this we deduce 
\begin{align*}
    T_{\mu}^{\lambda}H^i(\fl,\calO_{K^v}^{\lan,\mu})\isom H^i(\fl,T_{\mu}^{\lambda}\calO_{K^v}^{\lan,\mu})
\end{align*}
for $i\ge 0$.
\end{proof}
\end{proposition}

Recall that we have defined differential operators on $\calO_{K^v}^{\lan,(0,-k)}$ for $k\ge 0$, which are given by
\begin{align*}
    d:\calO_{K^p}^{\lan,(0,-k)}\to \calO_{K^p}^{\lan,(0,-k)}\ox _{\calO^{\sm}_{K^p}}(\Omega^{1,\sm}_{K^p})^{\ox k+1},\\
    \bar d:\calO_{K^p}^{\lan,(0,-k)}\to \calO_{K^p}^{\lan,(0,-k)}\ox _{\calO_{\fl}}(\Omega^{1}_{\fl})^{\ox k+1}
\end{align*}
for $k\in\bbZ_{\ge 0}$, such that $d$ is $\calO_{\fl}$-linear and $\bar d$ is $\calO_{K^p}^{\sm}$-linear. Moreover, they are $\GL_2(L)$-equivariant. Using results on translations of $\calO_{K^v}^{\lan,\mu}$, we can compute the translations of these differential operators. For simplicity, we study the case where $k=0$. The general case can be proved in a similar method. 

\begin{proposition}\label{prop:d-1}
For $k\in\bbZ_{\ge -1}$, the translation $T_{(0,0)}^{(k,0)}d$ is given by 
\[
    T_{(0,0)}^{(k,0)}d:\omega_{K^v}^{(-k,0),\lan,(k,0)}(-k)\to  \omega_{K^v}^{(-k,0),\lan,(k,0)}\ox_{\calO_{K^v}^{\sm}}\Omega_{K^v}^{1,\sm}(-k),
\]
which is a twist of $d$ by $\omega_{\fl}^{(k,0)}$.
\begin{proof}
As $\frg_\iota$ acts trivially on $\Omega_{K^v}^{1,\sm}$, the functor $-\ox_{\calO_{K^v}^{\sm}}\Omega_{K^v}^{1,\sm}$ commutes with translation functors. Hence the proposition follows from Proposition \ref{prop:translationOsmoxOfl1}.
\end{proof}
\end{proposition}

\begin{proposition}
For $k\ge 0$, the natural map $d\to T_{(k,0)}^{(0,0)}T_{(0,0)}^{(k,0)}d$ is an isomorphism. For $k=-1$, the natural map $d\to T_{(-1,0)}^{(0,0)}T_{(0,0)}^{(-1,0)}d=:d_{-1,1}$ is given by
\begin{center}
\begin{tikzpicture}[descr/.style={fill=white,inner sep=1.5pt}]
    \matrix (m) [
        matrix of math nodes,
        row sep=2.5em,
        column sep=2.5em,
        text height=1.5ex, 
        text depth=0.25ex
    ]
    { \calO_{K^v}^{\lan,(0,0)}&\calO_{K^v}^{\lan,(0,0)}\ox_{\calO_{K^v}^{\sm}} \Omega^{1,\sm}_{K^v}  \\
    \Theta_s(\calO_{K^v}^{\lan,(0,0)})&\Theta_s(\calO_{K^v}^{\lan,(0,0)}\ox_{\calO_{K^v}^{\sm}} \Omega^{1,\sm}_{K^v}) \\
    };
    \path[->,font=\scriptsize]
    (m-1-1) edge node[auto] {$d$} (m-1-2)
    (m-2-1) edge node[auto] {$d_{-1,1}$} (m-2-2)
    ;
    \path[->,font=\scriptsize]
    (m-1-1) edge node[right] {$\mathrm{incl}\comp \bar d$} (m-2-1)
    (m-1-2) edge node[right] {$\mathrm{incl}\comp \bar d$} (m-2-2)
    ;
\end{tikzpicture}
\end{center}
where the vertical maps are twists of $\bar d$ composed with the natural inclusion. The natural map $d_{-1,1}\to d$ is given by the natural projection, and is surjective.
\begin{proof}
This follows by the functoriality of the translation functor, and the computation on the wall-crossing of $\calO_{K^v}^{\lan,\mu}$, see Proposition \ref{prop:wallcrossingofOsmoxOfl}.
\end{proof}
\end{proposition}

\begin{proposition}
For $k\in\bbZ_{\ge 0}$, the translation $T_{(0,0)}^{(k,0)}\bar d$ is given by 
\[
    T_{(0,0)}^{(k,0)}\bar d:\omega_{K^v}^{(-k,0),\lan,(k,0)}(-k)\to \omega_{K^v}^{(1,-k-1),\lan,(-1,k+1)}(1),
\]
which is up to non-zero scalar given by the differential operator
\[
    \bar d:\calO_{K^v}^{\lan,(k,0)}\to \calO_{K^v}^{\lan,(k,0)}\ox_{\calO_{\fl}}(\Omega_{\fl}^1)^{\ox k+1}
\]
twisted by $\omega_{K^v}^{(-k,0),\sm}(-k)$. The natural map $\bar d\to T_{(k,0)}^{(0,0)} T_{(0,0)}^{(k,0)}\bar d$ is an isomorphism.
\begin{proof}
By Proposition \ref{prop:translationOsmoxOfl1}, we can compute the translation of the source and the target, so that it remains to work out $T_{(0,0)}^{(k,0)}\bar d$, which we can use a similar method in Proposition \ref{prop:connectionmap1}.
\end{proof}
\end{proposition}

\begin{proposition}\label{prop:dbar-1}
For $k=-1$, the translation $T_{(0,0)}^{(-1,0)}\bar d$ is given by the identity map on $\omega_{K^v}^{(1,0),\lan,(-1,0)}(1)$. Moreover, $\bar d_{-1,1}:=T_{(-1,0)}^{(0,0)}T_{(0,0)}^{(-1,0)}\bar d$ is the identity map. The natural map $\bar d\to \bar d_{-1,1}$ is given by 
\begin{center}
\begin{tikzpicture}[descr/.style={fill=white,inner sep=1.5pt}]
    \matrix (m) [
        matrix of math nodes,
        row sep=2.5em,
        column sep=2.5em,
        text height=1.5ex, 
        text depth=0.25ex
    ]
    { \calO^{\lan,(0,0)}_{K^v}&\calO^{\lan,(0,0)}\ox_{\calO_{\fl}} \Omega^{1}_{\fl}  \\
    \Theta_s(\calO_{K^v}^{\lan,(0,0)})&\Theta_s(\calO_{K^v}^{\lan,(0,0)}) \\
    };
    \path[->,font=\scriptsize]
    (m-1-1) edge node[auto] {$\bar d$} (m-1-2)
    (m-2-1) edge node[auto] {$\id$} (m-2-2)
    ;
    \path[->,font=\scriptsize]
    (m-1-1) edge node[right] {$\mathrm{incl}\comp \bar d$} (m-2-1)
    (m-1-2) edge node[right] {$\mathrm{incl}$} (m-2-2)
    ;
\end{tikzpicture}
\end{center}
and the natural map $\bar d_{-1,1}\to \bar d$ is given by 
\begin{center}
\begin{tikzpicture}[descr/.style={fill=white,inner sep=1.5pt}]
    \matrix (m) [
        matrix of math nodes,
        row sep=2.5em,
        column sep=2.5em,
        text height=1.5ex, 
        text depth=0.25ex
    ]
    { \Theta_s(\calO_{K^v}^{\lan,(0,0)})&\Theta_s(\calO_{K^v}^{\lan,(0,0)})  \\
    \calO^{\lan,(0,0)}_{K^v}&\calO^{\lan,(0,0)}_{K^v}\ox_{\calO_{\fl}} \Omega^{1}_{\fl} \\
    };
    \path[->,font=\scriptsize]
    (m-1-1) edge node[auto] {$\bar d_{-1,1}=\id$} (m-1-2)
    (m-2-1) edge node[auto] {$\bar d$} (m-2-2)
    ;
    \path[->,font=\scriptsize]
    (m-1-1) edge node[right] {$\mathrm{pr}$} (m-2-1)
    (m-1-2) edge node[right] {$\bar d\comp \pr$} (m-2-2)
    ;
\end{tikzpicture}.
\end{center}
\begin{proof}
Using Proposition \ref{prop:translationOsmoxOfl1}, it remains to identify the natural maps induced by wall-crossing. The key point is to prove that the translation $T_{(0,0)}^{(-1,0)}$ of $d_{\fl}:\calO_{\fl}\to \Omega^1_{\fl}$ is an isomorphism. See \cite[Lemma 2.3.6]{WConDrinfeld}.
\end{proof}
\end{proposition}

Finally, as the intertwining operator $I$ is given by the composition $\bar d'\comp d=d'\comp \bar d$ (Definition \ref{def:twistofd}), we can compute the translation of the intertwining operator using the results above.
\begin{corollary}
Let $I:\calO_{K^v}^{\lan,(0,0)}\to \calO_{K^v}^{\lan,(-1,1)}(1)$ be the intertwining operator. 
\begin{enumerate}[(i)]
    \item For $k\in\bbZ_{\ge 0}$, the natural map $I\to T_{(k,0)}^{(0,0)}T_{(0,0)}^{(k,0)}I$ is an isomorphism.
    \item In the singular case, let $I_{-1,1}:=T_{(-1,0)}^{(0,0)}T_{(0,0)}^{(-1,0)}I$. Then $I_{-1,1}\isom d_{-1,1}$.
\end{enumerate}
\end{corollary}

\begin{remark}
We may also consider translations of the subsheaf $\calO_{K^v}^{\iota\-\lan,\iota^c\-\lalg,(a,b)}\subset \calO_{K^v}^{\lan,(a,b)}$ consisting of sections which are algebraic outside the embedding $\iota$. The result is essentially the same as translations of $\calO_{K^v}^{\lan,(a,b)}$ since the translation we considered before only concerns about the locally $\iota$-analytic part. Translations of $\calO_{K^v}^{\iota\-\lan,\iota^c\-\lalg,(a,b)}$ will be useful later when we consider translations of certain de Rham cohomology groups defined using the differential operators we constructed before.
\end{remark}

\subsection{Cohomology of the intertwining operator III - Wall-crossing}\label{ss:wc}
In this section, we use our results on wall-crossing of the intertwining operator to compute wall-crossing of some de Rham cohomology groups defined using the intertwining operator. For simplicity and in order to get meaningful cohomology groups, we will focus on the restriction of $I$ on the subspace $\calO_{K^v}^{\iota\-\lan,(0,0)}$:
\begin{align*}
    I:\calO_{K^v}^{\iota\-\lan,(0,0)}\to \calO_{K^v}^{\iota\-\lan,(-1,1)}(1).
\end{align*}
We note that all results in this section can be naturally generalized to general weights.

We give a brief summary on the cohomology of the intertwining operator $I$ on $\calO_{K^v}^{\iota\-\lan,(0,0)}$. On the sheaf level, we have:
$$
\begin{tikzcd}
    &                                                 & 0 \arrow[d]                                                 & 0 \arrow[d]                                   &                                                    &   \\
    &                                                 & \ker \bar d \arrow[d] \arrow[r, "\theta"]                   & \ker \bar d' \arrow[d]                        &                                                    &   \\
0 \arrow[r] & \ker d \arrow[d, "\bar d|_{\ker d}"'] \arrow[r] & \calO \arrow[r, "d"] \arrow[d, "\bar d"'] \arrow[rd, "I" description] & \calO\ox\Omega \arrow[d, "\bar d'"] \arrow[r] & \coker d \arrow[r] \arrow[d, "\bar d|_{\coker d}"] & 0 \\
0 \arrow[r] & \ker d' \arrow[r]                               & \calO\ox\bar\Omega \arrow[r, "d'"'] \arrow[d]               & \calO^s \arrow[d] \arrow[r]                   & \coker d' \arrow[r]                                & 0 \\
    &                                                 & 0                                                           & 0                                             &                                                    &  
\end{tikzcd}
$$
where each term in the above diagram is given by:
\begin{itemize}
    \item $\calO=\calO_{K^v}^{\iota\-\lan,(0,0)}$.
    \item $\calO\ox\Omega=\calO_{K^v}^{\iota\-\lan,(0,0)}\ox_{\calO_{ K^v}^{\sm}}\Omega_{K^v}^{1,\sm}$.
    \item $\calO\ox\bar\Omega=\calO_{K^v}^{\iota\-\lan,(0,0)}\ox_{\calO_{\fl}}\Omega^1_{\fl}$.
    \item $\calO^s=\calO_{K^v}^{\iota\-\lan,(-1,1)}(1)$.
\end{itemize}
Also, the cohomologies of these differential operators are given as follows.
\begin{itemize}
    \item $\ker\bar d=\calO_{K^v}^{\sm}$. Here, $\calO_{K^v}^{\sm}=\pi_{\HT,*}(\dlim_{K_v}\pi_{K_v}^{-1}\calO_{\calX_{K^vK_v}})$ with $\pi_{K_v}:\calX_{K^v}\to \calX_{K^vK_v}$ the natural projection.
    \item $\ker \bar d'=\Omega_{K^v}^{1,\sm}$. Here, $\Omega_{K^v}^{1,\sm}=\pi_{\HT,*}(\dlim_{K_v}\pi_{K_v}^{-1}\Omega^1_{\calX_{K^vK_v}})$.
    \item $\coker d=i_*(i^*\calO_{\fl}\ox_C\calH^1_{\ord}( K^v,0))$, where $i:\bbP^1(L)\inj \fl$.
    \item $\coker d'=i_*(i^*\Omega_{\fl}^1\ox_C\calH^1_{\ord}( K^v,0))$.
    \item $\ker d=(\calA^{K^v}_{\bar G, 0}\ox_C j_!\bigoplus_{i\in\bbZ}\calO_{\Dr}^{\sm})^{D_L^\times}\oplus M_0(K^v)$, where $j:\Omega_L\inj\fl$, and $\calO_{\Dr}^{\sm}=\dlim_n\pi_{n,*}\pi_n^*\calO_{\fl}|_{\Omega}$ with $\pi_n:\calM_{\Dr,n}^{(0)}\ov{\pi_{\Dr,\GM}}\to \Omega$. 
    \item $\ker d'=(\calA^{K^v}_{\bar G, 0}\ox_C\bigoplus_{i\in\bbZ}\Omega_{\Dr}^{1,\sm})^{D_L^\times}$, where $\Omega_{\Dr}^{1,\sm}=\dlim_n\pi_{n,*}\pi_n^*\Omega^1_{\fl}|_{\Omega}$.
\end{itemize}

We also define some de Rham cohomology groups associated to these differential operators.
\begin{itemize}
    \item Define $\DR^{\sm}$ to be the two-term complex $\ker \bar d\ov{d}\to \ker \bar d'$, and $\bbH^*(\DR^{\sm})$ the hypercohomology of $\DR^{\sm}$.
    \item Define $\DR$ to be the two-term complex $\calO\ov{d}\to \calO\ox\Omega$, and $\bbH^*(\DR)$ the hypercohomology of $\DR$.
    \item Define $\DR'$ to be the two-term complex $\calO\ox\bar\Omega\ov{d'}\to \calO^s$, and $\bbH^*(\DR')$ the hypercohomology of $\DR'$.
\end{itemize}
Let $\bbT^S$ be the abstract spherical Hecke algebra defined before. All the above constructions are equivariant with respect to the Hecke action of $\bbT^S$. Let $\sigma_0^{K^v}$ be the set of Hecke eigenvalues appears in $\bbH^i(\DR^{\sm})$ where $i=0,1,2$. Let $\lambda\in \sigma_0^{K^v}$ be a Hecke eigenvalue. Let $\pi$ be the automorphic representation of $G(\bbA)$ over $C$ such that the associated system of Hecke eigenvalues is $\lambda$, and $\pi_f$ is the restriction of $\pi$ to $G(\bbA_f)$. Suppose $\bbH^1(\DR^{\sm})[\lambda]=(\pi_f^v)^{K^v}\ox_C\pi_v ^{\oplus 2}$, and $\tau_v $ the smooth irreducible representation of $D_L^\times$ which is the Jacquet--Langlands transfer of $\pi_v $. (If $\pi_v$ is a principal series then we put $\tau_v=0$.) Define $\text{SS}$ to be the map $\bar d$ on $H^1(\ker d)$, and $\text{ORD}$ to be the map $\bar d$ on $H^0(\coker d)$. For $\lambda\in\sigma_0^{K^v}$ such that $\pi_v$ is generic, we have the following result 
\begin{itemize}
    \item $\ker \mathrm{SS}[\lambda]=(\pi_f^v)^{K^v}\ox_C(\dlim_n H^1_{\dR,c}(\calM_{\Dr,n}^{(0)})\ox_C\tau_v )^{\calO_{D_L}^\times}$. Here $H^i_{\dR,c}(\calM_{\Dr,n}^{(0)})$ is the compactly supported de Rham cohomology of $\calM_{\Dr,n}^{(0)}$.
    \item $\coker \mathrm{SS}[\lambda]=(\pi_f^v)^{K^v}\ox_C(\dlim H^2_{\dR,c}(\calM_{\Dr,n}^{(0)})\ox_C\tau_v )^{\calO_{D_L}^\times}$.
    \item $H^1(\ker d)[\lambda]=(\pi_f^v)^{K^v}\ox_C(\dlim_n H^1_c(\calM_{\Dr,n}^{(0)},\calO_{\calM_{\Dr,n}^{(0)}})\ox_C\tau_v )^{\calO_{D_L}^\times}$. Here $H^1_c(\calM_{\Dr,n}^{(0)},\calO_{\calM_{\Dr,n}^{(0)}})$ is the first compactly supported cohomology group of $\calO_{\calM_{\Dr,n}^{(0)}}$ on $\calM_{\Dr,n}^{(0)}$.
    \item $H^1(\ker d')[\lambda]=(\pi_f^v)^{K^v}\ox_C(\dlim_n H^1_c(\calM_{\Dr,n}^{(0)},\Omega^1_{\calM_{\Dr,n}^{(0)}})\ox_C\tau_v )^{\calO_{D_L}^\times}$.
    \item $\ker\mathrm{ORD}[\lambda]=(\pi_f^v)^{K^v}\ox_C (\Ind^{\GL_2(L)}_{\bar B(L)}H^1_{\rig}(\Ig_{K^v})[\lambda])^\infty$. By Theorem \ref{thm:H0Ig}, we know $H^1_{\rig}(\Ig_{K^v})[\lambda]^{\ss}\isom J_{\bar N(L)}(\pi_v )^{\ss}$ where $J_{\bar N(L)}(-)$ is the Jacquet functor with respect to $\bar B(L)$.
    \item $H^0(\coker d)[\lambda]=(\pi_f^v)^{K^v}\ox_C(\Ind^{\GL_2(L)}_{\bar B(L)}H^1_{\rig}(\Ig_{K^v})[\lambda])^{\iota\-\lan}$.
    \item $H^0(\coker d')[\lambda]=(\pi_f^v)^{K^v}\ox_C(\Ind^{\GL_2(L)}_{\bar B(L)}H^1_{\rig}(\Ig_{K^v})[\lambda]\cdot z_{1,\iota}^{-1}z_{2,\iota})^{\iota\-\lan}$, where $z_{1,\iota}^{-1}z_{2,\iota}:\bar B(L)\to C,\left( \begin{matrix} a&0\\c&d \end{matrix} \right)\mapsto \iota(a)^{-1}\iota(d)$.  
\end{itemize}

We can calculate these de Rham cohomology groups, in terms of some classical objects. 
\begin{theorem}\label{thm:pvc}
Let $\lambda\in \sigma_0^{K^v}$ be a Hecke eigenvalue such that $\pi_v$ is generic. We have a morphism of exact sequences 
\begin{center}
\begin{tikzpicture}[descr/.style={fill=white,inner sep=1.5pt}]
    \matrix (m) [
        matrix of math nodes,
        row sep=2.5em,
        column sep=2.5em,
        text height=1.5ex, 
        text depth=0.25ex
    ]
    { 0 & H^1(\ker d)[\lambda] & \bbH^1(\DR)[\lambda] & H^0(\coker d)[\lambda] & 0  \\
      0 & H^1(\ker d')[\lambda] & \bbH^1(\DR')[\lambda] & \coker \mathrm{SS}[\lambda] & 0 \\
    };
    \path[->,font=\scriptsize]
    (m-1-1) edge (m-1-2)
    (m-1-2) edge (m-1-3)
    (m-1-3) edge (m-1-4)
    (m-1-4) edge (m-1-5)
    (m-2-1) edge (m-2-2)
    (m-2-2) edge (m-2-3)
    (m-2-3) edge (m-2-4)
    (m-2-4) edge (m-2-5)
    ;
    \path[->,font=\scriptsize]
    (m-1-2) edge node[right] {$\SS$} (m-2-2)
    (m-1-3) edge node[right] {$\bar d$} (m-2-3)
    (m-1-4) edge node[right] {$\ORD$} (m-2-4)
    ;
\end{tikzpicture}
\end{center}
which is equivariant for the $\GL_2(L)$-action, $\Gal(\bar L/L)$-action, and the Hecke action. Taking the associated long exact sequence, we get an exact sequence 
\[
    0\to \ker\SS[\lambda]\to \bbH^1(\DR^{\sm})[\lambda]\to \ker \ORD[\lambda]\to \coker \SS[\lambda]\to 0.
\] 
\end{theorem}

The relations between $\ker I^1$ and the cohomologies above are given as follows.
\begin{theorem}\label{thm:kerI1Hodge}
$\ker I^1$ has a spectral decomposition with respect to the Hecke action, and the action is semisimple. Let $\lambda\in \sigma_0^{K^v}$ be a Hecke eigenvalue such that $\pi_v$ is generic. Then we have exact sequences:
\begin{align*}
    &0\to H^0(\ker \bar d')[\lambda]\to \bbH^1(\DR)[\lambda]\to \ker I^1[\lambda]\to 0\\
    &0\to H^1(\ker \bar d)[\lambda]\to \ker I^1[\lambda]\to \bbH^1(\DR')[\lambda]\to 0
\end{align*}
which are equivariant for the $\GL_2(L)$-action, $\Gal(\bar L/L)$-action, and the Hecke action.
\end{theorem}

First we compute the effect of the translation functor $T_{(0,0)}^{(-1,0)}$ on the de Rham cohomology groups.
\begin{proposition}
Applying $T_{(0,0)}^{(-1,0)}$ to $I:\calO\to \calO^s$, we get a commutative diagram with exact rows:
$$
\begin{tikzcd}
0 \arrow[r] & \ker d_{-1} \arrow[d, "\id"'] \arrow[r] & T_{(0,0)}^{(-1,0)}\calO \arrow[r, "d_{-1}"] \arrow[d, "\bar d_{-1}=\id"'] \arrow[rd, "I_{-1}"] & T_{(0,0)}^{(-1,0)}\calO\ox\Omega \arrow[d, "\bar d'_{-1}=\id"] \arrow[r] & \coker d_{-1} \arrow[r] \arrow[d, "\id"] & 0 \\
0 \arrow[r] & \ker d'_{-1} \arrow[r]                  & T_{(0,0)}^{(-1,0)}\calO\ox\bar\Omega \arrow[r, "d'_{-1}"']                                     & T_{(0,0)}^{(-1,0)}\calO^s \arrow[r]                                      & \coker d'_{-1} \arrow[r]                 & 0 
\end{tikzcd}
$$
which are equivariant for the $\GL_2(L)$-action, $\Gal(\bar L/L)$-action, and the Hecke action. Here,  
\begin{itemize}
    \item $T_{(0,0)}^{(-1,0)}\calO=\omega_{K^v}^{(1,0),\iota\-\lan,(-1,0)}(1)$.
    \item $T_{(0,0)}^{(-1,0)}\calO\ox\Omega=\omega_{K^v}^{(0,1),\iota\-\lan,(-1,0)}(1)$
    \item $\ker d_{-1}=j_!(\calA_{\bar G,0}^{K^v}\ox_C\bigoplus_{i\in\bbZ}\omega_{\Dr}^{(-1,0),\sm})^{D_L^\times}$ with $\omega_{\Dr}^{(-1,0),\sm}=\dlim_n\pi_{n,*}\pi_n^*j^*\omega^{(-1,0)}_{\fl}$.
    \item $\coker d_{-1}=i_*(i^*\omega_{\fl}^{(-1,0)}\ox_C \calH^1_{\ord}(K^v,0))$.
\end{itemize}
\begin{proof}
This follows from Proposition \ref{prop:d-1}, \ref{prop:dbar-1}, and note that $T_{(0,0)}^{(-1,0)}$ is an exact functor.
\end{proof}
\end{proposition}

The results on the sheaf cohomology enables us to compute the translation of the de Rham cohomologies.
\begin{corollary}\label{cor:transtowall}
There are natural isomorphisms 
\[
    T_{(0,0)}^{(-1,0)}\ker I^1=T_{(0,0)}^{(-1,0)}\bbH^1(\DR)=T_{(0,0)}^{(-1,0)}\bbH^1(\DR')
\]
and we have an exact sequence 
\[
    0\to T_{(0,0)}^{(-1,0)}H^1(\ker d)\to T_{(0,0)}^{(-1,0)}\ker I^1\to T_{(0,0)}^{(-1,0)}H^0(\coker d)\to 0,
\]
which is equivariant for the $\GL_2(L)$-action, $\Gal(\bar L/L)$-action, and the Hecke action. Moreover, for $\lambda\in \sigma_0^{K^v}$ such that $\pi_v $ is generic, we have 
\begin{itemize}
    \item $T_{(0,0)}^{(-1,0)}H^1(\ker d)[\lambda]=(\pi_f^v)^{K^v}\ox_C(H^1(j_!\omega_{\Dr}^{(-1,0),\sm})\ox_C\tau_v )^{\calO_{D_L}^\times}$.
    \item $T_{(0,0)}^{(-1,0)}H^0(\coker d)[\lambda]=(\pi_f^v)^{K^v}\ox_C(\Ind^{\GL_2(L)}_{\bar B(L)}H^1_{\rig}(\Ig_{K^v})[\lambda]\cdot z_{1,\iota}^{-1})^{\iota\-\lan}$.
\end{itemize}
\begin{proof}
By Proposition \ref{prop:dbar-1}, $\bar d_{-1}$ is the identity map. Hence the cohomology of $I_{-1}$ is the same as the cohomology of $d_{-1}$. From this one deduces the desired results.
\end{proof}
\end{corollary}

\begin{remark}
From Corollary \ref{cor:transtowall} we see $T_{(0,0)}^{(-1,0)}\ker I^1[\lambda]$ only depends on the Weil--Deligne representation associated to $\pi_v$ via the local Langlands correspondence, which is a special case (locally $\sigma$-analytic part) of \cite[Conjecture 3.12]{Din24} in the $\GL_2(L)$-case.
\end{remark}

Now, we translate back from singular weight to generic weight, and see the relation between $\ker I^1$ and its wall-crossing. We first fix some notations. Let the subscript $(-)_{-1,1}$ denote the functor $T_{(-1,0)}^{(0,0)}T_{(0,0)}^{(-1,0)}(-)$. For example, $\DR_{-1,1}:=T_{(-1,0)}^{(0,0)}T_{(0,0)}^{(-1,0)}\DR$ is the two term complex 
\[
    \DR_{-1,1}= T_{(-1,0)}^{(0,0)}T_{(0,0)}^{(-1,0)}\calO\ov{d_{-1,1}}\to T_{(-1,0)}^{(0,0)}T_{(0,0)}^{(-1,0)}(\calO\ox\Omega).
\]
We first calculate the hypercohomology of the de Rham complex $\DR_{-1,1}$.
\begin{theorem}\label{thm:WCdeRham}
Let $\lambda\in \sigma_{0}^{K^v}$ such that $\pi_v $ is generic. There is a commutative diagram with exact rows and columns
\begin{center}
\begin{tikzpicture}[descr/.style={fill=white,inner sep=1.5pt}]
    \matrix (m) [
        matrix of math nodes,
        row sep=2.5em,
        column sep=2.5em,
        text height=1.5ex, 
        text depth=0.25ex
    ]
    {  & 0 & 0 & 0 &   \\
     0 & H^1(\ker d)[\lambda] & \bbH^1(\DR)[\lambda] & H^0(\coker d)[\lambda] & 0  \\
      0 & H^1(\ker d_{-1,1})[\lambda] & \bbH^1(\DR_{-1,1})[\lambda] & H^0(\coker d_{-1,1})[\lambda] & 0 \\
      0 & H^1(\ker d')[\lambda] & \bbH^1(\DR')[\lambda] & H^0(\coker d')[\lambda] & 0 \\
        & 0 & 0 & 0 &  \\
    };
    \path[->,font=\scriptsize]
    (m-2-1) edge   (m-2-2)
    (m-2-2) edge node[auto] {} (m-2-3)
    (m-2-3) edge node[auto] {} (m-2-4)
    (m-2-4) edge   (m-2-5)    
    (m-3-1) edge   (m-3-2)
    (m-3-2) edge node[auto] {} (m-3-3)
    (m-3-3) edge node[auto] {} (m-3-4)
    (m-3-4) edge   (m-3-5)    
    (m-4-1) edge   (m-4-2)
    (m-4-2) edge node[auto] {} (m-4-3)
    (m-4-3) edge node[auto] {} (m-4-4)
    (m-4-4) edge   (m-4-5)    
    (m-2-2) edge   (m-1-2)
    (m-3-2) edge node[auto] {} (m-2-2)
    (m-4-2) edge node[auto] {} (m-3-2)
    (m-5-2) edge   (m-4-2)
    (m-2-3) edge   (m-1-3)
    (m-3-3) edge node[auto] {} (m-2-3)
    (m-4-3) edge node[auto] {} (m-3-3)
    (m-5-3) edge   (m-4-3)
    (m-2-4) edge   (m-1-4)
    (m-3-4) edge node[auto] {} (m-2-4)
    (m-4-4) edge node[auto] {} (m-3-4)
    (m-5-4) edge   (m-4-4)
    ;
\end{tikzpicture}
\end{center}
which is equivariant for the $\GL_2(L)$-action, $\Gal(\bar L/L)$-action, and the Hecke action. In particular, the natural map $\bbH^1(\DR_{-1,1})[\lambda]\to \bbH^1(\DR)[\lambda]$ is surjective.
\begin{proof}
Let $\bbH^1(\DR_{-1,1})$ denote the hypercohomology of $\DR_{-1,1}$, which is also the wall-crossing of $\bbH^1(\DR)$. As the Hecke action on $\bbH^1(\DR)$ is semisimple, we see that the Hecke action on $\bbH^1(\DR_{-1,1})$ is also semisimple.

By our results on the wall-crossing of $\calO=\calO_{K^v}^{\iota\-\lan,(0,0)}$, see Proposition \ref{prop:wallcrossingofOsmoxOfl}, there is a short exact sequence 
\[
    0\to \DR'\to \DR_{-1,1}\to \DR\to 0.
\]
Taking hypercohomology, we get an exact sequence 
\[
    \bbH^1(\DR')\to \bbH^1(\DR_{-1,1})\to \bbH^1(\DR)\to 0,
\]
as $\bbH^2(\DR')=0$ by Proposition \ref{prop:DRcohprime}. Besides, for $\lambda\in \sigma_0^{K^v}$ such that $\pi_v $ is generic, we have $\bbH^0(\DR)[\lambda]=0$. Therefore, there is an exact sequence 
\[
    0\to \bbH^1(\DR')[\lambda]\to \bbH^1(\DR_{-1,1})[\lambda]\to \bbH^1(\DR)[\lambda]\to 0.
\]
Finally, it is easy to see that the wall-crossing functor is compatible with the filtration $H^1(\ker d)[\lambda]\subset \bbH^1(\DR)[\lambda]$.
\end{proof}
\end{theorem}

\begin{remark}
Let $\lambda\in \sigma_0^{K^v}$ such that $\pi_v $ is generic. Let $\bbH^1(\DR_{-1,1})[\lambda]_{\tilde{\chi}_{(0,0)}}$ be the infinitesimal co-semisimple part of $\bbH^1(\DR_{-1,1})[\lambda]$. Then the natural map given by the wall-crossing induces an isomorphism 
\[
    \bbH^1(\DR_{-1,1})[\lambda]_{\tilde{\chi}_{(0,0)}}\to \bbH^1(\DR)[\lambda].
\]
Similarly, let $\bbH^1(\DR_{-1,1})[\lambda]^{\tilde{\chi}_{(0,0)}}$ denote the infinitesimal semisimple part, which is an extension of $\bbH^1(\DR^{\sm})[\lambda]$ by $\bbH^1(\DR')[\lambda]$. The natural map $\bbH^1(\DR)[\lambda]\to T_{(-1,0)}^{(0,0)}T_{(0,0)}^{(-1,0)}\bbH^1(\DR)[\lambda]$ is given by the composition of $\bar d:\bbH^1(\DR)[\lambda]\to \bbH^1(\DR')[\lambda]$ and the natural inclusion $\bbH^1(\DR')[\lambda]\subset \bbH^1(\DR_{-1,1})[\lambda]$. To see this, we may consider the morphism of exact sequences 
\begin{center}
\begin{tikzpicture}[descr/.style={fill=white,inner sep=1.5pt}]
    \matrix (m) [
        matrix of math nodes,
        row sep=2.5em,
        column sep=2.5em,
        text height=1.5ex, 
        text depth=0.25ex
    ]
    { 0 & \bbH^1(\DR')[\lambda] & \bbH^1(\DR_{-1,1})[\lambda] & \bbH^1(\DR)[\lambda] & 0  \\
    0 & \bbH^1(\DR')[\lambda] & \bbH^1(\DR_{-1,1})[\lambda] & \bbH^1(\DR)[\lambda] & 0  \\
    };
    \path[->,font=\scriptsize]
    (m-1-1) edge (m-1-2)
    (m-1-2) edge (m-1-3)
    (m-1-3) edge (m-1-4)
    (m-1-4) edge (m-1-5)
    (m-2-1) edge (m-2-2)
    (m-2-2) edge (m-2-3)
    (m-2-3) edge (m-2-4)
    (m-2-4) edge (m-2-5)
    ;
    \path[->,font=\scriptsize]
    (m-1-2) edge node[right] {$\Omega $} (m-2-2)
    (m-1-3) edge node[right] {$\Omega $} (m-2-3)
    (m-1-4) edge node[right] {$\Omega $} (m-2-4)
    ;
\end{tikzpicture}
\end{center}
where $\Omega \in\calZ(\frg_{\iota})$ is the Casimir operator. As $\bbH^1(\DR')[\lambda]$ and $\bbH^1(\DR)[\lambda]$ is killed by $\Omega $, this morphism of exact sequence induces a connection map $\bbH^1(\DR)[\lambda]\to \bbH^1(\DR')[\lambda]$. By Proposition \ref{prop:connectionmap1}, this connection map is given by $\bar d$. Therefore, $\bbH^1(\DR_{-1,1})[\Omega =0]$ is an extension of $\bbH^1(\DR^{\sm})[\lambda]$ by $\bbH^1(\DR')[\lambda]$, and $\bbH^1(\DR_{-1,1})[\lambda]/\Omega \bbH^1(\DR_{-1,1})[\lambda]\isom \bbH^1(\DR)[\lambda]$.
\end{remark}

\subsubsection*{Example: $\pi_v $ is a principal series}
Let $\lambda\in \sigma_0^{K^v}$ such that $\pi_v $ is generic. Assume that $\pi_v $ is a principal series. Recall:
\begin{itemize}
    \item $\bbH^1(\DR)[\lambda]=H^0(\coker d)[\lambda]=(\pi_f^v)^{K^v}\ox_C(\Ind^{\GL_2(L)}_{\bar B(L)}H^1_{\rig}(\Ig_{K^v})[\lambda])^{\iota\-\lan}$.
    \item $\bbH^1(\DR')[\lambda]=H^0(\coker d')[\lambda]=(\pi_f^v)^{K^v}\ox_C(\Ind^{\GL_2(L)}_{B(L)}H^1_{\rig}(\Ig_{K^v})[\lambda]z_{1,\iota}^{-1}z_{2,\iota})^{\iota\-\lan}$.
\end{itemize}
We study the wall-crossing of $\ker I^1[\lambda]$ in this case using Corollary \ref{cor:transtowall} and Theorem \ref{thm:WCdeRham}, which can also be directly computed, see \cite[Theorem 4.1.12]{JLS22}. Applying the translation functor $T_{(0,0)}^{(-1,0)}(-)$ to $\bbH^1(\DR)[\lambda]$, we get 
\begin{itemize}
    \item $\bbH^1(\DR_{-1})[\lambda]=(\pi_f^v)^{K^v}\ox_C(\Ind^{\GL_2(L)}_{\bar B(L)}H^1_{\rig}(\Ig_{K^v})[\lambda] z_{1,\iota}^{-1})^{\iota\-\lan}$.
\end{itemize}
Translation back using $T_{(-1,0)}^{(0,0)}(-)$, we get
\begin{align*}
    0\to (\pi_f^v)^{K^v}\ox_C (\Ind^{\GL_2(L)}_{\bar B(L)}H^1_{\rig}(\Ig_{K^v})[\lambda] z_{1,\iota}^{-1}z_{2,\iota})^{\iota\-\lan}\to \bbH^1(\DR_{-1,1})[\lambda]\to (\pi_f^v)^{K^v}\ox_C (\Ind^{\GL_2(L)}_{\bar B(L)}H^1_{\rig}(\Ig_{K^v})[\lambda])^{\iota\-\lan}\to 0.
\end{align*}

\subsubsection*{Example: $\pi_v $ is a special series}
After twisting by a smooth character, we assume that $\pi_v =\St_2^\infty$, the smooth Steinberg representation of $\GL_2(L)$. This representation is the cosocle of $(\Ind^{\GL_2(L)}_{B(L)}1)^\infty$.

We recall the computation of $\bbH^1(\DR)[\lambda]$. Set $I=(\Ind^{\GL_2(L)}_{B(L)}1)^{\iota\-\an}$, with cosocle $I_s=(\Ind^{\GL_2(L)}_{B(L)}z_{1,\iota}z_{2,\iota}^{-1})^{\iota\-\an}$. Let $\St_2^{\iota\-\an}:=I/1$ be the locally $\iota$-analytic Steinberg representation over $C$. We also define $\tilde{I}=(\Ind^{\GL_2(L)}_{B(L)}|\cdot|\ox |\cdot|^{-1})^{\iota\-\an}$, with cosocle $\tilde{I}_s=(\Ind^{\GL_2(L)}_{B(L)}z_{1,\iota}z_{2,\iota}^{-1}|\cdot|\ox |\cdot|^{-1})^{\iota\-\an}$. Using the theory of Orlik-Strauch representations \cite{OS15}, we have exact sequences 
\begin{align*}
    &0\to (\Ind^{\GL_2(L)}_{B(L)}1)^\infty\to I\to I_s\to 0,\\
    &0\to (\Ind^{\GL_2(L)}_{B(L)}|\cdot|\ox |\cdot|^{-1})^\infty\to \tilde{I}\to \tilde{I}_s\to 0
\end{align*}
of $\GL_2(L)$-representations. Then
\begin{center}
    \begin{tikzpicture}[descr/.style={fill=white,inner sep=1.5pt}]
        \matrix (m) [
            matrix of math nodes,
            row sep=2.5em,
            column sep=2.5em,
            text height=1.5ex, 
            text depth=0.25ex
        ]
        { 0 & H^1(\ker d)[\lambda] & \bbH^1(\DR)[\lambda] & H^0(\coker d)[\lambda] & 0  \\
          0 & H^1(\ker d')[\lambda] & \bbH^1(\DR')[\lambda] & \coker \mathrm{SS}[\lambda] & 0 \\
        };
        \path[->,font=\scriptsize]
        (m-1-1) edge (m-1-2)
        (m-1-2) edge (m-1-3)
        (m-1-3) edge (m-1-4)
        (m-1-4) edge (m-1-5)
        (m-2-1) edge (m-2-2)
        (m-2-2) edge (m-2-3)
        (m-2-3) edge (m-2-4)
        (m-2-4) edge (m-2-5)
        ;
        \path[->,font=\scriptsize]
        (m-1-2) edge node[right] {$\SS$} (m-2-2)
        (m-1-3) edge node[right] {$\bar d$} (m-2-3)
        (m-1-4) edge node[right] {$\ORD$} (m-2-4)
        ;
    \end{tikzpicture}
\end{center}
is given by tensoring $(\pi_f^v)^{K^v}$ with
\begin{center}
    \begin{tikzpicture}[descr/.style={fill=white,inner sep=1.5pt}]
        \matrix (m) [
            matrix of math nodes,
            row sep=2.5em,
            column sep=2.5em,
            text height=1.5ex, 
            text depth=0.25ex
        ]
        { 0 & \St_2^{\iota\-\an} & (\St_2^{\infty})^{\oplus 2}-I_s-1-\tilde{I}_s & \tilde{I} & 0  \\
          0 & I_s-1 & I_s-1-\tilde I_s & \coker \tilde I_s & 0 \\
        };
        \path[->,font=\scriptsize]
        (m-1-1) edge (m-1-2)
        (m-1-2) edge (m-1-3)
        (m-1-3) edge (m-1-4)
        (m-1-4) edge (m-1-5)
        (m-2-1) edge (m-2-2)
        (m-2-2) edge (m-2-3)
        (m-2-3) edge (m-2-4)
        (m-2-4) edge (m-2-5)
        ;
        \path[->,font=\scriptsize]
        (m-1-2) edge node[right] {$\SS$} (m-2-2)
        (m-1-3) edge node[right] {$\bar d$} (m-2-3)
        (m-1-4) edge node[right] {$\ORD$} (m-2-4)
        ;
    \end{tikzpicture}
\end{center}
and
\[
    0\to \ker\SS[\lambda]\to \bbH^1(\DR^{\sm})[\lambda]\to \ker \ORD[\lambda]\to \coker \SS[\lambda]\to 0
\] 
is given by tensoring $(\pi_f^v)^{K^v}$ with
\begin{align*}
    0\to \St_2^\infty\to (\St_2^\infty)^{\oplus 2}\to (\Ind^{\GL_2(L)}_{B(L)}|\cdot|\ox |\cdot|^{-1})\to 1\to 0.
\end{align*}

We determine the wall-crossing of $\bbH^1(\DR)[\lambda]$.
\begin{proposition}
\begin{enumerate}[(i)]
    \item $T_{(0,0)}^{(-1,0)}I_s=I_{-1}:=(\Ind^{\GL_2(L)}_{B(L)}z_{2,\iota}^{-1})^{\iota\-\an}$, and $T_{(0,0)}^{(-1,0)}\tilde{I}_s=\tilde{I}_{-1}:=(\Ind^{\GL_2(L)}_{B(L)}z_{2,\iota}^{-1}|\cdot|\ox |\cdot|^{-1})^{\iota\-\an}$.
    \item the exact sequence 
    \[
        0\to T_{(0,0)}^{(-1,0)}H^1(\ker d)[\lambda]\to T_{(0,0)}^{(-1,0)}\bbH^1(\DR_{-1})[\lambda]\to T_{(0,0)}^{(-1,0)}H^0(\coker d_{-1})[\lambda]\to 0
    \]
    is given by tensoring $(\pi_f^v)^{K^v}$ with
    \[
        0\to I_{-1}\to I_{-1}-\tilde{I}_{-1}\to \tilde{I}_{-1}\to 0,
    \]
    where $I_{-1}-\tilde{I}_{-1}$ is the unique non-split extension of $\tilde{I}_{-1}$ by $I_{-1}$.
    \item The wall-crossing $\bbH^1(\DR_{-1,1})[\lambda]=T_{(-1,0)}^{(0,0)}T_{(0,0)}^{(-1,0)}\bbH^1(\DR')[\lambda]$is given by tensoring $(\pi_f^v)^{K^v}$ with
    $$
    \begin{tikzcd}
        &                                                                       & \tilde{I}_s \arrow[r, no head]                                  & \St_2^\infty \arrow[r, no head] & 1 \arrow[r, no head] & \tilde{I}_s \\
    I_s \arrow[r, no head] & 1 \arrow[r, no head] \arrow[ru, no head]  & \St_2^\infty \arrow[r, no head]  & I_s \arrow[ru, no head]           &                      &          
    \end{tikzcd}
    $$
\end{enumerate}
\begin{proof}
As $I_s=\calF^{\GL_2(L)}_{B(L)}(M(-1,1),1)$ and $\tilde I_s=\calF^{\GL_2(L)}_{ B(L)}(M(-1,1),|\cdot|\ox |\cdot|^{-1})$, we may calculate the translation of $I,\tilde{I}$ using the compatibility between the Orlik-Strauch induction functor and the translation functor, see \cite[Theorem 4.1.12]{JLS22}. In particular, we have 
\begin{align*}
    I_{-1}&=\calF^{\GL_2(L)}_{B(L)}(M(-1,0),1)=(\Ind^{\GL_2(L)}_{B(L)}z_{2,\iota}^{-1})^{\iota\-\an},\\
    \tilde I_{-1}&=\calF^{\GL_2(L)}_{B(L)}(M(-1,0),|\cdot|\ox |\cdot|^{-1})=(\Ind^{\GL_2(L)}_{B(L)}z_{2,\iota}^{-1}|\cdot|\ox |\cdot|^{-1})^{\iota\-\an}.
\end{align*}
We may also compute these translations using the translation of the differential operators. Using \cite[Proposition 8.15, 8.18]{Koh11Coh} (and a little devissage on results on translation functors), the extensions between them are given by 
\begin{align*}
    \dim\Ext^1_{\PGL_2(L),\iota}({I}_{-1},I_{-1}) =2,\qquad &\dim\Ext^1_{\PGL_2(L),\iota}({I}_{-1},\tilde I_{-1}) =0,\\
    \dim\Ext^1_{\PGL_2(L),\iota}(\tilde{I}_{-1},I_{-1}) =1,\qquad &\dim\Ext^1_{\PGL_2(L),\iota}(\tilde{I}_{-1},\tilde I_{-1}) =2.
\end{align*}
It is easy to see that $\bbH^1(\DR_{-1,1})[\lambda]$ fits into a diagram with exact rows and columns:

\begin{center}

\begin{tikzpicture}[descr/.style={fill=white,inner sep=1.5pt}]
    \matrix (m) [
        matrix of math nodes,
        row sep=2.5em,
        column sep=2.5em,
        text height=1.5ex, 
        text depth=0.25ex
    ]
    {  & 0 & 0 & 0 &   \\
     0 & \St_2^\infty-I_s & (\St_2^\infty)^{\oplus 2}-I_s-1-\tilde{I}_s & \St_2^\infty-1-\tilde{I}_s & 0  \\
      0 & I_s-(1-\St_2^\infty)-I_s & \bbH^1(\DR_{-1,1})[\lambda] & \tilde{I}_s-(\St_2^\infty-1)-\tilde{I}_s & 0 \\
      0 & I_s-1 & I_s-1-\tilde{I}_s & \tilde{I}_s & 0 \\
        & 0 & 0 & 0 &  \\
    };
    \path[->,font=\scriptsize]
    (m-2-1) edge   (m-2-2)
    (m-2-2) edge node[auto] {} (m-2-3)
    (m-2-3) edge node[auto] {} (m-2-4)
    (m-2-4) edge   (m-2-5)    
    (m-3-1) edge   (m-3-2)
    (m-3-2) edge node[auto] {} (m-3-3)
    (m-3-3) edge node[auto] {} (m-3-4)
    (m-3-4) edge   (m-3-5)    
    (m-4-1) edge   (m-4-2)
    (m-4-2) edge node[auto] {} (m-4-3)
    (m-4-3) edge node[auto] {} (m-4-4)
    (m-4-4) edge   (m-4-5)    
    (m-2-2) edge   (m-1-2)
    (m-3-2) edge node[auto] {} (m-2-2)
    (m-4-2) edge node[auto] {} (m-3-2)
    (m-5-2) edge   (m-4-2)
    (m-2-3) edge   (m-1-3)
    (m-3-3) edge node[auto] {} (m-2-3)
    (m-4-3) edge node[auto] {} (m-3-3)
    (m-5-3) edge   (m-4-3)
    (m-2-4) edge   (m-1-4)
    (m-3-4) edge node[auto] {} (m-2-4)
    (m-4-4) edge node[auto] {} (m-3-4)
    (m-5-4) edge   (m-4-4)
    ;
\end{tikzpicture}.
\end{center}
The first column is given by the wall-crossing of $I_s$, and the third column is given by the wall-crossing of $\tilde{I}_s$. From this, it is easy to see that the extension $T_{(0,0)}^{(-1,0)}\bbH^1(\DR')[\lambda]=I_{-1}-\tilde{I}_{-1}$ is non-split. As $\Ext^1_{\PGL_2(L),\iota}(\tilde{I}_{-1},I_{-1})=1$, this is the unique non-split extension of $\tilde{I}_{-1}$ by $I_{-1}$.

Finally we determine the structure of $\bbH^1(\DR_{-1,1})[\lambda]$. We first analyze the structure of $\bbH^1(\DR)[\lambda]$ in more detail. We already know that up to multiplicities it is a non-split extension $[(\St_2^\infty-I_s)-(\St_2^\infty-1-\tilde{I}_s)]$, given by a non-zero element in the $1$-dimensional space $\Ext^1_{\PGL_2(L),\iota}(\St_2^\infty-1-\tilde{I}_s,\St_2^\infty-I_s)$ (unique up to scalar). It is a quotient of a $2$-dimensional vector space $\Ext^1_{\PGL_2(L),\iota}(1-\tilde{I}_s,\St_2^\infty-I_s)\aisom \Ext^1_{\PGL_2(L),\iota}(1,\St_2^\infty-I_s)$ by a $1$-dimensional sub $\Hom_{\GL_2(L)}(\St_2^\infty,\St_2^\infty-I_s)$. (Note that $\Hom_{\GL_2(L)}(\St_2^\infty-1-\tilde{I}_s,\St_2^\infty-I_s)=0$.) Actually the line given by $\Hom_{\GL_2(L)}(\St_2^\infty,\St_2^\infty-I_s)$ is exactly the same as the line $\Ext^1_{\PGL_2(L),\iota}(1,\St_2^\infty)$ in $\Ext^1_{\PGL_2(L),\iota}(1,\St_2^\infty-I_s)$. Hence we see that there is no extra extension between $1$ in $\St_2^\infty-1-\tilde{I}_s$ and $\St_2^\infty$ in $\St_2^\infty-I_s$. Besides, as $\Ext^1_{\PGL_2(L),\iota}(\tilde{I}_s,I_s)=0$ and $\Ext^1_{\PGL_2(L),\iota}(\tilde{I}_s,\St_2^\infty)=0$ by \cite[Proposition 8.15, 8.18]{Koh11Coh}, from the above diagram it is easy to determine the structure of the extension of $\bbH^1(\DR)[\lambda]$ by $\bbH^1(\DR')[\lambda]$ inside $\bbH^1(\DR_{-1,1})[\lambda]$.   
\end{proof}
\end{proposition}

\begin{remark}
Although $\St_2^{\iota\-\an}-1\subset \ker I^1[\lambda]$ already carries Breuil's $\calL$-invariant \cite[Lemme 3.1.2]{Bre19}, one cannot recover the $\calL$-invariant from the natural map for wall-crossing of $\St_2^{\iota\-\an}-1$. Only by considering the wall-crossing of $\St_2^{\iota\-\an}-1-\tilde{I}_s$ can we recover the $\calL$-invariant inside $\St_2^{\iota\-\an}-1$.
\end{remark}

\subsubsection*{Example: $\pi_v $ is supercuspidal}
We assume that $\pi_v $ is supercuspidal. In this case, $J_{\bar N(L)}\pi_v =0$, and $\tau_v $ is not a character of $D_L^\times$. Hence 
\begin{itemize}
    \item $\bbH^1(\DR)[\lambda]=H^1(\ker d)[\lambda]=(\pi_f^v)^{K^v}\ox_C(\bigoplus_{i\in\bbZ} H^1(j_!\calO_{\Dr}^{\sm})\ox_C\tau_v )^{D_L^\times}=:\pi$.
    \item $\bbH^1(\DR')[\lambda]=H^1(\ker d')[\lambda]=(\pi_f^v)^{K^v}\ox_C(\bigoplus_{i\in\bbZ} H^1(j_!\Omega_{\Dr}^{1,\sm})\ox_C\tau_v )^{D_L^\times}=:\pi_c$.
    \item $\ker \SS[\lambda]=(\pi_f^v)^{K^v}\ox_C(\bigoplus_{i\in\bbZ} H^1_{\dR,c}(\calM_{\Dr,\infty})\ox_C\tau_v )^{D_L^\times}=(\pi_f^v)^{K^v}\ox_C\bbH^1(\DR^{\sm})[\lambda]$.
    \item $\coker\SS[\lambda]=0$.
\end{itemize}
Therefore, we have an exact sequence
\[
    0\to (\pi_f^v)^{K^v}\ox_C \pi_v \to \ker I^1[\lambda]\to (\pi_f^v)^{K^v}\ox_C \pi_c  \to 0.
\]
Using Corollary \ref{cor:transtowall} and Theorem \ref{thm:WCdeRham}, we can compute the wall-crossing of $\bbH^1(\DR)[\lambda]$.
\begin{proposition}
The wall-crossing of $\bbH^1(\DR)[\lambda]$ is isomorphic to the wall-crossing of $(\pi_f^v)^{K^v}\ox_C\pi_c$, and it fits into an exact sequence
\begin{align*}
    0\to (\pi_f^v)^{K^v}\ox_C \pi_c\to (\pi_f^v)^{K^v}\ox_C \Theta_s\pi_c\to (\pi_f^v)^{K^v}\ox_C \tilde \pi\to 0.
\end{align*}
\end{proposition}

\section{$p$-adic Hodge-theoretic interpretation of the intertwining operator}\label{padic}

In this section, we establish a $p$-adic Hodge-theoretic meaning of the intertwining operators $I_k$ defined in Section \ref{int}. Using this, we discuss some relations between de Rhamness at $p$ and classicality of a Galois representation.

\subsection{The Fontaine operator}\label{Fon}
In this subsection we recall the definition of the Fontaine operator and its relation with de Rham representations. The main reference is Section 6.1 of \cite{PanII}.

Let $K$ be a finite extension of $\Q_p$. For a Banach $C$-module $W$ with a continuous semilinear action of $G_K$, let $W^K\subset W$ be the subspace of $G_{K_\infty}$-fixed, $G_K$-analytic vectors in $W$. This is a $K$-Banach space and there is a natural map
$$\varphi^K_W:C\widehat{\otimes}_KW^K\ra W.$$

Now let $W$ be a flat Banach $B_{\text{dR}}^+/(t^{k+1})$-module equipped with a continuous semilinear action of $G_K$, and assume that $W$ is Hodge--Tate of weights  $0,k$, meaning that $W/tW$ is Hodge--Tate of weights $0,k$. Equip $W$ with the natural $t$-adic filtration and let $W_i=\text{gr}^iW=t^iW/t^{i+1}W$ ($0\leq i\leq k)$. As $W$ is flat over $B_{\text{dR}}^+/(t^{k+1})$, the subspace $t^iW\subset W$ is a closed subspace because it can be identified with the kernel of $W\xrightarrow{\times t^{k-i}}W$. Hence there is a natural isomorphism  $W_i\cong W_0(i)$. Let $W_0=W_{0,0}\oplus W_{0,-k}$ be the Hodge--Tate decomposition with the Sen operator acting by $0$ and $-k$ respectively. Then $W_i$ is also Hodge--Tate with $W_i=W_{i,0}\oplus W_{i,-k}$ such that $W_{i,j}\cong W_{0,j}(i)$ for $j\in \{0,-k\}$. By \cite[Lemma 6.1.11]{PanII} we get $W_0=C\widehat{\otimes}_KW_0^K$ and $\text{gr}^iW^K=W_i^K$. 

There is a natural action of $\text{Lie}(\Gal(K_\infty/K))$ on $W^K$. We denote by $\nabla\in \text{End}_K(W^K)$ the bounded operator defined by the action of $1\in \Z_p\cong \text{Lie}(\Gal(K_\infty/K))$. Let $E_0(W^K)$ be the generalized eigenspace associated with the eigenvalue $0$. Then there is an exact sequence
$$0\ra W_{k,-k}^K\ra E_0(W^K)\ra W^K_{0,0}\ra 0.$$
The action of $\nabla$ induces a $K$-linear map
$$N_W^K:W_{0,0}^K\ra W_{k,-k}^K\cong W_{0,-k}^K(k),$$
which can be extended to a $C$-linear map
$$N_W:W_{0,0}\ra W_{k,-k}\cong W_{0,-k}(k).$$
The map $N_W$ is called the Fontaine operator associated with $W$. We have the following result:
\begin{thm}\label{N=0}
	\textup{(\cite[Theorem 6.1.16]{PanII})} Let $V$ be a two-dimensional continuous representation of $G_{L}$ over $E$, a finite extension of $L$ containing all embeddings of $L$ into $\overline{\Q}_p$. Let $\sigma:L\hookrightarrow E$ be an embedding and  suppose that $V$ is $\sigma$-Hodge--Tate with $\sigma$-Hodge--Tate weights $0,k>0$. Let $
    W=V\otimes_{L,\sigma}B^+_{\text{dR}}/(t^{k+1})$. Then $N_W=0$ if and only if $V$ is a $\sigma$-de Rham representation.
\end{thm}
\begin{proof}
    Recall that $V$ is $\sigma$-de Rham if and only if $\dim_E D_{\dR,\sigma}(V)=\dim_E V=2$. As $V$ is of $\sigma$-Hodge--Tate weights $0,k$,  $(V\ox_{L,\sigma} C(l))^{\mathrm{Gal}_L}=0$ when $l\ge k+1$. Thus the natural map $(V\ox_{L,\sigma}B_{\dR}^+)^{\mathrm{Gal}_L}\to V\ox_{L,\sigma}B_{\dR}^+/(t^{k+1})=:W$ is injective. Hence it induces an inclusion 
\[
    D_{\dR,\sigma}(V)\subset E_0(W^{L}).
\]
Note that $E_0(W^L)$ is a $2$-dimensional $E$-vector space. If $\dim_ED_{\dR,\sigma}(V)=2$, then this inclusion becomes an isomorphism, so that $N_{W}=0$. 

Conversely, if $N_{W}=0$, then $E_0(W^L)$ is fixed by $\mathrm{Gal}_L$. Since $H^1(\mathrm{Gal}_L, V\ox_{L,\sigma} C(l))=0$ when $l\geq k+1$, the natural map $V \otimes_{L,\sigma}(B^+_{\text{dR}}/(t^{l}))^{\mathrm{Gal}_L}\ra W^{\mathrm{Gal}_L}$ is an isomorphism when $l\geq k+1$. Passing to the limit over $l$ we get that $\dim_E D_{\dR,\sigma}(V)=2$.
\end{proof}

Now let $\mathbb{B}_{\text{dR}}^+=\pi_{\text{HT}*}\B^+_{\text{dR},\mathcal{X}_{K^v}}$ and $\mathbb{B}_{\text{dR},k}^+=\mathbb{B}_{\text{dR}}^+/(t^k)$. Let $U\in \mathfrak{B}$ be an affinoid open subset of $\mathscr{F}\ell$ with $V_\infty=\pi_{\text{HT}}^{-1}(U)=\text{Spa}(B,B^+)$. Let $V_0={V_{K_p}}$ for some sufficiently small open subgroup $K_p\subset \GL_2(L)$. Then $\B^+_{\text{dR}}/(t^k)(U)$ is naturally a flat $B^+_\text{dR}/(t^k)$-Banach module with the unit open ball given by the image of the natural $\mathbb{A}_{\text{inf},\X_{K^v}}(V_\infty)\ra \B^+_{\text{dR},\X_{K^v}}(V_\infty)/(t^k)$. It is clear from the construction that $\text{GL}_2(L)$ acts on $\B^+_{\text{dR},k}$. We denote by 
$$\B^{+,\text{la}}_{\text{dR},k}\subset\B^+_{\text{dR},k}$$
the subsheaf of locally analytic vectors. The Lie algebra $\text{Lie}(\text{GL}_2(L))=\mathfrak{gl}_2(L)$ acts naturally on $\B^{+,\text{la}}_{\text{dR},k}$. Let $Z=Z(U(\mathfrak{gl}_2(L)))$ be the center of the universal enveloping algebra of $\mathfrak{gl}_2(L)$ over $L$. Given an infinitesimal character $\tilde{\chi}:Z\ra L\subset B^+_{\text{dR}}$, we denote by
$$\B^{+,\text{la},\tilde{\chi}}_{\text{dR},k}\subset\B^{+,\text{la}}_{\text{dR},k}$$
the $\tilde{\chi}$-isotypic part. Then there is an induced decreasing filtration on $\B^{+,\text{la},\tilde{\chi}}_{\text{dR},k}$.

\begin{lem}\label{chicomH1}
	\textup{(\cite[Lemma 6.2.2]{PanII})} For $i=0,\dots,k-1$, the natural maps
	
	\textup{(1)} $\textup{gr}^i\B^{+,\textup{la}}_{\textup{dR},k}\ra \mathcal{O}^{\textup{la}}_{K^v}(i)$,
	
	\textup{(2)} $\textup{gr}^i\B^{+,\textup{la},\tilde{\chi}}_{\textup{dR},k}\ra \mathcal{O}^{\textup{la},\tilde{\chi}}_{K^v}(i)$
    
	are isomorphisms. In particular, $\B^{+,\textup{la}}_{\textup{dR},k}$ and $\B^{+,\textup{la},\tilde{\chi}}_{\textup{dR},k}$ are flat over $B^+_\textup{dR}/(t^k)$.
\end{lem}
\begin{proof}
    Both are local statements on $\mathscr{F}\ell$, so we can fix $U\in \mathfrak{B}$. The first part follows from the fact that $\B^{+}_{\textup{dR},k}$ is filtered by $\mathcal{O}_{K^v}(U)(i)$ and $\mathcal{O}_{K^v}(U)$ is la-acyclic.

    For the second part, it suffices to show that
    $$\text{Ext}^i_Z(\tilde{\chi},\mathcal{O}^{\textup{la},\tilde{\chi}}_{K^v}(U))=0\text{ for } i\geq 1.$$
    Recall that the action of $Z$ on $\mathcal{O}^{\textup{la},\tilde{\chi}}_{K^v}(U)$ factors through $S(\mathfrak{h})$, where $\mathfrak{h}=\begin{pmatrix}
        *&0\\0&*
    \end{pmatrix}\subset \mathfrak{gl}_2(C)$ acts via the horizontal action $\theta_\mathfrak{h}$. Thus after choosing a character $\chi:S(\mathfrak{h})\ra C$ extending $\tilde{\chi}$, it suffices to show $$\text{Ext}^i_{S(\mathfrak{h})}(\tilde{\chi},\mathcal{O}^{\textup{la},\tilde{\chi}}_{K^v}(U))=0\text{ for } i\geq 1.$$
    The case of modular curve is proved in \cite[Lemma 5.1.2(1)]{Pan22}, and here the computation is the same via Proposition \ref{prop:tensordecomposition} the local expansion in Theorem \ref{thm:Oiotalalocal}. 
\end{proof}

	By Harish-Chandra's isomorphism, $Z\otimes_{L,\iota}C=Z(U(\mathfrak{gl}_2(C))\cong S(\mathfrak{h})^W$, where $W$ denotes the Weyl group of $\mathfrak{gl}_2$ and acts on $S(\mathfrak{h})$ via the dot action. Consider
	$$\tilde{\chi}_k=\{(0,1-k),(-k,1)\}\subset \mathfrak{h}^*,$$
	the infinitesimal character of the $(k-1)$-th symmetric power of the dual of the $\iota$-standard representation. It follows from the relation between $\theta_\mathfrak{h}$ and the infinitesimal character that on $\mathcal{O}_{K^v}^{\text{la},\tilde{\chi}_k}$,
	$$(\theta_\mathfrak{h}(\begin{pmatrix}
	a&0\\0&d
\end{pmatrix})-(1-k)a)(\theta_\mathfrak{h}(\begin{pmatrix}
a&0\\0&d
\end{pmatrix})-(a-kd)=0.$$
	Hence we have a natural decomposition
	$$\mathcal{O}_{K^v}^{\text{la},\tilde{\chi}_k}=\mathcal{O}_{K^v}^{\text{la},(0,1-k)}\oplus\mathcal{O}_{K^v}^{\text{la},(-k,1)}$$
	and $\mathcal{O}_{K^v}^{\text{la},(0,1-k)}(U)$ (resp. $\mathcal{O}_{K^v}^{\text{la},(-k,1)}(U)$) is Hodge--Tate of weight $0$ (resp. $k$) for any $U\in \mathscr{F}\ell$. This shows that $\mathcal{O}_{K^v}^{\text{la},\tilde{\chi}_k}(U)$ is Hodge--Tate of weights $0,k$. Therefore we get the Fontaine operator
	$$\mathcal{O}_{K^v}^{\text{la},(0,1-k)}(U)\ra \mathcal{O}_{K^v}^{\text{la},(-k,1)}(U)(k)$$
    associated to the Banach $C$-module $\B^{+,\textup{la},\tilde{\chi}_k}_{\textup{dR},k}$.
	By the functorial property of the Fontaine operator, this defines a map of sheaves
	$$N_k:\mathcal{O}_{K^v}^{\text{la},(0,1-k)}\ra \mathcal{O}_{K^v}^{\text{la},(-k,1)}(k).$$

Now we can state the main theorem of this section, and the proof will be given in the next two subsections. Recall we have introduced an intertwining operator $I_{k-1}:\mathcal{O}_{K^v}^{\text{la},(0,1-k)}\ra \mathcal{O}_{K^v}^{\text{la},(-k,1)}(k)$.
	
	\begin{thm}\label{N=dd}
		 The Fontaine operator $N_k$ and intertwining operator $I_{k-1}$ satisfy:   $N_k=c_k I_{k-1}$ for some $c_k\in L^\times$.
	\end{thm}

 \begin{remark}
     One can also state the theorem for the intertwining operator $I_{k-1}:\mathcal{O}_{K^v}^{\text{la},(n_1,n_2)}\ra \mathcal{O}_{K^v}^{\text{la},(n_2+1,n_1-1
     )}(k)$ where $n_2-n_1=k$. This only differs by a twist of $\text{det}^{n_2}$ on the $\text{GL}_2$ side and a Tate twist on the Galois side, so it is still true.
 \end{remark}

\subsection{A lifting of $\calO_{K^v}^{\iota^c\-\lan,\mathrm{Gal}_L\-\sm}$ into $\calO\bbB_{\dR}$}
Let $\calO\bbB_{\dR,\calX}^+$ be the structural positive de Rham sheaf on the pro-\'etale site of $\calX=\calX_{K^vK_v,L}$. We restrict it to the analytic site of $\calX_{K^v}$ and then pushforward to $\fl$ using $\pi_{\HT}$, which we denote by $\calO\bbB_{\dR}^+$. It is equipped with a natural map $\theta:\calO\bbB_{\dR}^+\to \calO_{K^v}$ such that $\calO\bbB_{\dR}^+$ is $\ker\theta$-adically complete. Let $\calO\bbB_{\dR}=\calO\bbB_{\dR}^+[\frac{1}{t}]$ with $t\in\bbB_{\dR}$ be the Fontaine's $p$-adic $2\pi i$. 

By the $p$-adic de Rham comparison theorem, there is a natural isomorphism 
\[
    V_{\eta}^{(a,b),\mathrm{Gal}_L\-\sm}\ox_{\bar L}\calO\bbB_{\dR}\isom D_{K^v,\eta}^{(-b,-a),\sm,\mathrm{Gal}_L\-\sm}\ox_{\calO_{K^v}^{\sm,\mathrm{Gal}_L\-\sm}}\calO\bbB_{\dR}
\]
for any $(a,b)\in\bbZ^2$ with $a\ge b$. Here $(-)^{\mathrm{Gal}_L\-\sm}$ means the smooth part under the $\Gal_L$-action. When $\eta\neq\iota$, the Hodge filtration on $D_{K^v,\eta}^{(a,b),\sm}$ is trivial, so that after taking the $\Fil^0$-part we get an isomorphism 
\[
    V_{\eta}^{(a,b),\mathrm{Gal}_L\-\sm}\ox_{\bar L}\calO\bbB_{\dR}^+\isom D_{K^v,\eta}^{(-b,-a),\sm,\mathrm{Gal}_L\-\sm}\ox_{\calO_{K^v}^{\sm,\mathrm{Gal}_L\-\sm}}\calO\bbB_{\dR}^+.
\]

\begin{proposition}\label{sect}
There exists a canonical map 
\[
    s:\calO^{\eta\-\lan,\mathrm{Gal}_L\-\sm}_{K^v}\to \calO\bbB_{\dR}^+
\]
whose composition with $\theta:\calO\bbB_{\dR}^+\to \calO_{{K^v}}$ is the natural inclusion. In particular, $s$ is injective.
\begin{proof}
We first construct the section of the map $\theta$ on $\calO^{\eta\-\lalg,\mathrm{Gal}_L\-\sm}_{K^v}$. By abuse of notation we also denote $V_{\eta}^{(a,b)}$ by its $\bar L$-structure. In particular, and $V_{\eta}^{(a,b)}\ox_{\bar L} D_{K^v,\eta}^{(a,b),\mathrm{Gal}_L\-\sm}$ is the $V_{\eta}^{(a,b)}$-isotypic part of $\calO_{K^v}^{\eta\-\lalg,\mathrm{Gal}_L\-\sm}$. Recall that we have a natural isomorphism 
\[
    V_{\eta}^{(a,b)}\ox_{\bar L}\calO\bbB_{\dR}^+\isom D_{K^v,\eta}^{(-b,-a),\mathrm{Gal}_L\-\sm}\ox_{\calO_{K^v}^{\sm,\mathrm{Gal}_L\-\sm}}\calO\bbB_{\dR}^+.
\]
For simplicity, let $V_{\eta}=V_{\eta}^{(a,b)}$, and $D_{\eta}^{\sm}=D_{K^v,\eta}^{(a,b),\sm}$. This induces the following map 
\begin{align*}
    V_\eta^*\ox_{\bar L} (D_\eta^{\sm,\mathrm{Gal}_L\-\sm})^*\to  V_\eta^*\ox _{\bar L}(D_{\eta}^{\sm, \mathrm{Gal}_L\-\sm})^*\ox _{\calO_{K^v}^{\sm,\mathrm{Gal}_L\-\sm}}\calO\bbB_{\dR}^+    \aisom V_\eta^*\ox _{\bar L}V_\eta \ox_{\bar L}\calO\bbB_{\dR}^+\ov{\ev}\to \calO\bbB_{\dR}^+\ov{\theta}\to \calO_{K^v},
\end{align*}
where $V_\eta^*$ is the dual representation of $V_\eta$, $\ev$ is induced by the natural evaluation map $V_\eta^*\ox V_\eta\to \bar L$. We claim that this map coincides with the natural inclusion $\calO_{K^v}^{V_{\eta}^{(a,b)}\-\lalg,\mathrm{Gal}_L\-\sm}\subset \calO_{K^v}$. Indeed, it is easy to check that this map induces an isomorphism after taking $V_{\eta}$-isotypic part. As $\calO_{K^v}^{\eta\-\lalg,\mathrm{Gal}_L\-\sm}$ decomposes as the direct sum of $V_{\eta}^{(a,b)}$-isotypic part, we get the desired map 
\[
    s:\calO_{K^v}^{\eta\-\lalg,\mathrm{Gal}_L\-\sm}\to \calO\bbB_{\dR}^+.
\]

Finally, we construct the section on $\calO_{K^v}^{\eta\-\lan,\mathrm{Gal}_L\-\sm}$. For $U\in\ffrb$ sufficiently small and $z\in \calO_{K^v}^{\eta\-\lan,\mathrm{Gal}_L\-\sm}(U)$, from Theorem \ref{thm:Oetalalocal} we know that there exists sections $e_1',e_2',e_3',e_4'\in \calO_{K^v}^{\lalg}$ for $i=1,\dots,4$, such that $z=\sum_{i,j,k,l}c_{ijkl}e_1'^ie_2'^je_3'^k e_4'^l$ where $c_{ijkl}\in \calO_{K_n}^{\mathrm{Gal}_L\-\sm}(U)$ for some sufficiently large $n$ and $c_{ijkl}\rightarrow 0$ when $i+j+k+l\rightarrow\infty$. Using similar results as in \cite[Lem 6.4.3]{PanII}, the element $s_k(z):=\sum_{i,j,k,l}c_{ijkl}s(e_1')^is(e_2')^js(e_3')^k s(e_4')^l$ converges in $\calO\bbB_{\dR}^+/\Fil^k$ for any $k\ge 0$. Let $s(z):=\lim_{k\ra \infty} s_k(z)\in\calO\bbB_{\dR}^+$, then this is a lifting of $z$. From the density of $\calO_{K^v}^{\eta\-\lalg,\mathrm{Gal}_L\-\sm}$ in $\calO_{K^v}^{\eta\-\lan,\mathrm{Gal}_L\-\sm}$ we know that $s$ is well defined, and gives a morphism of sheaves $s:\calO^{\eta\-\lan,\mathrm{Gal}_L\-\sm}_{K^v}\to \calO\bbB_{\dR}^+$ which is a section of $\theta$.

\end{proof}
\end{proposition}

Recall that $\calO_{K^v}^{\eta\-\lalg}\isom \bigoplus_{(a,b)\in\bbZ^2,a\ge b}V_\eta^{(a,b)}\ox_C D_{K^v,\eta}^{(a,b),\sm}$. The Gauss--Manin connection on $D_{K^v,\eta}^{(a,b),\sm}$ induces a connection $d_{\eta}:\calO_{K^v}^{\eta\-\lalg}\to \calO_{K^v}^{\eta\-\lalg}\ox_{\calO_{K^v}^{\sm}}\Omega^{1,\sm}_{K^v}$.
\begin{proposition}
The connection $d_{\eta}$ on $\calO_{K^v}^{\eta\-\lalg}$ extends to a continuous map 
\[
    d_{\eta}:\calO_{K^v}^{\eta\-\lan}\to \calO_{K^v}^{\eta\-\lan}\ox_{\calO_{K^v}^{\sm}}\Omega^{1,\sm}_{K^v}.
\] 
and is compatible with the connection 
\[
    \calO\bbB_{\dR}^+\to \calO\bbB_{\dR}^+\ox_{\calO_{K^v}^{\sm,\mathrm{Gal}_L\-\sm}}\Omega_{K^v}^{1,\sm,\mathrm{Gal}_L\-\sm}
\]
via the section $s:\calO_{K^v}^{\eta\-\lan,\mathrm{Gal}_L\-\sm}\to \calO\bbB_{\dR}^+$.
\begin{proof}
The construction of $d_\eta$ on $\calO_{K^v}^{\eta\-\lan}$ is given in Proposition \ref{prop:deta}. As the connection on $\calO\bbB_{\dR}^+$ restricts to the connection on $\calO_{K^v}^{\eta\-\lalg,\mathrm{Gal}_L\-\sm}$, the result follows by taking limits since $\calO_{K^v}^{\eta\-\lalg}$ is dense in $\calO_{K^v}^{\eta\-\lan}$.
\end{proof}
\end{proposition}

\subsection{Proof of Theorem \ref{N=dd}}
 Now we begin to prove Theorem \ref{N=dd}. The idea is to replace $\B^{+,\text{la}}_{\text{dR},k}$ by a resolution using the \textit{Poincar\'e lemma sequence}. We can define the sheaves $\mathcal{O}\B^+_{\text{dR}},\mathcal{O}\B^+_{\text{dR},k},\mathcal{O}\B^{+,\text{la}}_{\text{dR}},\mathcal{O}\B^{+,\text{la},\tilde{\chi}}_{\text{dR},k}$ on $\mathscr{F}\ell$ as before. Consider the following truncated Poincar\'e lemma sequence:
	
	$$\xymatrix{
	&\mathcal{O}^{\text{la},\tilde{\chi_k}}_{K^v}\\
	0 \ar[r] &\B^{+,\text{la},\tilde{\chi}_k}_{\text{dR},k+1}\ar[u]\ar[r]&\mathcal{O}\B^{+,\text{la},\tilde{\chi}_k}_{\text{dR},k+1}\ar[r]^{\nabla\qquad}\ar[ul]_{\theta}&\mathcal{O}\B^{+,\text{la},\tilde{\chi}_k}_{\text{dR},k}\otimes_{\mathcal{O}_{V_0}}\Omega^1_{V_0}\ar[r]&0\\
	0 \ar[r] &\text{gr}^k\B^{+,\text{la},\tilde{\chi}_k}_{\text{dR}}\ar[u]\ar[r]&\text{gr}^k\mathcal{O}\B^{+,\text{la},\tilde{\chi}_k}_{\text{dR}}\ar[r]\ar[u]&\text{gr}^{k-1}\mathcal{O}\B^{+,\text{la},\tilde{\chi}_k}_{\text{dR},k}\otimes_{\mathcal{O}_{V_0}}\Omega^1_{V_0}\ar[r]\ar[u]&0
	}$$
	Evaluating this diagram at $U\in \mathfrak{B}$, taking the subspace of $G_{K_\infty}$-fixed, $G_K$-analytic vectors, and taking the generalized eigenspace of $0$ with respect to the action of $1\in \Q_p\cong\text{Lie}(\Gal(K_\infty
	/K))$, we obtain the following diagram:
	
		$$\xymatrix{
		&\mathcal{O}^{\text{la},(0,1-k)}_{K^v}(U)^{G_K}\\
		0 \ar[r] &E_0(\B^{+,\text{la},\tilde{\chi}_k}_{\text{dR},k+1}(U)^K)\ar[u]\ar[r]&E_0(\mathcal{O}\B^{+,\text{la},\tilde{\chi}_k}_{\text{dR},k+1}(U)^K)\ar[r]^{E_0(\nabla)\qquad}\ar[ul]_{E_0(\theta)}&E_0(\mathcal{O}\B^{+,\text{la},\tilde{\chi}_k}_{\text{dR},k}(U)^K)\otimes_{\mathcal{O}_{V_0}}\Omega^1_{V_0}\ar[r]&0\\
		0 \ar[r] &\mathcal{O}^{\text{la},(-k,1)}_{K^v}(U)(k)^{G_K}\ar[u]\ar[r]&E_0(\text{gr}^k\mathcal{O}\B^{+,\text{la},\tilde{\chi}_k}_{\text{dR}}(U)^K)\ar[r]\ar[u]&\mathcal{O}^{\text{la},(0,1-k)}_{K^v}(U)^{G_K}\otimes_{\mathcal{O}_{V_0}}(\Omega^1_{V_0})^{\otimes k}\ar[r]\ar[u]&0
	}$$

    	As in \cite[Lemma 6.3.3]{PanII}, we want to construct a map
	$$s_{k+1}:\mathcal{O}^{\text{la},(0,1-k)}_{K^v}(U)^{G_K}\ra \mathcal{O}\B^{+,\text{la},\tilde{\chi}_k}_{\text{dR},k+1}(U)^{G_K}$$
	which  is a section of $E_0(\theta)$. On $\mathcal{O}^{\eta\text{-la}}_{K^v}(U)^{G_K}$ for $\eta\neq \iota$, this is already done in Proposition \ref{sect}. By the tensor decomposition in Proposition \ref{prop:tensordecomposition}, we can extend it to $\mathcal{O}^{\iota^c\text{-la}}_{K^v}(U)^{G_K}$. 
    Combining this and the strategy in Section 6.4 of \cite{PanII}, we can do the following construction.

    Recall that on $V_\infty$, we have the p-adic de Rham comparison isomorphism
	$$V_\iota\otimes_{\overline{L}}\mathcal{O}\B_{\text{dR}}(V_\infty)\cong D\otimes_{\mathcal{O}_{V_0}^{\mathrm{Gal}_L\text{-sm}}} \mathcal{O}\B_{\text{dR}}(V_\infty),$$   
	which preserves the filtration and connection, where $V_\iota$ is the standard $\iota$-representation of $\text{GL}_2(L)$ and $D$ is the associated filtered vector bundle with connection (i.e., $D$ is the $\iota$-part of the de Rham cohomology of the universal abelian scheme, and we denote $D^{(0,-1)}$ by $D$ before). The filtration gives
	$$V_\iota\otimes_{\overline{L}}t\mathcal{O}\B_{\text{dR}}^+(V_\infty)\subset D\otimes_{\mathcal{O}_{V_0}^{\mathrm{Gal}_L\text{-sm}}} \mathcal{O}\B_{\text{dR}}^+(V_\infty)\subset V_\iota\otimes_{\overline{L}}\mathcal{O}\B_{\text{dR}}^+(V_\infty).$$
	Thus there is a natural $\mathcal{O}_{V_0}$-linear, $G_0$-equivariant map
	$$l_1:D(V_0)\otimes_{\overline{L}}V^*_\iota\ra \mathcal{O}\B_{\text{dR}}^+(V_\infty)$$
	preserving the filtration and connection. Taking the composite with $\theta$ gives
	$$\theta\circ l_1:D(V_0)\otimes_{\overline{L}}V^*_\iota\ra \text{gr}^0D(V_0)\otimes_{\overline{L}}V^*_\iota\ra \mathcal{O}_{\mathcal{X}_{K^v}}(V_\infty).$$
	Note that this induces the Hodge--Tate filtration
	$$\text{gr}^0D(V_0)\cong \omega_+^{-1}\otimes \wedge^2 D(V_0)\hookrightarrow \mathcal{O}_{\mathcal{X}_{K^v}}(V_\infty) \otimes_{\overline{L}}V_\iota.$$
	Let $(1,0)^*,(0,1)^*\in V^*_\iota$ be the dual basis of $V_\iota=L^{\oplus 2}$. Then for $f\in D(V_0)$,
	\begin{gather*}
		\theta\circ l_1(f\otimes (1,0)^*)=\frac{\overline{f}e_2}{tc},\\
		\theta\circ l_1(f\otimes (0,1)^*)=-\frac{\overline{f}e_1}{tc},
	\end{gather*}
	where $\overline{f}\equiv f \mod \text{Fil}^1D(V_0)$, and $tc$ can be regarded as a Tate twist.

	Assume that $e_1$ is invertible on $V_\infty$. Fix a section $f_1\in D(V_0)(V_0)$  whose image in $\text{gr}^0 D(V_0)$ is a generator. Then $\theta\circ l_1(f_1\otimes (0,1)^*)\in \mathcal{O}_{K^v}(U)^\times$ and hence  $l_1(f_1\otimes (0,1)^*)$ is invertible in $\mathcal{O}\B_{\text{dR}}^+(U)$. Let
	$$\tilde{x}=-\frac{l_1(f_1\otimes (1,0)^*)}{l_1(f_1\otimes (0,1)^*)}\in \mathcal{O}\B_{\text{dR}}^+(U).$$ 
	Then $\theta(\tilde{x})=x$. It is clear that $G_K$ fixes $\tilde{x}$ and the Lie algebra $\mathfrak{gl}_2(L)$ acts on $\tilde{x}$ in the same way as $x$. Note that any element $f\in \mathcal{O}_{K^v}^{\text{la},(0,0)}(U)^{G_K}$ can be written as
	$$f=\sum_{i=0}^{+\infty}c_i(x-x_n)^i$$
	for some $n\geq 0$ and $c_i\in \mathcal{O}_{K^v}^{\iota^c\text{-la}}$ such that $c_ip^{(n-1)i}$ is uniformly bounded. Thus for any $l$, we can define
    $$\tilde{s}_l:\mathcal{O}_{K^v}^{\text{la},(0,0)}(U)^{G_K}\ra \mathcal{O}\B_{\text{dR},l}^{+,\text{la},\tilde{\chi}_1}(U)^{G_K}$$
    sending $f$ to $\sum_{i=0}^{+\infty}s(c_i)(\tilde{x}-x_n)^i$, which is well-defined by \cite[Lemma 6.4.3]{PanII}. In particular, letting $l=2$ we get the desired section $s_2$.
    
    For general $k\geq 1$, consider the map
	$$l_{k-1}: \text{Sym}^{k-1} D(V_0)\otimes_{\overline{L}} \text{Sym}^{k-1} V_\iota^*\ra \mathcal{O}\B^+_{\text{dR}}(U)$$
	induced by $l_1$. Using $\tilde{s}_{k+1}$, we can extend this to a map
	$$l'_{k-1}: \text{Sym}^{k-1} D(V_0)\otimes_{\mathcal{O}_{V_0}^{\mathrm{Gal}_L\text{-sm}}}\mathcal{O}_{K^v}^{\text{la},(0,0)}(U)^{G_K}\otimes_{\overline{L}} \text{Sym}^{k-1} V_\iota^*\ra \mathcal{O}\B^+_{\text{dR},k+1}(U).$$
    Note that the $\tilde{\chi}_{k}$-isotypic part of 
	$$\text{gr}^0\text{Sym}^{k-1} D(V_0)\otimes_{\mathcal{O}_{V_0}^{\mathrm{Gal}_L\text{-sm}}}\mathcal{O}_{K^v}^{\text{la},(0,0)}(U)^{G_K}\otimes_{\overline{L}} \text{Sym}^{k-1} V_\iota^*$$ is isomorphic to $\mathcal{O}_{K^v}^{\text{la},(0,1-k)}(U)^{G_K}$ via $\theta\circ l_{k-1}'$, and the Kodaira--Spencer isomorphism gives a natural section $\text{gr}^0\text{Sym}^{k-1} D(V_0)\ra \text{Sym}^{k-1} D(V_0)$. Then the composite 
    $$\mathcal{O}_{K^v}^{\text{la},(0,1-k)}(U)^{G_K}\xrightarrow{\psi_{k-1}}\text{Sym}^{k-1} D(V_0)\otimes_{\mathcal{O}_{V_0}^{\mathrm{Gal}_L\text{-sm}}}\mathcal{O}_{K^v}^{\text{la},(0,0)}(U)^{G_K}\otimes_{\overline{L}} \text{Sym}^{k-1} V^*_\iota\xrightarrow{l'_{k-1}}\mathcal{O}\B^+_{\text{dR},k+1}(U)$$
    gives the desired map $s_{k+1}:\mathcal{O}^{\text{la},(0,1-k)}_{K^v}(U)^{G_K}\ra \mathcal{O}\B^{+,\text{la},\tilde{\chi}_k}_{\text{dR},k+1}(U)^{G_K}$.

According to \cite[Lemma 6.3.3]{PanII}, $N_k=-N_k'\circ(E_0(\nabla)\circ s_{k+1})$ where $N_k':\mathcal{O}^{\text{la},(0,1-k)}_{K^v}(U)^{G_K}\otimes_{\mathcal{O}_{V_0}}(\Omega^1_{V_0})^{\otimes k}\ra \mathcal{O}^{\text{la},(-k,1)}_{K^v}(U)^{G_K}$ is obtained by applying $1\in \Z_p\cong \text{Lie}(\text{Gal}(K_\infty/K))$ to the middle term of the exact sequence
$$0\ra \mathcal{O}^{\text{la},(-k,1)}_{K^v}(U)(k)^{G_K}\ra E_0(\text{gr}^k\mathcal{O}\B^{+,\text{la},\tilde{\chi}_k}_{\text{dR},k+1}(U)^{G_K})\ra \mathcal{O}^{\text{la},(0,1-k)}_{K^v}(U)^{G_K}\otimes_{\mathcal{O}_{V_0}}(\Omega^1_{V_0})^{\otimes k}\ra 0.$$
Thus Theorem \ref{N=dd} follows from these two propositions:
 
\begin{prop}
    $E_0(\nabla)\circ s_{k+1}=d^k$, where both sides are viewed as maps $$\mathcal{O}^{\textup{la},(0,1-k)}_{K^v}(U)^{G_K}\ra \mathcal{O}^{\textup{la},(0,1-k)}_{K^v}(U)^{G_K}\otimes_{\mathcal{O}_{V_0}}(\Omega^1_{V_0})^{\otimes k}$$.
\end{prop}
\begin{proof}
    The proof parallels \cite[Proposition 6.3.13]{PanII}. Roughly speaking, let $DV_k:=\text{Sym}^{k-1} D(V_0)\otimes_{\mathcal{O}_{V_0}^{\mathrm{Gal}_L\text{-sm}}}\mathcal{O}_{K^v}^{\text{la},(0,0)}(U)^{G_K}\otimes_{\overline{L}} \text{Sym}^{k-1} V^*_\iota$, then we have the commutative diagram
    $$\xymatrix{
        \mathcal{O}^{\text{la},(0,1-k)}_{K^v}(U)^{G_K}\ar[r]^{\qquad\psi_{k-1}}\ar[dr]_{s_{k+1}\mod t} &DV_k \ar[r]^{\nabla\qquad}\ar[d]^{l_{k-1}'\mod t} &DV_k\otimes \Omega_{V_0}^1\ar[d]^{l_{k-1}'\otimes 1\mod t}\\
        &\mathcal{O}\B^{+}_{\text{dR},k+1}(U)/(t)\ar[r]^{\nabla\mod t} &\mathcal{O}\B^{+}_{\text{dR},k}(U)/(t)\otimes \Omega_{V_0}^1.
    }$$
    Thus $\nabla\circ s_{k+1}=d^k\mod t$. By looking at the Hodge--Tate weight of each graded part we know that the natural map
    $$\mathcal{O}\B^{+,\text{la},\tilde{\chi}_k}_{\text{dR},k+1}(U)^{G_K}\ra (\mathcal{O}\B^{+,\text{la},\tilde{\chi}_k}_{\text{dR},k+1}(U)/(t))^{G_K}$$
    is an isomorphism. Therefore $E_0(\nabla)\circ s_{k+1}=d^k$.
    
\end{proof}

\begin{prop}
    $N_k'=c_k \overline{d}^k$ for some $c_k\in L^\times$, where both sides are viewed as maps $$\mathcal{O}^{\textup{la},(0,1-k)}_{K^v}(U)^{G_K}\otimes_{\mathcal{O}_{V_0}}(\Omega^1_{V_0})^{\otimes k}\ra \mathcal{O}^{\textup{la},(-k,1)}_{K^v}(U)^{G_K}.$$
\end{prop}
\begin{proof}
    As both maps are $\mathcal{O}^{\iota^c\text{-la},(0,1-k)}_{K^v}(U)^{G_K}$-linear, it suffices to prove it on $\mathcal{O}^{\iota\text{-la},(0,1-k)}_{K^v}(U)^{G_K}\otimes_{\mathcal{O}_{V_0}}(\Omega^1_{V_0})^{\otimes k}$. Then the proof is the same as \cite[Proposition 6.3.14]{PanII}. Roughly speaking, we can show that $N_k'$ can also be obtained by looking at the action of $Z$ on the $G_K$-invariant up to a non-zero constant, then we can relate it with $\overline{d}^k$.
\end{proof}

\subsection{De Rhamness and classicality}
	Using the results in the previous sections, we can prove the following theorem:
	\begin{thm}\label{main}
		Let $E$ be a finite extension of $\Q_p$  containing all embeddings of $L$ into $\overline{\Q}_p$ and let 
		$$\rho:\mathrm{Gal}_F\ra \mathrm{GL}_2(E)$$ be a two-dimensional continuous absolutely irreducible representation of $\mathrm{Gal}_F$. We denote by $\Pi(\rho):=\Hom_{E[\Gal_F]}(\rho,\tilde{H}^1(K^v,E))$, which is a unitary Banach representation of $\mathrm{GL}_2(L)$.
        Suppose that 
        
		$\mathrm{(1)}$ $\Pi(\rho)^{\sigma\-\mathrm{la},\sigma^c\text{-}\mathrm{lalg}}\neq 0$.
		
		$\mathrm{(2)}$ $\rho|_{\mathrm{Gal}_L}$ is $\sigma$-de Rham of $\sigma$-Hodge--Tate weights $0,k$ for some embedding $\sigma:L\hookrightarrow E$ and integer $k>0$. 
		Then $\rho$ arises from a cuspidal eigenform of weight $k+1$.
	\end{thm}
	
	To prove the theorem, first we recall the Hecke algebra. Let $S$ be a finite set of rational primes containing $p$ such that $K_l\subset G(\Q_l)$ is hyperspecial for all $l\notin S$. For any open compact subgroup $K_v\subset \text{GL}_2(L)$, define
	$$\mathbb{T}(K^vK_v)\subset \text{End}_{\Z_p}(H^1_{\text{\'et}}(\mathcal{Y}_{K^vK_v},\Z_p))$$
	be the subalgebra generated by Hecke operators $\mathcal{H}(G(\Q_l)//K_l)$ at places $l\notin S$. Now we define the Hecke algebra of tame level $K^v$ as
	$$\T:=\varprojlim_{K_v\subset \text{GL}_2(L)}\T(K^vK_v).$$
	It acts faithfully on $\tilde{H}^1(K^v,E)$ and commutes with the action of $\text{GL}_2(L)\times \mathrm{Gal}_F$.
    We denote by 
		$$\tilde{H}^1(K^v,E)[\rho]\subset \tilde{H}^1(K^v,E)$$ the image of the evaluation map $\rho\otimes_E \mathrm{Hom}_{E[\mathrm{Gal}_F]}(\rho, \tilde{H}^1(K^v,E))\ra \tilde{H}^1(K^v,E)$.
    By the Eichler--Shimura relation, there exists a homomorphism $\lambda:\T(K^v)\ra E$ such that
	$$\tilde{H}^1(K^v,E)[\lambda]=\tilde{H}^1(K^v,E)[\rho].$$ This follows from \cite{BLR} and the discussion in \cite[6.1.1]{Pan22} (we will discuss this in more detail in the proof of Theorem \ref{thm:admissible}). Here $\tilde{H}^1(K^v,E)[\rho]:=\text{Hom}_{E[\mathrm{Gal}_F]}(\rho, \tilde{H}^1(K^v,E)[\lambda])\otimes_E\rho$. Since $\rho$ is $\sigma$-de Rham of $\sigma$-Hodge--Tate weight $0,k$, there is a Hodge--Tate decomposition
	$$\tilde{H}^1(K^v,C\otimes_{L,\iota,\sigma}E)[\lambda]=\tilde{H}^1(K^v,E)[\rho]\widehat{\otimes}_{L,\sigma,\iota}C=\tilde{H}^1(K^v,C\otimes_{L,\iota,\sigma}E)[\lambda]_0\oplus \tilde{H}^1(K^v,C\otimes_{L,\iota,\sigma}E)[\lambda]_k,$$
	where $C\otimes_{L,\iota,\sigma}E$ means the tensor product with respect to the embeddings $\iota:L\hookrightarrow C$ and $ \sigma:L\hookrightarrow E$, and $\tilde{H}^1(K^v,C\otimes_{L,\iota,\sigma}E)[\lambda]_i$ denotes the Hodge--Tate weight $i$ part for $i=0,k$.
	
	\begin{thm}\label{kerI}
		Let $\rho$ be as in Theorem \ref{main}. Then there is a natural $\mathrm{GL}_2(L)$-equivariant isomorphism of $\lambda$-isotypic parts
		$$\tilde{H}^1(K^v,C\otimes_{L,\iota,\sigma}E)[\lambda]_0^{\iota\text{-}\mathrm{la},\iota^c\text{-}\mathrm{lalg}}\cong \ker I^1_{k-1}\otimes_{L,\iota,\sigma}E[\lambda].$$
    Here $I^1_{k-1}$ is the intertwining operator on the cohomology group.
	\end{thm}
	\begin{proof}
		The proof is similar to \cite[Theorem 7.2.2]{PanII}. First, by \cite[Proposition 7.2.4]{PanII}, there is a natural $\mathrm{Gal}_F$-equivariant isomorphism of $B_{\text{dR},k}^+$-modules
        $$\tilde{H}^1(K^v,B_{\text{dR},k}^+)\cong H^1(\mathscr{F}\ell,\mathbb{B}_{\text{dR},k}^+).$$
        (The proof there is quite general so can be applied in any case of Hodge--Tate period map considered in \cite{Sch15}.) 
        
        Since $\text{det} \rho$ has $\sigma$-Hodge--Tate weight $k$, the center $Cz$ of $\mathfrak{g}_\sigma$ acts on $\tilde{H}^1(K^v,E)[\lambda]$ via $z\mapsto -k+1$. Then by Corollary \ref{cor:infSen} the $\text{GL}_2(L)$-locally analytic vectors $\tilde{H}^1(K^v,E)^\text{la}[\lambda]$ have infinitesimal character $\tilde{\chi}_k$ of the infinitesimal character of $(k-1)$-th symmetric power of dual of the standard $\sigma$-representation. Thus
        \begin{align*}
            \tilde{H}^1(K^v,B^+_{\text{dR},k+1}\otimes_{L,\iota,\sigma}E)[\lambda]^{\iota\text{-la},\iota^c\text{-}\mathrm{lalg}}&\cong \tilde{H}^1(K^v,B^+_{\text{dR},k+1}\otimes_{L,\iota,\sigma}E)[\lambda]^{\iota\text{-la},\iota^c\text{-}\mathrm{lalg},\tilde{\chi}_k}\\
            &\cong H^1(\mathscr{F}\ell,\mathbb{B}_{\text{dR},k+1}^+)^{\iota\text{-la},\iota^c\text{-}\mathrm{lalg},\tilde{\chi}_k}\otimes_{L,\iota,\sigma}E[\lambda]\\
            &\cong H^1(\mathscr{F}\ell,\mathbb{B}_{\text{dR},k+1}^{+,\iota\text{-la},\iota^c\text{-}\mathrm{lalg},\tilde{\chi}_k})\otimes_{L,\iota,\sigma}E[\lambda],
        \end{align*}
        where the last isomorphism is obtained from Theorem \ref{thm:JlacommutesH1} and Lemma \ref{chicomH1}. Taking $\text{gr}^0$-part we get
        $$\tilde{H}^1(K^v,C\otimes_{L,\iota,\sigma}E)[\lambda]^{\iota\text{-la},\iota^c\text{-}\mathrm{lalg}}\cong H^1(\mathscr{F}\ell,\mathcal{O}_{K^v}^{\iota\text{-la},\iota^c\text{-}\mathrm{lalg},\tilde{\chi}_k})\otimes_{L,\iota,\sigma}E[\lambda].$$
        Recall that
        $$H^1(\mathscr{F}\ell,\mathcal{O}_{K^v}^{\iota\text{-la},\iota^c\text{-}\mathrm{lalg},\tilde{\chi}_k})\cong H^1(\mathscr{F}\ell,\mathcal{O}_{K^v}^{\iota\text{-la},\iota^c\text{-}\mathrm{lalg},(0,1-k)})\oplus H^1(\mathscr{F}\ell,\mathcal{O}_{K^v}^{\iota\text{-la},\iota^c\text{-}\mathrm{lalg},(-k,1)}),$$
        with Hodge--Tate weight $0,k$ respectively. Therefore the Fontaine operator defined in section \ref{Fon} gives a map
        $$N:\tilde{H}^1(K^v,C\otimes_{L,\iota,\sigma}E)[\lambda]_0^{\iota\text{-}\mathrm{la},\iota^c\text{-}\mathrm{lalg}}\ra \tilde{H}^1(K^v,C\otimes_{L,\iota,\sigma}E)[\lambda]_k^{\iota\text{-}\mathrm{la},\iota^c\text{-}\mathrm{lalg}}(k),$$
        and by Theorem \ref{N=dd}  it coincides with $I^1_{k-1}$ up to a unit under the above isomorphism. As $\rho$ is $\sigma$-de Rham, $N=0$. Thus $I^1_{k-1}=0$ on $\tilde{H}^1(K^v,C\otimes_{L,\iota,\sigma}E)[\lambda]_0^{\iota\text{-}\mathrm{la},\iota^c\text{-}\mathrm{lalg}}$.
	\end{proof}

    Finally we can prove our main theorem in this section:
    
	\textit{Proof of Theorem \ref{main}.} Note that
	$$\tilde{H}^1(K^v,C\otimes_{L,\iota,\sigma}E)[\lambda]_0^{\iota\text{-}\mathrm{la},\iota^c\text{-}\mathrm{lalg}}=(\tilde{H}^1(K^v,E)^{\sigma\text{-la},\sigma^c\text{-}\mathrm{lalg}}[\lambda]\otimes_{L,\sigma,\iota}C)_0,$$
	thus it is non-zero by condition (1) of $\rho$. By the above theorem $ \ker I^1_{k-1}\otimes_{L,\iota,\sigma}E[\lambda]\neq 0$. Hence by Corollary \ref{classical} we get $\lambda$ is classical, and $\rho$ arises from a cuspidal eigenform. \hfill$\square$
	~\\

Next, we can prove the following theorem:
\begin{thm}\label{thm:sigmaalgimpliessigmadeRham}
    Let $E$ be a finite extension of $L$  containing all embeddings of $L$ into $\overline{\Q}_p$ and 
		$$\rho:\mathrm{Gal}_F\ra \mathrm{GL}_2(E)$$ be a two-dimensional continuous absolutely irreducible representation of $\mathrm{Gal}_F$. Suppose that  $\tilde{H}^1(K^v,E)^{\sigma^c\textup{-la}, \sigma\textup{-lalg}}[\rho]\neq 0$. Then $\rho$ is $\sigma$-de Rham of different Hodge--Tate weight.
\end{thm}
\begin{proof}
    As in the last section, there exists a homomorphism $\lambda:\T(K^v)\ra E$ such that
	$$\tilde{H}^1(K^v,E)[\lambda]=\tilde{H}^1(K^v,E)[\rho],$$
    where $\tilde{H}^1(K^v,E)[\lambda]$ is the $\lambda$-isotypic part. 
    After twisting by the cyclotomic character, we may assume that $\tilde{H}^1(K^v,E)^{\sigma^c\textup{-la}, V_\sigma^{(0,1-k)}\textup{-lalg}}[\lambda]\neq 0$ for some $k\geq 1$, where $V_\sigma^{(0,1-k)}$ is the dual of  $(k-1)$-th symmetric power of the $\sigma$-standard representation. As the center $Cz$ of $\mathfrak{gl}_{2,\sigma}$ acts on $\tilde{H}^1(K^v,E)^{\sigma^c\textup{-la}, V_\sigma^{(0,1-k)}\textup{-lalg}}[\lambda]$ via $z\mapsto -k+1$, $\textup{det}\rho$ has $\sigma$-Hodge--Tate weight $k$. 
    By Theorem \ref{thm:JlalgcommutesH1}, 
    $$\tilde{H}^1(K^v,E)^{\sigma^c\textup{-la}, V_\sigma^{(0,1-k)}\textup{-lalg}}[\lambda]\otimes_{L,\sigma,\iota}C\cong H^1(\fl,\calO_{K^v}^{\lan})^{\iota\-\lan, V_\iota^{(0,1-k)}\textup{-lalg}}\otimes_{L,\iota,\sigma}E[\lambda]$$
    is of Hodge--Tate weight $0,k$, and the natural inclusion
    $$H^1(\fl,\calO_{K^v}^{V_\iota^{(0,1-k)}\textup{-lalg},\iota^c\textup{-la}})\otimes_{L,\iota,\sigma}E[\lambda]\hookrightarrow H^1(\fl,\calO_{K^v}^{\lan})^{\iota\-\lan,V_\sigma^{(0,1-k)}\textup{-lalg}}\otimes_{L,\iota,\sigma}E[\lambda]$$
    coincides with the Hodge--Tate weight $0$ part. Thus $\rho$ is $\sigma$-Hodge--Tate of weight $0,k$. In particular, $H^1(\fl,\calO_{K^v}^{V_\iota^{(0,1-k)}\textup{-lalg},\iota^c\textup{-la}})\otimes_{L,\iota,\sigma}E[\lambda]$ is non-zero. By the same proof of Theorem $\ref{kerI}$, with $\iota^c$-locally algebraic vectors replaced by $\iota^c$-locally analytic vectors, we get that the Fontaine operator from
    $$\tilde{H}^1(K^v,C\otimes_{L,\iota,\sigma}E)[\lambda]_0^{\textup{la}}\cong H^1(\mathscr{F}\ell,\calO_{K^v}^{\mathrm{la},(0,-k+1)})\otimes_{L,\iota,\sigma}E[\lambda]$$
    to 
    $$\tilde{H}^1(K^v,C\otimes_{L,\iota,\sigma}E)[\lambda]_k^{\textup{la}}(k)\cong H^1(\mathscr{F}\ell,\calO_{K^v}^{\mathrm{la},(-k,1)})\otimes_{L,\iota,\sigma}E[\lambda](k)$$
    is (up to unit) induced by the intertwining operator $I_{k-1}:\calO_{K^v}^{\mathrm{la},(0,-k+1)}\ra \calO_{K^v}^{\mathrm{la},(-k,1)}(k)$. By Theorem \ref{thm:dbar} we know $\overline{d}$ is $0$ on $\calO_{K^v}^{\iota^c\textup{-la}, V_\iota^{(0,1-k)}\textup{-lalg}}$. Thus $I_{k-1}=0$ on $H^1(\fl,\calO_{K^v}^{\iota^c\textup{-la}, V_\iota^{(0,1-k)}\textup{-lalg}})\otimes_{L,\iota,\sigma}E[\lambda]$. Thus the $\sigma$-Fontaine operator associated to $\rho$ is not injective, hence must be zero. Therefore by Theorem  \ref{N=0} we get $\rho$ is $\sigma$-de Rham.
\end{proof}
\begin{remark}
    This theorem can be seen as a generalization of the theorem that ``$\tilde{H}^1(E)^{\mathrm{lalg}}[\rho]\neq 0$ implies $\rho$ is de Rham of different Hodge Tate weight" proved by Emerton in \cite{Eme06}.
\end{remark}

We also make the following conjecture. We are grateful to Yiwen Ding for insightful discussions on this conjecture.
\begin{conj}\label{conj:sigmaalgsigmadR}
    Let $E$ be a finite extension of $L$  containing all embeddings of $L$ into $\overline{\Q}_p$ and 
		$$\rho:\mathrm{Gal}_F\ra \mathrm{GL}_2(E)$$ be a two-dimensional continuous absolutely irreducible representation of $\mathrm{Gal}_F$. 
        Assume that $\Pi(\rho)\neq 0$. Then for any subset $J$ of embeddings of $L$ into $E$,  $\Pi(\rho)^{J^c\textup{-la}, J\textup{-lalg}}[\rho]\neq 0$ if and only if  $\rho$ is $J$-de Rham of different $J$-Hodge--Tate weight.
\end{conj}

\begin{remark}
   The above theorem proves the ``only if" part. If this is true, combining this with Theorem \ref{main} we can prove the classicality of every de Rham representation of different Hodge--Tate weights appearing in the completed cohomology. 
\end{remark}

\section{Applications to the $p$-adic local Langlands program for $\GL_2(L)$}\label{p-LLC}
In this section, we  discuss various applications to the $p$-adic local Langlands program for the group $\GL_2(L)$. Let $\rho:\Gal_F\to \GL_2(E)$ be a $2$-dimensional continuous $E$-linear absolutely irreducible representation. Its multiplicity space in the completed cohomology group 
\begin{align*}
    \Pi(\rho):=\Hom_{E[\Gal_F]}(\rho,\tilde{H}^1(K^v,E))
\end{align*}
is naturally a unitary Banach representation of $\GL_2(L)$. From $\rho$, we can associate a Hecke eigenvalue $\lambda$ using the Eichler--Shimura relation. In this final section, we will discuss the following applications:
\begin{itemize}
    \item When $\lambda$ is classical, we describe the locally $\sigma$-analytic vectors of the Hecke isotypic space purely in terms of the datum on the Galois side. From this we deduce $\Pi(\rho)^{\sigma\-\lan}$ only depends on the local Galois representation.
    \item We realize certain locally $\sigma$-analytic representations defined using coherent cohomology of Drinfeld tower in the completed cohomology group. From this we deduce the admissibility of such representations.
    \item Again when $\lambda$ is classical, such that the associated Weil--Deligne representation of $\rho|_{\Gal_L}$ is irreducible, we use the wall-crossing functor to handle the Breuil's locally analytic $\Ext^1$-conjecture in this case. We also show that there is a locally $\bbQ_p$-analytic subrepresentation of $\Pi(\rho)^{\lan}$, which we denoted by $\pi_1(\rho|_{\Gal_L})$ in Theorem \ref{thm:gal}, that completely determines $\rho|_{\Gal_L}$.
\end{itemize}

\subsection{Locality of $\Pi(\rho)^{\sigma\-\lan}$}
One of the main advantages of Pan's work is that it reveals the appearance of the Hodge filtration more transparently and geometrically \cite[Remark 5.6.12]{PanII}. Let $\lambda$ be a Hecke eigenvalue that appears in ${H}^1(K^v,E):=\dlim_{K_v}H^1_{\et}(\calX_{K^vK_v},E)$. From $\lambda$, we can construct a Galois representation $\rho:\Gal(\bar F/F)\to \GL_2(E)$, which is a $2$-dimensional continuous $E$-linear representation of $\Gal(\bar F/F)$, such that $\rho_w:=\rho|_{\Gal(\bar L/L)}$ is de Rham of parallel Hodge--Tate weights $\{0,1\}_{\sigma\in\Sigma}$. We assume $\rho$ is absolutely irreducible. Let $\tilde{H}^1(K^v,E)[\rho]$ be the image of the evaluation map of $\rho\ox_E\Hom_{E[\Gal(\bar F/F)]}(\rho,\tilde{H}^1(K^v,E))\to \tilde{H}^1(K^v,E)$.  Let $u:E\inj C$ be an embedding such that $\iota=u\comp \sigma$. By Theorem \ref{kerI}, there is a $\GL_2(L)$-equivariant isomorphism 
\[
    (\tilde{H}^1(K^v,E)^{\sigma\-\lan}[\rho]\ox_{E,u}C)[\Theta_{\Sen}=0]=\ker I^1[u\comp \lambda],
\]
where $I^1:H^1(\fl,\calO_{K^v}^{\iota\-\lan,(0,0)})\to H^1(\fl,\calO_{K^v}^{\iota\-\lan,(-1,1)}(1))$ is the intertwining operator introduced before, and $(-)[\Theta_{\Sen}=0]$ denotes the arithmetic Sen weight equal to $0$ part. Therefore, the study of the $\rho$-isotypic space in $\tilde H^1(K^v,E)^{\sigma\-\lan}$ is reduced to study the $\lambda$-isotypic space in $\ker I^1$. 

We note that $u\comp \lambda\in \sigma_0^{K^v}$, and for simplicity, we write $(-)[\lambda]$ for the Hecke isotypic space of $u\comp\lambda$. By Theorem \ref{thm:kerI1Hodge}, the representation $\ker I^1[\lambda]$ fits into two exact sequences
\begin{align*}
    0\to H^1(\fl,\calO_{K^v}^{\sm})[\lambda]\to \ker I^1[\lambda]\to \bbH^1(\DR')[\lambda]\to 0,\\
    0\to H^0(\fl,\Omega_{K^v}^{1,\sm})[\lambda]\to \bbH^1(\DR)[\lambda]\to \ker I^1[\lambda]\to 0,
\end{align*}
which are $\GL_2(L)$-equivariant. The first exact sequence gives the internal structure of $\ker I^1[\lambda]$, and the second exact sequence encodes the information of the Hodge filtration of $\rho_w$. Indeed, the image of $H^0(\Omega_{K^v}^{1,\sm})[\lambda]\to \bbH^1(\DR)[\lambda]$ lands in $\bbH^1(\DR^{\sm})[\lambda]\subset\bbH^1(\DR)[\lambda]$, and 
\[
    \bbH^1(\DR^{\sm})\isom \dlim_{K_v}H^1_{\dR}(\calX_{K^vK_v})
\]
is the inductive limit of de Rham cohomology of unitary Shimura curves of finite level. Therefore, by the $p$-adic de Rham comparison theorem \cite{Sch13}, we know that the position of the geometric Hodge filtration
\[
    H^0(\Omega_{K^v}^{1,\sm})[\lambda]\subset \bbH^1(\DR^{\sm})[\lambda]
\]
is exactly the position of the $\sigma$-Hodge filtration 
\[
    \Fil^1 D_{\dR,\sigma}(\rho_w)\ox_{E,u}C\subset D_{\dR,\sigma}(\rho_w)\ox_{E,u}C.
\]
We note that the above discussion has a natural generalization to arbitrary regular weights.

From our detailed study of the Hecke decomposition of the de Rham complex $\bbH^1(\DR)$ and $\ker I^1$, we can deduce the following locality result. Let $\rho$ be as in Theorem \ref{main}, with $\lambda:\bbT^S\to E$ the associated Hecke eigenvalue, such that $\pi_v$ is generic. Put $\rho_w:=\rho|_{\Gal(\bar F_w/F_w)}$. Fix $\sigma:L\inj E$, and let $\WD(\rho_w)$ be the underlying Weil--Deligne representation of $\rho_w$. Let $\HT(\rho_w)$ be the Hodge--Tate weights of $\rho_w$.
\begin{theorem}\label{thm:locality}
Let $\rho$ be as in Theorem \ref{main}, with $\rho_w:=\rho|_{\Gal(\bar F_w/F_w)}$. Then as $\GL_2(L)$-representations,
\begin{enumerate}[(i)]
    \item $\bbH^1(\DR)[u\comp \lambda]$ only depends on $\WD(\rho_w)$ and $\HT(\rho_w)$.
    \item $\ker I^1[u\comp \lambda]$ only depends on $\WD(\rho_w)$, $\HT(\rho_w)$, and the $\sigma$-Hodge filtration of $\rho_w$. Here $\sigma:L\to E$ is our fixed embedding to define $\WD(\rho_w)$. In particular, $\ker I^1[u\comp \lambda]$ only depends on the local Galois representation $\rho_w$.
\end{enumerate}
\begin{proof}
From our computations in Theorem \ref{thm:kerI2}, it is easy to see that when $\pi_v$ is a principal series or supercuspidal, then $\bbH^1(\DR)[\lambda]$ only depends on $\text{WD}(\rho_w)$ and $\HT(\rho_w)$. When $\pi_v$ is special, from our proof in \ref{universalSteinberg} we see $\bbH^1(\DR)[\lambda]$ is the unique non-split extension of $H^0(\coker d)[\lambda]$ by $H^1(\ker d)[\lambda]$. As $H^0(\coker d)[\lambda]$, $H^1(\ker d)[\lambda]$ only depends on $\WD(\rho_w)$ and $\HT(\rho_w)$, $\bbH^1(\DR)[\lambda]$ also only depends on $\WD(\rho_w)$ and $\HT(\rho_w)$. Finally, as $\ker I^1$ is the cokernel of the map $H^0(\calO\ox\Omega)[\lambda]\to H^1(\DR)[\lambda]$ which encodes the $\sigma$-Hodge filtration of $\rho_w$, we see $\ker I^1$ encodes the information about the $\sigma$-Hodge filtration of $\rho_w$.
\end{proof}
\end{theorem}

\subsection{Admissibility of coherent cohomology of Drinfeld tower}
In this section, we want to show certain locally analytic representations cut out from the coherent cohomology of Drinfeld towers using $D_L^\times$-actions are admissible.
Recall that $\calM_{\Dr,n}$ is the Drinfeld curve over $C$ of level $n$, equipped with an action of $\GL_2(L)\times(\calO_{D_L}/p^n\calO_{D_L})^\times$. Let $\Sigma_{n}$ be the $E$-model of $\calM_{\Dr,n}/{\varpi}^{\bbZ}$ obtained via the Weil descent datum on $\calM_{\Dr,n}$, which is a Stein space over $E$ such that $L$ acts via $\sigma$. Here $\varpi$ acts on $\calM_{\Dr,n}$ via the identification of $L^\times$ with the center of $D_L^\times$. We have a de Rham complex 
\begin{align*}
    0\to H^0_{\dR}(\Sigma_{n})\to \calO(\Sigma_{n})\to \Omega^1(\Sigma_{n})\to H^1_{\dR}(\Sigma_{n})\to 0,
\end{align*}
which is an exact sequence of Fr\'echet spaces over $E$. Applying $-\hat\ox_{E,u} C$, one can recover the de Rham complex of $\calM_{\Dr,n}/{\varpi}^{\bbZ}$.

Let $\fl_E$ be the natural $\GL_2(L)$-equivariant $E$-model of $\fl$, and let $j_E:\Omega_E\subset \fl_E$ be the $E$-model of $\Omega$. Let $\pi_{n,E}:\Sigma_{n}\to \Omega_E$ be the $E$-model of the Gross--Hopkins period map $\pi_{\GM,n}:\calM_{\Dr,n}/{\varpi}^{\bbZ}\to \Omega$. From
\begin{align*}
    \calO(\Sigma_n)\hat\ox_{E,u}C\isom \calO(\calM_{\Dr,n}/{\varpi}^{\bbZ}),
\end{align*}
by Serre duality, we deduce a $\GL_2(L)\times D_L^\times$-equivariant isomorphism 
\begin{align*}
    H^1(j_{E,!}\pi_{n,E,*}\Omega^1_{\Sigma_n})\hat\ox_{E,u} C\isom H^1(j_!\pi_{\GM,n,*}\Omega^1_{\calM_{\Dr,n}/{\varpi}^{\bbZ}})
\end{align*}
where $j_E$ denotes the inclusion $\Omega_E\inj \fl_E$ and $j$ denotes the inclusion $\Omega\inj\fl$. Define
\begin{align*}
    \pi_c=(\tau_v \hat\ox_C H^1(j_!\pi_{\GM,*}\calO_{\calM_{\Dr,n}^{(0)}}))^{\calO_{D_L}^\times}
\end{align*}
where $\tau_v $ is a smooth irreducible representation of $D_L^\times$ and $n$ is a sufficiently large integer such that $1+\varpi^n\calO_{D_L}$ acts trivially on $\tau_v $. After twisting $\tau_v $ by an unramified character, we may assume that $\varpi$ acts trivially on $\tau_v $. See for example the treatment in \cite[5.1]{CDN20}. Then 
\begin{align*}
    \pi_c&=(\tau_v \hat\ox_C H^1(j_!\pi_{\GM,n,*}\calO_{\calM_{\Dr,n}})^{{\varpi}^{\bbZ}})^{D_L^\times}\\
    &\isom (\tau_v \hat\ox_C H^1(j_!\pi_{\GM,n,*}\calO_{\calM_{\Dr,n}/{\varpi}^{\bbZ}}))^{D_L^\times}
\end{align*}
as the cohomology of $j_!\pi_{\GM,n,*}\calO_{\calM_{\Dr,n}/{\varpi}^{\bbZ}}$ only concentrated on $H^1$. Recall in Definition \ref{def:AutoonD}, $\calA^{K^v}_{\bar G,0}$ is defined as the $C$-valued continuous functions on $\bar G(\bbQ)\bs\bar G(\bbA^\infty)/K^v$ and are smooth for the $D_L^\times$-action. Hence $\calA^{K^v}_{\bar G,0}$ has a natural $E$-structure by looking at the subspace where functions are valued in $E$. Using this we can give a $E$-structure of $\tau_v $ which we denote by $\tau_{w,E}$, and define
\begin{align*}
    {\pi}_{c,E}:=(\tau_{w,E}\hat\ox_E H^1(j_!\pi_{\GM,n,E,*}\Omega^1_{\Sigma_{n}}))^{D_L^\times},
\end{align*}
which is a locally $\sigma$-analytic representation of $\GL_2(L)$. By definition, we see 
\begin{align*}
    {\pi}_{c,E}\hat\ox_E C\isom \pi_c
\end{align*}
and $(\pi_c)^{\Gal(\bar L/E)}\isom \pi_{c,E}$. Similarly one can define $E$-structures $\pi_{v,E}$, $\tilde\pi_{E}$ to $\pi_v$, $\tilde{\pi}$. For notation of $\pi_v$, $\tilde{\pi}$, see subsection \ref{cuspidal}.


\begin{theorem}\label{thm:admissible}
Let $\rho$ be as in Theorem \ref{main}, such that $\rho|_{\Gal(\bar F_w/F_w)}$ is de Rham of parallel Hodge--Tate weights $0,1$, and $\WD(\rho|_{\Gal(\bar F_w/F_w)})$ is irreducible. Write $\Pi_E(\rho):=\Hom_{\Gal_F}(\rho,\tilde{H}^1(K^v,E))$, and we assume $\Pi_E(\rho)^{\sigma\-\lan}\neq 0$. Then
$\tilde\pi_{E}$ is an admissible representation of $\GL_2(L)$.
\end{theorem}
The idea is as follows: we have shown that $\pi_c$ appears in $\tilde{H}^1(K^v,C)[\lambda_u]^{\Theta_{\Sen}=0}$, but note the coefficient is over $C$. One can show that  $\tilde{H}^1(K^v,C)[\lambda_u]^{\Theta_{\Sen}=0}\isom (\rho\ox\Pi_E(\rho)^{\sigma\-\lan}\hat \ox_{E,u}C)^{\Theta_{\Sen}=0}$, and $(\rho\ox_{E,u}C)^{\Gal(\bar L/E)}\inj (\rho\ox_{E,u}C)^{\Theta_{\Sen}=0}$ will gives a canonical $E$-structure on $\pi_c$ which is compatible with the $E$-structure of the Shimura curve.
\begin{proof}
Let $\lambda$ be the associated Hecke eigenvalue such that
\begin{align*}
    \tilde{H}^1(K^v,E)[\lambda]\isom \tilde{H}^1(K^v,E)[\rho].
\end{align*}
Then 
\begin{align*}
    \tilde{H}^1(K^v,E)[\lambda]\isom \Pi_E(\rho)\ox_E\rho.
\end{align*}
In the modular curve case, this was discussed in \cite[6.1.1]{Pan22}. Indeed, suppose $\rho:\Gal_F\to \GL_2(E)$ factors through the quotient $G_{F,S}$, where $S$ is a finite set of primes of $F$ and $G_{F,S}$ is the Galois group of the maximal unramified extension of $F$ outside of $S$. Let $G_0\subset G_{F,S}$ be the subgroup generated by the Frobenius element in $G_{F,S}$. By the Eichler--Shimura relation \cite[Remarque 7.1.1(2)]{Din17}, applied to $\rho_0=\rho|_{G_0}:G_0\to \GL_2(E)$, we know $\tilde{H}^1(K^v,E)[\lambda]$ is $\rho_0$-isotypic by the main theorem of \cite{BLR}. More precisely, there is a $G_0$-equivariant isomorphism 
\begin{align*}
    \tilde{H}^1(K^v,E)[\lambda]\isom \rho^{\oplus I}
\end{align*}
for some set $I$. Let $M=E^{\oplus I}$, then $\rho^{\oplus I}\isom \rho\ox_E M$. Applying $\Hom_{G_0}(\rho_0,-)$ to the above isomorphism, we get 
\begin{align*}
    \Hom_{G_0}(\rho_0,\tilde{H}^1(K^v,E)[\lambda])\isom \Hom_{G_0}(\rho_0,\rho_0)\ox_E M.
\end{align*}
By the Chebotarev density theorem, $G_0$ is dense in $G_{F,S}$. As $\rho$ is irreducible and finite dimensional, $\Hom_{G_0}(\rho_0,\rho_0)=E$. Hence 
\begin{align*}
    M\isom \Hom_{G_0}(\rho_0,\tilde{H}^1(K^v,E)[\lambda]).
\end{align*}
As $G_{F,S}$ acts continuously on $\tilde{H}^1(K^v,E)$, we deduce 
\begin{align*}
    \Hom_{G_0}(\rho_0,\tilde{H}^1(K^v,E)[\lambda])\isom \Hom_{G_{F,S}}(\rho,\tilde{H}^1(K^v,E)[\lambda]).
\end{align*}
Again by the Eichler--Shimura relation, 
\begin{align*}
    \Hom_{G_{F,S}}(\rho,\tilde{H}^1(K^v,E)[\lambda])\isom \Hom_{G_{F,S}}(\rho,\tilde{H}^1(K^v,E)).
\end{align*}
From this we deduce an abstract $G_{F,S}$-module isomorphism
\begin{align*}
    \tilde{H}^1(K^v,E)[\lambda]\isom \rho\ox_E\Hom_{G_{F}}(\rho,\tilde{H}^1(K^v,E)).
\end{align*}
However, the above map is induced by the  evaluation map
\begin{align*}
    \rho\ox_E\Hom_{G_{F}}(\rho,\tilde{H}^1(K^v,E))\to \tilde{H}^1(K^v,E),
\end{align*}
then $\Hom_{G_{F}}(\rho,\tilde{H}^1(K^v,E))$ inherits the topology and the $\GL_2(L)$-action from $\tilde{H}^1(K^v,E)$, making this isomorphism continous and $\GL_2(L)$-equivariantly.

By Theorem \ref{11}, there is a $\GL_2(L)\times \Gal_L$-equivariant and Hecke equivariant isomorphism
\begin{align*}
    H^1(K^v,C)\isom H^1(\fl,\calO_{K^v}).
\end{align*}
Therefore, there is a $\GL_2(L)\times \Gal_L$-equivariant isomorphism
\[
    H^1(K^v,C)[\lambda_u]\isom H^1(\fl,\calO_{K^v})[\lambda_u]
\]
with $\lambda_u=u\comp \lambda$, where $u:E\inj C$ is a fixed embedding. From the proof of Theorem \ref{kerI} (take $k=0$), we get a $\GL_2(L)\times \Gal_L$-equivariant isomorphism
\[
    H^1(K^v,C)[\lambda_u]^{\iota\-\lan,\tilde{\chi}_0,\Theta_{\Sen}=0}=H^1(\fl,\calO_{K^v}^{\iota\-\lan,(0,0)})[\lambda_u],
\]
where $\Theta_{\Sen}$ denote the arithmetic Sen operator. In the isomorphism $\tilde{H}^1(K^v,E)[\lambda]\isom \rho\ox\Pi_E(\rho)$, the action of $\Gal(\bar L/E)$ on $\rho\ox\Pi_E(\rho)$ given by the usual action on $\rho$ and trivial action on $\Pi_E(\rho)$. By \cite[Theorem 1.4]{DPSinf}, $\Pi_E(\rho)^{\sigma\-\lan}$ has trivial infinitesimal character. Therefore, we get an isomorphism
\begin{align*}
    \tilde H^1(K^v,C)[\lambda_u]^{\iota\-\lan,\tilde{\chi}_0}\isom \Pi_E(\rho)^{\sigma\-\lan}\hat\ox_E \rho\ox_{E,u}C.
\end{align*}
We remark that the action of $\Gal(\bar L/E)$ on $\rho\ox_{E,u}C$ is the diagonal action. Therefore, 
\begin{align*}
     H^1(\fl,\calO_{K^v}^{\iota\-\lan,(0,0)})[\lambda_u]\isom \tilde H^1(K^v,C)[\lambda_u]^{\iota\-\lan,\tilde{\chi}_0,\Theta_{\Sen}=0}\isom  (\Pi_E(\rho)^{\sigma\-\lan}\hat\ox_E (\rho\ox_{E,u}C))^{\Theta_{\Sen}=0}.
\end{align*}
Here we use
\begin{align*}
    (\Pi_E(\rho)^{\sigma\-\lan}\hat\ox_E (\rho\ox_{E,u}C))^{\Theta_{\Sen}=0}\isom \Pi_E(\rho)^{\sigma\-\lan}\hat\ox_E (\rho\ox_{E,u}C)^{\Theta_{\Sen}=0},
\end{align*}
and it is proved in Lemma \ref{lem:111}. As $\rho_w$ is of $\sigma$-Hodge--Tate weight $0,1$, $(\rho\ox_{E,u}C)^{\Theta_{\Sen}=0}\isom C$. Besides, by Theorem \ref{kerI}, we know 
\begin{align*}
    H^1(\fl,\calO_{K^v}^{\iota\-\lan,(0,0)})[\lambda_u]\isom \ker I^1[\lambda_u].
\end{align*}
In the case $\WD(\rho|_{\Gal(\bar F_w/F_w)})$ is irreducible, by \S \ref{cuspidal}  there is an exact sequence 
\begin{align*}
    0\to \pi_v\to \tilde{\pi}\to \ker I^1[\lambda_u]\to 0.
\end{align*}
Hence we have a $\GL_2(L)$-equivariant map 
\begin{align}\label{ex}
    0\to \pi_{v,E}\ox_{E}C\to \tilde\pi_{E}\hat\ox_E C\to \Pi_E(\rho)^{\sigma\-\lan}\hat\ox_E (\rho\ox_{E,u}C)^{\Theta_{\Sen}=0}\to 0.
\end{align}
We note the first map $\pi_{v,E}\ox_{E}C\to \tilde\pi_{E}\hat\ox_E C$ comes from the $\sigma$-Hodge filtration of $\rho_w$ and in particular it is defined over $E$. Indeed, the image of $\pi_v$ in $\tilde{\pi}$ lands in $\tilde{\pi}^{\sm}\isom \pi_v^{\oplus 2}$, which comes from the map $H^0(\calO\ox\Omega)[\lambda]\to \bbH^1(\DR^{\lalg})[\lambda]$ in Theorem \ref{thm:kerI1}. By the $p$-adic de Rham comparison theorem, this map is exactly $\Fil^1D_{\dR,\sigma}(\rho_w)\ox_{E,u}\pi_v\to D_{\dR,\sigma}(\rho_w)\ox_{E,u}\pi_v$. I claim the exact sequence (\ref{ex}) is also equivariant for the $\Gal(\bar L/E)$-action, where the action of $\Gal(\bar L/E)$ on $\pi_{v,E}$ and $\tilde\pi_E$ is trivial. Indeed, the Galois action on $ H^1(\fl,\calO_{K^v})$ comes from the $E$-structure of unitary Shimura curves. From \cite[Theorem III]{RZ96}, we know the $p$-adic uniformization of the supersingular locus of the unitary Shimura curves is defined over $L$, hence the $E$-structure of the Lubin--Tate space $\calM_{\LT,\infty}^{(0)}$ is compatible with the $E$-structure of the unitary Shimura curve. More precisely, we have an isomorphism $\pi_{\HT}^{-1}(\Omega)\isom \sqcup_{i\in I}\calM_{\LT,\infty}^{(0)}$ where $I$ is a finite index set, which induces a $\Gal(\bar L/E)$-equivariant map
\begin{align*}
    H^1(\fl,\bigoplus_{i\in I}j_!\calO_{\LT})\to H^1(\fl,\calO_{K^v}).
\end{align*}
Here $j:\Omega\inj \fl$ and $\calO_{\LT}=\pi_{\LT,\HT,*}\calO_{\calM_{\LT,\infty}^{(0)}}$. Let $\breve \calM_{\LT,\infty}$ be the $\breve L_0$-model of $\calM_{\LT,\infty}$ and same for $\breve \calM_{\Dr,\infty}$. In \cite[Theorem 7.2.3]{SW13}, they show 
\begin{align*}
    \breve \calM_{\LT,\infty}\isom \breve \calM_{\Dr,\infty}.
\end{align*}
From this we deduce the $E$-structure of $\calM_{\LT,\infty}^{(0)}$ is compatible with the $E$-structure of $\calM_{\Dr,\infty}^{(0)}$. This means the isomorphism
\begin{align*}
    H^1(\fl,j_!\calO_{\LT})\isom H^1(\fl,j_!\calO_{\Dr})
\end{align*}
is $\Gal(\bar L/E)$-equivariant, where $\calO_{\Dr}=\pi_{\LT,\HT,*}\calO_{\calM_{\Dr,\infty}^{(0)}}$. Finally, the $E$-structure $\tilde\pi_{E}$ of $\tilde\pi$ and the $E$-structure $\pi_{v,E}$ of ${\pi}_v$ is obtained from the $E$-structure of $\calM_{\Dr,\infty}/\varpi^{\bbZ}$. This shows the exact sequence 
\begin{align*}
    0\to \pi_{v,E}\ox_{E}C\to \tilde\pi_{E}\hat\ox_E C\to \Pi_E(\rho)^{\sigma\-\lan}\hat\ox_E (\rho\ox_{E,u}C)^{\Theta_{\Sen}=0}\to 0
\end{align*}
is $\Gal(\bar L/E)$-equivariant. Now taking $\Gal(\bar L/E)$-invariants on both sides, we get an $E$-linear $\GL_2(L)$-equivariant exact sequence
\begin{align*}
    0\to \pi_{v,E}\to \tilde\pi_{E}\to \Pi_E(\rho)^{\sigma\-\lan}\hat\ox_E (\rho\ox_{E,u}C)^{\Gal(\bar L/E)}\to 0.
\end{align*}
Indeed, $H^1(\Gal(\bar L/E),\pi_{v,E}\ox_{E}C)\to H^1(\Gal(\bar L/E),\tilde \pi_{E}\ox_{E}C)$ is injective because it comes from the base change of the injection $\pi_{v,E}\inj\tilde\pi_E$. Finally, as $\Pi_E(\rho)$ is a closed subrepresentation of the completed cohomology group $\tilde{H}^1(K^v,E)$ which is admissible by \cite[Theorem 0.1(i)]{Eme06}, we know $\Pi_E(\rho)^{\sigma\-\lan}$ is admissible. As $\pi_{v,E}$ is also admissible, we deduce $\tilde{\pi}_E$ is admissible by \cite[Lemme 2.1.1]{Bre19}. 
\end{proof}

\begin{lemma}\label{lem:111}
Let $K$ be a finite extension of $\bbQ_p$, with $C=\hat{\bar K}$. Let $W$ be a finite-dimensional semi-linear $C$-representation of $\Gal(\bar K/K)$, and let $V$ be a $C$-Banach space, with trivial semi-linear $\Gal(\bar K/K)$-action. Then the Sen operator on $V\ox_C W$ is given by $1\ox\Theta_{W}$, where $\Theta_{W}$ is the Sen operator on $W$.
\begin{proof}
First we recall the definition of Sen operator on $C$-Banach spaces, following \cite[Definition 6.1.9]{PanII}. Define $K(\zeta_{p^\infty})=\cup_n K(\zeta_{p^n})$ where $\zeta_{p^n}$ is a primitive $p^n$-th root of unity. Let $K_\infty\subset K(\zeta_{p^\infty})$ be the maximal $\bbZ_p$-extension inside $K(\zeta_{p^\infty})$. Let $W^K\subset W$ be the subspace of $\Gal(\bar K/K_\infty)$-invarint, $\Gamma:=\Gal(K_\infty/K)$-analytic vectors in $W$. Then by classical Sen theory \cite[Th\'eor\`eme 3.4]{BC16}, we know $W^K\ox_K C\isom W$. Since the action of $\Gal(\bar K/K)$ on $V$ is trivial, after choosing an orthonormal basis of $V$ which is fixed under the $\Gal(K_\infty/K)$-action, we obtain $(V\hat\ox_C W)^K\hat\ox_{K} C\isom V^K\hat\ox_{K} W^K\hat\ox_{K}C\isom V\hat\ox_C W$. The Sen operator on $V\hat \ox_C W$ is defined as the action of $1\in\Lie \Gamma$ on $(V\ox_CW)^K$, and extend $C$-linearly. As $(V\ox_CW)^K\isom V\ox_{K} W^K$, we see for any $X\in \Lie\Gamma$, $X(v\ox w)=v\ox X(w)$ for any $v\in V$ and $w\in W^K$. This shows the Sen operator on $V\ox_C W$ is given by $1\ox\Theta_{W}$, where $\Theta_{W}$ is the Sen operator on $W$.
\end{proof}
\end{lemma}

As translation functors preserve the admissibility \cite[Lemma 2.3.5]{JLH21}, we can deduce the admissibility for some equivariant line bundles on coverings of Drinfeld spaces. Let $\omega_{\fl_E}^{(a,b)}$ be the equivariant line bundle on $\fl_E$ associated to the character $\left(\begin{matrix} x&0\\y&z \end{matrix}\right)\mapsto \sigma(x)^a\sigma(z)^d$ with $x,y,z\in L$. Put $\omega_{\Sigma_n}^{(a,b)}:=\pi_{n,E}^*j_E^*\omega_{\fl}^{(a,b)}$.

\begin{corollary}
For any $(a,b)\in\bbZ^2$, the locally analytic $\GL_2(L)$-representation $H^1_c(\Sigma_n,\omega_{\Sigma_n}^{(a,b)})$ is admissible.
\begin{proof}
Since the $D_L^\times$-action on $H^1_c(\Sigma_n,\omega_{\Sigma_n}^{(a,b)})$ factors through the quotient $D_L^\times/1+\varpi^n\calO_{D_L}$ and the action of $\varpi$ on $H^1_c(\Sigma_n,\omega_{\Sigma_n}^{(a,b)})$ is trivial, we can decompose 
\begin{align*}
H^1_c(\Sigma_n,\omega_{\Sigma_n}^{(a,b)})\isom \bigoplus_{\sigma}\sigma\ox\Hom_{D_L^\times}(\sigma,H^1_c(\Sigma_n,\omega_{\Sigma_n}^{(a,b)})), 
\end{align*}
where $\sigma$ varies in the $E$-representations of $D_L^\times/1+\varpi^n\calO_{D_L}$ such that the action of $\varpi$ is trivial. There are only finitely many such $\sigma$. By Theorem \ref{thm:admissible}, $\Hom_{D_L^\times}(\sigma,H^1_c(\Sigma_n,\omega_{\Sigma_n}^{(0,0)}))$ is admissible. As $\Hom_{D_L^\times}(\sigma,H^1_c(\Sigma_n,\omega_{\Sigma_n}^{(-1,1)}))$ is a quotient of $\Hom_{D_L^\times}(\sigma,H^1_c(\Sigma_n,\omega_{\Sigma_n}^{(0,0)}))$, it is admissible. Then we may apply the translation functors to show $\Hom_{D_L^\times}(\sigma,H^1_c(\Sigma_n,\omega_{\Sigma_n}^{(a,b)}))$ is admissible for any $(a,b)\in\bbZ^2$. Finally, we deduce $H^1_c(\Sigma_n,\omega_{\Sigma_n}^{(a,b)})$ is admissible for any $(a,b)\in\bbZ^2$. 
\end{proof}
\end{corollary}

\begin{remark}
Beyond $\GL_2(\bbQ_p)$, in \cite{Patel_Schmidt_Strauch_2019, ardakov2023globalsectionsequivariantline}, they also give some coadmissiblity results for global sections of some equivariant line bundles on Drinfeld’s first \'etale covering of the $p$-adic upper half plane, based on explicit calculations of $\wideparen\calD$-modules.
\end{remark}

\subsection{Weil--Deligne representations and the locally analytic  $\Ext^1$ group}

Let $\lambda$ be a classical Hecke eigenvalue such that it appears in ${H}^1(K^v,E):=\dlim_{K_v}H^1_{\et}(\calX_{K^vK_v},E)$, and $\rho:\Gal(\bar F/F)\to \GL_2(E)$ the associated Galois representation, which is a $2$-dimensional continuous $E$-linear representation of $\Gal(\bar F/F)$, such that $\rho_w:=\rho|_{\Gal(\bar L/L)}$ is de Rham of parallel Hodge--Tate weights $\{0,1\}_{\sigma\in\Sigma}$. We assume that $\rho$ is absolutely irreducible.
\begin{itemize}
    \item Fix an embedding $\sigma:L\inj E$, and let $r_{w,E}$ be the underlying Weil--Deligne representation of $\rho_w$. We assume $r_{w,E}$ is irreducible.
    \item Recall $D_{\dR}(\rho_w)$ is a locally free $E\ox_{\bbQ_p}L$-module of rank $2$, which decomposes as 
    \[
        D_{\dR}(\rho_w)=\bigoplus_{\sigma\in\Sigma}D_{\dR,\sigma}(\rho_w)
    \]
    where $D_{\dR,\sigma}(\rho_w):=D_{\dR}(\rho_w)\ox_{E\ox_{\bbQ_p}L}E\ox_{L,\sigma}L$. Besides, we have a $(\vphi,N,\Gal(\bar L/L))$-module $D_{\pst}(\rho_w)$ over $E\ox_{L_0}L_0^{\ur}$, and a choice of a branch of $p$-adic logarithmic on $L^\times$ gives an isomorphism $\bar L\ox_{L_0^{\ur}}D_{\pst}(\rho_w)\isom \bar L\ox_LD_{\dR}(\rho_w)$. Suppose that $D_{\pst}(\rho_w)\isom D_{\st,L'}(\rho_w)\ox_{L_0'}L_0^{\ur}$ such that $\Gal(\bar L/L')$ acts trivially on $D_{\st,L'}(\rho_w)$, then by \cite[2.2]{Bre19} we see that
    \[
        D_{\pst}(\rho_w)_\sigma:=(D_{\st,L'}(\rho_w)\ox_{L_0'}L')^{\Gal(L'/L)}\ox_{E\ox_{\bbQ_p}L}E\ox_{L,\sigma}L
    \]
    is a $2$-dimensional $E$-linear representation of $\WD_L$, and there is a natural $E$-linear isomorphism $D_{\pst}(\rho_w)_\sigma\isom D_{\dR,\sigma}(\rho_w)$.
    \item By Theorem \ref{thm:pvc}, from $\lambda$ we can associate some locally $\iota$-analytic representation $\pi_{v}$, $\tilde{\pi}$ and $\pi_c$, which only depends on $r_{w,E}$, and has a natural $E$-model over $\pi_{v,E}$, $\tilde{\pi}_E$ and $\pi_{c,E}$. See the proof of Theorem \ref{thm:admissible} for details.
\end{itemize}
The inclusion $H^0(\Omega_{K^v}^{1,\sm})[\lambda]\subset \bbH^1(\DR^{\sm})[\lambda]$ is defined over $E$ as it is given by the $\sigma$-Hodge filtration of $\rho_w$. Therefore, the quotient 
\begin{align*}
    \ker I^1[\lambda]\isom \bbH^1(\DR)[\lambda]/\im(H^0(\Omega_{K^v}^{1,\sm})[\lambda]\subset \bbH^1(\DR^{\sm})[\lambda])
\end{align*}
has a model over $E$, which we denote by $\ker I^1_E[\lambda]$. It is an extension of $\pi_{c,E}$ by $\pi_{v,E}$. Define 
\begin{align*}
    \Ext^1_{\GL_2(L),\sigma}(\pi_{c,E},\pi_{v,E}):=\Ext^1_{D_\sigma(G,E)}(\pi_{v,E}',\pi_{c,E}'),
\end{align*}
where $D_\sigma(G,E)$ is the algebra of $E$-valued locally $\sigma$-analytic distribution on $\GL_2(L)$, and $(-)'$ denote the strong dual. Note that by Theorem \ref{thm:admissible}, $\pi_{c,E}$ is admissible.

\begin{proposition}\label{prop:FilExt1}
There is an $E$-linear isomorphism 
\begin{align*}
    D_{\dR,\sigma}(\rho_w)\aisom \Ext^1_{\GL_2(L)}(\pi_{c,E},\pi_{v,E})
\end{align*}
such that $\Fil^1D_{\dR,\sigma}(\rho_w)$ is mapped to the $E$-line generated by the extension class $\ker I^1_E[\lambda]$.
\begin{proof}
First of all, there is an $E$-linear isomorphism 
\begin{align*}
    D_{\dR,\sigma}(\rho_w)\ox_{E,u}C\aisom \Hom_{\GL_2(L)}(\pi_{v},\bbH^1(\DR^{\sm})).
\end{align*}
As $\Hom_{\GL_2(L)}(\pi_v,\bbH^1(\DR')[\lambda])=0$ by \cite[Proposition 2.5.2]{WConDrinfeld}, we deduce 
\begin{align*}
    D_{\dR,\sigma}(\rho_w)\ox_{E,u}C\aisom \Hom_{\GL_2(L)}(\pi_{v},\bbH^1(\DR)).
\end{align*}
Under this isomorphism, $\Fil^1D_{\dR,\sigma}(\rho_w)\ox_{E,u}C$ is mapped to the $C$-line generated by the map $H^0(\Omega_{K^v}^{1,\sm})[\lambda]\subset \bbH^1(\DR)[\lambda]$. But as we explained above, this map induces an isomorphism of $E$-vector spaces
\begin{align*}
    D_{\dR,\sigma}(\rho_w)\aisom \Hom_{\GL_2(L)}(\pi_{v,E},\tilde{\pi}_E).
\end{align*}
Finally, using a similar proof of Theorem \cite[Theorem 2.5.10]{WConDrinfeld}, we see the map 
\begin{align*}
    \Hom_{\GL_2(L)}(\pi_{v,E},\tilde{\pi}_E)\aisom \Ext^1_{\GL_2(L)}(\pi_{c,E},\pi_{v,E}),\alpha\mapsto \coker\alpha
\end{align*}
defines an isomorphism of $E$-vector spaces. Composing the above map with $D_{\dR,\sigma}(\rho_w)\aisom \Hom_{\GL_2(L)}(\pi_{v,E},\tilde{\pi}_E)$ gives the desired isomorphism.
\end{proof}
\end{proposition}

We aim to refine Proposition \ref{prop:FilExt1}. To achieve this, we need Weil--Deligne representations to rigidify the isomorphism. Recall that $\calM_{\LT,n}$ is the Lubin--Tate curve of level $n$ over $C$. Let $\calM_{\LT,n}^{(0)}\subset \calM_{\LT,n}$ be the subspace consisting of deformations with fixed height of quasi-isogeny to be $0$. By \cite{CN20}, \cite[Remark 5.2]{CGN23}, we get a $(\vphi,N,\Gal(\bar L/L))$-module $H^1_{\text{HK},c}(\calM_{\LT,n}^{(0)})$ over $L_0^{\ur}$, namely the compact supported Hyodo--Kato cohomology of $\calM_{\LT,n}^{(0)}$. By functoriality of the Hyodo--Kato cohomology groups, $H^1_{\text{HK},c}(\calM_{\LT,n}^{(0)})$ also carries an action of $(\GL_2(L)\times D_L^\times)^0$, which is a subgroup of $\GL_2(L)\times D_L^\times$:
\begin{align*}
    (\GL_2(L)\times D_L^\times)^0=\{(g,\check g)\in \GL_2(L)\times D_L^\times:v_p(\det g)+v_p(\Nrd(\check{g}))=0\}.
\end{align*}
Let $\calX_{K^v,n}$ be the unitary Shimura curve with level $K^vK_n$, where $K_n=\ker(\GL_2(\calO_L)\to \GL_2(\calO_L/\varpi^n))$. Let $\calX_{K^v,n}^{\ss}\subset \calX_{K^v,n}$ be the supersingular locus, which is the image of $\pi_{\HT}^{-1}(\Omega)$ under the projection map $\calX_{K^v}\to \calX_{K^v,n}$ (\cite[Lemma 3.2.3]{JLH21} for example). By \cite[Theorem III]{RZ96}, there is a $\GL_2(L)^0$-equivariant isomorphism 
\begin{align*}
    \calX_{K^v,n}^{\ss}\isom (\calM_{\LT,n}^{(0)}\times X_{K^v})/\calO_{D_L}^\times.
\end{align*}
where $X_{K^v}$ is the supersingular Igusa set discussed in \S \ref{padicuniformization}. The $\calO_{D_L}^\times$-space $X_{K^v}$ is isomorphic to finitely many copies of $\calO_{D_L}^\times$ and $\calA^{K^v}_{\bar G,0}$ is exactly the space of $C$-valued $\calO_{D_L}^\times$-smooth functions on $X_{K^v}$. From the Hyodo--Kato isomorphism, we get an embedding $H^1_{\text{HK},c}(\calM_{\LT,n}^{(0)})\inj H^1_{\dR,c}(\calM_{\LT,n}^{(0)})$. By \cite[Th\'eor\`eme 4.1]{CDN20}, $H^1_{\dR,c}(\calM_{\LT,n}^{(0)})$ is a smooth $\calO_{D_L}^\times$-representation. Therefore, the $\calO_{D_L}^\times$-action on $H^1_{\text{HK},c}(\calM_{\LT,n}^{(0)})$ is also smooth. From this, we get a Hecke equivariant isomorphism by \'etale descent
\begin{align*}
    H^1_{\text{HK},c}(\calX_{K^v,n}^{\ss})\hat\ox_{L_0^{\ur}}C\isom (\calA^{K^v}_{\bar G,0}\hat\ox_{L_0^{\ur}} H^1_{\text{HK},c}(\calM_{\LT,n}^{(0)})))^{\calO_{D_L}^\times}.
\end{align*}
Let $\lambda\in\sigma_0^{K^v}$ such that $\pi_v$ is supercuspidal. Then the natural map 
\begin{align*}
    \dlim_n H^1_{\text{HK},c}(\calX_{K^v,n}^{\ss})[\lambda]\to \dlim_n H^1_{\text{HK}}(\calX_{K^v,n})[\lambda]
\end{align*}
is an isomorphism of $(\vphi,N,\Gal(\bar L/L))$-modules. Indeed, after we base change the above map to $C$, from the Hyodo--Kato isomorphism we get a morphism 
\begin{align*}
    \dlim_n H^1_{\dR,c}(\calX_{K^v,n}^{\ss})[\lambda]\to \dlim_n H^1_{\dR}(\calX_{K^v,n})[\lambda]
\end{align*}
which is an isomorphism by Theorem \ref{thm:pvc}, combined with the isomorphism of the compactly supported de Rham cohomology of the Lubin--Tate tower and the Drinfeld tower \cite{CDN20}. In summary, we deduce 
\begin{align*}
    \dlim_n(\calA^{K^v}_{\bar G,0}[\lambda]\ox_{L_0^{\ur}} H^1_{\text{HK},c}(\calM_{\LT,n}^{(0)}))^{\calO_{D_L}^\times}\isom \dlim_n H^1_{\text{HK}}(\calX_{K^v,n})[\lambda].
\end{align*}
Recall $\rho$ is a 2-dimensional continuous absolutely irreducible $E$-linear representation of $\Gal(\bar F/F)$ associated to $\lambda$. Fix an embedding $\sigma:L\inj E$, and let $r_{w,E}$ be the underlying Weil--Deligne representation associated to $\rho|_{\Gal(\bar F_w/F_w)}$. Put $r_w:=r_{w,E}\ox_{E,u}C$. Define $H^1_{\text{HK},c}(\calM_{\LT,\infty}^{(0)}):=\dlim_n H^1_{\text{HK},c}(\calM_{\LT,n}^{(0)})$, and put $H^1_{\text{HK},c}(\calM_{\LT,\infty}):=\mathrm{ind}^{\GL_2(L)\times D_L^\times}_{(\GL_2(L)\times D_L^\times)^0}H^1_{\text{HK},c}(\calM_{\LT,\infty}^{(0)})$ where $\mathrm{ind}$ is the compact induction. By the semi-stable comparison theorem \cite{MR1705837} applied to $\dlim_{K_v}H^1_{\et}(\calX_{K^vK_v},E)$, we deduce $C\hat\ox_{L_0^{\ur}}\dlim_{K_v} H^1_{\text{HK}}(\calX_{K^vK_v})[\lambda]\isom r_w\ox_C\pi_v$. From this we see there is an isomorphism of $\GL_2(L)\times \WD_L$-representations
\[
    (\tau_v \hat\ox_{L_0^{\ur}}H^1_{\text{HK},c}(\calM_{\LT,\infty}))^{D_L^\times}\isom r_w\ox_C\pi_v 
\]
over $C$. Here $C\hat\ox_{L_0^{\ur}}H^1_{\text{HK},c}(\calM_{\LT,\infty})$ is equipped with the induced Weil--Deligne representation structure from the $(\vphi,N,\Gal(\bar L/L))$-action on $H^1_{\text{HK},c}(\calM_{\LT,\infty})$. In summary, we obtain the following result.
\begin{proposition}\label{prop:pvcw0}
The natural map 
\begin{align*}
    \dlim_n H^1_{\dR,c}(\calX_{K^v,n}^{\ss})[\lambda]\to \dlim_n H^1_{\dR}(\calX_{K^v,n})[\lambda]
\end{align*}
is an isomorphism. Moreover, there is an isomorphism of $\GL_2(L)\times\WD_L$-representations
\begin{align*}
    \dlim_n H^1_{\text{HK},c}(\calX_{K^v,n}^{\ss})[\lambda]\ox_{L_0^{\ur}} C\isom (\tau_v \hat\ox_{L_0^{\ur}}H^1_{\text{HK},c}(\calM_{\LT,\infty}))^{D_L^\times}\isom r_w\ox_C\pi_v.
\end{align*}
\end{proposition}

\begin{theorem}\label{thm:MAIN}
There exists an isomorphism of $E$-vector spaces
\[
    r_{w,E}\aisom \Ext^1_{\GL_2(L)}(\pi_{c,E},\pi_{v,E})
\]
which only depends on a choice of $\sigma$, and $r_{w,E}$. For any $2$-dimensional continuous absolutely irreducible $E$-linear representation $\rho:\Gal(\bar F/F)\to \GL_2(E)$, associated to a Hecke eigenvalue appears in $\dlim_{K_v}H^1_{\et}(\calX_{K^vK_v},E)$, with the following condition:
\begin{itemize}
    \item $\rho_w:=\rho|_{\Gal(\bar L/L)}$ is de Rham of parallel Hodge--Tate weight $0,1$,
    \item the underlying Weil--Deligne representation of $\rho_w$ is isomorphic to $r_{w,E}$,
\end{itemize} 
under the natural map 
\[
    D_{\dR,\sigma}(\rho_w)\isom D_{\pst}(\rho_w)_\sigma\aisom r_{w,E}\aisom \Ext^1_{\GL_2(L)}(\pi_{c,E},\pi_{v,E}),
\]
The $\sigma$-Hodge filtration $\Fil^1D_{\dR,\sigma}(\rho_w)$ corresponds to the $E$-line generated by the extension class $\ker I^1_E[\lambda]$.
\begin{proof}
Consider the following diagram
$$
\scriptsize
\begin{tikzcd}
r_{w,E} \arrow[d, Rightarrow, no head] \arrow[r, "\sim"]                                                                           & D_{\pst}(\rho_w)_\sigma \arrow[d, Rightarrow, no head] \arrow[r, "\sim"] \arrow[rd, "(*)", phantom]               & {D_{\dR,\sigma}(\rho_w)} \arrow[d, Rightarrow, no head]                   \\
{\Hom(\pi_v ,(\tau_v \hat\ox_{L_0^{\ur}}  H^1_{\text{HK},c}(\calM_{\LT,\infty}))^{D_L^\times})^{\Gal(\bar L/E)}} \arrow[r, "\sim"] \arrow[d, "\iota_{\text{HK}}"]        & {\Hom(\pi_v ,C\hat\ox_{L_0^{\ur}}H^1_{\text{HK}}(\calX_{K^v,\infty})[\lambda])^{\Gal(\bar L/E)}} \arrow[r, "\iota_{\text{HK}}"] \arrow[d, "\iota_{\text{HK}}"] & {\Hom(\pi_v ,H^1_{\dR}(\calX_{K^v,\infty})[\lambda])^{\Gal(\bar L/E)}} \arrow[d,"\alpha"] \\
{\Hom(\pi_v ,(\tau_v \hat\ox_{C}  H^1_{\dR,c}(\calM_{\LT,\infty}))^{D_L^\times})^{\Gal(\bar L/E)}} \arrow[r, "\beta_{\LT}"] \arrow[d,"\iota_{\LT,\Dr}"]  & {\Hom(\pi_v ,H^1_{\dR}(\calX_{K^v,\infty})[\lambda])^{\Gal(\bar L/E)}} \arrow[r,"\alpha"]                                              & {\Ext^1_{\GL_2(L)}(\pi_{c,E},\pi_{v,E})}\\
{\Hom(\pi_v ,(\tau_v \hat\ox_{C}  H^1_{\dR,c}(\calM_{\Dr,\infty}))^{D_L^\times})^{\Gal(\bar L/E)}} \arrow[rru,"\alpha_{\Dr}"'] \arrow[ru,"\beta_{\Dr}"]                                      
\end{tikzcd}
$$
Here, $\Hom(-,-)$ means $\Hom_{\GL_2(L)}(-,-)$ for simplicity, and $\tau_v =\calA_{\bar G,0}^{K^v}[\lambda]$ is a smooth irreducible representation of $D_L^\times$ corresponds to $\pi_v $ via the Jacquet--Langlands correspondence. The arrows labeled by $\iota_{\text{HK}}$ are the Hyodo--Kato isomorphisms. The map $\alpha$ is constructed in Proposition \ref{prop:FilExt1}, and the map $\alpha_{\Dr}$ is constructed similarly as in the proof of Proposition \cite[Theorem 2.5.10]{WConDrinfeld}. The map $\iota_{\LT,\Dr}$ is induced by the isomorphism between compactly supported de Rham cohomology of Lubin--Tate towers and Drinfeld towers. Finally, the map $\beta_{\LT}$ is construct in Proposition \ref{prop:pvcw0}, and the map $\beta_{\Dr}$ is the map $\ker\SS[\lambda]\to \bbH^1(\DR^{\sm})[\lambda]$ in Theorem \ref{thm:pvc}. All arrows in the above diagram are isomorphisms.

We claim that the above diagram commutes up to a non-zero constant in $E^\times$. Granting this, we see that the map 
\[
    D_{\dR,\sigma}(\rho_w)\ov{\sim}\ot D_{\pst}(\rho_w)_\sigma\isom r_{w,E}
\]
composed with 
\begin{align*}
    r_{w,E}&=\Hom(\pi_v ,(\tau_v \hat\ox_{L_0^{\ur}}  H^1_{\text{HK},c}(\calM_{\LT,\infty}))^{D_L^\times})^{\Gal(\bar L/E)}\\
    &\ov{\iota_{\text{HK}}}\to \Hom(\pi_v ,(\tau_v \hat\ox_{C}  H^1_{\dR,c}(\calM_{\LT,\infty}))^{D_L^\times})^{\Gal(\bar L/E)}\\
    &\ov{\iota_{\LT,\Dr}}\to \Hom(\pi_v ,(\tau_v \hat\ox_{C}  H^1_{\dR,c}(\calM_{\Dr,\infty}))^{D_L^\times})^{\Gal(\bar L/E)}\\
    &\ov{\alpha_{\Dr}}\to  \Ext^1_{\GL_2(L)}(\pi_{c,E},\pi_{v,E})
\end{align*}
is the same as the map 
\[
    D_{\dR,\sigma}(\rho_w)=\Hom(\pi_v ,H^1_{\dR}(\calX_{K^v,\infty})[\lambda])^{\Gal(\bar L/E)}\ov{\alpha}\to \Ext^1_{\GL_2(L)}(\pi_{c,E},\pi_{v,E})
\]
up to constants, hence we can conclude using Proposition \ref{prop:FilExt1}.

We show this diagram commutes. The top left corner square commutes up to constants as the arrows are equivariant for the action of $\WD_L$. The square labeled by $(*)$ commutes by the $p$-adic semi-stable comparison theorem \cite{MR1705837} applied to $\dlim_n H^1_{\et}(\calX_{K^v,n},E)[\lambda]$. We note the natural map $D_{\pst}(\rho_w)_\sigma\aisom D_{\dR,\sigma}(\rho_w)$ depends on the choice of a branch of $p$-adic logarithm $\log_\varpi$ and same for the Hyodo--Kato isomorphism 
\[
    \Hom(\pi_v ,(\tau_v \hat\ox_{L_0^{\ur}}  H^1_{\text{HK},c}(\calM_{\LT,\infty}))^{D_L^\times})^{\Gal(\bar L/E)}\ov{\iota_{\text{HK}}}\to \Hom(\pi_v ,H^1_{\dR}(\calX_{K^v,\infty})[\lambda])^{\Gal(\bar L/E)}.
\]
Here $\varpi$ is a uniformizer of $L$ we fixed throughout this paper and $\log_{\varpi}$ is a branch of $p$-adic logarithm such that $\log_\varpi(\varpi)=0$. Next, by functoriality of Hyodo--Kato isomorphism, the rest two square commutes. For the commutativity of the triangle with one edge labeled by $\iota_{\LT,\Dr}$, this follows from the explicit construction of $\iota_{\LT,\Dr}$ (see \cite{su2025rhamcohomologylubintatespaces} for more details) and $\beta_{\LT}$, $\beta_{\Dr}$. Finally, from the construction of $\alpha$, $\alpha_{\Dr}$, we see $\alpha_{\Dr}=\alpha\comp \beta_{\Dr}$.

Finally we note that by unraveling the definition of 
\[
    r_{w,E}\aisom \Ext^1_{\GL_2(L)}(\pi_{c,E},\pi_{v,E}),
\]
we see it only depends on $r_{w,E}$ and a choice of $\sigma$.
\end{proof}
\end{theorem}



\begin{remark}
The only place we use the assumption that $\rho_w$ is of parallel Hodge--Tate weight $0,1$ is to show the commutativity of the square $(*)$, where  we use the semi-stable (actually crystalline) comparison theorems for trivial coefficients. But this should not be a crucial assumption.
\end{remark}



Now we construct certain locally $\bbQ_p$-analytic representations in the completed cohomology group of unitary Shimura curves, such that they completely determine the original $p$-adic Galois representation if $\pi_v $ is supercuspidal. We rename $\ker I^1_E[\lambda]$ by $\pi_\sigma(\rho_w)$. Set 
\[
    \pi_1(\rho_w):=\bigoplus_{\pi_{v,E},\sigma\in\Sigma}\pi_\sigma(\rho_w),
\]
which is isomorphic to a locally $\bbQ_p$-analytic representation of $\GL_2(L)$ inside $\tilde{H}^1(K^v,E)^{\lan}[\lambda]$. The following theorem proves a special case of the conjecture (GAL) in \cite[2.3]{Bre19}.
\begin{theorem}\label{thm:gal}
Let $\rho:\Gal(\bar F/F)\to \GL_2(E)$ be as in Theorem \ref{thm:MAIN}. Then the continuous $E$-linear isomorphism class of the locally $\bbQ_p$-analytic $\GL_2(L)$-representation $\pi_1(\rho_w)$ depends only on $\rho_w$ and completely determines $\rho_w$.
\begin{proof}
By Theorem \ref{thm:locality}, we know $\pi_1(\rho_w)$ only depends on $\rho_w$. Next, by taking the locally $\sigma$-analytic subrepresentation of $\pi_1(\rho_w)$, we can recover $\pi_\sigma(\rho_w)$ for each $\sigma\in\Sigma$. This process also respects the $E$-linear automorphisms of $\pi_1(\rho_w)$. By \cite[Proposition 2.6.3]{WConDrinfeld}, together with Theorem \ref{thm:MAIN}, we know from the continuous isomorphism class of the $\GL_2(L)$-representation $\pi_\sigma(\rho_w)$ over $E$, we get a $E$-line inside $\Ext^1_{\GL_2(L)}(\pi_{c,E},\pi_{v,E})$, which corresponds to the $\sigma$-Hodge filtration inside $D_{\dR,\sigma}(\rho_w)$ via the isomorphism $D_{\dR,\sigma}(\rho_w)\aisom r_{w,E}\aisom \Ext^1_{\GL_2(L)}(\pi_{c,E},\pi_{v,E})$. This shows that we can recover the $\sigma$-Hodge filtration of $\rho_w$ for each $\sigma\in\Sigma$. As $r_{w,E}$ is irreducible (hence indecomposable), the Galois representation $\rho_w$ is completely determined by $r_{w,E}$ together with the $\sigma$-Hodge filtration of $\rho_w$ for each $\sigma\in\Sigma$ from the discussion in \cite[\S 3.1]{WConDrinfeld}. Therefore, the $E$-linear isomorphism class of the $\GL_2(L)$ representation $\pi_1(\rho_w)$ completely determines $\rho_w$.
\end{proof}
\end{theorem}

\bibliographystyle{amsalpha}
\bibliography{main.bib}

\end{document}